\documentclass[11pt]{amsart}

\usepackage{evearticle}

\makeatletter
\renewcommand*\env@matrix[1][\arraystretch]{
	\edef\arraystretch{#1}
	\hskip -\arraycolsep
	\let\@ifnextchar\new@ifnextchar
	\array{*\c@MaxMatrixCols c}}
\makeatother

\newcounter{jmstep}
\newcounter{jmstepanchor}
\renewcommand{\thejmstep}{\arabic{jmstep}}

\crefname{jmstep}{Step}{Steps}
\Crefname{jmstep}{Step}{Steps}
\newcommand{\step}[2]{
	\refstepcounter{jmstep}
	\stepcounter{jmstepanchor}
	\label{#1}
	\noindent\emph{Step \thejmstep. #2.}
}
\newcommand{\stepn}[2]{
	\refstepcounter{jmstep}
	\stepcounter{jmstepanchor}
	\label{#1}
	\noindent\emph{Step \thejmstep. #2.}
}
\AtBeginEnvironment{proof}{
	\setcounter{jmstep}{0}
}
\newcounter{jmsubstep}[jmstep]
\renewcommand{\thejmsubstep}{\thejmstep\alph{jmsubstep}}
\newcounter{jmsubstepanchor}

\crefname{jmsubstep}{Step}{Steps}
\Crefname{jmsubstep}{Step}{Steps}

\newcommand{\substepn}[2]{
	\refstepcounter{jmsubstep}
	\stepcounter{jmsubstepanchor}
	\label{#1}
	\noindent\emph{Step \thejmsubstep. #2.}
}

\newcommand{\one}{\mathbf{1}}

\begin{document}

\title[The conformally invariant metric on CLE$_4$ II: existence of geodesics]{The conformally invariant metric on CLE$_4$ II:\\ existence of geodesics}
\author[E.~Kammerer, K.~Kavvadias, J.~Miller, and Y.~Tian]{Emmanuel Kammerer, Konstantinos Kavvadias, Jason Miller, and Yi Tian}

\date{\today}

\maketitle
\begin{abstract}
We continue our study of the conformal loop ensemble (CLE) with parameter $\kappa=4$, the critical threshold at or below which the loops are simple and disjoint, touching neither each other nor the domain boundary. This paper is the second in a series of three establishing that the loops of a CLE$_4$ uniquely determine a conformally invariant, local, and geodesic metric such that the metric ball growth from the domain boundary coincides with the uniform exploration of Werner and Wu. In this second paper, we prove the existence of geodesics, showing that any geodesic between two loops is supported on the CLE$_4$ loops (off a set of Hausdorff dimension zero) and does not intersect the domain boundary. Along the way, we establish sharp quantitative estimates for the CLE$_4$ metric geometry, including exponential tail bounds for rectangle distances and multi-scale four-arm SLE$_4$ non-intersection bounds for metric balls.
\end{abstract}

\tableofcontents

\setlength{\parindent}{0pt}
\setlength{\parskip}{0.5\baselineskip plus 1pt minus 1pt}

\section{Introduction}
\subsection{Overview}

\begin{figure}[ht]
	\includegraphics[width=0.49\textwidth]{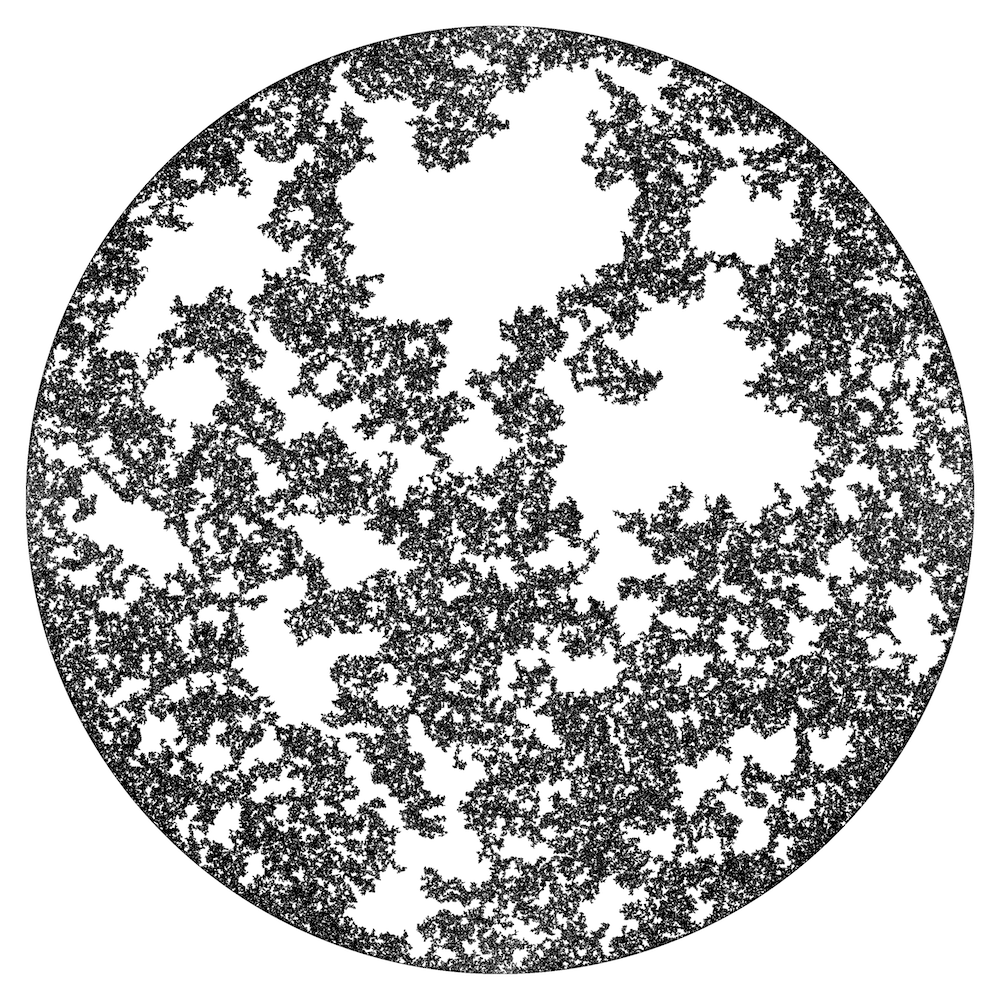}
	\includegraphics[width=0.49\textwidth]{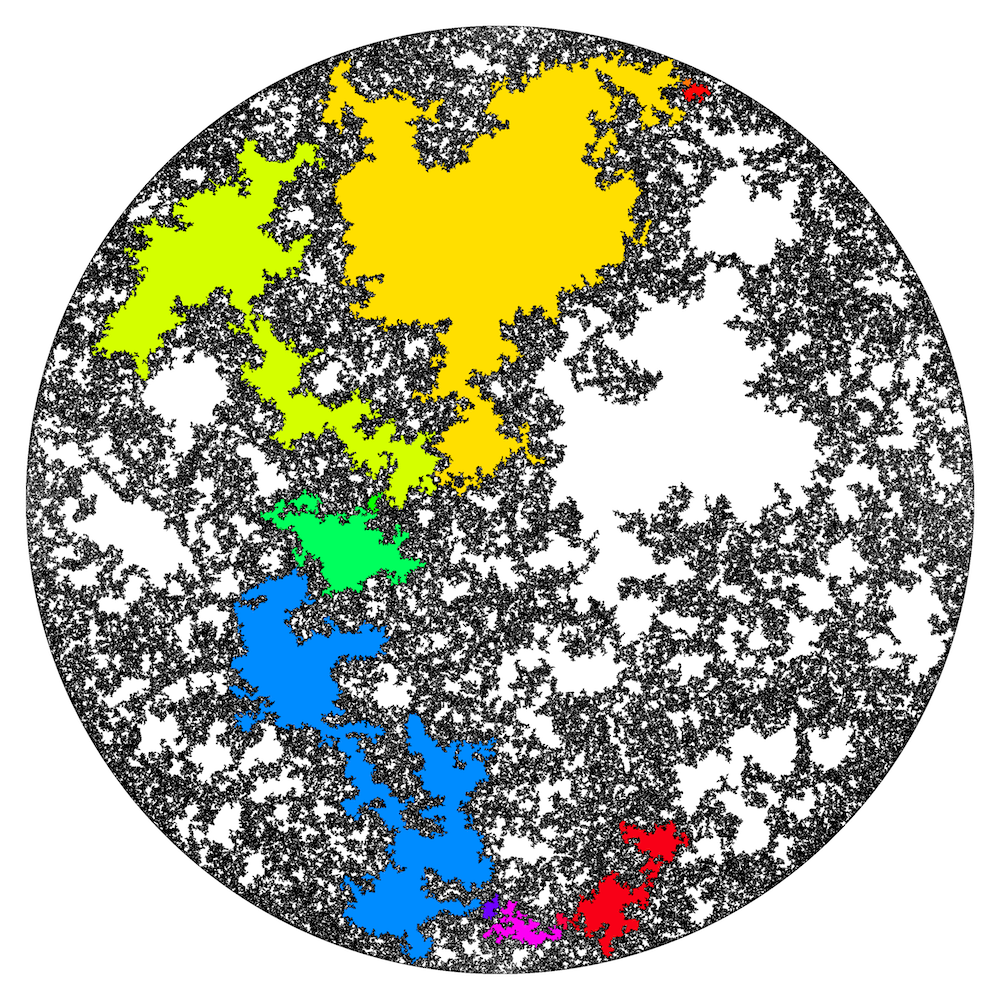}
	\caption{\label{fig:simulations}\textbf{Left:} A (discrete approximation of a) CLE$_4$ in the unit disk $\BD$. \textbf{Right:} The geodesic between the two red loops; loops along the geodesic are colored according to the length parameterization.}
\end{figure}

The \emph{conformal loop ensembles} (CLE$_\kappa$ for $\kappa \in (8/3, 8)$) are collections of random non-crossing loops within simply connected domains in $\BC$ \cite{TreeCLE,CLE}. They are conformally invariant, meaning their laws are preserved under conformal automorphisms of the domain. CLEs are the natural loop variants of the Schramm--Loewner evolution (SLE) \cite{s2000sle}. SLEs and CLEs are conjectured (and in several cases rigorously shown) to describe the scaling limits of interfaces in critical statistical mechanics models on planar lattices \cite{s2001percolation,lsw2004lerw,ss2009dgff,s2010ising} as well as on random planar maps \cite{s2016hc,lsw2017schnyder,gkmw2018active,kmsw2019bipolar,gm2021saw,gm2021percolation}.

In the same way as for SLE, the value $\kappa=4$ is critical for CLE. For $\kappa \in (8/3,4]$, the loops are simple and disjoint, touching neither each other nor the domain boundary. On the other hand, for $\kappa \in (4,8)$ the loops are self-touching and can touch each other and the domain boundary \cite{rs2005basic}. In this paper, we focus on the critical case $\kappa=4$.  For $\kappa \in (4,8)$ it is therefore possible to define a metric on the set of loops of a $\CLE_\kappa$ by considering the graph whose vertices are the loops and two distinct loops are adjacent if they intersect.  This metric is conformally invariant since $\CLE_\kappa$ is conformally invariant.  This is not naively possible for $\kappa = 4$ since the loops of a $\CLE_4$ are disjoint. Nevertheless, in our first paper \cite{kkmt2026cle4_part1}, we constructed a conformally invariant metric on the loops of a CLE$_4$ by taking subsequential scaling limits as $\kappa\downarrow 4$ of the renormalized graph metric on the CLE$_\kappa$.

The program of constructing the CLE$_4$ metric was initiated by Werner and Wu in \cite{CoInCLEExpl}, where the distance from any loop in CLE$_4$ to the domain boundary was defined via a Poissonian construction starting from the domain boundary. Equivalently, Werner and Wu in \cite{CoInCLEExpl} constructed metric balls from the domain boundary, where a loop of CLE$_4$ is discovered when it is contained in the metric ball. They referred to this exploration as the \emph{uniform exploration} (cf.~also \cite{LevelLineGFFI,TVSGFF}). This terminology arises because, as one grows the uniform exploration from the domain boundary, the CLE$_4$ loops are discovered in a Poissonian manner, with each new loop being rooted at a point chosen according to harmonic measure on the boundary of the unexplored region (seen from the target point of the exploration) just before the loop is discovered. It was not proved in \cite{CoInCLEExpl} that the uniform exploration corresponds to the metric ball growth from the boundary for a metric defined on the entire CLE$_4$, nor was it shown that such a metric would be determined by the CLE$_4$.

We proved in our first paper \cite{kkmt2026cle4_part1} that the growth process of metric balls from the domain boundary for our subsequential limiting metric indeed has the law of the uniform exploration. The aim of this second paper is to establish the existence of geodesics (see \Cref{fig:simulations}) and analyze their properties; these properties will play a crucial role in establishing the uniqueness of the metric in the third paper of the series \cite{kkmt2026cle4_part3}.

We note that a CLE$_4$ metric (which coincides with ours by the uniqueness proved in \cite{kkmt2026cle4_part3}) was previously constructed by Sheffield, Watson, and Wu in unpublished work, following an approach completely distinct from ours. In contrast to their indirect definition of the metric, we showed in the first paper that the metric arises naturally as a subsequential scaling limit as $\kappa \downarrow 4$ of the renormalized graph metric on the CLE$_\kappa$. Furthermore, they did not investigate the existence of geodesics, which is not automatically guaranteed by the subsequential limit as $\kappa \downarrow 4$: although a geodesic exists between any two loops for the graph metric on CLE$_\kappa$ for $\kappa>4$, one must rule out the possibility that the loops along this geodesic disappear in the limit as $\kappa \downarrow 4$; this is a subtle question due to the topology with respect to which the limit is taken. The results of this paper actually hold for any metric on CLE$_4$ satisfying a certain set of natural axioms. This paper can therefore be read independently of the first one, and the results proved here can be applied to other approximation schemes for the CLE$_4$ metric.

\subsection{Geodesics in random metric spaces}
\label{subsec:geodesic_background}
We provide a brief background on the study of geodesics in random metric spaces.

In many settings, the existence of geodesics is a direct consequence of the convergence (possibly along a subsequence) of geodesic metric spaces in the Gromov--Hausdorff or Gromov--Hausdorff--Prokhorov topology. This is the case, for example, for the Brownian continuum random tree \cite{Ald91}, the Brownian sphere (for which tightness was established in \cite{LG07}), and the scaling limits of random planar maps with large faces \cite{LGM11}. Similarly, for the Liouville quantum gravity (LQG) metric, the existence of geodesics in any subsequential limit of Liouville first passage percolation is a direct byproduct of tightness \cite{WeakLQGMet}. However, in our setting, since the subsequential limit is constructed with respect to a different topology, the existence of geodesics is not automatic and has to be proved directly.

More generally, the study of geodesics has become a central topic in the field of random geometry. In the Brownian sphere, detailed structural properties of geodesics and of geodesic trees directed toward typical points were established by Le Gall \cite{GeoBMap}. Beyond existence and uniqueness, a recurring phenomenon across planar random geometry is the \emph{confluence} (or coalescence) of geodesics: geodesics starting from different locations merge before reaching their destination. Geodesic confluence has been established for LQG metrics \cite{ConfLQG}, and the geometric nature of LQG geodesics was shown to be distinct from that of SLE curves \cite{GeoLQGnSLE}. Confluence and geodesic network structures also play a central role in the Kardar--Parisi--Zhang (KPZ) universality class, notably in the directed landscape \cite{DiLand, 27NetDiLand}. Conformally covariant metrics and their geodesics have also been studied on CLE carpets \cite{GeoCLECarp} and non-simple CLE gaskets \cite{ExUniCoCoGeoMetNonsimCLEGas}.

Furthermore, properties of geodesics were a crucial tool in establishing the uniqueness of the Brownian sphere \cite{LG13, Mie13} and of stable carpets and gaskets \cite{StabCarGas}. They were also useful in the proof of uniqueness of the LQG metric \cite{ExUniLQG}. In a similar way, the existence and structural properties of geodesics established in this paper will play a key role in proving the uniqueness of the CLE$_4$ metric in our next paper \cite{kkmt2026cle4_part3}.

\subsection{Setup}\label{subsec:setup}

We begin by introducing some notation for CLE$_4$ and associated objects. Let $U \subsetneq \BC$ be a simply connected domain. Let $\Gamma_U$ be a non-nested CLE$_4$ in $U$. We shall write $\Upsilon_U$ for the carpet of $\Gamma_U$, which is the set of points of $\overline{U}$ that are not surrounded by a loop of $\Gamma_U$ (we choose to include $\partial U$ in $\Upsilon_U$ so that the paths that we will define later can start from $\partial U$). 

For each open subset $V \subset U$, we shall write 
\begin{equation}\label{eq:def-V-star}
    V^\star \defeq V \setminus \overline{\bigcup_{\SCL \in \Gamma_U : \SCL \not\subset V} \mathop{\mathrm{int}}(\SCL)},
\end{equation}
where $\mathrm{int}(\SCL)$ is the open set encircled by the loop $\SCL$. Intuitively, $V^\star$ is obtained from $V$ by removing the regions encircled by the loops that are not contained in $V$. For $V \subset \overline{U}$, we write $\Gamma_U\vert_{V} \defeq \{\SCL \in \Gamma_U : \SCL \subset V\}$.

For every deterministic point $z \in U$, a.s.\ there is a unique loop of $\Gamma_U$ surrounding $z$; we denote it by $\SCL(z)$.

Our main object of interest is a metric on the set of loops of CLE$_4$. Since the underlying set is not a subset of $\BC$, the associated balls are not subsets of $\BC$. Still, we can define the metric balls in terms of subsets of $\BC$ as follows. Let $D$ be a metric on $\Gamma_U$. For all $\SCL \in \Gamma_U$ and $t \ge 0$, we denote by
\begin{equation*}
    \SCB_t(\SCL; D) \quad \text{(resp.\ } \SCB_t^-(\SCL; D)\text{)},
\end{equation*}
or simply $\SCB_t(\SCL)$ (resp.\ $\SCB_t^-(\SCL)$), when there is no danger of confusion, the closure of the union of the domains surrounded by $\SCL^\prime$ for $\SCL^\prime \in \Gamma_U$ with $D(\SCL, \SCL^\prime) \le t$ (resp.\ $D(\SCL, \SCL^\prime) < t$). Note that $\SCB_0^-(\SCL; D) = \emptyset$. Note that we have $\SCB_t(\SCL; D) = \SCB_t^-(\SCL; D)$ unless there exists $\SCL^\prime \in \Gamma_U$ with $D(\SCL, \SCL^\prime) = t$. 

Moreover, we can make sense of the distance to a subset of $\overline{U}$ in the following way. Let $A, B \subset \overline U$. Then, we write
\begin{itemize}
    \item $D(\SCL, A) \defeq \inf\{t \ge 0 : \SCB_t(\SCL; D) \cap A \neq \emptyset\}$ for all $\SCL \in \Gamma_U$; 
    \item for all $t \ge 0$, $\SCB_t(A; D)$ (resp.\ $\SCB_t^-(A; D)$) for the closure of the union of $A$ and the domains surrounded by $\SCL$ for $\SCL \in \Gamma_U$ with $D(\SCL, A) \le t$ (resp.\ $D(\SCL, A) < t$), except $\SCB_0^-(A; D)\defeq \emptyset$; 
    \item $D(A, B) \defeq \inf\{t \ge 0 : \SCB_t(A; D) \cap B \neq \emptyset\}$.
\end{itemize}
Here and below, $\inf\emptyset \defeq \infty$. Note that $D(A,B) = D(B,A)$ (since one can see that $D(A, B)$ is the limit as $\varepsilon \downarrow 0$ of the infimum of $D(\SCL, \SCL')$ over pairs of loops $\SCL, \SCL' \in \Gamma_U$ such that $\mathrm{dist}(A, \mathrm{int}(\SCL))\le \varepsilon$ and $\mathrm{dist}(B, \mathrm{int}(\SCL'))\le \varepsilon$). We also write $D(A, y)\defeq D(A, \{y\})$ for all $y \in \overline{U}$.

Let $P \colon [0, 1] \to \Upsilon_U$ be a continuous path. We say that $P$ is \emph{admissible} if $\bigl(\bigcup_{\SCL \in \Gamma_U} \SCL\bigr) \cap P([0,1])$ is dense in $P([0,1])$. Given $A, B \subset \overline{U}$, we say that $P$ \emph{connects} $A$ and $B$ if $P(0) \in A$ and $P(1) \in B$. We shall refer to
\begin{equation*}
    \len(P; D) \defeq \sup_{\SCL_0, \SCL_1, \dots, \SCL_n} \sum_{j = 1}^n D(\SCL_{j - 1}, \SCL_j)
\end{equation*}
as the \emph{$D$-length} of $P$, where $\SCL_0, \SCL_1, \dots, \SCL_n$ range over all sequences of (not necessarily distinct) loops of $\Gamma_U$ such that $n \ge 1$ and $P(t_j) \in \SCL_j$ for some $0 \le t_0 \le t_1 \le \cdots \le t_n \le 1$. Let $A, B \subset X \subset \overline U$. Then we shall write 
\begin{equation*}
    D(A, B; X) \defeq \inf_{P} \len(P; D),
\end{equation*}
where $P$ ranges over all admissible paths with $P([0,1]) \subset X$, $P(0) \in A$ and $P(1) \in B$ (with the convention that $\inf\emptyset = \infty$); note that admissible paths take values in $\Upsilon_U$, so the constraint is $P([0,1]) \subset X \cap \Upsilon_U$. The function $D(\bullet,\bullet;X)$ is called the internal distance on $X$.

We shall write $U_{\BQ} \defeq U \cap \BQ^2$. By abuse of notation, we shall also denote by $D$ the mapping
\begin{equation*}
    D \colon U_{\BQ} \times U_{\BQ} \to \BR \colon (x, y) \mapsto D(\SCL(x), \SCL(y)). 
\end{equation*}
This is only a pseudometric, since it vanishes whenever $\SCL(x) = \SCL(y)$. We shall equip $\BR^{U_{\BQ} \times U_{\BQ}}$ with the product topology. Note that the space $\BR^{U_{\BQ} \times U_{\BQ}}$ is a Polish space.

\subsection{Main results}

In the first paper of this series \cite{kkmt2026cle4_part1}, we proved that subsequential scaling limits of the renormalized graph distances on CLE$_\kappa$ give rise to a metric on CLE$_4$ which satisfies the following properties.

\begin{definition}\label{def:weak_axioms}
	We define a \emph{weak geodesic CLE$_4$ metric coupling} to be a family of couplings
	\begin{equation*}
		D= \left\{(\Gamma_U, D_{\Gamma_U}^U)\right\}_U
	\end{equation*}
	where $U \subsetneq \BC$ ranges over all simply connected domains, and for each $U$, $\Gamma_U$ is a non-nested CLE$_4$ in $U$, such that the following conditions are satisfied for every such $U$ and all deterministic choices of the auxiliary data ($V$, $I$, $\phi$, arcs) below:
	\begin{enumerate}[label=(\Roman*), ref=\Roman*]
		\item\label{it:weak_axiom_geodesic} {\bfseries (Weak geodesic metric)} $D_{\Gamma_U}^U$ is a.s.\ a metric on $\Gamma_U$ satisfying the following conditions:
		\begin{enumerate}[label=(\roman*), ref=\roman*]
			\item\label{it:weak_axiom_geodesic_0} Let $\SCL \in \Gamma_U$ and $t \ge 0$. Then the metric ball $\SCB_t(\SCL; D_{\Gamma_U}^U)$ is connected.
			\item\label{it:weak_axiom_geodesic_1} Let $\SCL \in \Gamma_U$ and $t \ge 0$. Let $\phi\colon U \to \BD$ be a conformal mapping. Then
			\begin{equation*}
				\lim_{t^\prime \downarrow t} \phi (\SCB_{t^\prime}(\SCL; D_{\Gamma_U}^U) )= \phi (\SCB_t(\SCL; D_{\Gamma_U}^U)) \quad \text{and} \quad \lim_{t^\prime \uparrow t} \phi (\SCB_{t^\prime}(\SCL; D_{\Gamma_U}^U)) = \phi (\SCB_t^-(\SCL; D_{\Gamma_U}^U))
			\end{equation*}
			with respect to the Hausdorff metric, the second limit being required only for $t > 0$.
			\item\label{it:weak_axiom_geodesic_2} Let $\SCL_1, \SCL_2 \in \Gamma_U$ and $t \ge 0$. Suppose that $\SCB_t(\SCL_1; D_{\Gamma_U}^U) \cap \SCL_2 = \emptyset$. Then
			\begin{equation*}
				D_{\Gamma_U}^U(\SCL_1, \SCL_2) = t + D_{\Gamma_U}^U(\SCB_t(\SCL_1; D_{\Gamma_U}^U), \SCL_2). 
			\end{equation*}
			\item\label{it:weak_axiom_geodesic_4} Let $\SCL_1, \SCL_2 \in \Gamma_U$. Then the closed set $K_{\SCL_1, \SCL_2}$ defined as the closure of $\{ z \in \overline U : D_{\Gamma_U}^U(\SCL_1, z)+ D_{\Gamma_U}^U(\SCL_2,z)= D_{\Gamma_U}^U(\SCL_1, \SCL_2) \}$ has a connected component containing $\SCL_1 \cup \SCL_2$.

		\end{enumerate}
		Items~\eqref{it:weak_axiom_geodesic_0}--\eqref{it:weak_axiom_geodesic_4} are also required with a deterministic connected arc (resp.\ two disjoint such arcs) of $\partial U$ in place of $\SCL$ or $\SCL_1$ (resp.\ $\SCL_1$ and $\SCL_2$).
		\item\label{it:weak_axiom_locality} {\bfseries (Locality)}
		Let $V \subset U$ be a deterministic simply connected subdomain. Let $I \subset \partial U$ be a deterministic connected arc. Let $\{V_j\}_j$ be the connected components of $V^\star$. Then there exists a coupling $(\Gamma_U, D_{\Gamma_U}^U, \{D_{\Gamma_U|_{V_j}}^{V_j}\}_j)$ such that the following hold: 
			\begin{itemize}
				\item Conditionally on the $\sigma$-algebra generated by
				\begin{multline}\label{eq:weak_axiom_locality}
					\bigl\{\SCL \in \Gamma_U : \SCL \not\subset V\bigr\}, \quad \bigl\{\SCB_t(\SCL; D_{\Gamma_U}^U) : \SCL \in \Gamma_U, \ \SCL \not\subset V, \ t \in [0, D_{\Gamma_U}^U(\SCL, \partial V^\star)]\bigr\}, \\
					\text{and} \quad \bigl\{\SCB_t(I; D_{\Gamma_U}^U) : t \in [0, D_{\Gamma_U}^U(I, \partial V^\star)]\bigr\}
				\end{multline}
				the $(\Gamma_U|_{V_j}, D_{\Gamma_U|_{V_j}}^{V_j})$'s are independent and their conditional laws are those of $(\Gamma_{V_j}, D_{\Gamma_{V_j}}^{V_j})$, respectively. 
				\item For all $j$, for all $\SCL_1, \SCL_2 \in \Gamma_U|_{V_j}$, we have $D_{\Gamma_U}^U(\SCL_1, \SCL_2)\le D_{\Gamma_U|_{V_j}}^{V_j}(\SCL_1, \SCL_2)$.
				\item For all $j$, for all $\SCL \in \Gamma_U|_{V_j}, \ t \in [0, D_{\Gamma_U}^U(\SCL, \partial V_j)]$, we have $\SCB_t(\SCL; D_{\Gamma_U}^U) = \SCB_t(\SCL; D_{\Gamma_U|_{V_j}}^{V_j})$.
				
		\end{itemize}

		\item\label{it:weak_axiom_conformal_invariance} {\bfseries (Conformal invariance)} Let $\phi \colon U \to \phi(U)$ be a deterministic conformal mapping. Then,
		\begin{equation*}
				\left(\phi(\Gamma_U), \left( D_{\Gamma_U}^U\left(\phi^{-1}(\SCL_1), \phi^{-1}(\SCL_2) \right) \right)_{\SCL_1, \SCL_2 \in \phi(\Gamma_U)} \right)
				\overset{(\mathrm{d})}{=} \left( \Gamma_{\phi(U)}, \left(D_{\Gamma_{\phi(U)}}^{\phi(U)}(\SCL_1, \SCL_2) \right)_{\SCL_1, \SCL_2 \in \Gamma_{\phi(U)}} \right).
		\end{equation*}
		Here both sides are viewed as random elements of the space of pairs consisting of a loop ensemble in $\phi(U)$ and a symmetric function on its pairs of loops; equivalently, by the reformulation of \Cref{subsec:setup}, as random elements of $\BR^{\phi(U)_\BQ \times \phi(U)_\BQ}$ together with the ensemble.

		\item\label{it:weak_axiom_uniform_exploration} {\bfseries (Uniform exploration)} The collection $\{(\SCL, D_{\Gamma_U}^U(\SCL, \partial U))\}_{\SCL \in \Gamma_U}$ has the law of a uniform exploration of $\Gamma_U$ (cf.~\cite{CoInCLEExpl}; see \Cref{subsec:labeled_cle_4}). That is, the pair consisting of the ensemble and its labels has the law of the labeled CLE$_4$ in $U$, namely of a CLE$_4$ decorated by its uniform exploration.
	\end{enumerate}
\end{definition}
Note that Axiom~\eqref{it:weak_axiom_geodesic}\eqref{it:weak_axiom_geodesic_2} implies that
\begin{equation*}
	D_{\Gamma_U}^U(\SCL_1, \SCL_2) = \inf\{t \ge 0 : \SCB_t(\SCL_1; D_{\Gamma_U}^U) \cap \SCL_2 \neq \emptyset\}, \quad \forall \SCL_1, \SCL_2 \in \Gamma_U. 
\end{equation*}
In particular, $D_{\Gamma_U}^U(\SCL_1, z) = D_{\Gamma_U}^U(\SCL_1, \SCL_2)$ for every $z \in \SCL_2$: here $\le$ holds because $\SCB_t(\SCL_1; D_{\Gamma_U}^U)$ contains the domain surrounded by $\SCL_2$ as soon as $t \ge D_{\Gamma_U}^U(\SCL_1,\SCL_2)$, and $\ge$ because $z \in \SCB_t(\SCL_1; D_{\Gamma_U}^U)$ forces $\SCB_t(\SCL_1; D_{\Gamma_U}^U) \cap \SCL_2 \neq \emptyset$. Hence $D_{\Gamma_U}^U(\SCL_1, \cdot)$ is constant on each loop, and $\SCL_1 \cup \SCL_2 \subset K_{\SCL_1,\SCL_2}$ in Axiom~\eqref{it:weak_axiom_geodesic}\eqref{it:weak_axiom_geodesic_4}.
Moreover, by conformal invariance, the laws of all the couplings $\left( \Gamma_U, D_{\Gamma_U}^U \right)$ are determined by the law of $\left( \Gamma_{\BD}, D_{\Gamma_{\BD}}^{\BD} \right)$. The arc $I$ enters Axiom~\eqref{it:weak_axiom_locality} only through the conditioning, and the balls grown from $I$ are included there because the applications in \Cref{sec:distances_across_rectangles} and \Cref{sec:vertical_vs_horizontal} condition on balls grown from boundary arcs as well as on balls grown from loops. In that axiom, the exterior balls are stopped at $D_{\Gamma_U}^U(\bullet, \partial V^\star)$ while the third bullet matches interior balls up to $D_{\Gamma_U}^U(\SCL, \partial V_j)$; the two thresholds concern different families of loops (those not contained in $V$, and those contained in $V_j$), which is why they differ. 
The main result of this paper is that every weak geodesic CLE$_4$ metric coupling is a geodesic CLE$_4$ metric coupling in the sense of the following definition.
\begin{definition}\label{def:axioms}
	We define a \emph{geodesic CLE$_4$ metric coupling} to be a family of couplings
	\begin{equation*}
	 	D = \left\{\left( \Gamma_U, D_{\Gamma_U}^U \right)\right\}_U
	\end{equation*}
	where $U \subsetneq \BC$ ranges over all simply connected domains and $\Gamma_U$ is a non-nested CLE$_4$ in $U$, such that:
	\begin{enumerate}[label=(\Roman*), ref=\Roman*]
		\item\label{it:axiom_geodesic} {\bfseries (Geodesic metric)} $D_{\Gamma_U}^U$ is a.s.\ a metric on $\Gamma_U$ such that the following hold.
			\begin{enumerate}[label=(\roman*), ref=\roman*]
				\item\label{it:axiom_geodesic_path} For each $\SCL_1, \SCL_2 \in \Gamma_U$, there exists an admissible path $P$ (which we refer to as a \emph{$D_{\Gamma_U}^U$-geodesic}) connecting $\SCL_1$ and $\SCL_2$ such that $D_{\Gamma_U}^U(\SCL_1, \SCL_2) = \len(P; D_{\Gamma_U}^U)$. The same is true with a deterministic connected arc (resp.\ two deterministic and disjoint connected arcs) of $\partial U$ in place of $\SCL_1$ (resp.\ $\SCL_1$ and $\SCL_2$).
				\item\label{it:axiom_geodesic_cadlag} Let $\phi\colon U \to \BD$ be a conformal mapping. For every $\SCL \in \Gamma_U$, the process $(\phi(\SCB_t(\SCL; D^U_{\Gamma_U}){)})_{t\ge 0}$ is c\`adl\`ag for the Hausdorff distance, and the same is true with a deterministic connected arc of $\partial U$ in place of $\SCL$.
				\item\label{it:axiom_geodesic_twoball} For all $t, s\ge 0$, if $\SCB_t(\SCL_1; D^U_{\Gamma_U}) \cap \SCB_s(\SCL_2; D^U_{\Gamma_U}) \neq \emptyset$, then $D^U_{\Gamma_U} (\SCL_1, \SCL_2) \le t+s$. The same is true with a deterministic connected arc (resp.\ two deterministic and disjoint connected arcs) of $\partial U$ in place of $\SCL_1$ (resp.\ $\SCL_1$ and $\SCL_2$).
			\end{enumerate}
			Note that $\len(P; D_{\Gamma_U}^U) \ge D_{\Gamma_U}^U(\SCL_1,\SCL_2)$ for every admissible path $P$ connecting $\SCL_1$ and $\SCL_2$, by the triangle inequality along the chains in the definition of $\len$; a path attaining equality is thus 
			what the term geodesic refers to. 
		\item\label{it:axiom_locality} {\bfseries (Locality)} Let $V \subset U$ be a deterministic simply connected subdomain. Let $I \subset \partial U$ be a deterministic connected arc. Let $\{V_j\}_j$ be the connected components of $V^\star$. Then, conditionally on
				\begin{multline}\label{eq:axiom_locality}
					\bigl\{\SCL \in \Gamma_U : \SCL \not\subset V\bigr\}, \quad \bigl\{\SCB_t(\SCL; D_{\Gamma_U}^U) : \SCL \in \Gamma_U, \ \SCL \not\subset V, \ t \in [0, D_{\Gamma_U}^U(\SCL, \partial V^\star)]\bigr\}, \\
					\text{and} \quad \bigl\{\SCB_t(I; D_{\Gamma_U}^U) : t \in [0, D_{\Gamma_U}^U(I, \partial V^\star)]\bigr\},
				\end{multline}
			the $(\Gamma_U|_{V_j}, D_{\Gamma_U}^U(\bullet, \bullet; V_j))$'s are independent and their conditional laws are those of $(\Gamma_{V_j}, D_{\Gamma_{V_j}}^{V_j})$, respectively. 
		\item\label{it:axiom_conformal_invariance} {\bfseries (Conformal invariance)} Let $\phi \colon U \to \phi(U)$ be a deterministic conformal mapping. Then,
		\begin{equation*}
			\left(\phi(\Gamma_U), \left( D_{\Gamma_U}^U\left(\phi^{-1}(\SCL_1), \phi^{-1}(\SCL_2) \right) \right)_{\SCL_1, \SCL_2 \in \phi(\Gamma_U)} \right)
			 \overset{(\mathrm{d})}{=} \left( \Gamma_{\phi(U)}, \left(D_{\Gamma_{\phi(U)}}^{\phi(U)}(\SCL_1, \SCL_2) \right)_{\SCL_1, \SCL_2 \in \Gamma_{\phi(U)}} \right).
		\end{equation*}
		The equality in law is understood as in \Cref{def:weak_axioms}.

		\item\label{it:axiom_uniform_exploration} {\bfseries (Uniform exploration)} The collection $\{(\SCL, D_{\Gamma_U}^U(\SCL, \partial U))\}_{\SCL \in \Gamma_U}$ has the law of a uniform exploration of $\Gamma_U$ (cf.~\cite{CoInCLEExpl}), in the sense of \Cref{def:weak_axioms}.
	\end{enumerate}
\end{definition}

\begin{remark}
	In the next paper \cite{kkmt2026cle4_part3}, we prove the uniqueness of the CLE$_4$ metric and show that it is a measurable function of CLE$_4$. Combining this measurability with Axiom~\eqref{it:axiom_locality}, the internal metric $D_{\Gamma_U}^U(\bullet, \bullet; V_j)$ is a.s.\ determined by $V_j$ and $\Gamma_U|_{V_j}$. Moreover, by Axiom~\eqref{it:axiom_conformal_invariance}, $D$ is uniquely determined by $D^U$ for a fixed simply connected domain $U \subsetneq \BC$. 
\end{remark}

We now state the main results of this paper.

\begin{theorem}[Existence of geodesics]\label{thm:existence_geodesics}
Let $D$ be a weak geodesic CLE$_4$ metric coupling in the sense of \Cref{def:weak_axioms}. Then, for each simply connected domain $U \subsetneq \BC$, a.s., for all $\SCL, \SCL' \in \Gamma_U$, there exists a $D^{U}_{\Gamma_U}$-geodesic between $\SCL$ and $\SCL'$. In fact, $D$ is a geodesic CLE$_4$ metric coupling in the sense of \Cref{def:axioms}. Conversely, every geodesic CLE$_4$ metric coupling is a weak geodesic CLE$_4$ metric coupling.
\end{theorem}

In this paper, we also establish three key properties of CLE$_4$ geodesics: there is no distance gap along them, they are supported on the loops of CLE$_4$, and they do not intersect the domain boundary.

\begin{theorem}[Behavior of geodesics]\label{thm:behavior-of-geodesics}
	Let $D$ be a weak geodesic CLE$_4$ metric coupling in the sense of \Cref{def:weak_axioms}. Let $U \subsetneq \BC$ be a deterministic simply connected domain. Then:
		\begin{enumerate}
			\item Almost surely, for each $\SCL_1, \SCL_2 \in \Gamma_U$, every $D_{\Gamma_U}^U$-geodesic $P$ connecting $\SCL_1$ and $\SCL_2$ is such that $\{D_{\Gamma_{U}}^{U}(\SCL_1,\SCL): \SCL \in \Gamma, \ \SCL\cap P([0,1]) \neq \emptyset\}$ forms a dense subset of the interval $[0, D_{\Gamma_{U}}^{U}(\SCL_1, \SCL_2)]$.
			\item Almost surely, for each $\SCL_1, \SCL_2 \in \Gamma_U$, every $D_{\Gamma_U}^U$-geodesic $P$ connecting $\SCL_1$ and $\SCL_2$ has the property that the subset $\{P(t) : t \in [0,1],\ P(t) \notin \bigcup_{\SCL \in \Gamma_U} \SCL\}$ of $\overline U$ is of Hausdorff dimension zero. 
			\item Almost surely, for each $\SCL_1, \SCL_2 \in \Gamma_U$, no $D_{\Gamma_U}^U$-geodesic connecting $\SCL_1$ and $\SCL_2$ touches $\partial U$. 
	\end{enumerate}
\end{theorem}
\begin{remark}
	In the next paper \cite{kkmt2026cle4_part3}, we also establish the uniqueness of geodesics. We postpone the uniqueness to \cite{kkmt2026cle4_part3} because the proof requires the fact that the metric is determined by the CLE$_4$, which is one of the main results of \cite{kkmt2026cle4_part3}.
\end{remark}

\subsection{Techniques}
Our toolbox combines several distinct perspectives on CLE$_4$: the connection between metric balls grown from the boundary and the uniform exploration, the couplings of CLE$_4$ and of SLE$_4$ with the Gaussian free field via level lines, and the Brownian loop soup construction of CLE$_4$.

It is well known that, due to the Poissonian construction of the uniform exploration \cite{CoInCLEExpl} (see \Cref{subsec:labeled_cle_4}), the time at which the loop surrounding the origin is discovered by the uniform exploration is an exponential random variable. We show that, similarly, the distance between the top and bottom sides of a conformal rectangle (defined in \Cref{sec:distances_across_rectangles}) has an exponential upper tail. This exponential tail is a fundamental ingredient in our analysis, as it yields bounds on all moments of the distance across a rectangle. In turn, these estimates provide tight a priori control on the macroscopic geometry of the space, preventing the metric from degenerating or blowing up.

Another key tool is the asymptotic behavior of the probability that two metric balls do not intersect before hitting a small Euclidean ball; see \Cref{prop:hitting-two-metric-balls}. More precisely, we show that the probability that the metric balls grown from two loops remain disjoint up until each of them first hits a Euclidean ball of radius $\varepsilon$ is at most $\varepsilon^{2 + o(1)}$ as $\varepsilon \to 0$. Geometrically, this event corresponds to a ``bottleneck'' where two distinct explorations come microscopically close without having merged. The main idea in the proof of \Cref{prop:hitting-two-metric-balls} is to translate this non-intersection event into a requirement for four macroscopic arms. This allows us to apply the four-arm exponents for bichordal SLE$_4$ across multiple scales, using the arm estimates obtained by Zhan \cite{zhan2020two,zhan2019two}. 

To establish the existence of geodesics, we employ a multi-scale construction that relies crucially on this $2+o(1)$ exponent bound from \Cref{prop:hitting-two-metric-balls}. To construct a geodesic between two given loops, we iteratively identify a sequence of intermediate loops that serve as nearly optimal ``waypoints''. At each step of the iteration, we refine the spatial scale and apply our non-intersection estimates to prove that this sequence of waypoints converges a.s. This procedure yields a continuous curve whose $D$-length equals the distance between the two loops, i.e., a geodesic.

\subsection{Outline}
The remainder of the paper is organized as follows. The subsequent sections contain several quantitative estimates that are used repeatedly in the final section of this paper and will also be essential in our sequel \cite{kkmt2026cle4_part3}.

\smallskip
\noindent\emph{\Cref{section:preliminaries}: Preliminaries.} We review the necessary background material on the Gaussian free field (GFF), Schramm--Loewner evolutions (SLE) with force points, conformal loop ensembles (CLE), the Brownian loop soup, couplings between GFF and CLE$_4$, Carath\'eodory convergence, and bichordal SLE$_4$ four-arm estimates.

\smallskip
\noindent\emph{\Cref{sec:distances_across_rectangles}: Distances across a rectangle.} We provide estimates for the distance across a conformal rectangle and control its $p$-th moment for all $p>0$.

\smallskip
\noindent\emph{\Cref{sec:hitting_two_metric_balls}: Hitting two disjoint metric balls.} We estimate the probability that the metric balls grown from two loops do not intersect up until each of them first hits a small Euclidean ball.

\smallskip
\noindent\emph{\Cref{sec:vertical_vs_horizontal}: Vertical versus horizontal distance across a rectangle.} We establish a bound that simultaneously controls both vertical and horizontal distances across a conformal rectangle; this result will be used in \cite{kkmt2026cle4_part3}.

\smallskip
\noindent\emph{\Cref{section:existence-of-geodesics}: Existence of geodesics.} We prove that every weak geodesic CLE$_4$ metric coupling is a geodesic CLE$_4$ metric coupling, the main points being Axioms~\eqref{it:axiom_geodesic} and~\eqref{it:axiom_locality} of \Cref{def:axioms}; the remaining two axioms are common to the two definitions. We also prove the converse assertion of \Cref{thm:existence_geodesics} at the end of the section. We also prove \Cref{thm:behavior-of-geodesics}.

\smallskip
\noindent\emph{\Cref{sec:cle_4_four_arm_exponents}: Bichordal SLE$_4$ four-arm exponents.} We derive the four-arm exponents for bichordal SLE$_4$ using the results of Zhan \cite{zhan2019two}.

\subsection*{Acknowledgements}

E.K.~acknowledges the support of a Research Fellowship from Emmanuel College, Cambridge. K.K.~was supported by the Simons Collaboration Grant \emph{Probabilistic Paths to Quantum Field Theory}. J.M.~received support from ERC Consolidator Grant ARPF (Horizon Europe UKRI G120614). Y.T.~was supported by a Cambridge International Scholarship from the Cambridge Trust. We thank Wendelin Werner for stimulating discussions at an early stage of this work. E.K.~also thanks Juhan Aru and Ellen Powell for insightful discussions on this topic before starting this work.

\section{Preliminaries}\label{section:preliminaries}

\subsection{Notation and conventions}

The notation $\BC$ (resp.\ $\BR$; $\BQ$; $\BZ$; $\BN$) will be used to denote the set of complex numbers (resp.\ real numbers; rational numbers; integers; positive integers). For $a \le b$, we shall write $[a, b]_\BZ \defeq [a, b] \cap \BZ$. For $0 \le r < s$ and $z \in \BC$, we write $A_{r,s}(z) \defeq \{w \in \BC : r < \vert w - z \vert < s\}$, and for $r > 0$ and $A \subset \BC$ we write $B_r(A) \defeq \{w \in \BC : \dist(w, A) < r\}$.

We call a nonempty connected open set $U \subset \BC$ a \emph{domain}. We call an open set $U \subset \BC$ a \emph{simply connected domain} if $U$ is connected, $U \neq \BC$, and $\BC \setminus U$ has no bounded connected component. We call a connected open set $A \subset \BC$ a \emph{doubly connected domain} if $\BC \setminus A$ consists of one unbounded connected component and one bounded connected component, and the bounded component is not a singleton. (These are exactly the domains conformally equivalent to a round annulus of finite modulus.)

\subsection{Gaussian free fields}
\label{subsec:gff}

Let $U \subsetneq \BC$ be a simply connected domain and let $H_0^1(U)$ denote the Hilbert space closure of $C_0^{\infty}(U)$ with respect to the Dirichlet inner product 
\begin{equation*}
    \langle f, g\rangle_{H_0^1(U)} = \frac{1}{2\pi} \int_{U} \nabla f(z) \cdot \nabla g(z) \, \rd z.
\end{equation*}
The \emph{zero-boundary Gaussian free field (GFF)} $\Psi$ is the random Schwartz distribution defined by setting
\begin{equation}\label{eqn:gff_series}
    \Psi = \sum_{n \geq 1} \alpha_n \psi_n,
\end{equation}
where $\{\psi_n\}_{n \geq 1}$ is an orthonormal basis for $H_0^1(U)$ and $\{\alpha_n\}_{n \geq 1}$ is a sequence of independent standard Gaussian random variables; the series converges a.s.\ in the space of distributions on $U$. The law of $\Psi$ does not depend on the choice of the orthonormal basis.

A GFF on $U$ is not a function but rather a random variable in the space of Schwartz distributions on $U$; moreover, the series in~\eqref{eqn:gff_series} a.s.\ does not converge in $H_0^1(U)$.

An important property of the GFF that we are going to use is the domain Markov property. More precisely, if $\Psi$ is a zero-boundary GFF on $U$ and $V \subset U$ is open, then we can write $\Psi = \kh_V + \Psi_V$, where $\Psi_V$ is a zero-boundary GFF on $V$ (defined as independent zero-boundary GFFs on the connected components of $V$, extended by zero to $U$) and $\kh_V$ is a distribution on $U$ that is a.s.\ harmonic on $V$, with $\Psi_V$ and $\kh_V$ being independent.

Let $f \colon \partial U \to \BR$ be bounded and measurable. We say that a GFF on $U$ has boundary conditions given by $f$ if it can be expressed as the sum of a zero-boundary GFF on $U$ and the harmonic extension of $f$ to $U$.

\subsection{Chordal SLE$_{\kappa}(\underline{\rho})$ processes}
\label{subsec:chordal_sle}

Chordal SLE$_{\kappa}(\underline{\rho})$ processes, first introduced in \cite{CoRestr}, are random fractal curves growing in a simply connected domain, from one marked boundary point to another. To define them, we fix $\kappa>0$, let $\underline{x}_{\rL} = (x_{\ell,\rL},\ldots,x_{1,\rL})$ and $\underline{x}_{\rR} = (x_{1,\rR},\ldots,x_{r,\rR})$, where $x_{\ell,\rL} < \cdots < x_{1,\rL} \leq 0 \leq x_{1,\rR} < \cdots < x_{r,\rR}$, and let $\underline{\rho}_\rL = (\rho_{1,\rL},\ldots,\rho_{\ell,\rL}) \in \BR^\ell$ and $\underline{\rho}_\rR = (\rho_{1,\rR},\ldots,\rho_{r,\rR}) \in \BR^r$. For $z \in \BH$, let $\{g_t(z)\}_{t \ge 0}$ denote the solution to the ordinary differential equation:
\begin{equation*}
\partial_t g_t(z) = \frac{2}{g_t(z) - W_t}; \quad g_0(z) = z, 
\end{equation*}
where $W$ is a solution to:
\begin{equation}\label{eqn:multiforce_point_sde}
    \begin{dcases}
        \mathrm{d} W_t = \sum_{j=1}^{\ell} \frac{\rho_{j,\rL}}{W_t - V_t^{j,\rL}} \, \rd t + \sum_{j=1}^r \frac{\rho_{j,\rR}}{W_t - V_t^{j,\rR}} \, \rd t + \sqrt{\kappa} \, \rd B_t, \quad W_0 = 0; \\
        \mathrm{d}V_t^{j,\rL} = \frac{2}{V_t^{j,\rL} - W_t} \, \rd t, \quad V_0^{j,\rL} = x_{j,\rL}, \quad \forall j \in [1, \ell]_{\BZ}; \\
        \mathrm{d}V_t^{j,\rR} = \frac{2}{V_t^{j,\rR} - W_t} \, \rd t, \quad V_0^{j,\rR} = x_{j,\rR}, \quad \forall j \in [1, r]_{\BZ}; 
    \end{dcases}
\end{equation}
where $\{B_t\}_{t\ge 0}$ is a standard Brownian motion. For each $z \in \BH$, the solution $g_t(z)$ exists up to the swallowing time $T(z) \defeq \sup\{t \ge 0 : \inf_{s \in [0,t]}\vert g_s(z) - W_s\vert > 0\}$, and we let $\BH_t \defeq \{z \in \BH : T(z) > t\}$ denote the domain of $g_t$.

For any value of $\underline{\rho}_\rL$ and $\underline{\rho}_\rR$, if $x_{1,\rL} < 0 < x_{1,\rR}$, then it is clear that~\eqref{eqn:multiforce_point_sde} has a unique strong solution until the first time $t$ such that $W_t=V^{j,\bullet}_t$ for some $\bullet \in \{\rL,\rR\}$ and some $j$. In fact, it was shown in \cite{IG1,LevelLineGFFI} that, when $\sum_{k = 1}^j \rho_{k,\rL} > -2$ for every $j \in [1, \ell]_\BZ$ and $\sum_{k = 1}^j \rho_{k,\rR} > -2$ for every $j \in [1, r]_\BZ$, there exists a solution to~\eqref{eqn:multiforce_point_sde} defined for all $t \ge 0$, for which the set of times $t$ with $W_t=V^{j,\bullet}_t$ for some $\bullet \in \{\rL,\rR\}$ and some $j$ a.s.\ has zero Lebesgue measure. The uniqueness in law of such a solution is also shown in \cite{IG1}. Moreover, it was shown in \cite{IG1,LevelLineGFFI} that there a.s.\ exists a continuous curve $\eta$ such that the domain $\BH_t$ of $g_t$ is given by the unbounded connected component of $\BH \setminus \eta([0,t])$, for all $t\ge 0$. The curve $\eta$ is called the chordal SLE$_{\kappa}(\underline{\rho}_\rL; \underline{\rho}_\rR)$ in $\BH$ from $0$ to $\infty$.

For a simply connected domain $U \subsetneq \BC$ with marked starting, target, and force points on $\partial U$, the SLE$_{\kappa}(\underline{\rho})$ in $U$ with this data is the image of an SLE$_{\kappa}(\underline{\rho})$ on $\BH$ under a conformal map $\phi \colon \BH \to U$ sending $0, \infty$ to the starting and target points, with force points the $\phi$-preimages of the marked ones.

\subsection{Conformal loop ensembles}
\label{subsec:cle}

Now, we will briefly review CLE and refer to \cite{CLE,CLEPerc,TreeCLE} for more details.

A \emph{conformal loop ensemble} (CLE) in $\BD$ is a random countable collection $\Gamma$ of non-nested and non-crossing loops in $\overline{\BD}$ that possesses the following properties.
\begin{enumerate}
\item \textbf{(Conformal invariance)} For any M\"obius transformation $\phi \colon \BD \to \BD$, the laws of $\Gamma$ and $\phi(\Gamma)$ are the same. This makes it possible to define CLE on any simply connected domain $U \subsetneq \BC$ as the image of a CLE on $\BD$ under a conformal mapping from $\BD$ onto $U$. (When the loops can touch $\partial\BD$, i.e., for $\kappa \in (4,8)$, the image is defined via prime ends.)
\item \textbf{(Domain Markov property)} For any simply connected subdomain $V \subset \BD$, we let $V^\star$ be defined as in~\eqref{eq:def-V-star} (with $\Gamma$ in place of $\Gamma_U$). Then, conditionally on the collection of loops of $\Gamma$ not contained in $V$, the restrictions $\{\SCL \in \Gamma : \SCL \subset \overline{V_j}\}$ for the connected components $V_j$ of $V^\star$ are independent, and each has the law of a CLE in $V_j$.
\end{enumerate}

It was shown in \cite{CLE,TreeCLE} that for every CLE there is a unique $\kappa \in (8/3,8)$ for which its loops can be constructed via branching variants of SLE$_\kappa$ called exploration trees; the loops are then SLE$_\kappa$-type loops. Conversely, for each $\kappa \in (8/3,8)$ there is a unique CLE with SLE$_\kappa$-type loops, and we denote it by CLE$_\kappa$.

If $\kappa \in (8/3,4]$, we have that CLE$_{\kappa}$ consists of disjoint simple loops which a.s.\ do not intersect the domain boundary. However, if $\kappa \in (4,8)$, we have that the loops in the CLE$_{\kappa}$ are self-touching and can touch each other and the domain boundary. In the present work, we will focus on the $\kappa = 4$ case.

\subsection{Brownian loop soup and CLE$_4$}
\label{subsec:brownian_loop_soup}

We now recall the construction of CLE$_4$ from the Brownian loop soup, referring to \cite{BLS,CLE} for details. The \emph{Brownian loop soup} with intensity $\lambda > 0$ in a domain $U \subset \BC$ is the Poisson point process with intensity measure $\lambda \mu_U^{\mathrm{loop}}$, where $\mu_U^{\mathrm{loop}}$ is the Brownian loop measure on $U$ of \cite{BLS}; it is conformally invariant, and for open $V \subset U$ the loops contained in $V$ form a loop soup in $V$. Two loops are in the same \emph{cluster} if they are joined by a finite chain of pairwise intersecting loops of the soup, and a cluster is \emph{outermost} if the closure of the union of its loops is not surrounded by that of another cluster. By \cite[Theorem~1.5]{CLE}, the outer boundaries of the closures of the outermost clusters of a loop soup in $\BD$ with intensity $1$ form a non-nested CLE$_4$ in $\BD$.

\subsection{Level lines of the GFF}
\label{subsec:level_lines_of_gff}

Let $\lambda = \pi / 2$ and let $\Psi$ be an instance of the GFF on $\BH$ with boundary conditions given by $-\lambda$ (resp.\ $\lambda$) on $\BR_-$ (resp.\ $\BR_+$). Then, it was shown in \cite{MR3101840} that a zero level line $\eta$ of $\Psi$ from $0$ to $\infty$ can be rigorously defined and has the law of a chordal SLE$_4$, and $\eta$ is a.s.\ determined by $\Psi$. Moreover, $\eta$ is characterized by the property that for all $t \geq 0$ simultaneously, the conditional law of $\Psi$ given $\eta|_{[0,t]}$ is that of a GFF on $\BH \setminus \eta([0,t])$ with boundary conditions given by $-\lambda$ (resp.\ $\lambda$) on the left (resp.\ right) side of $\eta([0,t])$ and $\BR_-$ (resp.\ $\BR_+$).

The results in \cite{MR3101840} were generalized in \cite{LevelLineGFFI} as follows. Fix a collection of force points $(\underline{x}_{\rL}; \underline{x}_{\rR})$ and weights $(\underline{\rho}_{\rL}; \underline{\rho}_{\rR})$ with $\sum_{i=1}^j \rho_{i,\bullet} > -2$ for all $j$ and $\bullet \in \{\rL, \rR\}$, and let $\Psi$ be a GFF on $\BH$ with boundary conditions given by $-\lambda( 1+ \sum_{i=0}^j \rho_{i,\rL})$ on $[x_{j+1,\rL}, x_{j,\rL})$ for $j \in [0, \ell]_\BZ$ and $\lambda (1+\sum_{i=0}^j \rho_{i,\rR})$ on $[x_{j,\rR}, x_{j+1,\rR})$ for $j \in [0, r]_\BZ$, where $\rho_{0,\rL} = \rho_{0,\rR} = 0$, $x_{0,\rL} = 0^-$, $x_{\ell + 1,\rL} = -\infty$, $x_{0,\rR} = 0^+$, and $x_{r+1,\rR} = \infty$. Then, the zero level line of $\Psi$ from $0$ to $\infty$ can be rigorously defined and it has the law of a chordal SLE$_4(\underline{\rho}_{\rL}; \underline{\rho}_{\rR})$, and it is a measurable function of $\Psi$. Moreover, the level line is characterized by the property that the conditional law of $\Psi$ given $\eta|_{[0,t]}$ is that of a GFF on $\BH \setminus \eta([0,t])$ with boundary conditions given by $-\lambda$ (resp.\ $\lambda$) on the left (resp.\ right) side of $\eta([0,t])$ and the same boundary values as $\Psi$ on $\BR$.

For all $u \in \BR$, the level line of $\Psi$ with height $u$ from $0$ to $\infty$ is given by the level line of $\Psi - u$ from $0$ to $\infty$. Then, we have the following interaction rules for level lines with different heights, which were proved in \cite{LevelLineGFFI}.

\begin{theorem}[{\cite[Theorem~1.1.4 and Remark 1.1.5]{LevelLineGFFI}}]\label{thm:level_line_interaction}
Suppose that $\Psi$ is a GFF on $\BH$ with piecewise constant boundary conditions that change only finitely many times. For all $u \in \BR$ and $x \in \BR$, we let $\gamma_u^x$ be the level line of $\Psi$ with height $u$ starting from $x$ and ending at $\infty$; the assertions below apply whenever the level lines in question are defined. Fix $u_1, u_2 \in \BR$ and $x_2 \leq x_1$.
\begin{enumerate}
\item If $u_2 < u_1$, then $\gamma_{u_2}^{x_2}$ a.s.\ stays to the left of $\gamma_{u_1}^{x_1}$ (possibly touching it).
\item If $u_2 = u_1$, then $\gamma_{u_2}^{x_2}$ may intersect $\gamma_{u_1}^{x_1}$ and, upon intersecting, the two curves merge and never separate.
\item If $u_1 - u_2 \geq 2 \lambda$, then $\gamma_{u_1}^{x_1}$ and $\gamma_{u_2}^{x_2}$ do not intersect each other a.s.
\end{enumerate}
\end{theorem}

\subsection{A coupling between CLE$_4$ and the GFF}
\label{subsec:nested_cle_gff}

Now we describe a coupling between a zero-boundary GFF and a nested CLE$_4$ that we are going to consider.
Let $U \subsetneq \BC$ be a simply connected domain. Recall that the nested CLE$_4$ $\overline{\Gamma}$ on $U$ is obtained from the non-nested CLE$_4$ $\Gamma$ on $U$ by the following inductive procedure. Let $\Gamma^1 \defeq \Gamma$ and inductively, we assume that for some $n \in \BN$ the collection of loops $\Gamma^n$ in $U$ (which we call the $n$-th generation loops) has been defined. Then conditional on $\{\Gamma^k\}_{1 \leq k \leq n}$, for each loop $\SCL$ in $\Gamma^n$, we let $\Gamma_{\SCL}$ be a non-nested CLE$_4$ in the domain surrounded by $\SCL$, and we take these CLEs to be conditionally independent given $\{\Gamma^k\}_{1 \leq k \leq n}$. We then let $\Gamma^{n+1} = \bigcup_{\SCL \in \Gamma^n} \Gamma_{\SCL}$. Finally, we let
\begin{equation*}
	\overline{\Gamma} \defeq \bigcup_{n=1}^{\infty} \Gamma^n.
\end{equation*}

Let us now review the coupling between $\overline{\Gamma}$ and a zero-boundary GFF $\Psi$ on $U$ that we are going to use (cf., e.g., \cite[Section~4]{aru2019bounded}). Let $\{X_{\SCL}\}_{\SCL \in \overline{\Gamma}}$ be conditionally independent Rademacher random variables given $\overline{\Gamma}$. 
For $n \in \BN$, we let $\Psi_n$ be the piecewise constant function on $U$ which is defined Lebesgue almost everywhere on $U$ (each fixed $z \in U$ is a.s.\ surrounded by a loop of $\Gamma^n$, so by Fubini's theorem the interiors of the $n$-th generation loops cover Lebesgue-almost all of $U$) and for each $n$-th generation loop $\SCL \in \Gamma^n$, satisfies 
\begin{align*}
	\Psi_n|_{\mathop{\mathrm{int}}(\SCL)} = 2\lambda \sum_{k=1}^n X_{\SCL^{(k)}},
\end{align*}
where $\SCL^{(n)} = \SCL$ and, for $k=1,\ldots,n-1, \SCL^{(k)}$ is the unique loop in $\Gamma^k$ which disconnects $\SCL$ from $\partial U$. Recall also that $\lambda = \frac{\pi}{2}$. Then we have that $\Psi_n$ converges a.s.\ in the space of distributions on $U$, as $n \to \infty$, to a zero-boundary GFF $\Psi$ on $U$; see \cite[Section~4]{aru2019bounded}.

In the coupling described above, we have that $\Psi$ and $\left(\overline{\Gamma}, \{X_{\SCL}\}_{\SCL \in \overline{\Gamma}}\right)$ are a.s.\ given by measurable functions of each other. Furthermore, the following Markov property holds. For each $n \in \BN$, the conditional law of $(\Psi,\overline{\Gamma})$ given 
\begin{align*}
	\bigcup_{k=1}^n \Gamma^k \quad \text{and} \quad \left\{X_{\SCL} : \SCL \in \bigcup_{k=1}^n \Gamma^k \right\}
\end{align*}
is described as follows. Let $\{\Psi_{\SCL}\}_{\SCL \in \Gamma^n}$ be conditionally independent zero-boundary GFFs on the domains $\mathop{\mathrm{int}}(\SCL)$ surrounded by the loops in $\Gamma^n$. Then, we have
\begin{align*}
	\Psi|_{\mathop{\mathrm{int}}(\SCL)} = \Psi_{\SCL} + 2\lambda \sum_{k=1}^n X_{\SCL^{(k)}}, \quad \forall\SCL \in \Gamma^n,
\end{align*}
where $\SCL^{(k)}$ is as above. This determines the conditional law of $\Psi$, since $U \setminus \bigcup_{\SCL \in \Gamma^n} \mathop{\mathrm{int}}(\SCL)$ a.s.\ carries none of the field, in the sense of condition~\eqref{it:thin_local_set} below. Moreover, the set $\overline{\Gamma}|_{\mathop{\mathrm{int}}(\SCL)}$ of loops of $\overline{\Gamma}$ which are contained in $\mathop{\mathrm{int}}(\SCL)$ is the nested CLE$_4$ coupled with $\Psi_{\SCL}$ as above.

\subsection{Uniform exploration of the CLE$_4$}
\label{subsec:labeled_cle_4}

The \emph{uniform exploration} of CLE$_4$ in $\BD$, also called the \emph{labeled} CLE$_4$ in $\BD$, was introduced in \cite{CoInCLEExpl} and corresponds to a Markovian exploration of CLE$_4$ loops, with labels on each loop which keep track of the time when each loop is discovered. We will briefly describe the construction and refer to \cite{CoInCLEExpl} and \cite{LevelLineGFFI} for more details.

To construct the exploration, we define the measure $M$ as the image of $\mu \otimes \omega$ under the mapping $(\ell, x) \mapsto x \cdot \ell$ (the rotation of $\ell$ by $x$), where $\omega$ denotes the harmonic measure on $\partial \BD$ as seen from $0$, and $\mu$ denotes the (infinite) SLE$_4$ bubble measure on $\BD$ pinned at $1$ introduced in \cite{CLE}, normalized so that the time at which the loop surrounding $0$ is discovered is an exponential random variable of parameter one. We also let $(\ell_t)_{t \geq 0}$ denote a Poisson point process (PPP) with intensity measure $M \otimes \mathrm{Leb}$, where $\mathrm{Leb}$ denotes the Lebesgue measure on $\BR_+$ ($\ell_t$ being defined only for the countably many atom times $t$). Almost surely no two atoms share a time, so this process defines an ordering on the (pinned) loops, also sometimes called bubbles, where each loop $\ell$ is equipped with time label $t_{\ell}$.

The exploration targeted at $0$ is defined as follows. Let $\tau_0$ denote the first time that the PPP discovers a (pinned) loop that surrounds $0$ and for all $\varepsilon \in (0,1)$, we let $\ell_{t_1^{\varepsilon}},\ldots, \ell_{t_{n_{\varepsilon}}^{\varepsilon}}$ denote the (pinned) loops discovered up to and including time $\tau_0$ with diameter (as subsets of $\overline\BD$, before any mapping) at least $\varepsilon$, and such that $t_1^{\varepsilon}<\cdots<t_{n_{\varepsilon}}^{\varepsilon}$. We define $D_0 = \BD$ and inductively let $D_i$ denote the connected component of $D_{i-1} \setminus \phi_i(\ell_{t_i^{\varepsilon}})$ containing $0$, where $\phi_i$ is the conformal mapping from $\BD$ onto $D_{i-1}$ such that $\phi_i(0) = 0$ and $\phi_i'(0) > 0$. We set $D^\varepsilon_t \defeq D_i$ for $t \in [t_i^\varepsilon, t_{i+1}^\varepsilon)$ and $0 \le i \le n_\varepsilon$, where $t_0^\varepsilon \defeq 0$ and $t^\varepsilon_{n_\varepsilon+1} \defeq \infty$. As $\varepsilon \downarrow 0$, a.s.\ $D^\varepsilon_t$ decreases and hence converges in the Carath\'eodory sense (seen from $0$), for each $t \ge 0$, to a simply connected open domain $C_t(0)$, called the connected component of the unexplored region containing the origin; see \cite{CoInCLEExpl}. We have $C_t(0) = C_{\tau_0}(0)$ for $t \ge \tau_0$ and, by \cite{CoInCLEExpl}, $C_{\tau_0}(0)$ is the interior of a loop which has the same law as the loop surrounding $0$ in a CLE$_4$. This loop is equipped with the time label $\tau_0$ and we say that it is discovered at time $\tau_0$.

It was shown in \cite[Lemma~6]{CoInCLEExpl} that the measure $M$ is invariant under M\"obius transformations on $\BD$. Therefore, by conformal invariance, we can define the exploration targeted at any fixed point $z \in \BD$ by applying to the exploration targeted at $0$ a M\"obius transformation of $\BD$ mapping $0$ to $z$; by the M\"obius invariance of $M$, the choice of transformation does not matter.

The labeled CLE$_4$ exploration is then defined by coupling the exploration processes targeted at all points $z \in \BD \cap \BQ^2$ simultaneously, so that for any two targets $z, w$, the corresponding explorations agree up until the first time that $z$ and $w$ no longer lie in the same connected component of the unexplored region, and then evolve independently. For all $z \in \BD$, we denote by $C_t(z)$ the connected component of the unexplored region containing $z$. The explored region is the compact set 
\[\SCB_t(\partial \BD) \defeq \overline{\BD} \setminus \bigcup_{z \in \BD \cap \BQ^2} C_t(z).\]
(The metric-ball notation is deliberate: by Axiom~\eqref{it:weak_axiom_uniform_exploration} of \Cref{def:weak_axioms}, the explored region agrees in law with the metric ball $\SCB_t(\partial\BD; D_{\Gamma_\BD}^\BD)$ of \Cref{subsec:setup}.)
Since every rational point is encircled by a loop and since $C_t(z)$ is open, and since $\mathop{\mathrm{int}}(\SCL(x)) \subset C_t(z)$ for all $x \in C_t(z) \cap \BQ^2$ (loops that are undiscovered at time $t$ lie in the unexplored region), we can also write $\overline{C_t(z)} = \overline{\bigcup_{x \in C_t(z) \cap \BQ^2} \mathrm{int}(\SCL(x)) }$.

\subsection{Two-valued sets of the GFF}
\label{subsec:two_valued_level_sets}

Next, we introduce the two-valued sets of a GFF $\Psi$ on $\BD$ that we are going to consider. Recall from \cite{MR3101840} that a random closed set $A$ coupled with a zero-boundary GFF $\Psi$ on a simply connected $U \subsetneq \BC$ is a \emph{local set} of $\Psi$ if there is a random distribution $\kh_A$, a.s.\ harmonic on $U \setminus A$, such that given $(A, \kh_A)$ the field $\Psi - \kh_A$ is a zero-boundary GFF on $U \setminus A$. We refer to \cite{aru2019bounded} and \cite{TVSGFF} for more detailed expositions of two-valued sets of the GFF.

\begin{definition}
Fix $a,b>0$ and let $\Psi$ be a zero-boundary GFF in a simply connected domain $U \subsetneq \BC$. We say that $\mathbb{A}_{-a,b}$ is a two-valued local set of $\Psi$ of levels $-a$ and $b$ if it is a local set of $\Psi$ and satisfies the following properties.
\begin{enumerate}
\item \label{it:boundary_conditions}
The distribution $\kh_{\mathbb{A}_{-a,b}}$ from the local set decomposition is a.s.\ constant on each connected component of $U \setminus \mathbb{A}_{-a,b}$, with value $-a$ or $b$.
\item \label{it:thin_local_set}
For any smooth test function $f \in C_0^{\infty}(U)$, we have that
\begin{align*}
\langle \kh_{\mathbb{A}_{-a,b}}, f\rangle = \int_{U \setminus \mathbb{A}_{-a,b}} \kh_{\mathbb{A}_{-a,b}}(x) f(x) \, \mathrm{d}x
\end{align*}
a.s.\ (that is, the distribution $\kh_{\mathbb{A}_{-a,b}}$ puts no mass on $\mathbb{A}_{-a,b}$: the set $\mathbb{A}_{-a,b}$ is a \emph{thin} local set, cf.~\cite{aru2019bounded}).
\item \label{it:finitely_many_components} The set $\mathbb{A}_{-a,b} \cup \partial U$ has a finite number of connected components.
\end{enumerate}
\end{definition}

It was shown in \cite{aru2019bounded} that if $a+b \geq 2\lambda$, then such a set $\mathbb{A}_{-a,b}$ exists and it is a.s.\ determined by $\Psi$. Moreover, $\mathbb{A}_{-a,b}$ is unique in the sense that if $A'$ is another local set of $\Psi$ satisfying~\eqref{it:boundary_conditions}, \eqref{it:thin_local_set} and \eqref{it:finitely_many_components}, we have that $A' = \mathbb{A}_{-a,b}$ a.s.

Let us now briefly review the construction of $\mathbb{A}_{-a,-a+2\lambda}$ for $a \in (0,2\lambda)$, which is going to be needed for our purposes. We will follow \cite[Section~3.2]{TVSGFF}. 

To begin with, we let $\eta$ denote the level line of $\Psi+a-\lambda$ from $-\ri $ to $\ri$ and set $A^1$ to be the range of $\eta$. In each connected component $O$ of $\BD \setminus A^1$ lying to the left of $\eta$, we let $x$ and $y$ be the two endpoints of $\partial O \cap \partial \BD$, labeled so that $-\ri$, $x$, $y$ are ordered counterclockwise. We then let $\eta^O$ denote the level line of $\Psi|_O + a -\lambda$ from $x$ to $y$.

As for the connected components $O$ of $\BD \setminus A^1$ lying to the right of $\eta$, we perform an analogous procedure except that we explore the level line of $\Psi|_O + a - \lambda$ from $y$ to $x$ (with $x$, $y$ labeled by the same convention). We denote those level lines by $\eta^O$.

In the next step of the exploration, we let $A^2$ denote the closure of the union of $A^1$ with the $\eta^O$'s for each connected component $O$ of $\BD \setminus A^1$. In the connected components of $\BD \setminus A^2$ whose boundaries are contained in $A^2$, we stop the exploration. Note that the boundary conditions of $\Psi$ on each of the aforementioned components are a.s.\ constant and given by either $-a$ or $2\lambda - a$. In the rest of the components, we iterate as before (i.e., in each such component $O$, we trace the level line of $\Psi|_O + a - \lambda$ between the endpoints of $\partial O \cap \partial\BD$, with the same orientation convention) to obtain $A^n$. We note that if $O$ is a connected component of $\BD \setminus A^2$ such that $\partial O \cap \partial \BD \neq \emptyset$, then we have that the boundary conditions of $\Psi|_O$ are given by $0$ on $\partial O \cap \partial \BD$, and $2\lambda - a$ (resp.\ $-a$) on $\partial O \setminus \partial \BD$ if $O$ is contained in a connected component of $\BD \setminus A^1$ lying to the left (resp.\ right) of $\eta$. Then, we set $\mathbb{A}_{-a,-a+2\lambda}\defeq\overline{\bigcup_{n \in \BN} A^n}$.

\begin{remark}\label{rem:target_invariance}
The uniqueness of $\mathbb{A}_{-a,b}$ implies that the starting and target points of the level lines and the order in which they were sampled to produce $\mathbb{A}_{-a,b}$ do not matter (the resulting closed set is a.s.\ the same).
\end{remark}

\begin{remark}\label{remark coupling TVS BCLE}
The set $\mathbb{A}_{-a,-a+2\lambda}$ constructed above is equal to the range of a BCLE$_4(\rho)$ process with $\rho = -a/\lambda$ introduced in \cite[Section~7]{CLEPerc}. By varying $a \in (0,2\lambda)$, we obtain BCLE$_4(\rho)$ for all $\rho \in (-2,0)$. (For a closed set $A \subset \overline\BD$, the \emph{loops} of $A$ are the boundaries of the connected components of $\BD \setminus A$.) More precisely, the loops labeled $2\lambda -a$ of $\mathbb{A}_{-a,{-a+2\lambda}}$ correspond to the \emph{true} loops of the BCLE$_4(\rho)$ and the loops labeled $-a$ of $\mathbb{A}_{-a,{-a+2\lambda}}$ correspond to the \emph{false} loops of the BCLE$_4(\rho)$.
\end{remark}

\subsection{Coupling between the GFF and the uniform exploration}
\label{subsec:gff_labeled_cle_4}

Now, we construct the coupling between a zero-boundary GFF $\Psi$ on $\BD$ and a labeled CLE$_4$ $\{(\SCL, t_{\SCL})\}_{\SCL\in \Gamma}$. We will follow the construction given in \cite[Section~6]{TVSGFF}.

Fix $r \in (0,2\lambda)$. For all $j \in \BN$, we define sets $\mathbb{B}_r^j$ iteratively as follows. First, we set $\mathbb{B}_r^1\defeq\mathbb{A}_{-r,-r+2\lambda}$. In each connected component $O$ of $\BD \setminus \mathbb{B}_r^1$ such that the boundary conditions of $\Psi|_O$ are given by $-r$, we explore the set $\mathbb{A}_{-r,-r+2\lambda}(O)$ which is defined in the same way as $\mathbb{A}_{-r,-r+2\lambda}$, but with the field $(\Psi - \kh_{\mathbb{B}_r^1})|_O$ in place of $\Psi$, where we recall that $\kh_{\mathbb{B}_r^1}$ is the harmonic extension of the boundary values of $\Psi|_{\BD \setminus \mathbb{B}_r^1}$. Then, we let $\mathbb{B}_r^2$ denote the closure of the union of $\mathbb{B}_r^1$ and the sets $\mathbb{A}_{-r,-r+2\lambda}(O)$ over all such components $O$, and note that the boundary conditions of $\Psi$ on the connected components of $\BD \setminus \mathbb{B}_r^2$ are constant and lie in $\{2\lambda-r,2\lambda-2r,-2r\}$.

Suppose that we have constructed $\mathbb{B}_r^j$ for some $j \in \BN$. Then, in each connected component $O$ of $\BD \setminus \mathbb{B}_r^j$ such that the boundary conditions of $\Psi|_O$ are given by $-jr$, we explore the set $\mathbb{A}_{-r,-r+2\lambda}(O)$ defined in the same way as $\mathbb{A}_{-r,-r+2\lambda}$ but with the field $(\Psi - \kh_{\mathbb{B}_r^j})|_O$ in place of $\Psi$. We define $\mathbb{B}_r^{j+1}$ as the closure of the union of $\mathbb{B}_r^j$ and the sets $\mathbb{A}_{-r,-r+2\lambda}(O)$ over all such components $O$.

We set $\mathbb{B}_r\defeq\overline{\bigcup_{j \in \BN} \mathbb{B}_r^j}$ and $\mathbb{B}_0\defeq\overline{\bigcup_{n \in \BN} \mathbb{B}_{2^{-n}}}$, where the union is increasing since $\mathbb{B}_{2^{-n}} \subset \mathbb{B}_{2^{-n-1}}$ for all $n \in \BN$ a.s.\ (see \cite[Section~6]{TVSGFF}). For every connected component $O$ of $\BD \setminus \mathbb{B}_0$, we let $\SCL \defeq \partial O$ and define the label $t_{\SCL} \defeq 2\lambda - v$, where $v$ is the constant boundary value of $\Psi|_O$. Then, we have the following.

\begin{proposition}[{\cite[Proposition~6.6]{TVSGFF}}]
\label{prop coupling TVS uniform explo}
Suppose that we have the setup described above. Then, the collection of loops of $\mathbb{B}_0$, together with the labels $t_{\SCL}$, has the law of a non-nested CLE$_4$ on $\BD$ decorated by its uniform exploration.
\end{proposition}

\subsection{Background on Carath\'eodory convergence}
We gather here some equivalent definitions of Carath\'eodory convergence. Let $(U_n)_{n\ge 0}$ be a sequence of simply connected domains in $\BC$. Let $U \subset \BC$ be a simply connected domain. Let $z \in U$. Let $f$ be the unique conformal map from $\BD$ to $U$ with $f(0)=z$ and $f'(0)>0$; when $z \in U_n$, let $f_n$ be the analogous map from $\BD$ to $U_n$. We say that $(U_n,z)$ converges to $(U,z)$ as $n\to \infty$ in the sense of Carath\'eodory if $z \in U_n$ for $n$ large enough and $f_n$ converges to $f$ uniformly on compact subsets of $\BD$. The convergence does not depend on the choice of $z\in U$, so we will often say that $U_n$ converges to $U$ in the sense of Carath\'eodory.

Let us also recall the original definition of Carath\'eodory convergence. 
\begin{definition}
	Let $U\subset \BC$ be a domain. Let $z \in U$. Let $(U_n)_{n\ge 0}$ be a sequence of domains in $\BC$. The sequence $(U_n)_{n\ge 0}$ converges in the sense of Carath\'eodory toward $U$ relative to $z$ if 
	\begin{itemize}
		\item For each compact set $K\subset U$, we have $K\subset U_n$ for all $n$ large enough;
		\item For each connected open set $V\ni z$, if $V \subset U_n$ for infinitely many $n\ge 0$, then $V \subset U$.
	\end{itemize}
\end{definition}
The above definition is more general since it does not assume that the domains are simply connected. The equivalence between the two definitions in the case of simply connected domains is the content of the well-known Carath\'eodory kernel theorem (see, e.g., \cite[Theorem~1.8]{Pom92}).

Furthermore, by the Arzel\`a--Ascoli theorem, the convergence in the sense of Carath\'eodory implies the uniform convergence of $f_n^{-1}$ toward $f^{-1}$ on every compact subset of $U$. 

Finally, let us recall a result that we proved in the previous paper \cite{kkmt2026cle4_part1}.
\begin{proposition}[{\cite[Proposition~2.12]{kkmt2026cle4_part1}}]\label{prop BCLE rho to zero}
	Recall that for each $t \in [0, \tau_0]$, $C_t(0)$ is the connected component of the unexplored region containing the origin. Then, as $t\to 0$, $C_t(0)$ converges in probability to $\BD$ in the sense of Carath\'eodory. Moreover, the maximum of the diameters of the connected components of $\BD \setminus \overline{C_t(0)}$ converges to zero in probability as $t\to 0$. In particular, for all $\delta>0$, with probability $1-o(1)$ as $t \to 0$, for all $a \in \partial \BD$, there exists $x \in \partial C_t(0) \cap \partial \BD$ such that $\vert a - x \vert \le \delta$.
\end{proposition}

\subsection{Four-arm estimates for the bichordal SLE$_4$}

For all $z_1, z_2, z_3, z_4 \in \partial\BD$ ordered counterclockwise, recall that a bichordal SLE$_4$ on $\BD$ with link pattern $\{\{z_1, z_2\}, \{z_3, z_4\}\}$, also sometimes called $2$-SLE$_4$, is made of two random curves, each one connecting one of the two pairs of points $\{z_1, z_2\}, \{z_3, z_4\}$, so that conditionally on one of the two curves, the other is a chordal SLE$_4$ in the connected component of the complement of the first curve whose closure contains the corresponding pair of points. There is a unique law with these properties \cite{IG2,beffara2021uniqueness}. It is a particular case of the multiple chordal SLEs defined in \cite{KL07} (see also \cite{NonsimSLERange,MultiSLE} for a construction of multiple SLE using CLE). We mention two useful four-arm estimates for bichordal SLE$_4$ which we are going to use in the subsequent sections. The first one (\Cref{lem:bichordal-bulk-4A}) estimates the probability that both curves of the bichordal SLE$_4$ intersect a small Euclidean ball centered at an interior point of the domain where the bichordal SLE$_4$ is defined. The second one (\Cref{lem:bichordal-boundary-4A}) estimates the probability that both curves of the bichordal SLE$_4$ intersect a small Euclidean ball centered at a boundary point. \Cref{lem:bichordal-bulk-4A} follows immediately from \cite[Theorem~1.1]{zhan2020two}. As for \Cref{lem:bichordal-boundary-4A}, we will give its proof in \Cref{sec:cle_4_four_arm_exponents}.

\begin{lemma}\label{lem:bichordal-bulk-4A}
    Let $z_1, z_2, z_3, z_4$ be points on $\partial\BD$ ordered counterclockwise. Let $\{\eta_{12}, \eta_{34}\}$ be a bichordal SLE$_4$ in $\BD$ with link pattern $\{\{z_1, z_2\}, \{z_3, z_4\}\}$. Then, there exists a universal constant $C \in (0,\infty)$ such that the following holds for all $\varepsilon \in (0,1)$. The probability that both $\eta_{12}$ and $\eta_{34}$ intersect $B_\varepsilon(0)$ is at most $C \varepsilon^2$. More generally, let $U \subsetneq \BC$ be a simply connected domain and let $z_1, z_2, z_3, z_4$ be prime ends of $U$ whose corresponding points of $\partial\BD$ are ordered counterclockwise. Let $\{\eta_{12}, \eta_{34}\}$ be a bichordal SLE$_4$ in $U$ with link pattern $\{\{z_1, z_2\}, \{z_3, z_4\}\}$. Then, there exists a universal constant $C \in (0,\infty)$ such that the following holds. Let $z \in U$ and $0 < \varepsilon < r < \dist(z, \partial U)$. The probability that both $\eta_{12}$ and $\eta_{34}$ intersect $B_\varepsilon(z)$ is at most $C (\varepsilon/r)^2$.
\end{lemma}

\begin{proof}
    The first assertion follows from \cite[Theorem~1.1]{zhan2020two}. The second assertion follows from the first assertion, together with the observation that if $\phi$ is a conformal map from $U$ onto $\BD$ with $\phi(z) = 0$, then $\zeta \mapsto \phi(z + \dist(z, \partial U)\zeta)$ maps $\BD$ into $\BD$ and fixes $0$, so the Schwarz lemma gives $\phi(B_\varepsilon(z)) \subset B_{\varepsilon/r}(0)$.
\end{proof}

\begin{lemma}\label{lem:bichordal-boundary-4A}
    Let $z_1, z_2, z_3, z_4$ be points on $\partial \BD$ ordered counterclockwise. Let $\{\eta_{12}, \eta_{34}\}$ be a bichordal SLE$_4$ in $\BD$ with link pattern $\{\{z_1, z_2\}, \{z_3, z_4\}\}$. Fix $\delta \in (0,1)$ and let $z_0 \in \partial \BD$ be such that $\dist(z_0, \{z_1,z_2,z_3,z_4\}) \geq \delta$. Then, there exists a constant $C =  C(\delta) > 0$, uniform over all such configurations, such that for all $\varepsilon > 0$, the probability that both $\eta_{12}$ and $\eta_{34}$ intersect $B_{\varepsilon}(z_0)$ is at most $C \varepsilon^4$. (For $\varepsilon$ bounded below, the estimate holds trivially after enlarging $C(\delta)$.)
\end{lemma}

\subsection{Locality property}

Throughout this subsection and the next, $D$ denotes a weak geodesic CLE$_4$ metric coupling in the sense of \Cref{def:weak_axioms}. In the present subsection, we record the fact that metric balls satisfy a locality property. Roughly speaking, this means that conditionally on a metric ball, the loop configuration in the complementary domain is a non-nested CLE$_4$, and the metric on this region provided by the coupling of Axiom~\eqref{it:weak_axiom_locality} is distributed as the metric associated with a CLE$_4$ there. This result is reminiscent of those in \cite{gwynne2020local} and \cite[Section~7]{TightSimCLE}.

\begin{lemma}\label{lem:metric-ball-locality}
Let $U \subsetneq \BC$ be a simply connected domain. Let $z \in U$ be deterministic, and let $V \subset U$ be a deterministic simply connected domain. Write $\{V_j\}_j$ for the connected components of $V^\star$ (recall~\eqref{eq:def-V-star} for the definition of $V^\star$). Let $(\Gamma_U, D_{\Gamma_U}^U, \{D_{\Gamma_U|_{V_j}}^{V_j}\}_j)$ be coupled as in \Cref{def:weak_axioms}, Axiom~\eqref{it:weak_axiom_locality} (locality). Let $\tau$ be a stopping time for the filtration $\{\mathcal{F}_t\}_{t \ge 0}$, where
\begin{equation*}
	\mathcal{F}_t \defeq \sigma\left(\SCB_s(\SCL(z); D_{\Gamma_U}^U), \Gamma_U|_{\SCB_s(\SCL(z); D_{\Gamma_U}^U)} : s \le t\right).
\end{equation*}
Then, conditionally on the $\sigma$-algebra generated by~\eqref{eq:weak_axiom_locality} and by $\SCB_\tau(\SCL(z); D_{\Gamma_U}^U)$, on the event that $\SCB_\tau(\SCL(z); D_{\Gamma_U}^U) \cap V^\star = \emptyset$, the pairs $(\Gamma_U|_{V_j}, D_{\Gamma_U|_{V_j}}^{V_j})$ are independent and their conditional laws are those of $(\Gamma_{V_j}, D_{\Gamma_{V_j}}^{V_j})$, respectively.

The same is true with a deterministic connected arc of $\partial U$ in place of $\SCL(z)$.
\end{lemma}

\begin{proof}
This follows immediately from \Cref{def:weak_axioms}, Axiom~\eqref{it:weak_axiom_locality} (locality), together with the fact that, on the event that $\SCB_\tau(\SCL(z); D_{\Gamma_U}^U) \cap V^\star = \emptyset$, the metric ball $\SCB_\tau(\SCL(z); D_{\Gamma_U}^U)$ is a.s.\ determined by~\eqref{eq:weak_axiom_locality}. Indeed, $\SCL(z) \not\subset V$ on this event: if $\SCL(z) \subset V_j$, then $\SCB_\tau(\SCL(z); D_{\Gamma_U}^U) \supseteq \mathop{\mathrm{int}}(\SCL(z))$ would meet $V^\star$; and if $\SCL(z) \subset V \setminus V^\star$, then $\SCL(z) \subset \mathop{\mathrm{int}}(\SCL^\prime)$ for some $\SCL^\prime \not\subset V$, contradicting non-nesting. The same applies to every loop contained in $\SCB_\tau(\SCL(z); D_{\Gamma_U}^U)$, so the stopped ball is built from the loops appearing in~\eqref{eq:weak_axiom_locality} and their stopped ball processes.
\end{proof}

A simple consequence of the above lemma is the following.

\begin{lemma}\label{lem:metric-ball-locality-segment}
	Let $\alpha \subset \partial \BD$ be a deterministic connected arc. Let $\tau$ be a stopping time for the filtration generated by 
	\begin{equation*}
		\left(\SCB_t(\alpha; D_{\Gamma_\BD}^\BD), \Gamma_\BD|_{\SCB_t(\alpha; D_{\Gamma_\BD}^\BD)}\right)_{t \ge 0}. 
	\end{equation*}
	Let $\{V_j\}_j$ denote the connected components of $\BD \setminus \SCB_\tau(\alpha; D^\BD_{\Gamma_\BD})$.
	Then, there is a coupling $\{(\Gamma_\BD\vert_{V_j}, D^{V_j}_{\Gamma_\BD \vert_{V_j}})\}_j$ such that, conditionally on $\SCB_\tau(\alpha; D^\BD_{\Gamma_\BD})$, the $(\Gamma_\BD\vert_{V_j}, D^{V_j}_{\Gamma_\BD \vert_{V_j}})$'s are independent and have the law of $(\Gamma_{V_j}, D^{V_j}_{\Gamma_{V_j}})$ respectively. Moreover,
		\begin{itemize}
	\item For all $j$, for all $\SCL_1, \SCL_2 \in \Gamma_\BD|_{V_j}$, we have $D_{\Gamma_\BD}^\BD(\SCL_1, \SCL_2)\le D_{\Gamma_\BD|_{V_j}}^{V_j}(\SCL_1, \SCL_2)$.
	\item For all $j$, for all $\SCL \in \Gamma_\BD|_{V_j}, \ t \in [0, D_{\Gamma_\BD}^\BD(\SCL, \partial V_j)]$, we have $\SCB_t(\SCL; D_{\Gamma_\BD}^\BD) = \SCB_t(\SCL; D_{\Gamma_\BD|_{V_j}}^{V_j})$.
	\end{itemize}
\end{lemma}
\begin{proof}
	We apply \Cref{lem:metric-ball-locality} to the interiors of unions of dyadic squares of the form $[k2^{-n}, (k+1)2^{-n}] \times [\ell2^{-n}, (\ell+1)2^{-n}]$ for $k,\ell \in \BZ$ and $n \in \BZ_{\ge 0}$. Then, we let $n\to \infty$ in order to approximate the $V_j$'s.
\end{proof}
\subsection{Lower semi-continuity}
Next, let us state the following lower semi-continuity property of the metric.
\begin{lemma}\label{lem:lower-semi-continuity}
	Let $U \subsetneq \BC$ be a simply connected domain. Almost surely, simultaneously for all $\SCL \in \Gamma_U$, the map $x \mapsto D_{\Gamma_U}^U(\SCL, x)$ is lower semi-continuous on $\overline{U}$. The same is true if we replace $\SCL$ by a deterministic connected arc $I$ of $\partial U$.
\end{lemma}
\begin{proof}
	Recall from \Cref{subsec:setup} that for all $x \in \overline{U}$, we have
	\[
	D_{\Gamma_U}^U(\SCL, x) = \inf \{t\ge 0 : x\in \SCB_t(\SCL; D_{\Gamma_U}^U) \}.
	\]
	Thus, it suffices to show that the right-hand side in the above display is lower semi-continuous in $x$. Let $x_n \to x$ as $n\to \infty$, and let $t_n = D_{\Gamma_U}^U(\SCL, x_n)$. Let $t = \liminf_{n \to \infty} t_n$. We may assume without loss of generality that $t < \infty$.
	
	By Axiom~\eqref{it:weak_axiom_geodesic}\eqref{it:weak_axiom_geodesic_1}, for all $n\ge 1$ we have $x_n \in \SCB_{t_n}(\SCL; D_{\Gamma_U}^U)$ (indeed, $x_n \in \SCB_s(\SCL; D_{\Gamma_U}^U)$ for every $s > t_n$ by the definition of $t_n$, and we let $s \downarrow t_n$ using right-continuity). Let $\varepsilon > 0$. Then there exists a subsequence $(t_{n_k})_{k \ge 1}$ of $(t_n)_{n \ge 1}$ such that $t_{n_k} \le t + \varepsilon$ for all $k \ge 1$. Since metric balls are non-decreasing in $t$, we have $x_{n_k} \in \SCB_{t_{n_k}}(\SCL; D_{\Gamma_U}^U) \subseteq \SCB_{t + \varepsilon}(\SCL; D_{\Gamma_U}^U)$ for all $k \ge 1$. Since $x_{n_k} \to x$ and the metric balls are closed subsets of $\overline{U}$, we deduce that $x \in \SCB_{t + \varepsilon}(\SCL; D_{\Gamma_U}^U)$. Since this holds for all $\varepsilon > 0$, we have
	\[
	x \in \bigcap_{\varepsilon > 0} \SCB_{t + \varepsilon}(\SCL; D_{\Gamma_U}^U). 
	\]
	By Axiom~\eqref{it:weak_axiom_geodesic}\eqref{it:weak_axiom_geodesic_1}, $\lim_{\varepsilon \downarrow 0} \SCB_{t + \varepsilon}(\SCL; D_{\Gamma_U}^U) = \SCB_t(\SCL; D_{\Gamma_U}^U)$ in the Hausdorff metric. Hence $x \in \SCB_t(\SCL; D_{\Gamma_U}^U)$, as this set is closed, which implies that $D_{\Gamma_U}^U(\SCL, x) \le t$. This completes the proof. The same proof works if we replace $\SCL$ with a boundary arc.
\end{proof}

\section{Distances across rectangles}
\label{sec:distances_across_rectangles}

Throughout this section, $D$ denotes a weak geodesic CLE$_4$ metric coupling in the sense of \Cref{def:weak_axioms}. The present section is devoted to proving the following result (\Cref{prop:distance_across_rectangle}), which, combined with the locality property, implies that the distance between the top and bottom sides of an $r \times 1$ rectangle is stochastically dominated by $1/r$ times a geometric random variable with a universal success probability when $r$ is large.

\begin{lemma}\label{rem:rectangle-monotonicity}
    For $0 < s < t$, the distance between the top and bottom sides of a $t \times 1$ rectangle is stochastically dominated by that of an $s \times 1$ rectangle.
\end{lemma}
\begin{proof}
    Conformally map the $t \times 1$ (resp.\ $s \times 1$) rectangle onto $\BH$ so that the bottom side is mapped to $[1, \infty)$ and the top side to $[0, x(t)]$ (resp.\ $[0, x(s)]$); the monotonicity of the conformal modulus gives $x(s) < x(t)$. By conformal invariance (Axiom~\eqref{it:weak_axiom_conformal_invariance}) we may compare the two distances within a single coupling on $\BH$, and $\SCB_u([0, x(s)]; D_{\Gamma_\BH}^\BH) \subset \SCB_u([0, x(t)]; D_{\Gamma_\BH}^\BH)$ for all $u \ge 0$ by the definition of the balls, so $D_{\Gamma_{\BH}}^{\BH}([0, x(t)], [1, \infty)) \le D_{\Gamma_{\BH}}^{\BH}([0, x(s)], [1, \infty))$ a.s.
\end{proof}

\begin{proposition}\label{prop:distance_across_rectangle}
For each $r > 0$, write
\begin{equation*}
	V_r \defeq \left\{z \in \BC : 0 < \Re(z) < r, \ 0 < \Im(z) < 1\right\};
\end{equation*}
write $T_r$ (resp.\ $B_r$) for the closed top (resp.\ bottom) side of $V_r$; let $\Gamma_{V_r}$ be a non-nested CLE$_4$ in $V_r$. (By convention, the quadrilateral $(V_r; T_r, B_r)$ has conformal modulus $1/r$.) Then there is a deterministic constant $c \in (0, 1)$ such that 
\begin{equation*}
	\BP\!\left\lbrack D_{\Gamma_{V_r}}^{V_r}(T_r, B_r) \le 1/r\right\rbrack \ge c
\end{equation*}
for all sufficiently large $r > 0$. 
\end{proposition}

\begin{corollary}\label{cor:moment_across_rectangle}
For each $p > 0$, there exists $C = C(p) > 0$ such that
\begin{equation*}
	\BE\!\left\lbrack D_{\Gamma_{V_r}}^{V_r}(T_r, B_r)^p\right\rbrack \le C r^{-p}
\end{equation*}
for all sufficiently large $r > 0$. In fact, for all sufficiently large $r > 0$, $r D_{\Gamma_{V_r}}^{V_r}(T_r, B_r)$ is stochastically dominated by a geometric random variable taking values in $\BN$ with a universal success probability.
\end{corollary}

\begin{proof}[Proof of \Cref{cor:moment_across_rectangle} assuming \Cref{prop:distance_across_rectangle}]
\stepn{step:mom-first}{The first stage} Let $c$ be as in \Cref{prop:distance_across_rectangle}. Write $X \defeq \left\lceil rD_{\Gamma_{V_r}}^{V_r}(T_r, B_r)\right\rceil$. We claim that for all $r>0$ sufficiently large, $X$ is stochastically dominated by $Y+1$ where $Y$ is a geometric random variable with success probability $c$ and $\BP[Y = k] = c(1-c)^k$ for $k \ge 0$ (so $Y + 1$ is the $\BN$-valued geometric variable in the statement). Indeed, let $r>0$ be sufficiently large such that the statement of \Cref{prop:distance_across_rectangle} holds. Suppose that we are working on the event that $\SCB_{1/r}(B_r;D_{\Gamma_{V_r}}^{V_r})$ does not touch the top side of $V_r$ and note that the probability of this event is at most $1-c$ by \Cref{prop:distance_across_rectangle}. (This event equals $\{D_{\Gamma_{V_r}}^{V_r}(T_r, B_r) > 1/r\}$ by right-continuity and closedness of the ball process and the symmetry of the set-to-set distance; on its complement, $X \le 1$.) Let $a_1$ (resp.\ $b_1$) be the point on $(\{0\} \times [0,1]) \cap \SCB_{1/r}(B_r;D_{\Gamma_{V_r}}^{V_r})$ (resp.\ $(\{r\} \times [0,1]) \cap \SCB_{1/r}(B_r;D_{\Gamma_{V_r}}^{V_r})$) with the largest imaginary part; these intersections are a.s.\ nonempty, as the bottom corners lie in the closed ball (loops accumulating there are discovered at arbitrarily small times). Moreover, let $\phi_1$ denote the conformal mapping from $U_1$ onto $V_{r_1}$ for some $r_1>0$ such that $\phi_1(a_1) = 0, \phi_1(\ri) = \ri, \phi_1(r+\ri ) = r_1 + \ri $, and $\phi_1(b_1) = r_1$ (such a map exists and is unique, for a unique $r_1 > 0$, by the uniformization of quadrilaterals; see \cite{CoInAhl}), where $U_1$ denotes the connected component of $V_r \setminus \SCB_{1/r}(B_r;D_{\Gamma_{V_r}}^{V_r})$ whose boundary contains $T_r$. By the monotonicity of conformal modulus (cf., e.g., \cite[Section~4-3]{CoInAhl}), we have $r_1 \geq r$. Therefore, combining with \Cref{lem:metric-ball-locality} and \Cref{prop:distance_across_rectangle}, we obtain that, given $U_1$, with conditional probability at least $c$, $D_{\Gamma_{V_{r_1}}}^{V_{r_1}}(T_{r_1}, B_{r_1}) \leq 1 / r$. In particular, given $U_1$, with conditional probability at least $c$, we have
\begin{align*}
D_{\Gamma_{V_r}}^{V_r}(T_r,\SCB_{1/r}(B_r;D_{\Gamma_{V_r}}^{V_r}); U_1) \leq 1 / r,
\end{align*}
and hence $D_{\Gamma_{V_r}}^{V_r}(T_r, B_r) \leq 2 / r$.

\stepn{step:mom-induction}{The inductive step} Suppose that for some $n \ge 2$, we have defined simply connected domains $U_n \subset \cdots \subset U_1 \subset V_r$ and conformal mappings $\phi_1,\ldots,\phi_n$, such that for all $2 \le j \le n$, $U_j$ is the connected component of $U_{j-1} \setminus \phi_{j-1}^{-1}(\SCB_{1/r_{j-1}}(B_{r_{j-1}};D_{\Gamma_{V_{r_{j-1}}}}^{V_{r_{j-1}}}))$ whose boundary contains $T_r$, and $\phi_j$ maps $U_j$ onto $V_{r_j}$ for some $r_j>0$ with $\phi_j(a_j) = 0$, $\phi_j(\ri) = \ri$, $\phi_j(r+\ri ) = r_j+\ri$, and $\phi_j(b_j) = r_j$, where $a_j$ (resp.\ $b_j$) is the point on $(\{0\} \times [0,1]) \cap \phi_{j-1}^{-1}( \SCB_{1/r_{j-1}}(B_{r_{j-1}};D_{\Gamma_{V_{r_{j-1}}}}^{V_{r_{j-1}}}))$ (resp.\ $(\{r\} \times [0,1]) \cap \phi_{j-1}^{-1}(\SCB_{1/r_{j-1}}(B_{r_{j-1}};D_{\Gamma_{V_{r_{j-1}}}}^{V_{r_{j-1}}}))$) with the largest imaginary part, and such that the following holds:
\begin{align*}
 D_{\Gamma_{V_{r_j}}}^{V_{r_j}}(T_{r_j}, B_{r_j}) > 1 / r,\quad \forall j \in [1, n - 1]_\BZ.
\end{align*}
Again, $r_n \geq r$. Therefore, \Cref{lem:metric-ball-locality} together with \Cref{prop:distance_across_rectangle} imply that conditionally on $U_n$, we have, off an event of probability at most $1-c$, that $D_{\Gamma_{V_{r_n}}}^{V_{r_n}}(B_{r_n}, T_{r_n}) \leq 1 / r$. In particular, on this event, we have that
\begin{align*}
D_{\Gamma_{V_r}}^{V_r}(T_r,\phi_{n-1}^{-1}( \SCB_{1/r_{n-1}}(B_{r_{n-1}};D_{\Gamma_{V_{r_{n-1}}}}^{V_{r_{n-1}}})); U_n) \leq 1 / r.
\end{align*}
It follows that $r D_{\Gamma_{V_r}}^{V_r}(T_r, B_r) \leq N+1$, where $N$ is the first $n \in \BN$ such that $D_{\Gamma_{V_{r_n}}}^{V_{r_n}}(B_{r_n}, T_{r_n}) \leq 1 / r$ (each successive ball lies within distance $1/r_j \le 1/r$ of the previous one in the corresponding internal metric, internal distances dominate $D_{\Gamma_{V_r}}^{V_r}$ by Axiom~\eqref{it:weak_axiom_locality}, and these $N+1$ bounds telescope via the triangle inequality). This proves the claim since $N$ is stochastically dominated by a geometric random variable with success probability $c$ (the step-$n$ bound holds conditionally on $\sigma(U_1, \ldots, U_n)$).

\stepn{step:mom-concl}{Conclusion of the proof} Thus, we obtain that $\BE\lbrack D_{\Gamma_{V_r}}^{V_r}(T_r, B_r)^p\rbrack \le \BE\lbrack X^p\rbrack r^{-p} = O(r^{-p})$ as $r \to \infty$. This completes the proof. 
\end{proof}

The primary ingredient in the proof of \Cref{prop:distance_across_rectangle} is the following property of SLE$_4$ with force points.

\begin{lemma}\label{lem:guiding_sle4}
Let $\gamma$ be a continuous path in $\overline\BD$ from $-\ri$ to $\ri$ that meets $\partial \BD$ only at $-\ri$ and $\ri$, and fix $\varepsilon \in (0,1)$. Then there exist constants $c = c(\gamma, \varepsilon) > 0$ and $t_\ast = t_\ast(\gamma, \varepsilon) \in (0, \pi)$ such that for each $t \in (0, t_\ast]$, the following is true: Let $\eta|_{[0, \tau]}$ be an SLE$_4(-2t/\pi; 2t/\pi - 2)$ curve in $\BD$ starting from $-\ri$, with force points at $(-\ri)^-$ and $(-\ri)^+$, targeting $\ri$, and stopped at the first time $\tau$ at which it hits $B_\varepsilon(\ri) \cap \partial\BD$. (Since $-2t/\pi + (2t/\pi - 2) = \kappa - 6$, this curve is target invariant \cite{SLE-CoC}, so the law of $\eta|_{[0,\tau]}$ does not depend on the choice of target.) Then it holds with probability at least $ct$ that $\eta([0, \tau]) \subset B_\varepsilon(\gamma)$. 
\end{lemma}

For the rest of this section, write
\begin{align*}
	V_{x,y} &\defeq \left\{z \in \BC : x < \Re(z) < y, \ 0 < \Im(z) < 1\right\}, \quad \forall x, y \in \BR \text{ with } x < y; \\
	I_x &\defeq \left\{z \in \BC : \Re(z) = x, \ 0 \le \Im(z) \le 1\right\}, \quad \forall x \in \BR.
\end{align*}

\begin{proof}[Proof of \Cref{prop:distance_across_rectangle} assuming \Cref{lem:guiding_sle4}]
\stepn{step:dar-levellines}{Level lines lie in metric balls} By Axiom~\eqref{it:weak_axiom_uniform_exploration}, the loops of $\Gamma_{V_r}$ with labels $t_\SCL \defeq D_{\Gamma_{V_r}}^{V_r}(\SCL, \partial V_r)$ form a CLE$_4$ decorated by its uniform exploration. Let $\Psi$ be a zero-boundary GFF in $V_r$ attached to it via the coupling of \Cref{subsec:gff_labeled_cle_4}, so that the boundary values of $\Psi$ on $\mathop{\mathrm{int}}(\SCL)$ are $\pi - t_\SCL$. First, we observe that for each $t \in (0, \pi)$, the following holds: Let $\eta$ be a level line of $\Psi + t - \pi/2$ starting from a boundary point. (Thus, $\eta$ is an SLE$_4(-2t/\pi; 2t/\pi - 2)$ curve.) Then $\eta$ is contained in the BCLE$_4(-2t/\pi)$ associated with $\Psi$. In particular, each point of $\eta$ is of $D_{\Gamma_{V_r}}^{V_r}$-distance at most $t$ to the boundary: $\eta$ lies in the range $\mathbb{A}_{-t,-t+\pi} = \mathbb{B}_t^1$ of this BCLE$_4$ (\Cref{remark coupling TVS BCLE}), which lies in the region explored by time $t$, and the latter equals $\SCB_t(\partial V_r; D_{\Gamma_{V_r}}^{V_r})$ by \Cref{prop coupling TVS uniform explo} and Axiom~\eqref{it:weak_axiom_uniform_exploration}.

\stepn{step:dar-crossing}{A crossing level line exists with positive probability}

\substepn{step:darc-event}{The crossing event} For each $x \in B_r$, write $\eta_x$ for the level line of $\Psi + 1/(2r) - \pi/2$ starting from $x$, targeting $x+\ri$, and stopped upon exiting $V_{x - 1/4,x + 1/4}$. Consider the event that
\begin{equation}\label{eq:distance_across_rectangle_proof_0}
	\parbox{.85\linewidth}{there exists $x \in [r / 3 + 1, 2r / 3 - 1]_{\BZ}$ such that $\eta_x$ hits $T_r$ before exiting $V_{x - 1/4,x + 1/4}$.}
\end{equation}

\substepn{step:darc-prob}{Comparison with an independent curve} We \emph{claim} that there is a deterministic constant $p \in (0, 1)$ such that~\eqref{eq:distance_across_rectangle_proof_0} holds with probability at least $p$ for all sufficiently large $r > 0$. For each $x \in [r / 3 + 1, 2r / 3 - 1]_{\BZ}$, let $\eta_x^\prime$ be an SLE$_4(-1/(\pi r) ; 1/(\pi r) - 2)$ curve in $V_{x - 1/2,x + 1/2}$ starting from $x$, targeting $x+\ri$, and stopped upon exiting $V_{x - 1/4,x + 1/4}$. By \Cref{lem:guiding_sle4} (applied with a conformal map from $V_{x-1/2,x+1/2}$ onto $\BD$ sending $x$ to $-\ri$ and $x+\ri$ to $\ri$, $\gamma$ the image of the segment from $x$ to $x+\ri$, and $\varepsilon$ small enough that the preimages of $B_\varepsilon(\gamma)$ and $B_\varepsilon(\ri) \cap \partial\BD$ lie in $V_{x-1/4,x+1/4}$ and $T_r$), for all $r>0$ sufficiently large, it holds with probability at least a universal constant multiple of $1/r$ that $\eta_x^\prime$ hits $T_r$ before exiting $V_{x - 1/4,x + 1/4}$. On the other hand, by the Markov property, given $\{\eta_y : y < x\}$ the field is a GFF on the complement of these curves, with boundary data bounded by universal constants uniformly over the realizations of the $\eta_y$, over $x$, and over $r$; since $\eta_x$ and $\eta_x^\prime$ stay in $V_{x - 1/4, x + 1/4}$, at distance $1/4$ from where the two fields differ, \cite[Lemma~2.8]{miller2017intersections} gives that the conditional law of $\eta_x$ given $\{\eta_y : y \in [r / 3 + 1, 2r / 3 - 1]_{\BZ}, \ y < x\}$ and the law of $\eta_x^\prime$ are mutually absolutely continuous, and the Radon--Nikodym derivative is bounded below and above by deterministic constants. Thus, we conclude that the random variable
\begin{equation*}
	\#\!\left\{x \in [r / 3 + 1, 2r / 3 - 1]_{\BZ} : \eta_x \text{ hits } T_r \text{ before exiting } V_{x - 1/4,x + 1/4}\right\}
\end{equation*}
stochastically dominates a binomial random variable with $\#[r / 3 + 1, 2r / 3 - 1]_{\BZ}$ trials and success probability given by a deterministic constant multiple of $1/r$ (see \cite[Lemma~2.6]{GeoLQGnSLE}). Such a binomial random variable is positive with probability bounded below by a universal constant, since its number of trials is of order $r$ and its success probability is of order $1/r$; this completes the proof of the claim.

\stepn{step:dar-concl}{Conclusion of the proof} Now, by the following \Cref{lem:distance_across_rectangle-proof} (applied with $p/4$ in place of $p$, together with the left-right reflection symmetry and a union bound), we may choose a sufficiently small $\delta > 0$ so that
\begin{equation}\label{eq:distance_across_rectangle_proof_1}
	\BP\!\left\lbrack D_{\Gamma_{V_r}}^{V_r}(I_0, I_{r/3}) \wedge D_{\Gamma_{V_r}}^{V_r}(I_{2r/3}, I_r) \ge \delta\right\rbrack \ge 1 - p/2
\end{equation}
for all sufficiently large $r > 0$. Thus, it suffices to show that, on the event that~\eqref{eq:distance_across_rectangle_proof_0} and~\eqref{eq:distance_across_rectangle_proof_1} hold, we have $D_{\Gamma_{V_r}}^{V_r}(T_r, B_r) \le 1/r$. 
Indeed, suppose that~\eqref{eq:distance_across_rectangle_proof_0} and~\eqref{eq:distance_across_rectangle_proof_1} hold and note that this event has probability at least $p/2$ for all $r>0$ sufficiently large. Fix $x \in [r / 3 +1, 2r/3 - 1]_{\BZ}$ such that $\eta_x$ hits $T_r$ before exiting $V_{x-1/4, x+1/4}$. Also, it follows from the discussion of the first paragraph of the proof that each point of $\eta_x$ has $D_{\Gamma_{V_r}}^{V_r}$-distance at most $1 / (2r)$ to $\partial V_r$. On the other hand, when $1 / r < \delta$ and since $\eta_x$ does not intersect $I_{r/3} \cup I_{2r/3}$, it follows from~\eqref{eq:distance_across_rectangle_proof_1} that any point of $\eta_x$ has $D_{\Gamma_{V_r}}^{V_r}$-distance at least $1/r$ from both the left and right sides of $V_r$ (by Axiom~\eqref{it:weak_axiom_geodesic}\eqref{it:weak_axiom_geodesic_0}, a ball from $I_0$ reaching $\eta_x$ is connected, so it meets the closed segment $I_{r/3}$, forcing $D_{\Gamma_{V_r}}^{V_r}(I_0, I_{r/3}) < 1/r < \delta$; likewise on the right). Thus, we conclude that each point of $\eta_x$ has $D_{\Gamma_{V_r}}^{V_r}$-distance at most $1/(2r)$ to either $T_r$ or $B_r$. This, together with the lower semi-continuity (\Cref{lem:lower-semi-continuity}), implies that there exists a point on $\eta_x$ that has $D_{\Gamma_{V_r}}^{V_r}$-distance at most $1/(2r)$ to both $T_r$ and $B_r$. 
\end{proof}

\begin{lemma}\label{lem:distance_across_rectangle-proof}
    For each $p \in (0, 1)$, we may choose a sufficiently small $\delta = \delta(p) > 0$ and a sufficiently large $r_\ast = r_\ast(p) > 0$ such that
    \begin{equation*}
    	\BP\!\left\lbrack D_{\Gamma_{V_r}}^{V_r}(I_0, I_{r/3}) \ge \delta\right\rbrack \ge 1 - p, \quad \forall r \ge r_\ast. 
    \end{equation*}
\end{lemma}

\begin{proof}
    First, after a $90^\circ$ rotation, the left-right crossing of $V_r$ is the top-bottom crossing of a $(1/r) \times 1$ rectangle, so \Cref{rem:rectangle-monotonicity} shows that $D_{\Gamma_{V_r}}^{V_r}(I_0, I_r)$ stochastically dominates the crossing distance of the unit square for $r \ge 1$; the latter is a.s.\ positive (by right-continuity of the ball process at $0$), so there exists $\delta > 0$ such that $\BP\lbrack D_{\Gamma_{V_r}}^{V_r}(I_0, I_r) \ge \delta\rbrack \ge 1 - p/2$ for all $r \ge 1$. Since $D_{\Gamma_{V_r}}^{V_r}(I_0, I_{r/3})$ is equal to the $D_{\Gamma_{V_r}|_{V_{0,r/3}^\star}}^{V_{0,r/3}^\star}$-distance between the left and right sides of the conformal quadrilateral $V_{0,r/3}^\star$, it suffices to show that there exists $r_\ast > 0$ such that for each $r \ge r_\ast$, it holds with probability at least $1 - p/2$ that $V_{0,r/3}^\star$ is conformally equivalent to an $r^\prime \times 1$ rectangle with $r^\prime \ge 1$ (i.e., the extremal distance between the left and right sides of $V_{0,r/3}^\star$ is at least $1$) (on this event, by the conditional application of Axiom~\eqref{it:weak_axiom_locality} to the deterministic $V_{0,r/3}$ together with \Cref{rem:rectangle-monotonicity}, the crossing distance stochastically dominates that of the unit square, and the two $1-p/2$ bounds combine). It suffices to show that there exists $r_\ast > 3$ such that for each $r \ge r_\ast$, it holds with probability at least $1 - p/2$ that there is no loop of $\Gamma_{V_r}$ connecting $I_1$ and $I_{r/3}$ (on that event, every left-right crossing of $V_{0,r/3}^\star$ crosses $V_{0,1}$, so the extremal distance is at least $1$). Consider the domain $W \defeq \{z \in \BC : \Re(z) > 0, \ 0 < \Im(z) < 1\}$ and a non-nested CLE$_4$ $\Gamma_W$ in $W$. Couple $\Gamma_{V_r}$ and $\Gamma_W$ by constructing both from a single Brownian loop soup in $W$ and its restriction to $V_r$ (\Cref{subsec:brownian_loop_soup}). Note that if there is a loop of $\Gamma_{V_r}$ connecting $I_1$ and $I_{r/3}$, then by the monotonicity of this coupling (each cluster of the soup in $V_r$ is contained in a cluster of the soup in $W$), there is a loop of $\Gamma_W$ satisfying the same property. On the other hand, we may choose a sufficiently large $r_\ast > 3$ such that for each $r \ge r_\ast$, it holds with probability at least $1 - p/2$ that there is no loop of $\Gamma_W$ connecting $I_1$ and $I_{r/3}$ (continuity from above: the union of the loops of $\Gamma_W$ that intersect $I_1$ is a.s.\ bounded, so these events decrease to a null event as $r \to \infty$). This completes the proof of \Cref{lem:distance_across_rectangle-proof}.
\end{proof}

It remains to prove \Cref{lem:guiding_sle4}. As an intermediate step, we will first prove a version in which we only have one force point.
\begin{lemma}\label{lem:guiding_sle4_one_force_point}
Let $\gamma$ be a continuous path in $\overline\BD$ from $-\ri$ to $\ri$ that meets $\partial \BD$ only at $-\ri$ and $\ri$, and fix $\varepsilon \in (0,1)$. Then there exist $c = c(\gamma, \varepsilon) > 0$ and $t_\ast = t_\ast(\gamma, \varepsilon) \in (0, \pi)$ such that for each $t \in (0, t_\ast]$, the following is true: Suppose that $\eta|_{[0, \tau]}$ is an SLE$_4(2t/\pi - 2)$ curve in $\BD$ starting from $-\ri$ with a single force point located at $(-\ri)^+$ (i.e., immediately counterclockwise of $-\ri$), targeting $\ri$, and stopped upon hitting $B_\varepsilon(\ri) \cap \partial\BD$. Then it holds with probability at least $ct$ that $\eta([0, \tau]) \subset B_\varepsilon(\gamma)$.
\end{lemma}

Before proceeding, we recall how to construct the SLE$_4(\rho)$ process for $\rho > -2$ from a Bessel process (see, e.g., \cite{lawler2008conformally,IG1}). Let $X$ be a Bessel process of dimension $d = 1 + (\rho+2)/2$ driven by a standard Brownian motion~$B$, so that $\mathrm{d}X_s = \mathrm{d}B_s + \frac{\rho+2}{4X_s} \mathrm{d}s$. We define the distance between the force point $V$ and the driving function $W$ as $Z \defeq 2X$. By the Loewner equation, $\mathrm{d}V_s = \frac{2}{Z_s} \mathrm{d}s$, and the driving function satisfies $\mathrm{d}W_s = 2 \mathrm{d}B_s - \frac{\rho}{Z_s} \mathrm{d}s$. Taking a linear combination yields $\mathrm{d}(W_s + \frac{\rho}{2}V_s) = 2 \mathrm{d}B_s$. We can therefore define $V_s \defeq \frac{2}{\rho+2}(2B_s + Z_s)$ and $W_s \defeq V_s - Z_s$ ( so that $W_0 = V_0 = 0$). This construction is well-defined for all $s \ge 0$, even when the curve hits the boundary (which corresponds to $X$ hitting zero), because it avoids directly integrating $1/Z_s$.
    
Furthermore, we recall that the scale function for a Bessel process of dimension $d \in (1,2)$ is given by $s(x) = x^{2-d}$. Consequently, the probability that such a Bessel process starting from $X_0 = \varepsilon \in (0,1)$ hits $1$ before $0$ is exactly $\varepsilon^{2-d}$.

\begin{lemma}
\label{lem:bessel_int_lbd}
For each $d_0 > 1$, there exist constants $c > 0$ and $c_0 \in (0,1)$ such that the following is true. Suppose that $X$ is a Bessel process of dimension $d \in (1,d_0)$ starting from $0$. Then
\[ \BP\!\left[ \int_0^1 X_s^{-1} \mathrm{d}s \leq 1, \ \sup_{0 \le s \le 1} X_s \le c_0^{-1}, \ X_1 \ge c_0\right] \geq c(d-1).\]
\end{lemma}
\begin{proof}
\stepn{step:bes-reduce}{Reduction to $d \in (1,3/2)$} We first observe that for any $d > 1$, the integral $\int_0^1 X_s^{-1} \mathrm{d}s$ is a.s.\ finite. For $d \in [3/2, d_0]$ and fixed $c_0 \in (0,1/2)$, the event of the lemma has positive probability, continuous in $d$, hence bounded below on this compact interval; shrinking $c$ so that this bound is at least $c(d_0 - 1) \ge c(d-1)$ gives the lemma in this range. Therefore, for the remainder of the proof, we may assume without loss of generality that $d \in (1, 3/2)$.

\stepn{step:bes-first}{The process up to time $t_0$} We bound the integral by analyzing the process in two phases, the first up to time $t_0 \defeq ((d-1)/2)^2$. By Brownian scaling, we may write $X_t = \sqrt{t_0}\, \widetilde{X}_{t/t_0}$ for all $t \ge 0$, where $\widetilde{X}$ is a Bessel process of dimension $d$ starting from $0$. Then $X_{t_0} = \frac{d-1}{2} \widetilde{X}_1$, and the integral up to this time satisfies
    \begin{equation*}
        \int_0^{t_0} X_s^{-1} \mathrm{d}s = \sqrt{t_0} \int_0^1 \widetilde{X}_u^{-1} \mathrm{d}u = \frac{d-1}{2} \int_0^1 \widetilde{X}_u^{-1} \mathrm{d}u.
    \end{equation*}

  By integrating the stochastic differential equation for $\widetilde{X}$,  we have the exact identity $\frac{d-1}{2} \int_0^1 \widetilde{X}_u^{-1} \mathrm{d}u = \widetilde{X}_1 - \widetilde{B}_1$,  where $\widetilde{B}$ is a standard one-dimensional Brownian motion starting from zero.  As $d \downarrow 1$,  the pair $(\widetilde{X} ,  \widetilde{X} - \widetilde{B})$ is tight in $C([0,1])^2$ by a simple application of the Burkholder-Davis-Gundy inequality to the square Bessel process $\widetilde{X}^2$ and its associated SDE at the stopping time $\inf\{t\ge 0: \widetilde{X}_t^2 \ge n\}\wedge1$ for $n\ge 1$. Actually, one can see that $\widetilde{X}$ converges in law on $C([0,1])$ to a reflected Brownian motion $\widehat{B}$ (see the discusion following \cite[Definition 1.9, Chapter XI]{RY05} for the density of the semi-group of $\widetilde{X}$ which can be used to identify the limit of $\widetilde{X}$ as $d \downarrow 1$). More precisely (but this is not needed for our purpose), $(\widetilde{X} ,  \widetilde{X} - \widetilde{B})$ converges in law on $C([0,1])^2$ to  $(\widehat{B}, \ell)$, where $\widehat{B}$ is a reflected Brownian motion and $\ell$ is the local-time term in the decomposition $\widehat{B} = W + \ell$, where $W$ is a standard Brownian motion (the decomposition comes from Tanaka's formula, \cite[Theorem 1.2, Chapter VI]{RY05}); the identification of the limit of $\widetilde{X} - \widetilde{B}$ as $\ell$ is obtained via the uniqueness of the solution to the Skorokhod problem. Thus,  we can find constants $c_0 \in (0,1)$ and $M>0$ such that the event $E_1 \defeq \{X_{t_0} \geq c_0 \frac{d-1}{2},  \,  \sup_{0 \leq t \leq t_0} X_t \leq c_0^{-1} \frac{d-1}{2},\,\,\text{and} \,\,  \int_0^{t_0} X_s^{-1} \mathrm{d}s \leq M\}$ has a probability bounded from below by a uniform constant $p_1>0$ for all $d \in (1,3/2)$ (apply the convergence to the open version of $E_1$, whose limiting probability is positive for suitable $c_0$, $M$).

\stepn{step:bes-hit}{Reaching level one} Next, we condition on $E_1$ and apply the Markov property at time $t_0$. The process starts the second phase from $x_0 \defeq X_{t_0} \ge c_0 \frac{d-1}{2}$. By the properties of the scale function $s(x) = x^{2-d}$, the probability that $X$ subsequently hits $1$ before returning to $0$ is exactly $(x_0 \wedge 1)^{2-d}$ (it equals one if $x_0 \ge 1$). Since $2-d < 1$ and $c_0 \frac{d-1}{2} \le 1$, we have $(x_0 \wedge 1)^{2-d} \ge x_0 \wedge 1 \ge c_0 \frac{d-1}{2}$. Thus, the conditional probability of reaching $1$ before $0$ is bounded below by $c_0 \frac{d-1}{2}$.
    
    Conditioned on hitting $1$ before $0$, the process up to the hitting time $\tau_1$ has the law of a Bessel process of dimension $4-d$ 
    (this comes from \cite[Exercise 1.22, Chapter XI]{RY05} and from the Doob $h$-transform construction of the conditioned process, via the harmonic function $s(x) = x^{2-d}$)
    . For $d \in (1, 3/2)$, we have $4-d \in (5/2, 3)$. For such a process, $X_s^{-1}$ is integrable even when starting from $0$. Since starting from $x_0 > 0$ only stochastically decreases the hitting time of $1$ and the corresponding inverse integral compared to starting from $0$, there exist uniform constants $M' > 0$ and $T > 0$ such that the transformed process satisfies $\int_{t_0}^{\tau_1} X_s^{-1} \mathrm{d}s \le M'$ and $\tau_1 - t_0 \le T$ with some uniform probability $p_2 > 0$. These constants can be taken uniform over $d \in (1,3/2)$, again by the  continuity of the Bessel law in its dimension.
    
\stepn{step:bes-concl}{Conclusion of the proof} Combining these estimates, with probability at least $p_1 \times c_0 \frac{d-1}{2} \times p_2 = c'(d-1)$, the process reaches $1$ before time $t_0 + T$ and accumulates a total inverse integral of at most $M + M'$. Now fix $T^\prime > 0$ and $L > 4$ such that $L/8 \ge M + M^\prime + 2T^\prime$ 
and $\lambda \defeq L^2/4 \ge t_0 + T + T^\prime$, so that $M + M^\prime + 2T^\prime + (3/2)\lambda/L \le \sqrt{\lambda}$. After shrinking $c_0$ if necessary, conditionally on the above and with probability uniform over $d \in (1,3/2)$, the process reaches $L$ before time $\tau_1 + T^\prime$ while staying above $1/2$, and then stays in $(2L/3, c_0^{-1}\sqrt{\lambda})$ on $(\tau_L,\lambda)$, where $\tau_L$ is the first time that $X$ hits $L$. On this event $\int_0^{\lambda} X_s^{-1}\, \mathrm{d}s \le \sqrt{\lambda}$, $\sup_{0 \le s \le \lambda} X_s \le c_0^{-1}\sqrt{\lambda}$, and $X_\lambda \ge c_0\sqrt{\lambda}$. Indeed, the third inequality comes from the fact that $X_\lambda \ge 2L/3=(2/3)\sqrt{4\lambda}=(4/3)\sqrt{\lambda}\ge c_0 \sqrt{\lambda}$ and for the integral, we have
\[
\int_0^\lambda X_s^{-1} \mathrm{d}s =\int_0^{\tau_1}X_s^{-1} \mathrm{d}s+ \int_{\tau_1}^{\tau_L}X_s^{-1} \mathrm{d}s + \int_{\tau_L}^\lambda X_s^{-1} \mathrm{d}s \le M+M'+2T'+\frac{3\lambda}{2L}\le \sqrt{\lambda}.
\]
 
Since $(\lambda^{-1/2} X_{\lambda s})_{s \ge 0}$ is a Bessel process of dimension $d$ from $0$, this implies the event of the lemma for the rescaled process, completing the proof.
\end{proof}

\begin{proof}[Proof of \Cref{lem:guiding_sle4_one_force_point}]

\stepn{step:g1f-bessel}{A Bessel estimate} We first consider an $\SLE_4(2t / \pi - 2)$ process in $\BH$ from $0$ to $\infty$ with a single force point at $0^+$.  The distance between the force point and the driving function is given by $Z_s = 2X_s$,  where $X$ is a Bessel process of dimension $d = 1+t / \pi$.  The force point evolves as $\mathrm{d}V_s = X_s^{-1} \mathrm{d}s$,  so its position at time $s$ is $V_s = \int_0^s X_r^{-1} \mathrm{d}r$.  By \Cref{lem:bessel_int_lbd} (with $c_0 \in (0,1)$ as in its statement),  there is a uniform constant $c \in (0,1)$ such that
\begin{align*}
\BP\!\left[\int_0^1 X_r^{-1} \mathrm{d}r \leq 1,\, \sup_{0 \leq s \leq 1} X_s \leq c_0^{-1},\,\,  \text{and} \,\,  X_1 \geq c_0 \right] \geq c (t/\pi).
\end{align*}
By Brownian scaling,  for any $u>0$,  the process $u^{-1} X_{u^2 r}$ is a Bessel process of the same dimension,  yielding
\begin{align*}
\BP\!\left[\int_0^{u^2} X_r^{-1} \mathrm{d}r \leq u,  \,\,  \sup_{0 \leq s \leq u^2} X_s \leq c_0^{-1} u, \,\,  \text{and} \,\,  X_{u^2} \geq c_0 u \right] \geq c(t / \pi).
\end{align*}

\stepn{step:g1f-confine}{Confining the curve up to time $u^2$} Fix $u>0$ (to be chosen) and suppose that we are working on the event that $\int_0^{u^2} X_r^{-1} \mathrm{d}r \leq u$,  $\sup_{0 \leq s \leq u^2} X_s \leq c_0^{-1} u$,  and $X_{u^2} \geq c_0 u$.  We parameterize $\eta$ according to half-plane capacity.  Then,  we have that $0 \leq V_s \leq u$,  $-2c_0^{-1}u \leq W_s \leq u$ for all $0 \leq s \leq u^2$,  and $V_{u^2} - W_{u^2} = 2 X_{u^2} \geq 2 c_0 u$.  Hence,  combining with \cite[Lemma~4.13]{lawler2008conformally},  we obtain that there exists a universal constant $C>1$ such that $\eta([0,u^2]) \subset B_{C c_0^{-1} u}(0)$.  Moreover,  \cite[Corollary~3.44]{lawler2008conformally} implies that $|g_{u^2}(z) - z| \leq 6 C c_0^{-1} u$ for all $z$ lying in the unbounded connected component of $\BH \setminus \eta([0,u^2])$,  where $(g_s)$ denotes the family of conformal maps solving the Loewner equation with driving function $W$.  It follows that if $\psi : \BD \to \BH$ is any fixed conformal transformation such that $\psi(-\ri) = 0$ and $\psi(\ri) = \infty$,  then there exist constants $0<c_3 < c_2 < c_1 < 1/2$ depending only on $\varepsilon,\gamma,\psi$,  and $c_0$,  such that, on the above event, the following hold with $u = c_1 \varepsilon$.  We have that
\begin{align*}
&f_{u^2}^{-1}(\psi(B_{\varepsilon/2}(\gamma))) \cup \eta([0,u^2]) \subset \psi(B_{\varepsilon}(\gamma)) \text{ and } f_{u^2}^{-1}(\psi(B_{\varepsilon/2}(\ri) \cap \partial \BD)) \subset \psi(B_{\varepsilon}(\ri) \cap \partial \BD),\\
&\dist(V_{u^2} - W_{u^2} ,  \psi(B_{c_2 \varepsilon}(\gamma))) \geq c_3 \varepsilon,
\end{align*}
where $f_{u^2}(\cdot) = g_{u^2}(\cdot) - W_{u^2}$.

\stepn{step:g1f-concl}{The continuation after time $u^2$, and conclusion} Note that the conditional law of $\widetilde{\eta} = f_{u^2}(\eta|_{[u^2,\infty)})$ given $\eta|_{[0,u^2]}$ is that of a chordal $\SLE_4(2t / \pi - 2)$ process in $\BH$ from $0$ to $\infty$ with the force point located at $V_{u^2}-W_{u^2}$.  Thus,  \cite[Lemma~2.8]{miller2017intersections} implies that the conditional law of $\widetilde{\eta}$ given $\eta|_{[0,u^2]}$,  stopped at the first time it either exits $\psi(B_{c_2 \varepsilon}(\gamma))$ or hits $\psi(B_{c_2 \varepsilon}(\ri) \cap \partial \BD)$,  is mutually absolutely continuous with respect to the law of an ordinary chordal $\SLE_4$ process in $\BH$ from $0$ to $\infty$.  Moreover,  the corresponding Radon--Nikodym derivatives are uniformly bounded away from $0$ and $\infty$ by constants depending only on $\psi,\gamma,\varepsilon,c_2$,  and $c_3$.  Furthermore,  \cite[Lemma~A.1]{CoInCLERiemSph} implies that the probability that a chordal $\SLE_4$ in $\BH$ from $0$ to $\infty$ hits $\psi(B_{c_2 \varepsilon}(\ri) \cap \partial \BD)$ before exiting $\psi(B_{c_2 \varepsilon}(\gamma))$ is bounded away from $0$ by a constant $p \in (0,1)$ depending only on $\psi,\gamma,\varepsilon$,  and $c_2$.

On these events, the initial segment $\psi^{-1}(\eta([0,u^2]))$ lies in $B_\varepsilon(\gamma)$ by the above displays, and the continuation $\psi^{-1}(f_{u^2}^{-1}(\widetilde\eta))$ stays in $B_\varepsilon(\gamma)$ until hitting $B_\varepsilon(\ri) \cap \partial\BD$ since $c_2 \le 1/2$. Multiplying the two probabilities, $\psi^{-1}(\eta)$ remains in $B_{\varepsilon}(\gamma)$ and hits $B_{\varepsilon}(\ri) \cap \partial \BD$ with probability at least a constant multiple of $t$, completing the proof.
\end{proof}

\begin{proof}[Proof of \Cref{lem:guiding_sle4}]

\stepn{step:gs4-setup}{Setup} Fix $0 < \delta < \delta_0 <1$ (to be chosen).  Let also $\psi : \BD \to \BH$ be a fixed conformal transformation mapping $\BD$ onto $\BH$ such that $\psi(-\ri) = 0$ and $\psi(\ri) = \infty$.  Let $\eta$ be an $\SLE_4(\rho_1 ; \rho_2)$ process in $\BH$ from $0$ to $\infty$ with the force points located at $0^-$ and $0^+$ respectively,  and with $\rho_1 = -2t / \pi$ and $\rho_2 = 2t / \pi -2$.  Then,  the exact same argument used to prove \cite[Proposition~3.1]{DKM24} (using level lines in place of flow lines) implies that $\eta$ converges in probability to $\BR_+$ as $t \downarrow 0$ with respect to the bounded Hausdorff metric (Hausdorff convergence of $\eta \cap B_M(0)$ for every $M > 0$).  Consider the rectangle $R = [-\delta , 2 \delta_0] \times [0,\delta]$ and let $\sigma$ denote the first time that $\eta$ intersects $\partial R \setminus \partial \BH$.  Then,  the convergence of $\eta$ to $\BR_+$ in probability as $t \downarrow 0$ implies that there exists $t_{\star} = t_\star(\delta, \delta_0) > 0$ such that the following holds for all $t \in (0,t_{\star}]$.  With probability at least $1/2$,  we have that $\eta(\sigma) \in \{2\delta_0\} \times [0,\delta]$.  Fix $t \in (0,t_{\star}]$ and let $E$ denote the above event.

We also parameterize $\eta$ by half-plane capacity and let $W$ denote its driving function.  Let $V^1$ and $V^2$ be the force points starting at $0^-$ and $0^+$ associated with the weights $\rho_1$ and $\rho_2$.  Moreover,  we let $\BH_s$ denote the unbounded connected component of $\BH \setminus \eta([0,s])$ for all $s \geq 0$.   We \emph{claim} that if we choose $\delta \in (0,\delta_0)$ sufficiently small (in a way that depends only on $\delta_0$),  we have on the event $E$ that $W_{\sigma} - V_{\sigma}^1 > \delta_0$ and $V_{\sigma}^2 - W_{\sigma} \leq \delta_0 / 100$.  Indeed,  let $\beta = (\beta_s)$ denote a planar Brownian motion which is independent from $\eta$,  and let $\tau_s$ denote the first time that $\beta$ exits $\BH_s$ for all $s \geq 0$.  Let also $(g_s)$ denote the family of conformal maps solving the Loewner equation associated with $W$. All $\BP_{\ri y}$-probabilities below are conditional on $\eta|_{[0,\sigma]}$.

\stepn{step:gs4-hm}{A harmonic measure estimate} Fix $y>1$ large and let $A$ denote the event that $\beta$ exits $\{z \in \BH : \Im(z) > \delta\}$ for the first time on $[-\delta + \delta^{1/2} ,  2\delta_0 - \delta^{1/2}] \times \{\delta\}$ (its left endpoint is positive for small $\delta$; this is intentional and harmless).  Let also $\eta^{\mathrm{L}},  \eta^{\mathrm{R}}$ denote the left and right sides of $\eta$,  respectively.  Then,  the Beurling estimate implies that conditional on $A$,  the probability that $\beta$ exits $\BH_{\sigma}$ for the first time on $\BR \cup \eta^{\mathrm{R}}([0,\sigma])$ is at most $\widetilde{C} \delta^{1/4}$ for some universal constant $\widetilde{C}>1$.  Thus,  we obtain that $\BP_{\ri y}[\beta_{\tau_{\sigma}} \notin \eta^{\mathrm{L}}([0,\sigma]) \,\, |\,\, A] \leq \widetilde{C} \delta^{1/4}$,  which implies that
\begin{align*}
y \BP_{\ri y}[\beta_{\tau_{\sigma}} \in \eta^{\mathrm{L}}([0,\sigma])] \geq y (1 - \widetilde{C} \delta^{1/4}) \BP_{\ri y}[A].
\end{align*}
Note that
\begin{align*}
y \BP_{\ri y}[A] \to \tfrac{1}{\pi}\bigl(2 \delta_0 - 2\delta^{1/2} +\delta\bigr) \quad \text{and} \quad y\BP_{\ri y}[\beta_{\tau_{\sigma}} \in \eta^{\mathrm{L}}([0,\sigma])] \to \tfrac{1}{\pi}\bigl(W_{\sigma} - V_{\sigma}^1\bigr) \quad \text{as} \quad y \to \infty.
\end{align*}
Hence,  we can choose $\delta \in (0,\delta_0)$ sufficiently small (depending only on $\delta_0$) such that $2\delta_0 - 2\delta^{1/2} + \delta > \delta_0$ and so $W_{\sigma} - V_{\sigma}^1 > \delta_0$.  Similarly,  the Beurling estimate implies that conditionally on the event that $\beta$ exits $\{z \in \BH : \Im(z) > \delta\}$ for the first time on $\bigl((-\infty, -\delta - \delta^{1/2}] \cup [2\delta_0 + \delta^{1/2}, \infty)\bigr) \times \{\delta\}$,  the conditional probability that $\beta$ exits $\BH_{\sigma}$ on $\eta^{\mathrm{R}}([0,\sigma])$ is at most $\widetilde{C} \delta^{1/4}$.  Hence,  by possibly taking $\delta \in (0,\delta_0)$ to be smaller (in a way that depends only on $\delta_0$) and arguing as above,  we obtain that $V_{\sigma}^2 - W_{\sigma} < \delta_0 / 100$.  This proves the claim.

\stepn{step:gs4-confine}{Confining the curve after time $\sigma$} Using \cite[Corollary~3.44]{lawler2008conformally} and arguing as in the proof of \Cref{lem:guiding_sle4_one_force_point},  we obtain that we can choose $\delta_0 \in (0,\varepsilon / 2)$ sufficiently small (in a way that depends only on $\varepsilon,\gamma$,  and $\psi$) such that 
\begin{align*}
f_{\sigma}^{-1}(\psi(B_{\varepsilon/2}(\gamma))) \cup (\BH \cap B_{4\delta_0}(0)) \subset \psi(B_{\varepsilon}(\gamma)),\quad f_{\sigma}^{-1}(\psi(B_{\varepsilon/2}(\ri) \cap \partial \BD)) \subset \psi(B_{\varepsilon}(\ri) \cap \partial \BD),
\end{align*}
where $f_t(\cdot) = g_t(\cdot) - W_t$ for all $t \geq 0$.  Moreover,  conditional on $\eta|_{[0,\sigma]}$,  the curve $\widetilde{\eta} = f_{\sigma}(\eta|_{[\sigma,\infty)})$ has the law of an $\SLE_4(\rho_1 ; \rho_2)$ process in $\BH$ from $0$ to $\infty$ with the force points located at $V_{\sigma}^1 - W_{\sigma}$ and $V_{\sigma}^2 - W_{\sigma}$ respectively.  Note that $W_{\sigma} - V_{\sigma}^1 > \delta_0$ and so \cite[Lemma~2.8]{miller2017intersections} implies that the law of $\widetilde{\eta}$ stopped at the first time that it exits $\psi(B_{\delta_0 /2}(\gamma) \setminus (B_{\delta_0 / 2}(\ri) \cap \partial \BD))$ (i.e., the first time that it leaves the tube or reaches the arc) is mutually absolutely continuous with respect to the law of an $\SLE_4(\rho_2)$ process in $\BH$ from $0$ to $\infty$ stopped at the first time that it exits $\psi(B_{\delta_0 / 2}(\gamma) \setminus (B_{\delta_0 / 2}(\ri) \cap \partial \BD))$ and with force point located at $V_{\sigma}^2 - W_{\sigma}$.  Moreover,  the corresponding Radon--Nikodym derivatives are bounded away from $0$ and $\infty$ by deterministic constants depending only on $\psi,\delta_0$,  and $\gamma$.

\stepn{step:gs4-concl}{Conclusion of the proof} Furthermore,  since starting the force point at a non-negative distance $V_{\sigma}^2 - W_{\sigma} \geq 0$ stochastically increases the associated Bessel process and thereby decreases the inverse integral compared to starting at distance zero,  the exact same argument used to prove \Cref{lem:guiding_sle4_one_force_point} works in the case that the force point of the $\SLE_4(2t / \pi -2)$ process is located at $V_{\sigma}^{2} - W_{\sigma}$.  It follows that by possibly taking $t_{\star}>0$ to be smaller (in a way that depends only on $\varepsilon,\delta_0,\gamma$,  and $\psi$),  we have that there exists a constant $c>0$ such that the following holds for all $t \in (0,t_{\star}]$.  With conditional probability at least $ct$ given $\eta|_{[0,\sigma]}$ and on the event $E$,  the curve $\widetilde{\eta}$ intersects $\psi(B_{\delta_0 / 2}(\ri) \cap \partial \BD)$ before exiting $\psi(B_{\delta_0 / 2}(\gamma))$ for the first time.  Since $\BP[E] \geq 1/2$ for all $t \in (0,t_{\star}]$,  this completes the proof of the lemma.
\end{proof}

\section{Hitting two disjoint metric balls}\label{sec:hitting_two_metric_balls}

The main result of the present section (\Cref{prop:hitting-two-metric-balls}) shows that for deterministic $x, y, z \in \BD$, the metric balls around $\SCL(x)$ and $\SCL(y)$ stopped upon hitting $B_\varepsilon(z)$ remain disjoint with probability at most $\varepsilon^{2 + o(1)}$ as $\varepsilon \to 0$. (We will also present a boundary version of this result, stated as \Cref{prop:boundary-hitting-two-metric-balls}.) This will be used later to establish the existence of geodesics. Observe that if a geodesic connecting $\SCL(x)$ and $\SCL(y)$ exists and passes through $B_\varepsilon(z)$, then the aforementioned event must occur with the stopped balls replaced by their left limits.

\subsection{Some facts about extremal distances}
Before proceeding, we recall the following estimate for extremal distances.

\begin{lemma}[{\cite[Proposition~3.69]{lawler2008conformally}}]
	\label{lem:extremal_length}
	For $W>0$, we set $R_{W} \defeq (0,W) \times (0,\pi)$ and let $\partial_1 = [0,\pi \ri]$, $\partial_2 = \partial_{2,W} = [W,W+\pi \ri]$, $\partial_3 = \partial_{3,W} = [0,W]$, and $\partial_4 = \partial_{4,W} = [\pi \ri, W+\pi \ri]$. Let $\tau = \tau_{R_{W}}$ denote the first time that a complex Brownian motion $B$ started from $z \in R_W$ exits $R_{W}$, and set
    \begin{align*}
    	&f_1(z) = f_{1,W}(z) = 2 \min\{\BP_z[B_{\tau} \in \partial_1], \BP_z[B_{\tau} \in \partial_2]\}\\
    	&f_2(z) = f_{2,W}(z) = 2 \min\{\BP_z[B_{\tau} \in \partial_3], \BP_z[B_{\tau} \in \partial_4]\},
    \end{align*}
    and set $\Theta(W) \defeq \Theta(R_{W}; \partial_1, \partial_2) = \sup\{f_1(z) : z \in R_{W}\}$. Then $\Theta(W)$ is a continuous, strictly decreasing function of $W$ such that 
	\begin{align*}
		\Theta(W) = \left(\frac{8}{\pi}\right) e^{-W/2} + O(e^{-W}) \quad \text{as} \quad W \to \infty.
	\end{align*}
	Moreover, the supremum in the definition of $\Theta$ is attained when $z$ is the center $W/2 + \ri \pi/2$; in particular, $f_1(W/2 + \ri\pi/2) = \Theta(W)$, and likewise $\sup\{f_2(z) : z \in R_W\} = \Theta(\pi^2/W)$, by exchanging the two pairs of sides.
\end{lemma}

The next lemma gives an upper bound for the extremal distance between two compact subsets of the unit disk.

\begin{lemma}\label{lem:Beurling}
    There is a universal constant $C > 0$ such that the following is true: Let $K_1, K_2 \subset \overline\BD$ be two disjoint, connected, simply connected, and compact subsets such that $\partial \BD \not\subset K_1 \cup K_2$. Then the extremal distance in $\BD$ between $K_1$ and $K_2$ is at most 
    \begin{equation}\label{eq:Beurling}
        \left.C\middle/\log\left(\frac{\diam(K_1) \wedge \diam(K_2)}{\dist(K_1, K_2)} \vee 1\right)\right., 
    \end{equation}
    where $\diam$ and $\dist$ denote Euclidean diameter and distance, respectively, with the convention that the bound is $+\infty$ when the logarithm vanishes.
\end{lemma}

\begin{figure}[ht!]
    \centering
    \includegraphics[width=.4\linewidth]{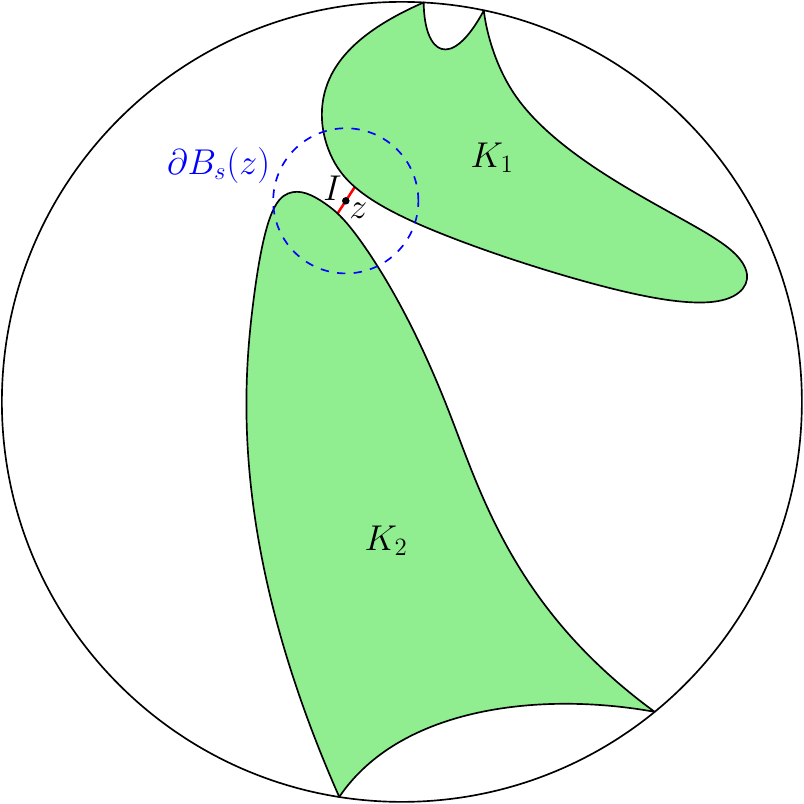}
    \caption{Illustration of the proof of \Cref{lem:Beurling}. The red line segment $I$ is a shortest line segment connecting $K_1$ and $K_2$, and the point $z$ is its midpoint, so that $\dist(z, K_1) = \dist(z, K_2) = \dist(K_1,K_2)/2$. For every $s$ between $\dist(K_1,K_2)/2$ and $(\diam(K_1) \wedge \diam(K_2))/2$, the circle $\partial B_s(z)$ meets both $K_1$ and $K_2$, since each of them has a point within distance $\dist(K_1,K_2)/2$ of $z$ and a point at distance greater than $s$ from $z$. The resulting family of circular arcs joining $K_1$ and $K_2$ has extremal length at most $2\pi/\log\bigl((\diam(K_1) \wedge \diam(K_2))/\dist(K_1,K_2)\bigr)$.}
    \label{fig:Beurling}
\end{figure}

\begin{proof}
\stepn{step:beu-setup}{Setup} Write $d \defeq \dist(K_1, K_2)$ and $D \defeq \diam(K_1) \wedge \diam(K_2)$. We may assume that $d < D$, since otherwise the asserted bound is $+\infty$. Let $z$ be the midpoint of a shortest line segment connecting $K_1$ and $K_2$; since $\BD$ is strictly convex and the segment has positive length, $z \in \BD$, and $\dist(z, K_i) = d/2$ for $i = 1,2$.

\stepn{step:beu-arcs}{Circular arcs joining $K_1$ and $K_2$} Fix $s \in (d/2, D/2)$. Each $K_i$ meets $B_s(z)$, since $\dist(z,K_i) = d/2 < s$, and is not contained in $\overline{B_s(z)}$, since $\diam(K_i) \ge D > 2s$; being connected, it therefore meets $\partial B_s(z)$. As two distinct circles meet in at most two points, $\partial B_s(z) \cap \overline{\BD}$ is a single closed arc, which consequently meets both $K_1$ and $K_2$ (since $s<D/2$) and hence contains a sub-arc $\sigma_s$ joining them whose interior lies in $\BD$.

\stepn{step:beu-el}{The extremal length computation} The extremal distance in question is the extremal length of the family of curves in $\BD$ joining $K_1$ and $K_2$; since $\{\sigma_s\}_{s \in (d/2,\, D/2)}$ is contained in that family, it suffices to bound the extremal length of the smaller family from above. Let $\rho\colon \BD \to [0,\infty]$ be a Borel function and put $L \defeq \inf_s \int_{\sigma_s} \rho(w) \lvert \mathrm{d}w\rvert$. Parameterizing $\sigma_s$ by the angle $\theta$ and applying the Cauchy--Schwarz inequality,
    \begin{equation*}
        L^2 \le \left(\int_{\sigma_s} s \, \mathrm{d}\theta\right)\left(\int_{\sigma_s} \rho^2 s \, \mathrm{d}\theta\right) \le 2\pi s \int_{\sigma_s} \rho^2 s \, \mathrm{d}\theta, \quad \forall s \in (d/2, D/2),
    \end{equation*}
    so that, integrating over $s$,
    \begin{equation*}
        \int_{\BD} \rho^2 \, \mathrm{d}A \ \ge \ \int_{d/2}^{D/2} \int_{\sigma_s} \rho^2 s \, \mathrm{d}\theta \, \mathrm{d}s \ \ge \ \frac{L^2}{2\pi} \int_{d/2}^{D/2} \frac{\mathrm{d}s}{s} \ = \ \frac{L^2}{2\pi} \log\!\left(\frac{D}{d}\right).
    \end{equation*}
    Taking the supremum over $\rho$ shows that the extremal distance is at most $2\pi/\log(D/d)$, which is~\eqref{eq:Beurling} with $C = 2\pi$.
\end{proof}
\subsection{Absolute continuity of the hitting point by a metric ball} Next, let us prove that a deterministically chosen hitting point by a metric ball of a segment is absolutely continuous.
\begin{lemma}\label{lemma hitting point absolutely continuous}
	The following absolute continuity relationships hold:
	\begin{itemize}
	\item Let $\alpha$, $\beta$ be two disjoint non-trivial segments of $\partial \BD$. Let $\tau$ be the first time $t\ge 0$ that the ball $\SCB_t(\alpha; D^\BD_{\Gamma_\BD})$ hits $\beta$. Let $X$ be a deterministically selected point (e.g., the leftmost in some parameterization) in $\beta \cap \SCB_\tau(\alpha;D^\BD_{\Gamma_\BD})$. Then the law of $X$ is absolutely continuous with respect to the Lebesgue measure on $\beta$. 
	\item Similarly, let $\tau'$ be the first time $t\ge 0$ that $\SCB_t(\alpha;D^\BD_{\Gamma_\BD})$ hits $\SCL(0)$ and let $X$ be a deterministically selected point in $\SCL(0) \cap \SCB_{\tau'}(\alpha;D^\BD_{\Gamma_\BD})$. Then, conditionally on $\SCL(0)$, the conditional law of $X$ is absolutely continuous with respect to the harmonic measure on $\SCL(0)$ seen from $\infty$ in $\BC \setminus \overline{\mathop{\mathrm{int}}(\SCL(0))}$.
	\item Furthermore, let $\tau''$ be the first time $t\ge 0$ that $\SCB_t(\SCL(0); D^\BD_{\Gamma_{\BD}})$ hits $\partial \BD$ and let $X$ be a rotationally equivariant selected point on $\partial \BD \cap \SCB_{\tau''}(\SCL(0); D^\BD_{\Gamma_\BD})$ (i.e., the selection commutes with rotations of $\BD$). Then, the law of $X$ is absolutely continuous with respect to the Lebesgue measure on $\partial \BD$ and is actually the uniform distribution on $\partial \BD$.
	\end{itemize}
\end{lemma}

Before proving the above result, let us gather a few useful lemmas. The first one states that a.s., one can find polygonal paths from $0$ to $\partial \BD$ such that the associated unions of loops (without $\SCL(0)$) along these paths are disjoint.
\begin{lemma}\label{lemma existence of disjoint paths from zero to the boundary}
	For all $k\ge 1$, there exist a.s.\ $k$ polygonal paths $\gamma_1, \ldots, \gamma_k:[0,1] \to \overline{\BD}$ from $0$ to $\partial \BD$, such that $\gamma_1((0,1)), \ldots, \gamma_k((0,1))$ are disjoint subsets of $\BD$, which are made of finitely many segments whose vertices lie in $\BQ^2$ except possibly for their endpoints on $\partial \BD$ which  lie in a countable subset of $\partial \BD$, and such that the
	following holds. For all $i \in [1,k]_{\BZ}$, let $F_i$ denote the closure of the union of the domains encircled by loops of $\Gamma_\BD \setminus \{\SCL(0)\}$ which intersect $\gamma_i$. Then, the sets $F_1,\ldots,F_k$ are disjoint.
\end{lemma}
\begin{proof}
	We start with $\gamma_1=[0,1]$. Then, let $F_1$ be defined as in the statement. Since $\gamma_1$ is a line segment and $\mathrm{int}(\SCL(0))$ is an open neighborhood of $0$ disjoint from $F_1$, we can choose a point $z_1 \in \BQ^2 \cap \mathrm{int}(\SCL(0))$ such that the segment $[0, z_1]$ is disjoint from $\gamma_1((0,1))$. Note that $\BC \setminus F_1$ is an open connected set, so it is path-connected by polygonal arcs whose vertices have rational coordinates. Take one such polygonal path from $z_1$ to $\BC \setminus \BD$ and let $\gamma_2$ be the concatenation of $[0, z_1]$ with this path, stopped when it first exits $\overline{\BD}$. Note that the vertices of $\gamma_2$ belong to $\BQ^2$ except for its endpoint on $\partial \BD$, which ensures that these paths can be drawn from a fixed countable set. The desired result follows from a simple induction using that for all $k\ge 1$, the set $\BC \setminus (F_1 \cup\ldots \cup F_k)$ is open, connected, and unbounded. It remains to check that $F_1, \ldots, F_k$ are disjoint, and for this it suffices to treat $F_i$ and $F_{i^\prime}$ with $i < i^\prime$. No loop $\SCL$ contributing to $F_{i^\prime}$ meets $\gamma_i$: it would then contribute to $F_i$, so that $\overline{\mathop{\mathrm{int}}(\SCL)} \subset F_i$, contradicting $\gamma_{i^\prime} \cap F_i = \emptyset$. Nor does $\overline{\mathop{\mathrm{int}}(\SCL)}$ contain a point of $\gamma_i$, since $\gamma_i$ starts at $0 \in \mathop{\mathrm{int}}(\SCL(0))$, which is disjoint from $\overline{\mathop{\mathrm{int}}(\SCL)}$, so $\gamma_i$ would have to cross $\SCL$. The loops contributing to $F_i$ and to $F_{i^\prime}$ are therefore distinct, whence their filled interiors are disjoint. Finally, since only finitely many loops have diameter at least any given $\varepsilon > 0$, all but finitely many of the filled interiors contributing to $F_i$ lie within distance $\varepsilon$ of $\gamma_i$; hence every point of $F_i$ that is not in one of these filled interiors lies on $\gamma_i$. Combining the three observations with $\gamma_i((0,1)) \cap \gamma_{i^\prime}((0,1)) = \emptyset$ gives $F_i \cap F_{i^\prime} = \emptyset$.
\end{proof}

The second one states that a metric ball from a segment in $\partial \BD$ a.s.\ does not hit a given point in the complement of this segment in $\partial \BD$.
\begin{lemma}\label{lemma ball from segment does not hit a given point}
	Let $\alpha$ be a non-trivial segment of $\partial \BD$ and $z \in \partial \BD \setminus \alpha$. Let $t>0$. Then, a.s., $ z \not\in\SCB_t(\alpha;D^\BD_{\Gamma_\BD}) $.
\end{lemma}
\begin{proof}
	Assume by contradiction that $\BP[ z\in \SCB_t(\alpha;D^\BD_{\Gamma_\BD})]=c>0$. By conformal invariance,  this probability does not depend on the choice of $z \in \partial \BD \setminus \alpha$ or on the length of $\alpha$. 
	
	Let $k\ge 1$. Let $P_k$ be the set of families of $k$ deterministic polygonal paths from $0$ to $\partial \BD$ which are disjoint except at their starting point and made of finitely many segments with vertices lying in a countable subset of $\overline{\BD}$. We also assume that the $k$ paths are ordered in counterclockwise order around zero. Let $((\gamma_1^j, \ldots, \gamma_k^j))_{j\ge 1}$ be an enumeration of $P_k$. For all $i \in [1,k]_\BZ$ and $j\ge 1$, let $F^j_i$ be the closure of the union of the domains encircled by the loops of $\Gamma_\BD \setminus \{\SCL(0)\}$ which touch $\gamma^j_i$. Let $\mathcal{A}_j$ be the event that the $F^j_i$'s for $i \in [1,k]_\BZ$ are disjoint.

	Moreover, on the event $\mathcal{A}_j$, let $R^j_1, \ldots, R^j_k$ be the conformal rectangles which are connected components of $\BD \setminus (\overline{\mathrm{int}(\SCL(0))}  \cup \bigcup_{i=1}^k F^j_i) $ and such that for all $i\in [1,k]_\BZ$, the top boundary $\beta^j_i$ of $R^j_i$ is included in $\SCL(0)$, the bottom boundary 
	of $R^j_i$ is included in $\partial \BD$, the left boundary of $R^j_i$ is included in $F^j_{i-1}$ where by convention $F^j_{0}= F^j_k$, and the right boundary of $R^j_i$ is included in $F^j_i$.
	
	For all $j\ge 1$, for all $i \in [1,k]_\BZ$, let $\alpha_i^j$ be a segment of $\partial \BD$ between $\gamma^j_{i-1}$ and $\gamma^j_i$ which does not touch $\gamma^j_{i-1}$ and $\gamma^j_i$ (where by convention $\gamma^j_0= \gamma^j_k$).
	
	Consider the metric balls $\SCB^j_i \defeq \SCB_{t}(\alpha^j_i; D^{R^j_i}_{\Gamma_\BD\vert_{R^{j}_i}})$ for $i \in [1,k]_\BZ$, defined via the coupling in Axiom~\eqref{it:weak_axiom_locality} in \Cref{def:weak_axioms}. Note that, by Axiom~\eqref{it:weak_axiom_locality} in \Cref{def:weak_axioms}, on the event $\mathcal{A}_j$, conditionally on the $F^{j}_i$'s for $1 \le i \le k$ and on $\SCL(0)$, the restrictions $\Gamma_\BD\vert_{R^{j}_i}$ equipped with the distances $D^{R^j_i}_{\Gamma_\BD\vert_{R^{j}_i}}$ are independent and their conditional laws are those of $(\Gamma_{R^j_i}, D^{R^j_i}_{\Gamma_{R^j_i}})$ respectively. 
	
	But, on the event $\mathcal{A}_j$, conditionally on the $F^{j}_i$'s for $1 \le i \le k$ and on $\SCL(0)$, the probability that the ball $\SCB^j_i$ intersects  $\beta^{j}_i$ is larger than $\BP[z \in \SCB_t(\alpha;D^\BD_{\Gamma_\BD})]/2 \ge c/2$. Therefore, on the event $\mathcal{A}_j$, conditionally on $\SCL(0)$ and on the $F^j_i$'s, the number of $i \in [1,k]_\BZ$ such that $\SCB^j_i$ intersects $\SCL(0)$ stochastically dominates a binomial random variable of parameters $k$ and $c/2$. 
	
	 Since the union of the events $\mathcal{A}_j$ has full probability by~\Cref{lemma existence of disjoint paths from zero to the boundary}, and using that $\SCB^j_i \subset \SCB_t(\partial \BD)$ for all $i \in [1,k]_\BZ$ and $j\ge 1$, we conclude that the probability that $\SCB_t(\partial \BD;D^\BD_{\Gamma_\BD})$ intersects $\SCL(0)$ is at least $1-(1-c/2)^k$. Since $k$ can be taken arbitrarily large, we deduce that $\BP[\SCB_t(\partial \BD;D^\BD_{\Gamma_\BD}) \cap \SCL(0) \neq \emptyset ]=1$. This is absurd, because it would entail that the loop surrounding the origin is a.s.\ discovered before time $t$ by the uniform exploration, contradicting the fact that the time of discovery of that loop is an exponential random variable.
\end{proof}

Then, let us state a lemma that says that balls from small segments are small.
\begin{lemma}\label{lemma balls from small segment are small}
	Consider a sequence $\alpha_n$ of segments of $\partial \BD$ whose length goes to zero. Let $t>0$. Then,
	\[
	\diam (\SCB_t(\alpha_n;D^\BD_{\Gamma_\BD})) \mathop{\longrightarrow}\limits_{n\to \infty}^{(\BP)} 0.
	\]
\end{lemma}

\begin{proof}
	By conformal invariance (specifically, rotation invariance, which preserves Euclidean diameters), the law of $\diam(\SCB_t(\alpha_n);D^\BD_{\Gamma_\BD})$ depends only on the length of $\alpha_n$. Thus, using also that $\SCB_t(\cdot)$ is monotone in the source arc (after a rotation, a shorter arc yields a smaller ball), it suffices to focus on a specific family of segments whose lengths go to zero. Let $\alpha$ be a segment containing $1$ and not $-1$. For all $u>0$, let $f_u: \overline{\BD} \to \overline{\BD}$ be the conformal mapping defined by setting for all $z \in \overline{\BD}$,
	\[
	f_u(z) = \frac{z+ e^{-u}}{e^{-u} z+ 1}.
	\]
	Note that $f_u(-1) = -1$ and for all $z \in \overline{\BD}\setminus \{-1\}$ we have
	$f_u(z) \to 1$ as $u\downarrow 0$. This convergence is actually uniform on $\overline{\BD} \setminus B_\varepsilon(-1)$ for all $\varepsilon >0$, given that for all $z \in \overline{\BD} $ such that $\Re (z)>-1+\varepsilon$,
	\[
	\left\vert 1- \frac{z+e^{-u}}{e^{-u} z+1}  \right\vert = \vert 1-e^{-u}\vert \frac{\vert z-1 \vert}{\vert e^{-u} z +1 \vert} \le \vert 1-e^{-u} \vert \frac{2}{e^{-u} \varepsilon}.
	\]
	By~\Cref{lemma ball from segment does not hit a given point}, we know that a.s., $\SCB_t(\alpha;D^\BD_{\Gamma_\BD})$ is a compact set which does not contain $-1$. Thus, there a.s.\ exists some $\varepsilon > 0$ such that $\SCB_t(\alpha;D^\BD_{\Gamma_\BD}) \subset \{ z \in \overline{\BD} : \Re(z) > -1+\varepsilon \}$. Therefore, a.s.,
	\[
	\diam(f_u(\SCB_t(\alpha;D^\BD_{\Gamma_\BD}))) \mathop{\longrightarrow}\limits_{u\to 0} 0.
	\]
	We conclude by conformal invariance, since the random variable $f_u(\SCB_t(\alpha;D^\BD_{\Gamma_\BD}))$ has the same law as $\SCB_t(f_u(\alpha);D^\BD_{\Gamma_\BD})$, and the length of the segment $f_u(\alpha)$ goes to zero as $u \downarrow 0$.
\end{proof}
We use this lemma to prove \Cref{lemma hitting point absolutely continuous}.

\begin{proof}[Proof of \Cref{lemma hitting point absolutely continuous}]
\stepn{step:hp-setup}{Setup in the strip} Let $\varepsilon>0$. It suffices to show that, on an event with probability at least $1-3\varepsilon$, the law of $X$ restricted to this event is absolutely continuous with respect to the Lebesgue measure. 
	
	By conformal invariance, we conformally map $\BD$ to the infinite strip $S \defeq (-\infty, \infty) \times (0,1)$ and we may assume that $\beta$ is the top boundary of the strip, so that $\alpha$ is included in the bottom boundary of the strip, i.e.\@ we can identify the bottom boundary with $\BR$ and write $\alpha= [a,b]  \subset\BR$ with $a<b$. 
	
	Note that $\tau <\infty$ a.s.\ since $D^S_{\Gamma_S}(\alpha, \SCL(\ri/2)) + D^S_{\Gamma_S}(\SCL(\ri/2), \beta)<\infty$ a.s.\ and since for all $t \ge D^S_{\Gamma_S}(\alpha, \SCL(\ri/2))$, we have $\SCB_{t-D^S_{\Gamma_S}(\alpha, \SCL(\ri/2))}(\SCL(\ri/2);D^S_{\Gamma_S}) \subseteq \SCB_t(\alpha)$ (by the triangle inequality for $D^S_{\Gamma_S}$). One can thus take $t\ge 0$ such that $\BP[\tau <t]\ge 1-\varepsilon$.
	
\stepn{step:hp-event}{A good event} Next, using~\Cref{lemma balls from small segment are small}, let $\delta<(b-a)/2$ be small enough that for any segment $\gamma \subset \BR$ of length at most $2\delta$ contained in $[a-2\delta, b+2\delta]$,
	\[
	\BP[\diam(\SCB_t(\gamma;D^S_{\Gamma_S})) \ge 1/2] \le \varepsilon.
	\]
	But note that for all $\theta \in [0, \delta]$,
	\[
	\SCB_t(\alpha_{\mathrm{I}};D^S_{\Gamma_S}) \subseteq \SCB_t( \alpha + \theta;D^S_{\Gamma_S}) \subseteq \SCB_t([a-\delta, a+\delta];D^S_{\Gamma_S}) \cup \SCB_t(\alpha_{\mathrm{I}};D^S_{\Gamma_S}) \cup \SCB_t([b-\delta, b+\delta];D^S_{\Gamma_S}),
	\]
	where $ \alpha_\mathrm{I}\defeq [a+\delta, b-\delta]$.
	
	Let $E$ be the event that $\diam(\SCB_t([a-\delta, a+\delta];D^S_{\Gamma_S})) < 1/2$, $\diam(\SCB_t([b-\delta, b+\delta];D^S_{\Gamma_S})) < 1/2$, and $\tau<t$. By a union bound, $\BP[E] \ge 1-3\varepsilon$. On $E$, neither $\SCB_t([a-\delta, a+\delta];D_{\Gamma_S}^S))$ nor $\SCB_t([b-\delta, b+\delta];D^S_{\Gamma_S})$ intersects the top boundary $\beta$. Thus, by the containment relations above, for any $\theta \in [0, \delta]$, the intersection of $\SCB_s(\alpha+\theta;D^S_{\Gamma_S})$ with $\beta$ is exactly the intersection of $\SCB_s(\alpha_{\mathrm{I}};D^S_{\Gamma_S})$ with $\beta$ for all $s \le t$, which is in turn exactly the intersection of $\SCB_s(\alpha;D^S_{\Gamma_S})$ with $\beta$. Consequently,  $X$ is identically equal to the point on $\beta$ chosen in the same deterministic way with the balls growing from $\alpha$ replaced by the balls growing from $\alpha + \theta$.
	
\stepn{step:hp-first}{Conclusion for the first assertion} Now, let $U$ be an independent uniform random variable in $[0, \delta]$. By the above, on the event $E$, $X$ coincides with the random variable $Y$,  defined in the same deterministic way as $X$ with the balls growing from $\alpha$ replaced by the balls growing from $\alpha + U$.  However, by horizontal translation invariance of the CLE in the strip (which follows from conformal invariance), the unconditional law of $Y$ is the same as the law of $X'+U$, where $X'$ is an independent copy of $X$. Since the law of $X'+U$ is absolutely continuous with respect to the Lebesgue measure, the law of $Y$ is as well. This proves that on an event with probability at least $1-3 \varepsilon$, the law of $X$ is absolutely continuous with respect to the Lebesgue measure.
	
\stepn{step:hp-rest}{The second and third assertions} The second point of the lemma follows from a similar reasoning, except that instead of mapping the disk to the infinite strip, we conformally map the annulus $\BD \setminus \overline{\mathrm{int}(\SCL(0))}$ to an annulus $\BD \setminus \overline{r\BD}$ for some random $r \in (0,1)$, and the role of the translation $U$ is replaced by an independent uniform rotation.
	
	The third point of the lemma comes from the invariance by rotation of $(\Gamma_\BD, D^\BD_{\Gamma_\BD})$.
\end{proof}

\subsection{Metric balls a.s.\ hit a set with a loop} Using \Cref{lemma hitting point absolutely continuous} which was proved in the previous subsection, let us show that a metric ball a.s.\ hits a deterministic compact set by discovering a loop whose interior intersects the compact set.
\begin{lemma}\label{lemma metric ball hitting a loop}
	Let $K \subset \overline\BD$ be a connected compact set. Let $x \in \BD_{\BQ}$. Let $\tau$ be the first time $t$ that $\SCB_t(\SCL(x); D^\BD_{\Gamma_\BD})$ intersects $K$. Then, a.s., there exists exactly one loop $\SCL$ with $\overline{\mathop{\mathrm{int}}(\SCL)} \cap K \neq \emptyset$ such that $D^\BD_{\Gamma_\BD}(\SCL(x), \SCL)= \tau$. The same holds for a ball from a non-trivial arc $\alpha \subset \partial \BD$.
\end{lemma}
 Let us start with a lemma about the harmonic measure of non-loop points is zero.
\begin{lemma}\label{lemma harmonic measure of non loop points is zero}
	Let $K\subset \BD$ be a compact set and let $F$ be the closure of the union of the domains which are encircled by loops of $\Gamma_\BD$ and intersect $K$. Fix $z \in \BD$. Then, a.s.\ on the event $\{z \notin F\}$, the harmonic measure of $\partial F \setminus \bigcup_{\SCL \in \Gamma_\BD} \SCL$ as a subset of $\partial F$ seen from $z$ is zero.
\end{lemma}
\begin{proof}
	Let $(B_t)_{t\ge 0}$ be a planar Brownian motion, independent from $\Gamma_\BD$, starting from $z$ and stopped when it hits $\partial F$. Let $X$ be the first hitting point of $K$ by $(B_t)_{t\ge 0}$. Since $X$ is independent from $\Gamma_\BD$, the point $X$ is a.s.\ encircled by a loop $\SCL(X)$. Let $T$ be the last time that $(B_t)_{t\ge 0}$ intersects $\SCL(X)$ before hitting $X$. By continuity of $(B_t)_{t\in [0, T]}$ we have
	\[
	\varepsilon \defeq \min_{t \in [0, T]} \mathrm{dist}(B_t, K)>0.
	\]
	Furthermore, by local finiteness of $\Gamma_\BD$, there are finitely many loops $\Gamma_\BD$ with diameter at least $\varepsilon/2$. In particular, $(B_t)_{t\in [0, T]}$ intersects finitely many loops of $\Gamma_\BD$ that are included in $F$ and does not intersect the closure of the union of the domains encircled by the loops of $\Gamma_\BD$ that are included in $F$ and have diameter at most $\varepsilon/2$. Moreover, the Brownian motion hits $\partial F$ by time $T$, since $B_T \in \SCL(X)$ and $\mathop{\mathrm{int}}(\SCL(X)) \subset F$; its first hitting point of $\partial F$ therefore lies on one of the finitely many loops of diameter at least $\varepsilon/2$ met by $B|_{[0,T]}$, hence on $\bigcup_{\SCL \in \Gamma_\BD} \SCL$. This concludes the proof.
\end{proof}

Before proving \Cref{lemma metric ball hitting a loop}, let us state the following lemma which says that loops are attached at a single point when they are first hit by the metric ball from a segment.

\begin{lemma}\label{lemma intersection with boundary}
	Let $\alpha \subset\partial \BD$ be a segment. Let $\SCL \in \Gamma_\BD$ and let $t=D^\BD_{\Gamma_\BD}(\alpha, \SCL)$. Then, the intersection $\SCL \cap \SCB^-_t(\alpha; D^\BD_{\Gamma_\BD})$ contains only one point and $\SCL$ is the only loop at $D^\BD_{\Gamma_\BD}$-distance $t$ from $\alpha$.
\end{lemma}
\begin{proof}
\stepn{step:att-locality}{Locality in the complementary component} Let $z \in \BD$. For all $t\ge 0$, let $C_t$ be the connected component of $\BD \setminus\SCB_t(\alpha; D^\BD_{\Gamma_\BD})$ that contains $z$. By \Cref{lem:metric-ball-locality-segment}, we know that there is a coupling of $(\Gamma_\BD, D^\BD_{\Gamma_\BD})$ with $(\Gamma\vert_{C_t}, D^{C_t}_{\Gamma\vert_{C_t}})$ such that
		\begin{itemize}
			\item Conditionally on $\SCB_t(\alpha; D^\BD_{\Gamma_\BD})$, the pair $(\Gamma\vert_{C_t}, D^{C_t}_{\Gamma\vert_{C_t}})$ has the same law as $(\Gamma_{C_t}, D^{C_t}_{\Gamma_{C_t}})$.
			\item For all $\SCL_1, \SCL_2 \in \Gamma_\BD\vert_{C_t}$, we have $D^\BD_{\Gamma_\BD}(\SCL_1, \SCL_2) \le D^{C_t}_{\Gamma_\BD\vert_{C_t}}(\SCL_1, \SCL_2)$.
			\item For all $\SCL \in \Gamma_\BD\vert_{C_t}$, for all $s \in [0, D^\BD_{\Gamma_\BD}(\SCL, \partial C_t)]$, we have $\SCB_s(\SCL; D^\BD_{\Gamma_\BD})= \SCB_s(\SCL; D^{C_t}_{\Gamma_\BD\vert_{C_t}})$.
		\end{itemize}  
		One can thus consider $\SCB_s(\partial C_t; D^{C_t}_{\Gamma_\BD\vert_{C_t}})$ for all $s \ge 0$. Note that for all $\SCL \in \Gamma\vert_{C_t}$, we have $D^\BD_{\Gamma_\BD}(\partial C_t, \SCL)= D^{C_t}_{\Gamma_\BD\vert_{C_t}}(\partial C_t, \SCL)$ (recall that $D^\BD_{\Gamma_\BD}(\partial C_t, \SCL)$ is the first time $s\ge 0$ such that $\SCB_s(\SCL; D^\BD_{\Gamma_\BD}) \cap \partial C_t \neq \emptyset $ by definition).
	
\stepn{step:att-domination}{Stochastic domination by independent samples} For all $t\ge 0$ and $s\ge 0$, we see that $C_{t+s}$ is included in the connected component of $C_t \setminus \SCB_s(\partial C_t; D^{C_t}_{\Gamma_\BD\vert_{C_t}})$ containing $z$. For all $n\ge 1$ and $t\ge 1/n$, consider $Y_n(t)\defeq C_{t-1/n} \setminus C_{t}$. 
		Let $\phi_t\colon C_{t} \to \BD$ be the unique conformal mapping such that $\phi_t(z)=0$ and $\phi_t'(z)>0$. By the previous paragraph and a straightforward induction using \Cref{lem:metric-ball-locality-segment}, the family $(\phi_{(k-1)/n}(Y_n(k/n)))_{k\ge 1}$ is stochastically dominated by a family $(X^n_k)_{k\ge 1}$ of i.i.d.\@ random variables with the same law as $\BD \setminus \overline{C_{1/n}(0)}$, where we recall that $C_{1/n}(0)$ is the connected component containing zero of $\BD \setminus \SCB_{1/n}(\partial \BD; D^\BD_{\Gamma_\BD})$, in the sense that there is a coupling such that for all $k\ge 1$,
		\begin{equation}\label{eq stochastic domination uniform exploration}
			 \phi_{(k-1)/n}(Y_n(k/n)) \subseteq X^n_k
		\end{equation}
		Actually, one can couple the $ \phi_{(k-1)/n}(\SCB_{k/n}(\alpha; D^\BD_{\Gamma_\BD}) \cap \overline{C_{(k-1)/n}})$ for $k\ge 1$ with i.i.d.\@ random variables $\widetilde{X}^n_k$ having the same law as $\SCB_{1/n}(\partial \BD; D^\BD_{\Gamma_\BD})$ so that 
		\[
		\phi_{(k-1)/n}(\SCB_{k/n}(\alpha; D^\BD_{\Gamma_\BD}) \cap \overline{C_{(k-1)/n}}) \subseteq \widetilde{X}^n_k
		\]
		for all $k\ge 1$. More precisely, every loop forming $ \phi_{(k-1)/n}(\SCB_{k/n}(\alpha; D^\BD_{\Gamma_\BD}) \cap \overline{C_{(k-1)/n}})$ is a loop discovered by $\widetilde{X}^n_k$. Moreover, $X^n_k$ is the complement of the connected component of $\BD \setminus \overline{\widetilde{X}^n_k}$ that contains the origin.

\stepn{step:att-convergence}{Convergence of the point processes} 
		Note that $t \mapsto C_t$ is c\`adl\`ag for the Carath\'eodory topology seen from $z$ (because $t\mapsto \SCB_t(\alpha; D^\BD_{\Gamma\vert_\BD})$ is c\`adl\`ag for the Hausdorff distance). In particular, $t\mapsto \phi_t^{-1}$ is c\`adl\`ag for the uniform convergence on compact sets. Therefore, 
		we have the convergence of the point process
		\[
		\sum_{k\ge 1}	 
		\delta_{\left( \frac{k}{n}, \BD \setminus \phi_{(k-1)/n}(Y_n(k/n)) \right)} \mathop{\longrightarrow}\limits_{n\to \infty}^{(\mathrm{a.s.})} \sum_{t>0} 
		\delta_{\left(t, \phi_{t-}(C_{t})\right)},
		\]
		where the second coordinate is equipped with the Carath\'eodory topology seen from zero and where $\BD$ is seen as a cemetery point. More precisely, the above convergence holds in the following sense: for all continuous bounded function $f$ defined on $\BR_{\ge 0}\times \SCO$, where $\SCO$ is the set of simply connected open subsets of $\BD$ containing $0$ (equipped with the Carath\'eodory topology seen from $0$), with compact support in the first variable and vanishing on a neighborhood of $\BD$, the integral of $f$ against the left-hand side converges to the integral of $f$ against the right-hand side. Moreover, since $(\SCB_t(\partial \BD; D^\BD_{\Gamma_\BD}))_{t\ge 0}$ is the uniform exploration, and since the uniform exploration is also càdlàg in the same sense as above, we also get the convergence
		\[
		\sum_{k\ge 1}	 
		\delta_{\left( \frac{k}{n}, \BD \setminus X^n_k \right)} \mathop{\longrightarrow}\limits_{n\to \infty}^{(\mathrm{d})} \sum_{t>0}  \delta_{\left(t, \BD \setminus \overline{\mathop{\mathrm{int}}(\ell_t)}\right)},
		\]

Here $(\ell_t)_{t\ge 0}$ is the PPP of $\SLE_4$ bubbles rooted according to the harmonic measure defined in \Cref{subsec:labeled_cle_4}. Recall that the $\SLE_4$ bubbles only intersect $\partial \BD$ at one point, their root. Combining the two convergences above with the inclusion~\eqref{eq stochastic domination uniform exploration}, we see that there is a coupling of $(\SCB_t(\alpha; D^\BD_{\Gamma_\BD}))_{t\ge 0}$ with a PPP of $\SLE_4$ bubbles rooted according to the harmonic measure so that a.s., for all $t\ge 0$,
		\[
		\BD \setminus \overline{\mathop{\mathrm{int}}(\ell_t)} \subseteq \phi_{t-}(C_t).
		\]
		In particular, any component that is discovered is attached at one point. Moreover, recall that for all $n$, a.s., the loops forming $ \phi_{(k-1)/n}(\SCB_{k/n}(\alpha; D^\BD_{\Gamma_\BD}) \cap \overline{C_{(k-1) /n}})$ are loops of $\widetilde{X}^n_k$. Therefore, a.s., for all $t\ge 0$, when $\phi_{t-}(C_t) \neq \BD$, we have that $\phi_{t-}(\SCB_t(\alpha; D^\BD_{\Gamma_\BD}) \cap \overline{C_{t^-}})$ is made of one loop that is equal to $\ell_t$.
		
\stepn{step:att-concl}{Conclusion of the proof} Thus, a.s., when it is discovered at some time $t\ge 0$, the loop $\SCL(z)$ only intersects $\partial C_{t-}$ at one point. This gives the desired result.
\end{proof}
Let us now prove \Cref{lemma metric ball hitting a loop}.
\begin{proof}[Proof of \Cref{lemma metric ball hitting a loop}]
\stepn{step:mbl-setup}{Setup} Let $F$ be the closure of the union of the domains that are encircled by loops that intersect $K$. Let us deal with the case of the metric ball from a loop $\SCL(x)$ with $x \in \BD_{\BQ}$. Without loss of generality, we may assume that $x \not\in F$. 
		
\stepn{step:mbl-simple}{The case of a simply connected complementary component} For simplicity, let us first explain the proof in the case where $K \cap \partial\BD \neq \emptyset$ (so that here $K \subset \overline\BD$) and on the event that $\SCB_t(\SCL(x)) \cap \partial\BD \neq \emptyset$ for some $t<\tau$. 
		We consider the connected component $R$ of $\BD \setminus (\SCB_t(\SCL(x))  \cup F)$ whose closure intersects $\SCB_t(\SCL(x)) $ and $F$. Note that $R$ is simply connected. We see $R$ as a conformal rectangle whose top boundary is $\partial R \cap F$ and whose bottom boundary is $\partial R \cap \SCB_t(\SCL(x))$. By locality and by the first case of \Cref{lemma hitting point absolutely continuous}, we deduce that, conditionally on $F$ and $\SCB_t(\SCL(x))$, a deterministically chosen hitting point $X$ of $F$ by the metric ball starting from the bottom boundary of $R$ is absolutely continuous with respect to the harmonic measure on $\partial F $ seen from $x$. We deduce that conditionally on $F$, the law of a deterministically chosen hitting point of $F$ by the metric ball from $\SCL(x)$ is absolutely continuous with respect to the harmonic measure on $\partial F$ seen from $x$. By applying \Cref{lemma harmonic measure of non loop points is zero}, we deduce that this hitting point of $F$ is on a loop $\SCL$ encircling a point of $K$. Therefore, $\tau$ is a time of discovery of a loop. By \Cref{lemma intersection with boundary}, no other loop than $\SCL$ is discovered at time $\tau$. This is the desired result.
		
\stepn{step:mbl-general}{The general case} The main difficulty in the general case is that the above connected component $R$ may not be simply connected. We proceed as follows.

		\substepn{step:mbl-case1}{The ball does not disconnect $K$ from $\partial\BD$} In this case, there exists a.s.\ a polygonal path $P$ made of finitely many segments with rational endpoints (recall that $\BQ^2 \cap \partial \BD$ is dense in $\partial \BD$) connecting $F$ to $\partial \BD$ and such that $\SCB_\tau(\SCL(x))\cap P =\emptyset$. Since there are countably many such paths, we may fix one deterministic $P$ and work on the event that its starting point lies in $F$ and $\SCB_\tau(\SCL(x)) \cap P = \emptyset$, taking a union over $P$; on this event $\tau$ is the hitting time of $K \cup P$. We may also assume that $x \notin F$, since otherwise $\tau = 0$ and the statement is trivial. Let $\widetilde{P}$ be the closure of the union of the domains that are encircled by loops of $\Gamma_\BD$ and that intersect $P$. 
		\begin{itemize}
			\item If there exists $t \in (0, \tau)$ such that $\SCB_t (\SCL(x))\cap \partial \BD \neq \emptyset$, then we take the first time $t$ such that this occurs and we consider the connected component $R$ of $\BD \setminus (\SCB_t(\SCL(x)) \cup F \cup \widetilde{P})$. Note that $R$ is simply connected. We view $R$ as a conformal rectangle whose bottom boundary is $\partial R \cap \SCB_t(\SCL(x))$ and whose top boundary is $\partial R \cap (F \cup \widetilde{P})$. By locality and the first case of \Cref{lemma hitting point absolutely continuous}, we deduce that, conditionally on $F$, $\widetilde{P}$ and $\SCB_t(\SCL(x))$, a deterministically chosen hitting point $X$ of $F \cup \widetilde{P}$ by the metric ball starting from the bottom boundary of $R$ is absolutely continuous with respect to the harmonic measure on $\partial F $ seen from $x$. We conclude as in the first part of the proof on the event that $\SCB_\tau(\SCL(x))\cap P =\emptyset$. The desired result then follows by taking a union over $P$.
			\item If for all $t \in (0, \tau)$, we have $\SCB_t(\SCL(x)) \cap \partial \BD = \emptyset$, then $\tau$ is the first hitting time of $\partial \BD \cup \widetilde{P} \cup F$ by the metric ball from $\SCL(x)$. Consider the connected component $C$ of $\BD \setminus (\widetilde{P} \cup F)$. Note that $C$ is simply connected. We apply locality to the simply connected component $C$ and the third point of \Cref{lemma hitting point absolutely continuous} to get the absolute continuity of a deterministically chosen hitting point of $\partial \BD \cup \widetilde{P} \cup F$ by the metric ball starting from $\SCL(x)$. Then, as before, we conclude by applying \Cref{lemma harmonic measure of non loop points is zero} and \Cref{lemma intersection with boundary}.
		\end{itemize}

		\substepn{step:mbl-case2}{The ball disconnects $K$ from $\partial\BD$ before hitting it} In this case, we take the first time $t$ such that this occurs and we let $C$ be the connected component of $\BD \setminus \SCB_t(\SCL(x))$ that contains $K$. By locality, conditionally on $\SCB_t(\SCL(x))$, the intersection $(\SCB_{t+s}(\SCL(x)) \cap C)_{s\ge 0}$ has the law of a uniform exploration of a $\CLE_4$ in $C$. We then conformally map the annulus $C \setminus F$ to an annulus of the form $\BD \setminus \overline{\mathrm{int}(\widetilde{\SCL}(0))}$, where $\widetilde{\SCL}(0)$ has the law of a $\CLE_4$ origin containing loop conditioned so that $\BD \setminus \overline{\mathrm{int}(\widetilde{\SCL}(0))}$ has the same conformal modulus as $C \setminus F$. We then apply the second point of \Cref{lemma hitting point absolutely continuous} with $\alpha =\partial \BD$.
		
		\substepn{step:mbl-case3}{The ball disconnects $K$ from $\partial\BD$ exactly when it hits it} In this case, $\tau$ is a continuity time of $(\SCB_t(\SCL(x)))_{t\ge 0}$ for the Hausdorff topology (since the discovery of a loop cannot disconnect $K$ from the boundary: by \Cref{lemma intersection with boundary}, a discovered loop intersects the boundary of the unexplored region at one point only, so its filled closure is attached to the explored region at a single point, and removing such a set from a domain leaves the complement of the explored region connected). For all $\varepsilon>0$, let 
		\[\tau_\varepsilon\defeq \inf \{t\ge 0: \mathrm{dist}(\SCB_t(\SCL(x)), K)\le \varepsilon\}.\]
		Let $Z_\varepsilon$ be a point in the connected component of $\BD \setminus (\SCB_{\tau_\varepsilon}(\SCL(x)) \cup F)$ whose closure intersects both $\SCB_{\tau_\varepsilon}(\SCL(x)) $ and $ F$ such that $\mathrm{dist}(Z_\varepsilon,  \SCB_{\tau_\varepsilon}(\SCL(x)) ) \le \varepsilon$ and $\mathrm{dist}(Z_\varepsilon,  K ) \le \varepsilon$,  where we choose $Z_{\varepsilon}$ in a way that is measurable with respect to $\SCB_{\tau_{\varepsilon}}(\SCL(x)),  F$,  and $F_1$.
		
		\substepn{step:mbl-avoidbdry}{The ball avoids $\partial\BD$ for every $\varepsilon$} Suppose that $\SCB_{\tau_\varepsilon}(\SCL(x)) \cap \partial \BD= \emptyset $ for all $\varepsilon>0$. Draw a segment $P_1$ whose endpoints have rational coordinates (recall that $\partial \BD \cap \BQ^2$ is dense in $\partial \BD$) and work on the event that it connects $\SCB_{\tau}(\SCL(x))$ to $\partial \BD$ and that $P_1$ does not intersect $F$. Such a segment exists since $\SCB_{\tau}(\SCL(x))$ disconnects $K$ from $\partial \BD$. Let $F_1$ be the closure of the union of the domains that are encircled by loops and that intersect $P_1$. Let $\delta>0$. Let $\widetilde{F_1}$ be the closure of the union of the domains that are encircled by loops and that intersect $\partial B_\delta(F_1)$. Let $\eta>0$. Note that by taking $\delta$ small enough, there exists $\varepsilon_0$ such that for all $\varepsilon\in (0, \varepsilon_0)$, with probability at least $1-\eta$, the point $Z_\varepsilon$ does not belong to $\overline{B_\delta(F_1)}$. Consider the conformal rectangle $R_1^{\mathrm{left}}$ (resp.\@ $R_1^{\mathrm{right}}$) which is given by the connected component of $\BD \setminus (\SCB_{\tau_\varepsilon}(\SCL(x)) \cup F \cup \widetilde{F}_1 \cup F_1)$ that is encircled by $\partial B_\delta(P_1)$, that lies on the left (resp\@ right) of $F_1$ and such that $\partial R_1^{\mathrm{left}}$ (resp.\ $\partial R_1^{\mathrm{right}}$) intersects both $\SCB_{\tau_\varepsilon}(\SCL(x))$ and $\partial \BD$. These conformal rectangles are well defined with probability at least $1-\eta$ by taking $\delta$ small enough. The bottom boundary of $R_1^{\mathrm{left}}$ (resp.\@ $R_1^{\mathrm{right}}$) is $\partial R_1^{\mathrm{left}} \cap \SCB_{\tau_\varepsilon}(\SCL(x))$ (resp.\@ $\partial R_1^{\mathrm{right}} \cap \SCB_{\tau_\varepsilon}(\SCL(x))$) and its top boundary is $\partial R_1^{\mathrm{left}} \cap \partial \BD$ (resp.\@ $\partial R_1^{\mathrm{right}} \cap \partial \BD$). One can couple $(\Gamma_\BD, D^\BD_{\Gamma_{\BD}})$ with $(\Gamma_\BD\vert_{R_1^{\mathrm{left}}}, D^{R_1^{\mathrm{left}}}_{\Gamma_{\BD}\vert_{R_1^{\mathrm{left}}}})$ and $(\Gamma_\BD\vert_{R_1^{\mathrm{right}}}, D^{R_1^{\mathrm{right}}}_{\Gamma_{\BD}\vert_{R_1^{\mathrm{right}}}})$ as in Axiom~\eqref{it:weak_axiom_locality} (locality) in \Cref{def:weak_axioms}. Note that the $D^{R^{\mathrm{left}}_1}_{\Gamma_{\BD}\vert_{R^{\mathrm{left}}_1}}$-distance from the left to the right boundary of $R^{\mathrm{left}}_1$ stochastically dominates a positive random variable which does not depend on $\varepsilon$ since $\SCB_{\tau_\varepsilon}(\SCL(x))$ converges to $\SCB_\tau(\SCL(x))$ in the sense of Hausdorff when $\varepsilon \to 0$. Same for $R^{\mathrm{right}}_1$. Let $A$ be the connected component of $\BD \setminus (\SCB_{\tau_\varepsilon}(\SCL(x)) \cup F \cup F_1)$. Note that $A$ is a topological annulus. Let us conformally map $A$ to an annulus of the form $A'=\BD \setminus r_\varepsilon \overline{\BD}$ using some conformal mapping $\phi_\varepsilon$, where, by \Cref{lem:Beurling}, we have $r_\varepsilon \to 1$ in probability as $\varepsilon \to 0$ on the event that $Z_\varepsilon$ does not belong to $F_1$.  Moreover,  by possibly applying a rotation,  we can assume that $\phi_{\varepsilon}(A \cap B_{\varepsilon}(Z_{\varepsilon})) \cap B_{\varepsilon}([r_{\varepsilon},1]) = \emptyset$ with probability tending to $1$ as $\varepsilon \to 0$.  Let $\widetilde{F}_2$ (resp.\@ $F_2$) be the closure of the union of the domains that are encircled by loops of $\phi_\varepsilon(\Gamma_\BD\vert_A)$ and that intersect $\partial B_{\delta}([r_\varepsilon, 1])$ (resp.\@ $[r_\varepsilon, 1]$). Consider the conformal rectangle $R_2^{\mathrm{left}}$ (resp.\@ $R_2^{\mathrm{right}}$) given by the connected component of $A'\setminus (\widetilde{F}_2 \cup F_2)$ that is surrounded by $\partial B_{\delta}([r_\varepsilon, 1])$, that lies on the left (resp.\@ right) of $F_2$ and whose closure intersects the inner and outer boundaries of $A'$. Its top and bottom boundaries are given by the intersection of $\partial R_2^{\mathrm{left}}$ (resp.\@ $ \partial R_2^{\mathrm{right}}$) with the inner and outer boundaries of $A'$. These conformal rectangles are well defined with probability going to $1$ as $\varepsilon \to 0$ since $r_\varepsilon \to 1$ as $\varepsilon \to 0$.  Also,  we have that $\phi_{\varepsilon}(A \cap B_{\varepsilon}(Z_{\varepsilon})) \cap (R_2^{\mathrm{left}} \cup R_2^{\mathrm{right}}) = \emptyset$ with probability tending to $1$ as $\varepsilon \to 0$.  Note that the $\CLE_4$ metric distances between the left and right boundaries of both $R_2^{\mathrm{left}}$ and $R_2^{\mathrm{right}}$ stochastically dominate a positive random variable whose law does not depend on $\varepsilon$.  By locality, we couple $(\Gamma_\BD, D^\BD_{\Gamma_{\BD}})$, $(\Gamma_\BD\vert_{R_1^{\mathrm{left}}}, D^{R_1^{\mathrm{left}}}_{\Gamma_{\BD}\vert_{R_1^{\mathrm{left}}}})$ and $(\Gamma_\BD\vert_{R_1^{\mathrm{right}}}, D^{R_1^{\mathrm{right}}}_{\Gamma_{\BD}\vert_{R_1^{\mathrm{right}}}})$ with $(\Gamma_\BD\vert_{\phi_\varepsilon^{-1}(R_2^{\mathrm{left}})}, D^{\phi_\varepsilon^{-1}(R_2^{\mathrm{left}})}_{\Gamma_\BD\vert_{\phi_\varepsilon^{-1}(R_2^{\mathrm{left}})}})$ and $(\Gamma_\BD\vert_{\phi_\varepsilon^{-1}(R_2^{\mathrm{right}})}, D^{\phi_\varepsilon^{-1}(R_2^{\mathrm{right}})}_{\Gamma_\BD\vert_{\phi_\varepsilon^{-1}(R_2^{\mathrm{right}})}})$ as in Axiom~\eqref{it:weak_axiom_locality} (locality) in \Cref{def:weak_axioms} so that the $D^{\phi_\varepsilon^{-1}(R_2^{\mathrm{left}})}_{\Gamma_\BD\vert_{\phi_\varepsilon^{-1}(R_2^{\mathrm{left}})}}$-distance from left to right boundary of $\phi_\varepsilon^{-1}(R_2^{\mathrm{left}})$ stochastically dominates a positive random variable whose law does not depend on $\varepsilon$. The same holds when we replace $R_2^{\mathrm{left}}$ by $R_2^{\mathrm{right}}$. Finally, we define the conformal rectangle $R$ as the pre-image by $\phi_\varepsilon$ of the connected component of $A'\setminus {F}_2$ whose closure meets both boundary circles of $A'$. We view it as a conformal rectangle whose top (resp.\@ bottom) boundary is the intersection of $\partial R$ with the inner (resp.\@ outer) boundary of $A$. Again using locality, since $r_\varepsilon \to 1$ in probability as $\varepsilon \downarrow 0$, by \Cref{cor:moment_across_rectangle} we see that the $D^R_{\Gamma_\BD\vert_R}$-distance from bottom to top in $R$ goes to zero in probability.  Furthermore,  the $\CLE_4$ metric distance in $R$ between $\phi_{\varepsilon}(Z_{\varepsilon})$ and the left and right boundaries of $R$ stochastically dominates a positive random variable that does not depend on $\varepsilon$,  and both the $\CLE_4$ metric distances in $\widetilde{A}$ between $Z_{\varepsilon}$ and $F$ and $\SCB_{\tau_{\varepsilon}}(\SCL(x))$ respectively tend to $0$ in probability as $\varepsilon \to 0$,  where $\widetilde{A} = \BD \setminus (\SCB_{\tau_{\varepsilon}}^{\bullet}(\SCL(x)) \cup F_1)$ and $\SCB_{\tau_{\varepsilon}}^{\bullet}(\SCL(x))$ denotes the set of points that $\SCB_{\tau_{\varepsilon}}(\SCL(x))$ disconnects from $\infty$.  It follows that with probability tending to $1$ as $\varepsilon \to 0$,  we have that the $\CLE_4$ metric distance in $\phi_{\varepsilon}^{-1}(R)$ between the top and bottom boundaries of $\phi_{\varepsilon}^{-1}(R)$ is the same as the $\CLE_4$ metric distance in $\widetilde{A}$ between $\SCB_{\tau_{\varepsilon}}(\SCL(x))$ and $F$.  Suppose that we are working on the event that the above hold.
			
\substepn{step:mbl-compare}{Comparing the two explorations} We grow the $\CLE_4$ metric ball from the bottom boundary of $\phi_{\varepsilon}^{-1}(R)$ and stop it at the first time that it hits its top boundary.  Let $X$ be a point on the top boundary of $\phi_{\varepsilon}^{-1}(R)$ which also belongs to the aforementioned metric ball and is chosen according to some arbitrary but fixed deterministic way.  Then,  since we have assumed that the $\CLE_4$ metric distance in $\widetilde{A}$ between $F$ and $\SCB_{\tau_{\varepsilon}}(\SCL(x))$ is the same as the $\CLE_4$ metric distance in $\phi_{\varepsilon}^{-1}(R)$ between the top and bottom boundaries of $\phi_{\varepsilon}^{-1}(R)$,  we obtain that $X \in \SCB_{\tau}(\SCL(x)) \cap F$.  Also,  \Cref{lemma hitting point absolutely continuous} implies that the law of $X$ is absolutely continuous with respect to the harmonic measure of $F$ in $\BD \setminus F$ as seen from $x$.  Therefore,  combining with \Cref{lemma harmonic measure of non loop points is zero},  we obtain that it is a.s.\ the case that $X$ lies in some loop in $\Gamma_{\BD}$ intersecting $K$.  In particular,  we have that there exists a loop $\SCL \in \Gamma_{\BD}$ such that $D_{\Gamma_{\BD}|_{\widetilde{A}}}^{\BD}(\SCB_{\tau_{\varepsilon}}(\SCL(x)) ,  \SCL) = \tau-\tau_{\varepsilon}$.  Hence,  the desired claim follows from combining with \Cref{lemma intersection with boundary}.

		\substepn{step:mbl-meetbdry}{The ball meets $\partial\BD$ for some $\varepsilon$} If there exists $\varepsilon>0$ such that $\SCB_{\tau_\varepsilon}(\SCL(x)) \cap \partial \BD\neq \emptyset $, then we do exactly the same proof as above except that we do not need the segment $P_1$.
\end{proof}

\subsection{Estimate for two disjoint metric balls to hit a small Euclidean ball}
Let us now state the main results of the section, which estimate the probability that two metric balls do not meet before hitting a small Euclidean ball.
\begin{proposition}
\label{prop:hitting-two-metric-balls}
Let $\Gamma_{\BD}$ be a non-nested CLE$_4$ in $\BD$. Let $\delta > 0$. Then for each $x, y, z \in {\BD}$ with $\lvert x - z\rvert \wedge \lvert y - z\rvert \ge \delta$, 
	\begin{equation}\label{eq:hitting-two-metric-balls}
		\BP\!\left\lbrack\SCB^{B_\varepsilon(z)}(\SCL(x); D_{\Gamma_{\BD}}^{\BD}) \cap \SCB^{B_\varepsilon(z)}(\SCL(y); D_{\Gamma_{\BD}}^{\BD}) = \emptyset\right\rbrack \le \varepsilon^{2 + o(1)} \quad \text{as } \varepsilon \to 0, 
	\end{equation}
	at a rate depending only on $\delta$, where $\SCB^{B_\varepsilon(z)}(\SCL; D_{\Gamma_{\BD}}^{\BD}) \defeq \SCB_{D_{\Gamma_{\BD}}^{\BD}(\SCL, B_\varepsilon(z))}(\SCL; D_{\Gamma_{\BD}}^{\BD})$. The same is true with two deterministic and disjoint connected arcs of $\partial\BD$ in place of $\SCL(x)$ and $\SCL(y)$. (In this case, the rate of~\eqref{eq:hitting-two-metric-balls} depends only on the deterministic arcs and is uniform over all $z \in \BD$.)
	\end{proposition}
	
	We also have the following boundary version of \Cref{prop:hitting-two-metric-balls}.
	
	\begin{proposition}
		\label{prop:boundary-hitting-two-metric-balls}
Let $\Gamma_{\BH}$ be a non-nested CLE$_4$ in $\BH$. Let $\delta > 0$. Then for each $x, y \in \BH$ with $\Im(x) \wedge \Im(y) \ge \delta$ and each $z \in \partial\BH$, 
	\begin{equation}\label{eq:boundary-hitting-two-metric-balls}
		\BP\!\left\lbrack\SCB^{B_\varepsilon(z)}(\SCL(x); D_{\Gamma_{\BH}}^{\BH}) \cap \SCB^{B_\varepsilon(z)}(\SCL(y); D_{\Gamma_{\BH}}^{\BH}) = \emptyset\right\rbrack \le \varepsilon^{4 + o(1)} \quad \text{as } \varepsilon \to 0, 
	\end{equation}
	at a rate depending only on $\delta$, where $\SCB^{B_\varepsilon(z)}(\SCL; D_{\Gamma_{\BH}}^{\BH}) \defeq \SCB_{D_{\Gamma_{\BH}}^{\BH}(\SCL, B_\varepsilon(z))}(\SCL; D_{\Gamma_{\BH}}^{\BH})$. The same is true with two deterministic and disjoint connected arcs $I_1$ and $I_2$ of $\partial\BH$ in place of $\SCL(x)$ and $\SCL(y)$. (In this case, the rate of~\eqref{eq:boundary-hitting-two-metric-balls} depends only on $I_1$ and $I_2$ and is uniform over all $z \in \partial\BH \setminus (I_1 \cup I_2)$.)
	\end{proposition}
	
	For simplicity, we only present the proof of \Cref{prop:hitting-two-metric-balls}. \Cref{prop:boundary-hitting-two-metric-balls} then follows from \Cref{lem:bichordal-boundary-4A} combined with an argument analogous to that used in the proof of \Cref{prop:hitting-two-metric-balls}.
	
	We first establish the following lemma. Heuristically, it states that in an $r \times 1$ rectangle, the ``geodesic'' connecting the top and bottom sides passes through a fixed square with probability $O(1/r)$ as $r \to \infty$, even though the existence of such a geodesic has not yet been established at this stage. Recall that $V_r$ is the $r\times 1$ rectangle and that $T_r,B_r$ are the top and bottom sides of $V_r$, respectively.

\begin{lemma}\label{lem:avoiding_two_metric_balls}
Fix constants $0<c_1<c_2<c_2^\prime<c_3<c_4<1$.  Then,  there exists a constant $C>1$ depending only on $c_1,c_2,c_2^\prime,c_3$,  and $c_4$,  such that the following is true for each $0 < r \leq \widetilde{r}$ and each compact set $K \subset V_{\widetilde{r}}$ with $\dist(K , \partial V_{\widetilde{r}}) \geq c_4,  \sup_{z \in K} \Re(z) - \inf_{z \in K} \Re(z) \leq c_4^{-1}$,  and $K \subset V_{c_3 r , \widetilde{r}}$: with probability at least $1 - C / r$,  we have that
\begin{align*}
D_{\Gamma_{V_{\widetilde{r}}}}^{V_{\widetilde{r}}}(K ,  B_{\widetilde{r}}) > D_{\Gamma_{V_{\widetilde{r}}}|_{V_{c_1 r,c_2 r}^{\star}}}^{V_{c_1 r ,c_2  r}^{\star}}(T_{\widetilde{r}} \cap \partial V_{c_1 r, c_2 r}^\star ,  B_{\widetilde{r}} \cap \partial V_{c_1 r, c_2 r}^\star).
\end{align*}
\end{lemma}

The main ingredient in the proof of \Cref{lem:avoiding_two_metric_balls} is \Cref{lem:uniform_exploration_conformal_radius} below. Before stating and proving \Cref{lem:uniform_exploration_conformal_radius}, we establish the following useful elementary lemma.

\begin{lemma}\label{lem:set_almost_the_unit_disk}
Fix $\varepsilon \in (0,1)$. Then there exists $\delta = \delta(\varepsilon) \in (0,1)$ such that the following holds. Let $G \subset \BD$ be a simply connected domain such that $0 \in G$ and let $h : \BD \to G$ be the conformal transformation with $h(0) = 0$ and $h'(0) > 0$.  Suppose that $h'(0) \geq 1 - \delta$.  Then $B_{1-\varepsilon}(0) \subset G$.
\end{lemma}
\begin{proof}
We will argue by contradiction. Suppose that the claim in the statement of the lemma is not true. Then there exists $\varepsilon \in (0,1)$ such that for all $n \in \BN$, there exist $\delta_n \in (0,1/n)$ and a simply connected domain $U_n \subset \BD$ having the following properties. First, $0 \in U_n$. Next, let $h_n : \BD \to U_n$ be the conformal transformation with $h_n(0) =0$ and $h_n'(0) > 0$. Then $h_n'(0) \geq 1 - \delta_n$ and $B_{1-\varepsilon}(0) \not\subset U_n$.

By applying the Montel and Hurwitz theorems, we can assume without loss of generality that there exists an injective holomorphic function $h : \BD \to G$ for some simply connected domain $G \subset \BC$ such that $0 \in G, h(\BD) = G$, and $h_n \to h$ as $n \to \infty$ locally uniformly. Note that $G \subset \BD$ since $h_n(\BD) = U_n \subset \BD$ for all $n$. Moreover, we have that $h(0) = 0$ and $h'(0) \geq 1$ since $h_n'(0) \to h'(0)$ as $n \to \infty$ and $\liminf_{n \to \infty} h_n'(0) \geq 1$. Also, Schwarz's lemma implies that $|h'(0)| \leq 1$, and so $h'(0) = 1$. By applying Schwarz's lemma again, we obtain that $h(z) = z$ for all $z \in \BD$, and so $G = \BD$. Furthermore, for all $n \in \BN$, there exists $w_n \in B_{1-\varepsilon}(0) \setminus U_n$. Since $h_n\to h$ uniformly on compact subsets of $\BD$, we deduce that $U_n \to G$ in the sense of Carath\'eodory, so that $\overline{B_{1-\varepsilon}(0)} \subset U_n$ for all $n$ large enough. This contradicts the choice of the $w_n$'s and completes the proof of the lemma.
\end{proof}

\begin{lemma}\label{lem:uniform_exploration_conformal_radius}
For each $\varepsilon \in (0, 1)$, there exists $C = C(\varepsilon) > 0$ such that the following is true: Let $\Gamma_{\BD}$ be a non-nested CLE$_4$ in $\BD$. Then for each $t \ge 0$, it holds with probability at least $1 - Ct$ that $\SCB_t(\partial\BD; D_{\Gamma_{\BD}}^{\BD})$ does not hit $B_{1 - \varepsilon}(0)$. 
\end{lemma}

\begin{proof}
	By the definition of the uniform exploration from a PPP of SLE$_4$ bubbles (see~\cite{CoInCLEExpl}), $-\log(\mathop{\mathrm{CR}}(0, \BD \setminus \SCB_t(\partial\BD; D_{\Gamma_{\BD}}^{\BD})))$ is an increasing L\'evy process (killed at an exponential time), i.e., a (killed) subordinator, where $\mathop{\mathrm{CR}}(0, \BD \setminus \SCB_t(\partial\BD; D_{\Gamma_{\BD}}^{\BD}))$ denotes the conformal radius of $\BD \setminus \SCB_t(\partial\BD; D_{\Gamma_{\BD}}^{\BD})$ viewed from the origin. Let us denote this killed subordinator by $(S_u)_{u\ge 0}$. Note that the jumps of $S_u$ are given by the time-length of the excursions $e_u$ which have a height smaller than $2\pi$, where we recall from \cite[Section~3.3]{kkmt2026cle4_part1} that $(e_u)_{u\ge 0}$ is the PPP of the excursions of $\theta$, where $\theta$ is two times a Brownian motion in $[0,\pi]$ which is reflected at $0$ and $\pi$. Let us show that $\BE[S_1]<\infty$. To see this, by our description of the jumps of $(S_t)_{t\ge 0}$, it suffices to show that
\begin{equation}\label{eq expectation duration of the excursion finite}
	\int R(e) \one_{M(e)\le \pi } n^+(\mathrm{d}e) < \infty,
\end{equation}
where $n^+$ is the Itô excursion measure (see e.g.~\cite[Chapter XII, Paragraph 2]{RY05}) on positive excursions, where $R(e)$ is the duration of the excursion $e$ defined by $R(e)= \inf\{t>0: e(t)=0\}$ and $M(e)$ is the maximum of $e$. By \cite[Chapter~XII,  Theorem (4.5)]{RY05}, we can write the above integral in the form
\[
\int R(e) \one_{M(e)\le \pi} n^+(\mathrm{d}e) = \frac{1}{2} \int_0^\pi \frac{1}{x^2} \BE[ R(Z^x) ] \mathrm{d} x,
\]
where $Z^x$ is a process whose expected duration is twice the expected hitting time $\tau_x$ of $x>0$ by a $3$-dimensional Bessel process starting from $0$. One can for instance compute this expected time from the Laplace transform given in \cite[Equation (2.1)]{HM13}: the Laplace transform is given by $\BE[e^{-\lambda \tau_x}]= x\sqrt{2\lambda}/{\sinh(x \sqrt{2\lambda})}$ for all $\lambda>0$. By taking the derivative at zero, we deduce that the expectation of $\tau_x$ is $x^2/3$, hence~\eqref{eq expectation duration of the excursion finite} (one can alternatively compute this expectation using optional sampling and a suitable martingale). As a consequence, since $(S_t)_{t\ge 0}$ is a killed subordinator, we deduce that $\BE[S_t] \le ct$ for some constant $c>0$.

Using Markov's inequality, this implies that for each $\delta \in (0, 1)$, there exists $C = C(\delta) > 0$ such that for each $t \ge 0$, it holds with probability at least $1 - Ct$ that $\mathop{\mathrm{CR}}(0, \BD \setminus \SCB_t(\partial\BD; D_{\Gamma_{\BD}}^{\BD})) \ge 1 - \delta$. Note that \Cref{lem:set_almost_the_unit_disk} implies that there exists $\delta = \delta(\varepsilon) \in (0, 1)$ such that $B_{1 - \varepsilon}(0) \subset \BD \setminus \SCB_t(\partial\BD; D_{\Gamma_{\BD}}^{\BD})$ whenever $\mathop{\mathrm{CR}}(0, \BD \setminus \SCB_t(\partial\BD; D_{\Gamma_{\BD}}^{\BD})) \ge 1 - \delta$. 

Thus, we conclude from the above discussion that there exists $C = C(\varepsilon) > 0$ such that for each $t \ge 0$, the probability that $\SCB_t(\partial\BD; D_{\Gamma_\BD}^\BD)$ hits $B_{1-\varepsilon}(0)$ is at most $Ct + \BP[D_{\Gamma_\BD}^\BD(\SCL(0), \partial\BD) \le t]$, the first term bounding the probability that this happens before $\SCL(0)$ is discovered. The latter distance being exponential, the second term is $1 - \re^{-ct} \le ct$, so increasing $C$ completes the proof of \Cref{lem:uniform_exploration_conformal_radius}.
\end{proof}
We now turn to the proof of \Cref{lem:avoiding_two_metric_balls}.

\begin{proof}[Proof of \Cref{lem:avoiding_two_metric_balls}]
\step{step:avoid-outline}{Outline and setup} Let $\{V_{x,y}\}_{x < y}$ and $\{I_x\}_x$ be as in the proof of \Cref{prop:distance_across_rectangle}.  Let $\Gamma_{V_{\widetilde{r}}},  D_{\Gamma_{V_{\widetilde{r}}}|_{V_{c_1 r,c_2 r}^{\star}}}^{V_{c_1 r,c_2 r}^{\star}}$,  and $D_{\Gamma_{V_{\widetilde{r}}}|_{V_{c_2 r,\widetilde{r}}^{\star}}}^{V_{c_2 r,\widetilde{r}}^{\star}}$ be coupled as in \Cref{def:weak_axioms}, Axiom~\eqref{it:weak_axiom_locality} (locality).  

Let us briefly explain the strategy of the proof.  First,  we will show in \Cref{step:avoid-wide} that with very high probability,  the conformal rectangle $V_{c_2 r,\widetilde{r}}^{\star}$ is wide.  The exact same argument works to prove that $V_{c_1 r,c_2 r}^{\star}$ is very wide with high probability.  In \Cref{step:avoid-conclusion},  we will use the results of \Cref{step:avoid-wide} to complete the proof of the lemma.  In particular,  if we set
\begin{align*}
d\defeq D_{\Gamma_{V_{\widetilde{r}}}|_{V_{c_1 r,c_2r}^{\star}}}^{V_{c_1 r,c_2 r}^{\star}}(T_{\widetilde{r}},B_{\widetilde{r}}),
\end{align*}
\Cref{cor:moment_across_rectangle} implies that $d$ is small with high probability,  so combining with \Cref{lem:uniform_exploration_conformal_radius},  we deduce that with very high probability,  the $\CLE_4$ distance in $V_{c_2 r,\widetilde{r}}^{\star}$ between $\partial V_{c_2 r,\widetilde{r}}^{\star}$ and $K$ is larger than $d$.  This will complete the proof of the lemma.

\step{step:avoid-wide}{$V_{c_2 r, \widetilde{r}}^{\star}$ is wide with very high probability}

\substepn{step:avw-levellines}{The level lines $\eta_j$} Fix a constant $c_2' \in (c_2 ,  c_3)$.  We claim that there exist constants $\alpha ,  C_1>0$ depending only on $c_2'$ and $c_2$ such that $V_{c_2' r ,  \widetilde{r}} \subset V_{c_2 r,\widetilde{r}}^{\star}$ with probability at least $1 - C_1 e^{-\alpha r}$.  Indeed,  let $\Psi$ be a zero-boundary GFF on $V_{\widetilde{r}}$ coupled with $\Gamma_{V_{\widetilde{r}}}$ such that each loop $\SCL \in \Gamma_{V_{\widetilde{r}}}$ is labeled $\pm\pi$ (see \Cref{subsec:nested_cle_gff}).  For all $j \in [c_2 r , c_2' r -1]_{\BZ}$,  we set $R_j\defeq[j,j+1] \times [0,1]$ and let $x_j$ (resp.\ $y_j$) be the midpoint of the bottom (resp.\ top) boundary of $R_j$.  Fix $u \in (0,\pi/2)$ and let $\widetilde{\eta}_j$ denote an $\SLE_4(-1+2u/\pi ; -1 -2u / \pi)$ process in $R_j$ from $x_j$ to $y_j$ with the force points located at $x_j^-$ and $x_j^+$ respectively.   Also,  let $\eta_j$ denote the level line of $\Psi$ from $x_j$ to $y_j$ with height $u$,  stopped at the first time that it exits the Euclidean $1/2$-neighborhood of $[x_j,y_j]$.

\substepn{step:avw-lower}{A uniform lower bound} Note that \cite[Lemma~2.5]{miller2017intersections} implies that there exists some universal constant $p \in (0,1)$ such that for all $j \in [c_2 r ,  c_2' r -1]_{\BZ}$,  the curve $\widetilde{\eta}_j$ hits the top boundary of $R_j$ before leaving the Euclidean $1/2$-neighborhood of $[x_j,y_j]$ with probability at least $p$.  Note that the boundary conditions of $\Psi$ on $\cup_j \eta_j$ lie in $[-\pi,\pi]$ and that $\eta_j$ has the law of an $\SLE_4(-1+2u/\pi ; -1-2u/\pi)$ process in $V_{\widetilde{r}}$ from $x_j$ to $y_j$ with force points located at $x_j^-$ and $x_j^+$ respectively.  Therefore,  \cite[Lemma~2.8]{miller2017intersections} implies that the following is true for all $j \in [c_2 r ,  c_2' r -1 ]_{\BZ}$.  The Radon--Nikodym derivative between the conditional law of $\eta_j$ given $(\eta_{\lceil c_2 r \rceil},\cdots,\eta_{j-1},\eta_{j+1},\cdots, \eta_{\lfloor c_2' r - 1\rfloor})$ and that of $\widetilde{\eta}_j$ when both curves are stopped at the first time that they exit $B_{1/2}([x_j,y_j])$ is bounded from above and below by constants depending only on $u$.  In particular,  by possibly decreasing $p \in (0,1)$ (depending only on $u$),  we deduce that for all $j \in [c_2 r ,  c_2' r -1]_{\BZ}$,  the conditional probability that $\eta_j$ hits the top boundary of $R_j$ before exiting $B_{1/2}([x_j,y_j])$ given $(\eta_{\lceil c_2 r \rceil},\cdots,\eta_{j-1},\eta_{j+1},\cdots, \eta_{\lfloor c_2' r - 1\rfloor})$ is at least $p$.

\substepn{step:avw-count}{Counting the successful boxes} It follows that the number of $j \in [c_2 r,c_2' r-1]_{\BZ}$ for which $\eta_j$ hits the top boundary of $R_j$ before exiting $B_{1/2}([x_j,y_j])$ for the first time stochastically dominates a binomial random variable with success probability $p$ and $\lfloor (c_2' - c_2) r -2 \rfloor$ trials.  In particular,  there exist constants $\alpha ,  C_1>0$ depending only on $c_2 , c_2'$, and $u$, such that for all $r>0$ sufficiently large,  it holds with probability at least $1 - C_1 e^{-\alpha r}$ that there exists $j \in [c_2 r ,  c_2' r-1]_{\BZ}$ such that $\eta_j$ hits the top boundary of $R_j$ before exiting $B_{1/2}([x_j,y_j])$ for the first time.  Since there are no loops in $\Gamma_{V_{\widetilde{r}}}$ that cross any of the $\eta_j$'s (see \Cref{subsec:nested_cle_gff} together with the level line interaction rules),  we deduce that on this event,  we have that $V_{c_2' r,\widetilde{r}} \subset V_{c_2 r, \widetilde{r}}^{\star}$.  This proves the claim.

\step{step:avoid-conclusion}{Conclusion of the proof}

\substepn{step:avc-explore}{The exploration of $V_{c_2 r,\widetilde{r}}^{\star}$} Next,  given $V_{c_2 r,\widetilde{r}}^{\star}$,  we consider the uniform exploration
\begin{equation*}
	\left\{\SCB_t\left(\partial V_{c_2 r,\widetilde{r}}^\star; D_{\Gamma_{V_{\widetilde{r}}}|_{V_{c_2 r,\widetilde{r}}^\star}}^{V_{c_2 r,\widetilde{r}}^\star}\right)\right\}_{t \ge 0}
\end{equation*} 
associated with $\Gamma_{V_{\widetilde{r}}}|_{V_{c_2 r,\widetilde{r}}^{\star}}$.  We claim that there exists a constant $C_2>0$ depending only on $c_2,c_2',c_3$,  and $c_4$,  such that for all $t \geq 0$,  conditionally on $V_{c_2 r,\widetilde{r}}^{\star}$ and on the event that $V_{c_2' r , \widetilde{r}} \subset V_{c_2 r,\widetilde{r}}^{\star}$,  it holds with conditional probability at most $C_2 t$ that 
\begin{align}\label{eq:no_hit}
\SCB_t\left(\partial V_{c_2 r,\widetilde{r}}^{\star} ; D_{\Gamma_{V_{\widetilde{r}}}|_{V_{c_2 r,\widetilde{r}}^{\star}}}^{V_{c_2 r,\widetilde{r}}^{\star}}\right) \cap K \neq \emptyset.
\end{align}
Indeed,  let $\phi_r$ denote the conformal mapping from $V_{c_2 r,\widetilde{r}}^{\star}$ onto $\BD$ such that $\phi_r(z) = 0$ and $\phi_r'(z)>0$,  where $z \in K$ is fixed.  Note that the assumptions on $K$ (and since we are working on the event that $V_{c_2' r,\widetilde{r}} \subset V_{c_2 r,\widetilde{r}}^{\star}$) imply that there exists a constant $q \in (0,1)$ (depending only on $c_2,c_2',c_3$,  and $c_4$) such that the following holds for all $r>0$ sufficiently large (how large depending only on $c_2,c_2',c_3$,  and $c_4$).  With probability at least $q$,  a planar Brownian motion starting from $w = \Re(z) + \ri(1-c_4 / 2)$ makes a loop in $V_{c_2 r,\widetilde{r}}^{\star}$ that disconnects $K$ from $\partial V_{\widetilde{r}}$.  Therefore,  there exists a constant $\delta \in (0,1)$ depending only on $q$ such that $\phi_r(K) \subset B_{1-\delta}(0)$.  Hence,  the claim follows from combining \Cref{lem:uniform_exploration_conformal_radius} with the conformal invariance of the uniform $\CLE_4$ exploration.

\substepn{step:avc-expect}{Bounding the expectation of $d$} Next,  we note that by the locality property,  $d$ is conditionally independent of $\Gamma_{V_{\widetilde{r}}}|_{V_{c_2 r,\widetilde{r}}^{\star}}$ and $D_{\Gamma_{V_{\widetilde{r}}}|_{V_{c_2 r,\widetilde{r}}^{\star}}}^{V_{c_2 r,\widetilde{r}}^{\star}}$ given $V_{c_1 r, c_2 r}^{\star}$.  Thus,  conditionally on $V_{c_1 r ,  c_2 r}^{\star},  V_{c_2 r,\widetilde{r}}^{\star}$,  and $d$,  and on the event that $V_{c_2' r ,\widetilde{r}} \subset V_{c_2 r,\widetilde{r}}^{\star}$,  the probability that~\eqref{eq:no_hit} holds with $t=d$ is at most $C_2 d$.  Moreover,  for fixed constants $c_1 < \tilde{c}_1 < \tilde{c}_2 < c_2$,  we have that $d \one_{\{V_{\tilde{c}_1 r ,\tilde{c}_2 r} \subset V_{c_1 r,c_2 r}^{\star}\}}$ is stochastically dominated by $D_{\Gamma_{V_{(\tilde{c}_2-\tilde{c}_1)r}}}^{V_{(\tilde{c}_2-\tilde{c}_1)r}}(T_{(\tilde{c}_2 - \tilde{c}_1)r} ,  B_{(\tilde{c}_2 - \tilde{c}_1)r})$.  Thus,  combining with \Cref{cor:moment_across_rectangle},  we obtain that there exists a constant $C_3>0$ depending only on $c_1,\tilde{c}_1,\tilde{c}_2$,  and $c_2$,  such that
\begin{align*}
\BE\lbrack d \one_{\{V_{\tilde{c}_1 r,\tilde{c}_2 r} \subset V_{c_1 r, c_2 r}^{\star}\}}\rbrack \leq \BE\left\lbrack D_{\Gamma_{V_{(\tilde{c}_2-\tilde{c}_1)r}}}^{V_{(\tilde{c}_2-\tilde{c}_1)r}}(T_{(\tilde{c}_2-\tilde{c}_1)r}, B_{(\tilde{c}_2-\tilde{c}_1)r})\right \rbrack  \leq \frac{C_3}{r}
\end{align*}
for all $r>0$ sufficiently large.  It follows that there exists a constant $C_4>0$ depending only on $c_1,\tilde{c}_1,\tilde{c}_2,c_2,c_2',c_3$,  and $c_4$,  such that for all sufficiently large $r>0$ (how large depending only on $c_1,\tilde{c}_1,\tilde{c}_2,c_2,c_2',c_3$,  and $c_4$),  the following holds with probability at least $1-C_4 / r$: on the event that $V_{c_2' r,\widetilde{r}} \subset V_{c_2 r,\widetilde{r}}^{\star}$ and $V_{\tilde{c}_1 r ,  \tilde{c}_2 r} \subset V_{c_1 r,c_2 r}^{\star}$,  we have that~\eqref{eq:no_hit} does not hold with $t=d$.

\substepn{step:avc-combine}{Combining the estimates} By applying the same argument used to deduce that $V_{c_2' r,\widetilde{r}} \subset V_{c_2 r,\widetilde{r}}^{\star}$ with high probability,  we obtain (possibly by decreasing $\alpha$ and increasing $C_1$ in a way that depends only on $c_1,\tilde{c}_1,\tilde{c}_2,c_2,c_2'$) that 
\begin{align*}
\BP\lbrack \{V_{c_2' r,\widetilde{r}} \subset V_{c_2 r,\widetilde{r}}^{\star}\} \cap \{V_{\tilde{c}_1 r,\tilde{c}_2 r} \subset V_{c_1 r,c_2 r}^{\star}\} \rbrack \geq 1 - C_1 e^{-\alpha r}.
\end{align*}
Therefore,  combining these estimates shows that there exists a constant $C_5>0$ (depending only on $c_1,\tilde{c}_1,\tilde{c}_2,c_2,c_2',c_3$,  and $c_4$) such that for all $r>0$ sufficiently large,  the metric ball $\SCB_d\left(\partial V_{c_2 r,\widetilde{r}}^{\star} ; D_{\Gamma_{V_{\widetilde{r}}}|_{V_{c_2 r,\widetilde{r}}^{\star}}}^{V_{c_2 r,\widetilde{r}}^{\star}}\right)$ does not hit $K$ with probability at least $1-C_5 / r$.  If this occurs,  we have that
\begin{align*}
D_{\Gamma_{V_{\widetilde{r}}}}^{V_{\widetilde{r}}}(B_{\widetilde{r}} ,  K) > D_{\Gamma_{V_{\widetilde{r}}}|_{V_{c_1 r,c_2 r}^{\star}}}^{V_{c_1 r,c_2 r}^{\star}}(T_{\widetilde{r}} ,  B_{\widetilde{r}}).
\end{align*}
This completes the proof of the lemma.
\end{proof}
Another tool in the proof of \Cref{prop:hitting-two-metric-balls} is the following lemma.
\begin{lemma}\label{lem:from-CLE-2-SLE}
Let $\Gamma_{\BH}$ be a non-nested CLE$_4$ in $\BH$. Let $\gamma$ be an independent chordal SLE$_{8/3}$ curve in $\BH$ from $0$ to $\infty$. Write $U$ for the connected component of $\BH \setminus \gamma$ that is to the right of $\gamma$. Write $\eta \defeq \partial U^\star \setminus \partial\BH$, where $U^\star$ is defined by~\eqref{eq:def-V-star} with the random domain $U$ in place of $V$. Then $\eta$ is a chordal SLE$_4$ curve in $\BH$ from $0$ to $\infty$, and given $\eta$, $\Gamma_{\BH}|_{U^\star}$ is conditionally a non-nested CLE$_4$ in $U^\star$.
\end{lemma}
\begin{proof}
See \cite[Theorem~1]{MR3035764}.
\end{proof}

Finally,  we state and prove one more lemma before giving the proof of \Cref{prop:hitting-two-metric-balls}.  Recall from \cite{IG1} that if $(A,\Psi)$ is a coupling between a closed set $A \subset \overline{\BD}$ and a GFF $\Psi$ on $\BD$,  then $A$ is said to be a \textbf{local set} of $\Psi$ if there exists a law on pairs $(A,\Psi_1)$ where $\Psi_1$ is a distribution on $\BD$ with $\Psi_1|_{\BD \setminus A}$ harmonic such that a sample from $(A,\Psi)$ can be produced as follows.  First,  we choose the pair $(A,\Psi_1)$,  and then we sample an instance $\Psi_2$ of the zero-boundary GFF on $\BD \setminus A$ and then setting $\Psi = \Psi_1 + \Psi_2$.  Then,  we have the following lemma which states that $\CLE_4$ metric balls are local sets for the GFF $\Psi$ on $\BD$ for an appropriate coupling between $\Gamma_{\BD}$ and $\Psi$.

\begin{lemma}\label{lem:metric_ball_local}
Fix distinct points $x,y \in \BD$ and $\delta \in (0,1)$.  Let also $I,J$ be two disjoint segments connecting $x$ and $y$ respectively to $\partial\BD$.  Then,  there exists a coupling between $(\Gamma_{\BD} ,  D_{\Gamma_{\BD}}^{\BD})$ and a zero-boundary GFF $\Psi$ on $\BD$ such that $\Psi$ and $\Gamma_{\BD}$ are coupled as in \Cref{subsec:nested_cle_gff} and the following holds.  Let $I^{\star}$ (resp.\ $J^{\star}$) denote the closure of the union of the loops in $\Gamma_{\BD}$ that intersect $I$ (resp.\ $J$),  each intersected with the closed Euclidean $\delta$-neighborhood of $I$ (resp.\ $J$).  Let also $\tau$ (resp.\ $\sigma$) be a stopping time with respect to the filtration generated by $(\SCB_t(\SCL(x) ; D^{\BD}_{\Gamma_{\BD}}))_{t \geq 0}$ (resp.\ $(\SCB_t(\SCL(y) ; D_{\Gamma_{\BD}}^{\BD}))_{t \geq 0}$).  Consider the set
\begin{align*}
A\defeq\SCB_{\tau}(\SCL(x) ; D_{\Gamma_{\BD}}^{\BD}) \cup \SCB_{\sigma}(\SCL(y) ; D_{\Gamma_{\BD}}^{\BD}) \cup I^{\star} \cup J^{\star}.
\end{align*}
Then,  $A$ is local for $\Psi$.
\end{lemma}

\begin{proof}
\stepn{step:mbloc-coupling}{Construction of the coupling} First,  we will construct the coupling between $(\Gamma_{\BD} ,  D_{\Gamma_{\BD}}^{\BD})$ and $\Psi$.  Let $\mathcal{P}$ denote the collection of polygonal paths $P : [0,1] \to \overline{\BD}$ such that $P((0,1)) \subset \BD,  \{P(0) ,  P(1)\} \subset \{e^{\ri\theta} : \theta \in [0,2\pi) \cap \BQ\},  P|_{(0,1)}$ connects points in $\BQ^2 \cap \BD$,  and $P \cap (I \cup J) =\emptyset$.  For all $P \in \mathcal{P}$,  we let $V_{I,P}$ (resp.\ $V_{J,P}$) denote the connected component of $\BD \setminus P$ containing $I$ (resp.\ $J$).

Since $\mathcal{P}$ is a countable family of paths,  \Cref{def:weak_axioms}, Axiom~\eqref{it:weak_axiom_locality} implies that we can find a coupling 
\begin{align*}
\mathcal{K} = \left(\Gamma_{\BD} ,  D_{\Gamma_{\BD}}^{\BD} ,  \left \{D_{\Gamma_{\BD}|_{V_{I,P}^j}}^{V_{I,P}^j}\right\}_j, \left \{D_{\Gamma_{\BD}|_{V_{J,P}^j}}^{V_{J,P}^j}\right\}_j \right)_{P \in \mathcal{P}}
\end{align*}
such that the conditions in \Cref{def:weak_axioms}, Axiom~\eqref{it:weak_axiom_locality} hold simultaneously for all $V \in \{V_{I,P} ,  V_{J,P}\}$ and all $P \in \mathcal{P}$. Indeed, the coupling for a given $P$ is provided by Axiom~\eqref{it:weak_axiom_locality}, and one can couple terms of $\mathcal{K}$ for $P \in \mathcal{P}$ by taking the elements of the family $\mathcal{K}$ conditionally independent given $(\Gamma_\BD, D^\BD_{\Gamma_\BD})$. Moreover,  inside every loop in $\Gamma_{\BD}$,  we sample independently from $\mathcal{K}$ and from each other a nested $\CLE_4$ in the region bounded by the loop,  and so we obtain a nested $\CLE_4$ $\overline{\Gamma}_{\BD}$ on $\BD$.  Furthermore,  conditionally on $(\overline{\Gamma}_{\BD} ,  \mathcal{K})$,  we sample conditionally independent Rademacher random variables $\{X_{\SCL}\}_{\SCL \in \overline{\Gamma}_{\BD}}$ and construct a random field $\Psi$ on $\BD$ as in \Cref{subsec:nested_cle_gff}.  For the rest of the proof,  we will assume that we are working with the above coupling between $\mathcal{K}$ and $\Psi$.

\stepn{step:mbloc-criterion}{Verifying the criterion for local sets} Next,  we will show that $A$ is local for $\Psi$.  Indeed,  by \cite[Lemma~3.6]{IG1},  it suffices to show that for every fixed and deterministic open set $U \subset \BD$,  we have that the following holds.  Let $\Psi_0$ denote the projection of $\Psi$ onto $H_0^1(U)$ and set $\Psi_1 = \Psi - \Psi_0$.  Then,  conditionally on $\Psi_1$,  we have that the event $\{A \cap U = \emptyset\}$ is independent from $\Psi_0$. 

Fix an open set $U \subset \BD$.  Note that if either $I \cap \overline{U} \neq \emptyset$ or $J \cap \overline{U} \neq \emptyset$,  we have that $A \cap U \neq \emptyset$ a.s.,  and so without loss of generality,  we can assume that $(I \cup J) \cap \overline{U} = \emptyset$.  
Let $\delta_n \uparrow \delta$. For all $n\ge 1$, let $I^\star_n= I^\star \cap \overline{B_{\delta_n}(I)}$ and $J^\star_n= J^\star \cap \overline{B_{\delta_n}(J)}$. 
Suppose first that $I$ and $J$ lie in different connected components of $\BD \setminus \overline{U}$.  Then,  we have that $A \cap U = \emptyset$ if and only if for all $n \in \BN$,  there exist $P,Q \in \mathcal{P}$ such that $V_{I,P} \cap V_{J,Q} = \emptyset,  V_{I,P} \cup V_{J,Q} \subset \BD \setminus \overline{U}$,  and
\begin{align*}
\SCB_{\tau}(\SCL(x) ; D_{\Gamma_{\BD}}^{\BD}) \cup I_{n}^{\star} \subset V_{I,P}^{\star,1},\quad \SCB_{\sigma}(\SCL(y) ; D_{\Gamma_{\BD}}^{\BD}) \cup J_{n}^{\star} \subset V_{J,Q}^{\star,1},
\end{align*}
where $V_{I,P}^{\star,1}$ (resp.\ $V_{J,Q}^{\star,1}$) denotes the connected component of $V_{I,P}^{\star}$ (resp.\ $V_{J,Q}^{\star}$) containing $I$ (resp.\ $J$). Indeed, by \Cref{lemma metric ball hitting a loop}, one can see that $\SCB_{\tau}(\SCL(x) ; D_{\Gamma_{\BD}}^{\BD})\cap \overline{U} \neq\emptyset $ if and only if $\SCB_{\tau}(\SCL(x) ; D_{\Gamma_{\BD}}^{\BD}) \cap U \neq \emptyset$ because the metric ball from $\SCL(x)$ hits $\overline{U}$ with a loop that encircles a point of $\overline{U}$ and thus encircles a point of $U$ as well.  
Note also that if $n \in \BN,  P,Q \in \mathcal{P}$ are as above,  the events
\begin{align*}
\left\{\SCB_{\tau}(\SCL(x) ; D_{\Gamma_{\BD}}^{\BD}) \cup I_n^{\star} \subset V_{I,P}^{\star,1} \right\}, \quad \left\{\SCB_{\sigma}(\SCL(y) ; D_{\Gamma_{\BD}}^{\BD}) \cup J_{n}^{\star} \subset V_{J,Q}^{\star,1} \right\}
\end{align*}
are determined by $\left(\Gamma_{\BD}|_{V_{I,P}^{\star,1}} ,  D_{\Gamma_{\BD}|_{V_{I,P}^{\star,1}}}^{V_{I,P}^{\star,1}}\right)$ and $\left(\Gamma_{\BD}|_{V_{J,Q}^{\star,1}} ,  D_{\Gamma_{\BD}|_{V_{J,Q}^{\star,1}}}^{V_{J,Q}^{\star,1}}\right)$ respectively by \Cref{def:weak_axioms}, Axiom~\eqref{it:weak_axiom_locality}.

\stepn{step:mbloc-indep}{Conditional independence} Therefore,  to complete the proof of the lemma in the case that $I,J$ lie in different connected components of $\BD \setminus \overline{U}$,  it suffices to prove that the following holds for all $n \in \BN,  P,Q \in \mathcal{P}$ such that $V_{I,P} \cap V_{J,Q} = \emptyset$.  Conditionally on $\Psi_1$,  we have that the random variable
\begin{align*}
X \defeq \left(\Gamma_{\BD}|_{V_{I,P}^{\star,1}} ,  D_{\Gamma_{\BD}|_{V_{I,P}^{\star,1}}}^{V_{I,P}^{\star,1}},  \Gamma_{\BD}|_{V_{J,Q}^{\star,1}} ,  D_{\Gamma_{\BD}|_{V_{J,Q}^{\star,1}}}^{V_{J,Q}^{\star,1}}\right)
\end{align*}
and $\Psi_0$ are conditionally independent.  Let $Y$ be the restriction of $\overline{\Gamma}_\BD$ to $\BD \setminus \overline{V_{I,P}^{\star,1} \cup V_{J,Q}^{\star,1} \cup U}$ together with the signs of its loops. Let $Z$ be the restriction of $\Gamma_\BD$ to $V_{I,P}^{\star,1} \cup V_{J,Q}^{\star,1} $. Let $\overline{Z}$ be the restriction of $\overline{\Gamma}_\BD$ to $V_{I,P}^{\star,1} \cup V_{J,Q}^{\star,1}$ together with the signs of its loops. Note that conditionally on $(Y, Z)$, the random variables $X, \overline{Z}$ and $\Psi_0$ are independent. In particular, since $Z$ is a function of $\overline{Z}$, conditionally on $(Y, \overline{Z})$, the random variables $X$ and $\Psi_0$ are independent. Moreover, note that

$V_{I,P}^{\star}$ and $V_{J,Q}^{\star}$ are both a.s.\ determined by $\Psi_1$

 and that $\Psi_1$ is a.s.\ determined by $(Y, \overline{Z})$.

Therefore, the $\sigma$-field generated by $(Y, \overline{Z})$ is the same as the $\sigma$-field generated by $\Psi_1$. Thus, conditionally on $\Psi_1$, the random variables $X$ and $\Psi_0$ are independent.

\stepn{step:mbloc-rest}{The remaining case} Finally,  the case that both $I$ and $J$ lie in the same connected component of $\BD \setminus \overline{U}$ follows from a similar argument as the one applied above where we consider a path $P \in \mathcal{P}$ such that $V_{I,P}$ contains both $I$ and $J$,  and $V_{I,P} \subset \BD \setminus \overline{U}$.  This completes the proof of the lemma.
\end{proof}

\begin{figure}[ht!]
\centering
\includegraphics[width=.49\linewidth]{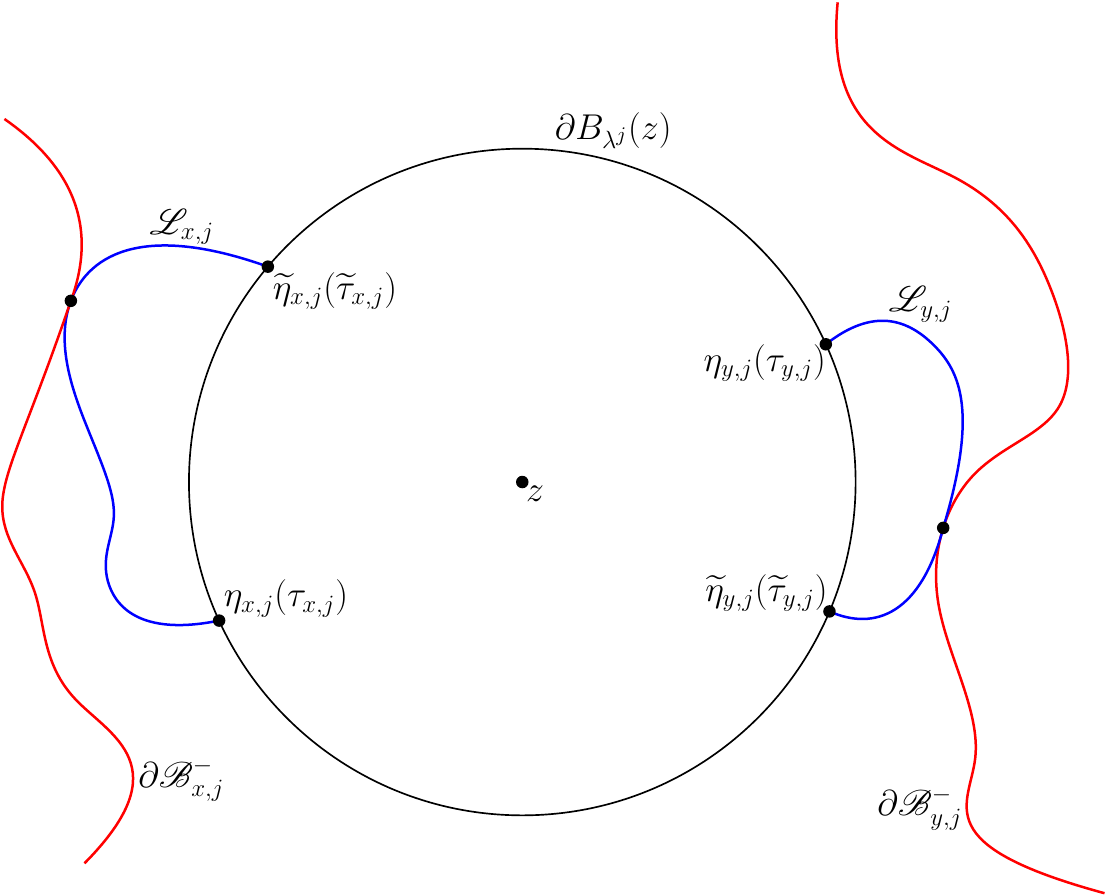}
\includegraphics[width=.49\linewidth]{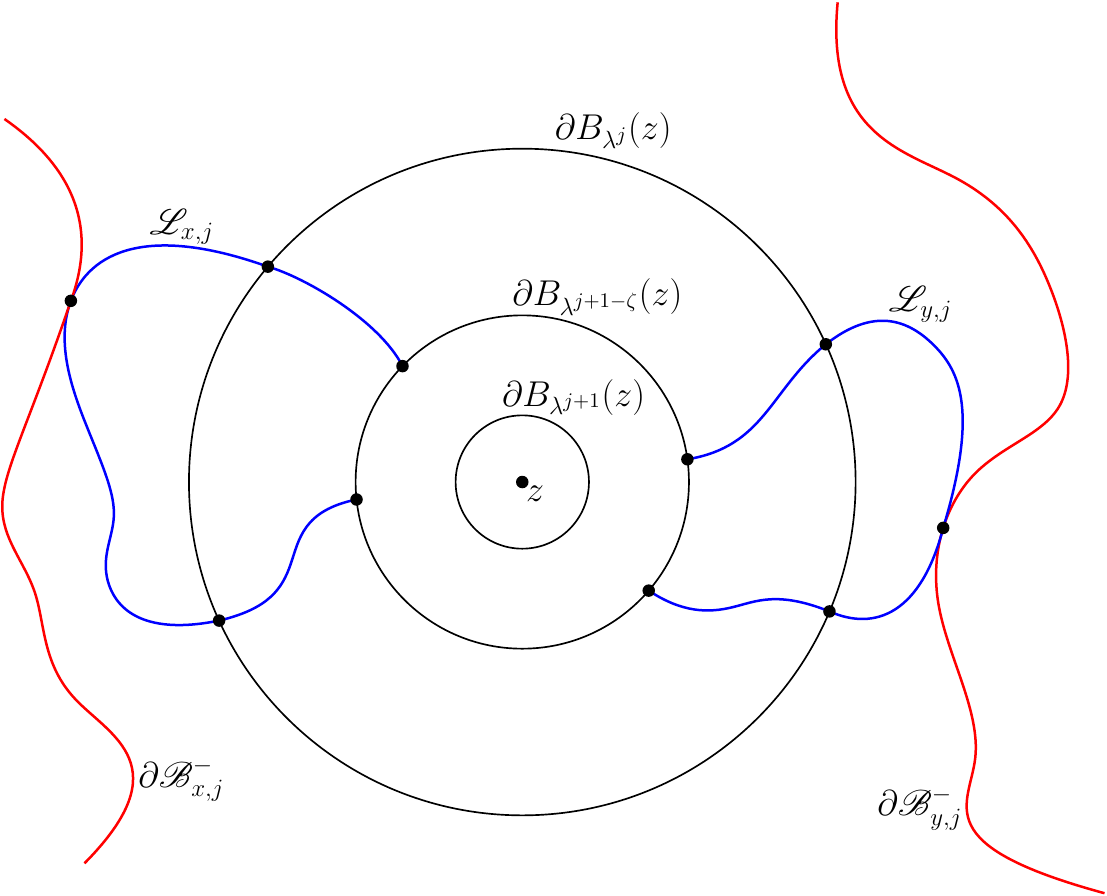}
\includegraphics[width=.49\linewidth]{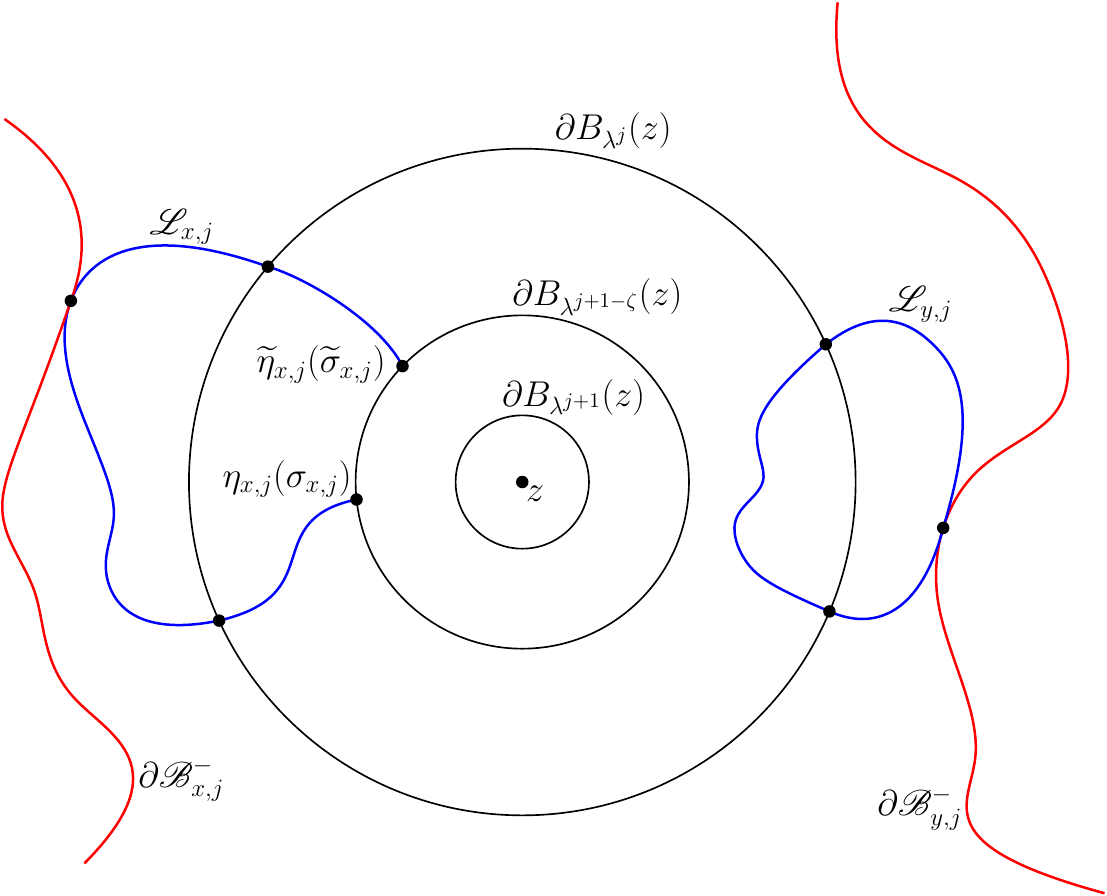}
\includegraphics[width=.49\linewidth]{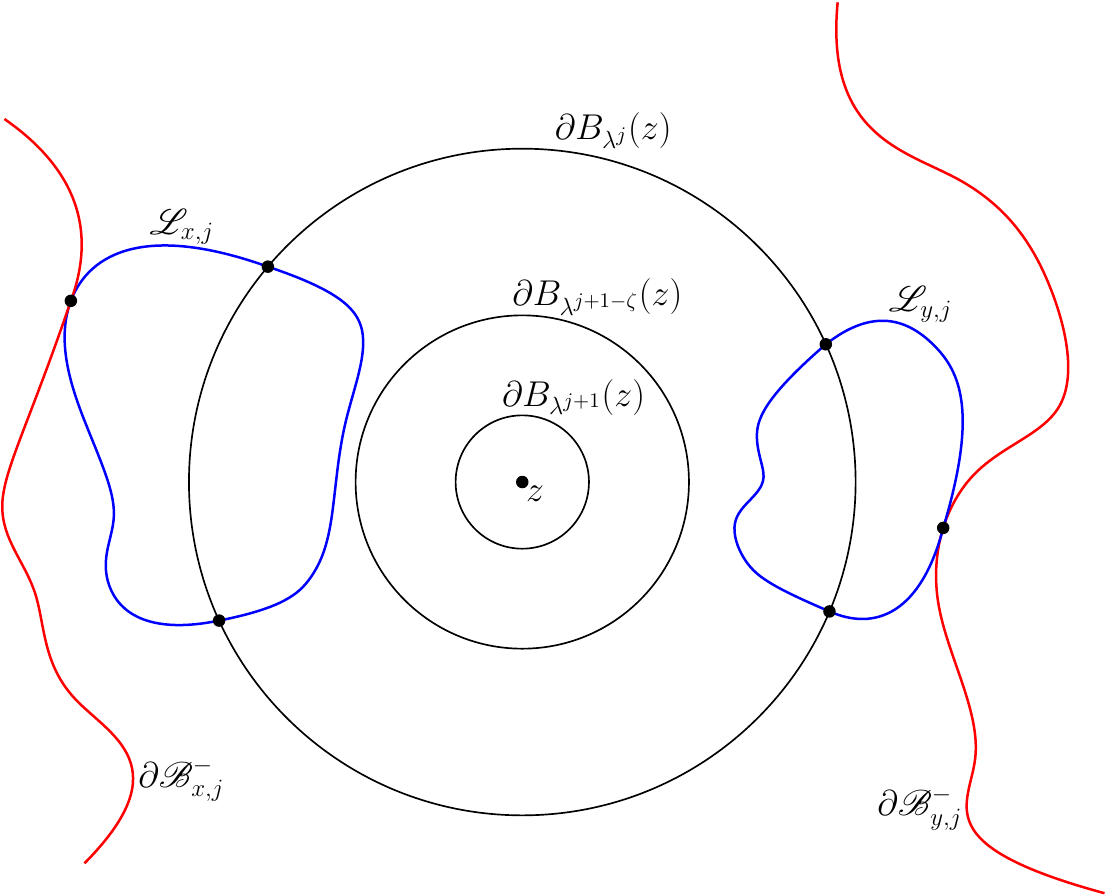}
\caption{Illustration of the proof of \Cref{prop:hitting-two-metric-balls}. {\bfseries Top left:} Illustration of the metric ball explorations around $\SCL(x)$ and $\SCL(y)$ stopped upon hitting $\partial B_{\lambda^j}(z)$. It is a.s.\ the case that the explorations stopped while exploring some loops $\SCL_{x,j}$ and $\SCL_{y,j}$, respectively, and $\eta_{x,j}(\tau_{x,j})$, $\widetilde\eta_{x,j}(\widetilde\tau_{x,j})$, $\eta_{y,j}(\tau_{y,j})$, and $\widetilde\eta_{y,j}(\widetilde\tau_{y,j})$ denote the four hitting points. {\bfseries Top right/bottom left/bottom right:} Illustration of the events $F_j^1$, $F_j^2$, $F_j^3$, respectively. We will control the conditional probability $\BP\lbrack E_{j + 1} \mid E_j\rbrack$ by estimating $\BP\lbrack F_j^1 \mid E_j\rbrack$, $\BP\lbrack E_{j + 1} \mid F_j^2 \cap E_j\rbrack$, and $\BP\lbrack E_{j + 1} \mid F_j^3 \cap E_j\rbrack$. See \Cref{fig:two-metric-balls-rectangle} for an illustration of the argument applied to estimate $\BP\lbrack E_{j + 1} \mid F_j^2 \cap E_j\rbrack$.}
\label{fig:two-metric-balls}
\end{figure}

\begin{figure}
\centering
\includegraphics[width=\linewidth]{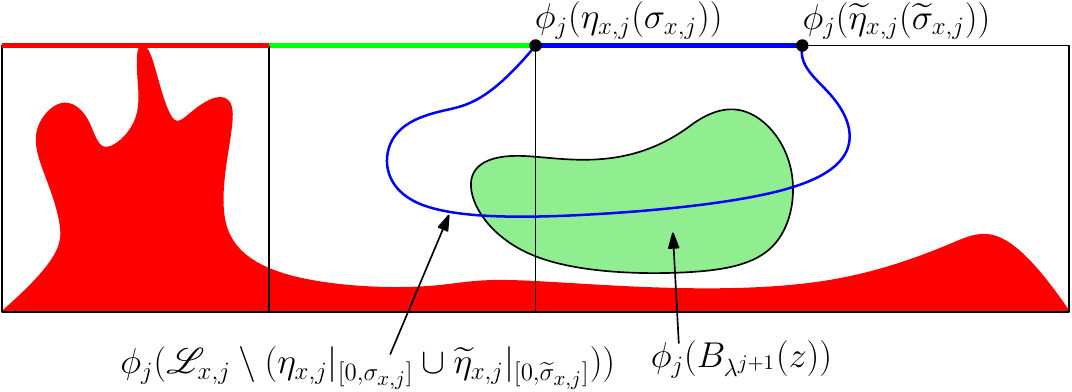}
\caption{Illustration of the argument applied to estimate $\BP\lbrack E_{j + 1} \mid F_j^2 \cap E_j\rbrack$, where the largest rectangle depicts $\phi_j(U_j) = V_{r_j}$. The arc $[\eta_{x,j}(\sigma_{x,j}), \widetilde\eta_{x,j}(\widetilde\sigma_{x,j})]_{\partial U_j}^\circlearrowright$ is mapped to the blue segment on the top of the rectangle. The image of $B_{\lambda^{j + 1}}(z)$ is bounded away from the top and the bottom and has bounded width. Both the red and green segments on the top have length at least of order $j$. The red segment being long implies that if we start the metric ball exploration with respect to a non-nested CLE$_4$ inside the rectangle from the bottom, it will hit the red segments before hitting $\phi_j(B_{\lambda^{j + 1}}(z))$ with probability $1 - O(1/j)$. The green segment being long implies that, with extremely high probability, adding the SLE$_4$ curve $\phi_j(\SCL_{x,j} \setminus (\eta_{x,j}|_{[0, \tau_{x,j}]} \cup \widetilde\eta_{x,j}|_{[0, \widetilde\tau_{x,j}]}))$ will not affect the situation where the metric ball exploration from the bottom hits the red segments. Also, it will not make the metric ball exploration from the bottom hit $\phi_j(B_{\lambda^{j + 1}}(z))$ more easily (because by the definition of the event $E_{j + 1}$, we only need to consider the situation where the metric ball exploration from the bottom hits $\phi_j(B_{\lambda^{j + 1}}(z))$ without using $\phi_j(\SCL_{x,j} \setminus (\eta_{x,j}|_{[0, \tau_{x,j}]} \cup \widetilde\eta_{x,j}|_{[0, \widetilde\tau_{x,j}]}))$).}
\label{fig:two-metric-balls-rectangle}
\end{figure}

\begin{proof}[Proof of \Cref{prop:hitting-two-metric-balls}]
\step{step:hit2-outline}{Outline and setup}

\substepn{step:hit2o-reduce}{Reduction and choice of parameters} For simplicity, we only consider the case of $\SCL(x)$ and $\SCL(y)$; the case of deterministic and disjoint connected arcs of $\partial\BD$ is entirely similar. The requirement that $\lvert x - z\rvert \wedge \lvert y - z\rvert \ge \delta$ is to make sure that $\SCB^{B_\varepsilon(z)}(\SCL(x); D_{\Gamma_{\BD}}^{\BD})$ and $\SCB^{B_\varepsilon(z)}(\SCL(y); D_{\Gamma_{\BD}}^{\BD})$ have diameters of order $\delta$ so that we may apply \Cref{lem:Beurling}. In the case of deterministic and disjoint connected arcs of $\partial\BD$, the metric balls automatically have diameters at least the diameters of the arcs.

See \Cref{fig:two-metric-balls} for an illustration. 
Fix $\zeta \in (0,1/2)$ and $\lambda \in (0,1)$.

\substepn{step:hit2o-notation}{Notation and the events $E_j$} To lighten notation, for each $j \in \BN$, write 
	\begin{gather*}
		\SCB_{x,j} \defeq \SCB^{B_{\lambda^j}(z)}(\SCL(x); D_{\Gamma_{\BD}}^{\BD}); \quad \SCB_{y,j} \defeq \SCB^{B_{\lambda^j}(z)}(\SCL(y); D_{\Gamma_{\BD}}^{\BD}); \quad E_j \defeq \{\SCB_{x,j} \cap \SCB_{y,j} = \emptyset\}; \\
		\SCB_{x,j}^- \defeq \SCB_{D_{\Gamma_{\BD}}^{\BD}(\SCL(x), B_{\lambda^j}(z))}^-(\SCL(x); D_{\Gamma_{\BD}}^{\BD}); \quad \SCB_{y,j}^- \defeq \SCB_{D_{\Gamma_{\BD}}^{\BD}(\SCL(y), B_{\lambda^j}(z))}^-(\SCL(y); D_{\Gamma_{\BD}}^{\BD}). 
	\end{gather*}
	Since $B_{\lambda^{j+1}}(z) \subset B_{\lambda^j}(z)$, we have $\SCB_{x,j}^- \subseteq \SCB_{x,j+1}^-$ and likewise for $y$, so that $E_{j+1} \subseteq E_j$ and
	\begin{equation}\label{eq:hitting-two-metric-balls-proof-2}
		\BP\lbrack E_n\rbrack = \BP\lbrack E_1\rbrack \prod_{j = 1}^{n - 1} \BP\lbrack E_{j + 1} \mid E_j\rbrack.
	\end{equation}
	By \Cref{lemma metric ball hitting a loop}, a.s.\ there exists $\SCL_{x,j} \in \Gamma_{\BD}$ with $D_{\Gamma_{\BD}}^{\BD}(\SCL(x), B_{\lambda^j}(z)) = D_{\Gamma_{\BD}}^{\BD}(\SCL(x), \SCL_{x,j})$, $\SCL_{x,j} \cap B_{\lambda^j}(z) \neq \emptyset$, and $\SCB_{x,j}^- \cap B_{\lambda^j}(z) = \emptyset$. Let $\eta_{x,j} \colon [0, 1] \to \SCL_{x,j}$ be a counterclockwise parameterization of $\SCL_{x,j}$ such that $\eta_{x,j}(0) = \eta_{x,j}(1) = (\text{the unique point of }\SCB_{x,j}^- \cap \SCL_{x,j})$. Write $\widetilde\eta_{x,j}$ for the time-reversal of $\eta_{x,j}$. Write $\tau_{x,j}$ (resp.\ $\widetilde\tau_{x,j}$) for the first time at which $\eta_{x,j}$ (resp.\ $\widetilde\eta_{x,j}$) hits $\partial B_{\lambda^j}(z)$. Let $\SCL_{y,j}$, $\eta_{y,j}$, $\widetilde\eta_{y,j}$, $\tau_{y,j}$, and $\widetilde\tau_{y,j}$ be defined in the same manner as $\SCL_{x,j}$, $\eta_{x,j}$, $\widetilde\eta_{x,j}$, $\tau_{x,j}$, and $\widetilde\tau_{x,j}$, respectively, but with $\SCL(y)$ in place of $\SCL(x)$. Note that $\SCL_{x,j} \neq \SCL_{y,j}$ on $E_j$, since a common loop would be contained in both $\SCB_{x,j}$ and $\SCB_{y,j}$. 
    
    \substepn{step:hit2o-decomp}{Decomposition according to $F_j^1$, $F_j^2$ and $F_j^3$} Write
	\begin{itemize}
		\item $F_j^1$ for the event that both $\SCL_{x,j} \setminus (\eta_{x,j}|_{[0, \tau_{x,j}]} \cup \widetilde\eta_{x,j}|_{[0, \widetilde\tau_{x,j}]})$ and $\SCL_{y,j} \setminus (\eta_{y,j}|_{[0, \tau_{y,j}]} \cup \widetilde\eta_{y,j}|_{[0, \widetilde\tau_{y,j}]})$ intersect $B_{\lambda^{j + 1 - \zeta}}(z)$; 
		\item $F_j^2$ for the event that exactly one of $\SCL_{x,j} \setminus (\eta_{x,j}|_{[0, \tau_{x,j}]} \cup \widetilde\eta_{x,j}|_{[0, \widetilde\tau_{x,j}]})$ and $\SCL_{y,j} \setminus (\eta_{y,j}|_{[0, \tau_{y,j}]} \cup \widetilde\eta_{y,j}|_{[0, \widetilde\tau_{y,j}]})$ intersects $B_{\lambda^{j + 1 - \zeta}}(z)$; 
		\item $F_j^3$ for the event that neither $\SCL_{x,j} \setminus (\eta_{x,j}|_{[0, \tau_{x,j}]} \cup \widetilde\eta_{x,j}|_{[0, \widetilde\tau_{x,j}]})$ nor $\SCL_{y,j} \setminus (\eta_{y,j}|_{[0, \tau_{y,j}]} \cup \widetilde\eta_{y,j}|_{[0, \widetilde\tau_{y,j}]})$ intersects $B_{\lambda^{j + 1 - \zeta}}(z)$.
	\end{itemize}
	Then 
	\begin{align}\label{eq:hitting-two-metric-balls-proof-3}
		\BP\lbrack E_{j + 1} \mid E_j\rbrack &= \BP\lbrack E_{j + 1} \cap F_j^1 \mid E_j\rbrack + \BP\lbrack E_{j + 1} \cap F_j^2 \mid E_j\rbrack + \BP\lbrack E_{j + 1} \cap F_j^3 \mid E_j\rbrack \\
		&\le \BP\lbrack F_j^1 \mid E_j\rbrack + \BP\lbrack E_{j + 1} \mid F_j^2 \cap E_j\rbrack + \BP\lbrack E_{j + 1} \mid F_j^3 \cap E_j\rbrack. \notag
	\end{align}

Our goal for the rest of the proof will be to bound separately from above the terms appearing on the right hand side of~\eqref{eq:hitting-two-metric-balls-proof-3} and then conclude the proof of the lemma.  For the rest of the proof,  we will assume that $B_{\lambda^j}(z) \subset \BD$.  A similar argument works in the case that $B_{\lambda^j}(z) \cap \partial \BD \neq \emptyset$.

\step{step:hit2-Fj1}{Bounding $\BP\lbrack F_j^1 \mid E_j\rbrack$}

\substepn{step:hit2F-claim}{The claim} Consider
\begin{equation}\label{eq:hitting-two-metric-balls-proof-0}
	\SCB_{x,j}^- \cup \eta_{x,j}|_{[0, \tau_{x,j}]} \cup \widetilde\eta_{x,j}|_{[0, \widetilde\tau_{x,j}]} \cup \SCB_{y,j}^- \cup \eta_{y,j}|_{[0, \tau_{y,j}]} \cup \widetilde\eta_{y,j}|_{[0, \widetilde\tau_{y,j}]}
\end{equation}
We claim that 
\begin{equation}\label{eq:hitting-two-metric-balls-proof-4}
		\BP\lbrack F_j^1 \mid~\eqref{eq:hitting-two-metric-balls-proof-0}, \ E_j\rbrack \le \lambda^{(1 - \zeta)(2 - \zeta)} \quad \text{a.s.} 
	\end{equation}
if $\lambda \in (0,1)$ is sufficiently small  and $j \in \BN$ is sufficiently large (how small and how large depend only on $\delta$ and $\zeta$).

    \substepn{step:hit2F-bothbdry}{The case where both pre-balls meet $\partial\BD$} We first consider the case where $\SCB_{x,j}^- \cap \partial\BD \neq \emptyset$ and $\SCB_{y,j}^- \cap \partial\BD \neq \emptyset$ (i.e., the case where each connected component of $\BD \setminus (\SCB_{x,j}^- \cup \SCB_{y,j}^-)$ is simply connected). By the locality property, given $\SCB_{x,j}$ and $\SCB_{y,j}^-$, $\SCL_{y,j}$ is a SLE$_4$ loop rooted at $\eta_{y,j}(0)$ conditioned so that $\SCL_{y,j} \cap B_{\lambda^j}(z) \neq \emptyset$, and the same holds with the roles of $\SCL_{x,j}$ and $\SCL_{y,j}$ interchanged. This implies that, given $\SCB_{x,j}$, $\SCB_{y,j}^-$, $\eta_{y,j}|_{[0, \tau_{y,j}]}$, and $\widetilde\eta_{y,j}|_{[0, \widetilde\tau_{y,j}]}$, $\SCL_{y,j} \setminus (\eta_{y,j}|_{[0, \tau_{y,j}]} \cup \widetilde\eta_{y,j}|_{[0, \widetilde\tau_{y,j}]})$ is conditionally a chordal SLE$_4$ from $\eta_{y,j}(\tau_{y,j})$ to $\widetilde\eta_{y,j}(\widetilde\tau_{y,j})$, and the same holds with the roles of $\SCL_{x,j}$ and $\SCL_{y,j}$ interchanged. Thus, by definition, given~\eqref{eq:hitting-two-metric-balls-proof-0} and the event $E_j$, $\{\SCL_{x,j} \setminus (\eta_{x,j}|_{[0, \tau_{x,j}]} \cup \widetilde\eta_{x,j}|_{[0, \widetilde\tau_{x,j}]}), \SCL_{y,j} \setminus (\eta_{y,j}|_{[0, \tau_{y,j}]} \cup \widetilde\eta_{y,j}|_{[0, \widetilde\tau_{y,j}]})\}$ is conditionally a bichordal SLE$_4$ in $(\BD \setminus~\eqref{eq:hitting-two-metric-balls-proof-0}; \eta_{x,j}(\tau_{x,j}), \widetilde\eta_{x,j}(\widetilde\tau_{x,j}), \eta_{y,j}(\tau_{y,j}), \widetilde\eta_{y,j}(\widetilde\tau_{y,j}))$ with link pattern 
    \begin{align*}	
    \{\{\eta_{x,j}(\tau_{x,j}), \widetilde\eta_{x,j}(\widetilde\tau_{x,j})\}, \{\eta_{y,j}(\tau_{y,j}), \widetilde\eta_{y,j}(\widetilde\tau_{y,j})\}\}.
    \end{align*}
    (Recall that $\{\SCL_{x,j} \setminus (\eta_{x,j}|_{[0, \tau_{x,j}]} \cup \widetilde\eta_{x,j}|_{[0, \widetilde\tau_{x,j}]}), \SCL_{y,j} \setminus (\eta_{y,j}|_{[0, \tau_{y,j}]} \cup \widetilde\eta_{y,j}|_{[0, \widetilde\tau_{y,j}]})\}$ is called a \emph{bichordal SLE$_4$} in $(\BD \setminus~\eqref{eq:hitting-two-metric-balls-proof-0}; \eta_{x,j}(\tau_{x,j}), \widetilde\eta_{x,j}(\widetilde\tau_{x,j}), \eta_{y,j}(\tau_{y,j}), \widetilde\eta_{y,j}(\widetilde\tau_{y,j}))$ if given~\eqref{eq:hitting-two-metric-balls-proof-0} and $\SCL_{x,j} \setminus (\eta_{x,j}|_{[0, \tau_{x,j}]} \cup \widetilde\eta_{x,j}|_{[0, \widetilde\tau_{x,j}]})$, $\SCL_{y,j} \setminus (\eta_{y,j}|_{[0, \tau_{y,j}]} \cup \widetilde\eta_{y,j}|_{[0, \widetilde\tau_{y,j}]})$ is conditionally a chordal SLE$_4$ in the connected component of $\BD \setminus (\eqref{eq:hitting-two-metric-balls-proof-0} \cup \SCL_{x,j} \setminus (\eta_{x,j}|_{[0, \tau_{x,j}]} \cup \widetilde\eta_{x,j}|_{[0, \widetilde\tau_{x,j}]}))$ that contains $\SCL_{y,j} \setminus (\eta_{y,j}|_{[0, \tau_{y,j}]} \cup \widetilde\eta_{y,j}|_{[0, \widetilde\tau_{y,j}]})$ from $\eta_{y,j}(\tau_{y,j})$ to $\widetilde\eta_{y,j}(\widetilde\tau_{y,j})$, and the same holds with the roles of $\SCL_{x,j} \setminus (\eta_{x,j}|_{[0, \tau_{x,j}]} \cup \widetilde\eta_{x,j}|_{[0, \widetilde\tau_{x,j}]})$ and $\SCL_{y,j} \setminus (\eta_{y,j}|_{[0, \tau_{y,j}]} \cup \widetilde\eta_{y,j}|_{[0, \widetilde\tau_{y,j}]})$ interchanged.)

Let $K_1$ (resp.\ $K_2$) denote the closure of the union of $\SCB_{x,j}^- \cup \eta_{x,j}|_{[0,\tau_{x,j}]} \cup \widetilde{\eta}_{x,j}|_{[0,\widetilde{\tau}_{x,j}]}$ (resp.\ 	$\SCB_{y,j}^- \cup \eta_{y,j}|_{[0,\tau_{y,j}]} \cup \widetilde{\eta}_{y,j}|_{[0,\widetilde{\tau}_{y,j}]}$) together with the points that it disconnects from $z$. Since both $\SCB_{x,j}^- \cup \eta_{x,j}|_{[0,\tau_{x,j}]} \cup \widetilde{\eta}_{x,j}|_{[0,\widetilde{\tau}_{x,j}]}$ and $\SCB_{y,j}^- \cup \eta_{y,j}|_{[0,\tau_{y,j}]} \cup \widetilde{\eta}_{y,j}|_{[0,\widetilde{\tau}_{y,j}]}$ are connected sets, we have that both $K_1$ and $K_2$ are connected, simply connected, and compact sets. 
Let $W_j$ denote the connected component of $\BD \setminus (K_1 \cup K_2)$ that contains $z$ and note that $W_j$ is simply connected and $B_{\lambda^j}(z) \subset W_j$. Let also $\psi_j$ denote the conformal mapping from $W_j$ onto $\BD$ such that $\psi_j(z) = 0$ and $\psi_j'(z) > 0$.

Note that $\dist(w,\partial W_j) > \lambda^j/2
	$ for all $w \in B_{\lambda^{1+j-\zeta}}(z)$ if $\lambda \in (0,1)$ is sufficiently small (depending only on $\zeta$) and hence \cite[Corollary~3.25]{lawler2008conformally} (applied with $r=2\lambda^{1-\zeta}$ with $\lambda$ small enough that $2\lambda^{1-\zeta}<1/2$) implies that
\begin{align*}
\psi_j(B_{\lambda^{1+j-\zeta}}(z)) \subset B_{32\lambda^{1-\zeta} /3}(0).
\end{align*}
Therefore, \Cref{lem:bichordal-bulk-4A} implies that $\BP\lbrack F_j^1 \mid~\eqref{eq:hitting-two-metric-balls-proof-0}\rbrack \le \lambda^{(1 - \zeta)(2 - \zeta)}$ a.s.\ on the event that $E_j$ occurs and $\SCB_{x,j}^- \cap \partial\BD \neq \emptyset$ and $\SCB_{y,j}^- \cap \partial\BD \neq \emptyset$.

\substepn{step:hit2F-general}{The general case: the event $G_j$} We next consider the general case. Fix deterministic,  disjoint line segments $I$ and $J$ from $x$ and $y$, respectively, to $\partial\BD$, in such a way that their distances from $z$ are also at least $\delta$. Let $G_j$ be the event that there is no loop of $\Gamma_\BD$ contained in $\BD \setminus \eqref{eq:hitting-two-metric-balls-proof-0}$ that intersects both $B_{\lambda^{1+j-\zeta}}(z)$ and $I$, or both $B_{\lambda^{1+j-\zeta}}(z)$ and $J$.  We will show that
\begin{equation}\label{eq proba Gj goes to one}
	\BP[G_j \vert  \, \eqref{eq:hitting-two-metric-balls-proof-0}, \, E_j] \mathop{\longrightarrow}\limits_{j\to \infty}1
\end{equation}

at a rate depending only on $\delta,\lambda$,  and $\zeta$.

\substepn{step:hit2F-localset}{The local set $A$ and the domain $U_j$} Following the notation of \Cref{lem:metric_ball_local},  we let $I^{\star}$ (resp.\ $J^{\star}$) denote the closure of the union of the loops in $\Gamma_{\BD}$ that intersect $I$ (resp.\ $J$),  each intersected with the closed Euclidean $\delta/2$-neighborhood of $I$ (resp.\ $J$).  We also consider the set
\begin{align*}
A = \SCB_{x,j}^- \cup \eta_{x,j}([0,\tau_{x,j}]) \cup \widetilde{\eta}_{x,j}([0,\widetilde{\tau}_{x,j}]) \cup \SCB_{y,j}^- \cup \eta_{y,j}([0,\tau_{y,j}]) \cup \widetilde{\eta}_{y,j}([0,\widetilde{\tau}_{y,j}]) \cup I^{\star} \cup J^{\star}.
\end{align*}
Then,  \Cref{lem:metric_ball_local} implies that there exists a coupling between $\Gamma_{\BD}$ and a zero-boundary GFF $\Psi$ on $\BD$ such that $A$ is local for $\Psi$.  For the rest of the proof,  we will assume that we are working with the above coupling.

Let $U_j$ denote the connected component of $\BD \setminus A$ containing $z$ and note that $\partial U_j \cap \partial \BD$ consists of two disjoint subarcs of $\partial \BD$.  Let also $\psi_j : U_j \to V_{r_j}$ be the conformal transformation mapping $U_j$ onto $V_{r_j}$ for some $r_j>0$ such that $\psi_j$ maps $\SCB_{x,j}^- \cup \eta_{x,j}([0,\tau_{x,j}]) \cup \widetilde{\eta}_{x,j}([0,\widetilde{\tau}_{x,j}])$ (resp.\ $\SCB_{y,j}^- \cup \eta_{y,j}([0,\tau_{y,j}]) \cup \widetilde{\eta}_{y,j}([0,\widetilde{\tau}_{y,j}])$) onto $(0,r_j) \times \{1\}$  (resp.\ $(0,r_j) \times \{0\}$) and $\Re(\psi_j(\eta_{x,j}(\tau_{x,j}))) < \Re(\psi_j(\widetilde{\eta}_{x,j}(\widetilde{\tau}_{x,j})))$.  Let $I^{\star,\mathrm{L}}$ (resp.\ $I^{\star,\mathrm{R}}$) denote the left (resp.\ right) side of $I^{\star}$ when seen as prime ends on $\partial U_j$.  We define $J^{\star,\mathrm{L}}$ and $J^{\star,\mathrm{R}}$ similarly.  Then,  the Beurling estimate implies that for all $q \in \{\mathrm{L},\mathrm{R}\}$ and all $w \in \psi_j(B_{\lambda^{1+j}}(z))$,  the probability that a planar Brownian motion starting from $w$ exits $U_j$ for the first time on $I^{\star,\mathrm{L}} \cup I^{\star,\mathrm{R}} \cup J^{\star,\mathrm{L}} \cup J^{\star,\mathrm{R}}$ is at most $\lesssim \lambda^{j/2}$,  where the implicit constant depends only on $\lambda,\zeta$,  and $\delta$.  

Moreover,  there exists a constant $p \in (0,1)$ depending only on $\zeta$ and $\lambda$,  such that with probability at least $p$,  a planar Brownian motion starting from $z$ makes a loop around $B_{\lambda^{1+j-\zeta}}(z)$ before exiting $U_j$ for the first time.  This implies that there exists a constant $d \in (0,1)$ depending only on $\zeta$ and $\lambda$ such that $\dist(\psi_j(B_{\lambda^{1+j-\zeta}}(z)) ,  \partial V_{r_j}) \geq d$.  Thus,  combining with \Cref{lem:extremal_length},  we obtain that there exists a constant $c>0$ depending only on $\zeta$ and $\lambda$ such that
\begin{align*}
\dist(\psi_j(B_{\lambda^{1+j-\zeta}}(z)) ,  \psi_j(I^{\star,\mathrm{L}} \cup I^{\star,\mathrm{R}} \cup J^{\star,\mathrm{L}} \cup J^{\star,\mathrm{R}})) \geq c j.
\end{align*}

Set $x_j = \inf\{\Re(w) : w \in \psi_j(B_{\lambda^{1+j-\zeta}}(z))\}$ and $ y_j = \sup\{\Re(w) : w \in \psi_j(B_{\lambda^{1+j-\zeta}}(z))\}$,  and note that
\begin{align*}
\partial V_{x_j - cj ,  x_j} \cap (I^{\star,\mathrm{L}} \cup I^{\star,\mathrm{R}} \cup J^{\star,\mathrm{L}} \cup J^{\star,\mathrm{R}}) = \partial V_{y_j ,  y_j+cj} \cap (I^{\star,\mathrm{L}} \cup I^{\star,\mathrm{R}} \cup J^{\star,\mathrm{L}} \cup J^{\star,\mathrm{R}}) = \emptyset.
\end{align*}
Suppose that $\Re(\psi_j(\widetilde{\eta}_{y,j}(\widetilde{\tau}_{y,j})))$ and $\Re(\psi_j(\eta_{y,j}(\tau_{y,j}))) $ both lie in $[x_j,y_j]$.  The other cases for the locations of the points $\psi_j(\widetilde{\eta}_{y,j}(\widetilde{\tau}_{y,j})) $ and $\psi_j(\eta_{y,j}(\tau_{y,j}))$ are treated by a similar argument and so we will only give a proof in the case that the above holds.

Since $A$ is a local set,  we have that the random field $\Psi \circ \psi_j^{-1}$ can be decomposed given $A$ into the sum of a zero-boundary GFF on $V_{r_j}$ plus an independent harmonic function on $V_{r_j}$,  which is a.s.\ determined by $A$ and such that its boundary conditions are piecewise constant and change only a countable number of times.  Moreover,  the boundary conditions of $\Psi \circ \psi_j^{-1}$ on $[x_j-cj ,   x_j] \cup [y_j ,  y_j + cj]$ take values in $[-\pi,\pi]$.  By arguing as in \Cref{step:avoid-wide} of the proof of \Cref{lem:avoiding_two_metric_balls},  we obtain that for fixed $u \in (0,\pi/2)$,  there exists deterministic constants $C_1,\alpha>0$ depending only on $c,u$ such that given $A$,  with conditional probability at least $1-C_1 e^{-\alpha j}$,  we have that there exist $m,k \in [1,cj - 1]_{\BZ}$ such that the level line of $\Psi \circ \psi_j^{-1}$ of height $u$ starting from $x_j -cj + (m+1/2)$ (resp.\ $y_j + (k+1/2)$) hits the top boundary of $V_{x_j-cj+m,x_j-cj+m+1}$ (resp.\ $V_{y_j+k,y_j+k+1}$) before hitting either its left or right boundaries for the first time.  But then,  if the latter holds,  we have by the level line interaction rules (\Cref{thm:level_line_interaction}, applied to the height-$u$ line and the continuations of the loops, which are level lines of heights $\pm\pi/2$

) that none of the continuations of the loops in $\Gamma_{\BD}$ that intersect $I \cup J$ can hit $B_{\lambda^{1+j-\zeta}}(z)$.  This proves~\eqref{eq proba Gj goes to one}.

\substepn{step:hit2F-concl}{Conclusion in the general case}  By an argument similar to the preceding ones,
a.s.\ $\BP\lbrack F_j^1 \mid~\eqref{eq:hitting-two-metric-balls-proof-0}, \ E_j, \ G_j\rbrack \le \lambda^{(1 - \zeta)(2 - \zeta)} / 2$ for all sufficiently small $\lambda \in (0, 1)$. Moreover, by~\eqref{eq proba Gj goes to one}, $\BP[G_j^c \mid~\eqref{eq:hitting-two-metric-balls-proof-0}, \ E_j] \to 0$ as $j \to \infty$ at a rate depending only on $\delta, \lambda, \zeta$, so there are $\lambda_0 = \lambda_0(\delta,\zeta) \in (0,1)$ and $j_0 = j_0(\lambda,\delta,\zeta) \in \BN$ with this conditional probability at most $\lambda^{(1-\zeta)(2-\zeta)}/2$ for $\lambda \le \lambda_0$ and $j \ge j_0$. Since $\BP[F \mid H] \le \BP[F \mid G \cap H] + \BP[G^c \mid H]$,
\begin{equation*}
    \BP\lbrack F_j^1 \mid~\eqref{eq:hitting-two-metric-balls-proof-0}, \ E_j\rbrack \le \BP\lbrack F_j^1 \mid~\eqref{eq:hitting-two-metric-balls-proof-0}, \ E_j, \ G_j\rbrack + \BP\lbrack G_j^c \mid~\eqref{eq:hitting-two-metric-balls-proof-0}, \ E_j\rbrack \le \lambda^{(1 - \zeta)(2 - \zeta)}
\end{equation*}
a.s.\ for all $\lambda \le \lambda_0$ and all $j \ge j_0$. This proves~\eqref{eq:hitting-two-metric-balls-proof-4}.

\step{step:hit2-Ej1}{Bounding $\BP\lbrack E_{j + 1} \mid F_j^2 \cap E_j\rbrack$}

\substepn{step:hit2Ej-claim}{The claim} As in \Cref{step:hit2-Fj1},  we assume first that both $\SCB_{x,j}^-$ and $\SCB_{y,j}^-$ intersect $\partial \BD$. On the event $F_j^2 \cap E_j$, we may assume without loss of generality that $\SCL_{x,j} \setminus (\eta_{x,j}|_{[0,\tau_{x,j}]} \cup \widetilde{\eta}_{x,j}|_{[0,\widetilde{\tau}_{x,j}]})$ intersects $B_{\lambda^{1+j-\zeta}}(z)$ but $\SCL_{y,j} \setminus (\eta_{y,j}|_{[0,\tau_{y,j}]} \cup \widetilde{\eta}_{y,j}|_{[0,\widetilde{\tau}_{y,j}]})$ does not; write $\sigma_{x,j}$ (resp.\ $\widetilde{\sigma}_{x,j}$) for the first time at which $\eta_{x,j}$ (resp.\ $\widetilde{\eta}_{x,j}$) hits $\partial B_{\lambda^{1+j-\zeta}}(z)$. We claim that
\begin{equation}\label{eq:hitting-two-metric-balls-proof-5}
		\BP\!\left\lbrack E_{j + 1} \ \middle\vert \ \SCB_{x,j}^-, \ \eta_{x,j}|_{[0, \sigma_{x,j}]}, \ \widetilde\eta_{x,j}|_{[0, \widetilde\sigma_{x,j}]}, \ \SCB_{y,j}, \ F_j^2 \cap E_j\right\rbrack = O(1/j) \quad \text{as } j \to \infty, 
	\end{equation}
at a rate depending only on $\zeta$ and $\lambda$. Indeed, let $\widetilde{K}_1$ (resp.\ $\widetilde{K}_2$) denote the closure of the union of $\SCB_{x,j}^- \cup \eta_{x,j}|_{[0,\sigma_{x,j}]} \cup \widetilde{\eta}_{x,j}|_{[0,\widetilde{\sigma}_{x,j}]}$ (resp.\ $\SCB_{y,j}$) together with the points that it disconnects from $\infty$. Since both $\SCB_{x,j}^- \cup \eta_{x,j}|_{[0,\sigma_{x,j}]} \cup \widetilde{\eta}_{x,j}|_{[0,\widetilde{\sigma}_{x,j}]}$ and $\SCB_{y,j}$ are connected sets, we have that both $\widetilde{K}_1$ and $\widetilde{K}_2$ are connected, simply connected, and compact sets.

\substepn{step:hit2Ej-domain}{The domain $U_j$ and the Beurling estimate} Let $U_j$ denote the connected component of $\BD \setminus (\widetilde{K}_1 \cup \widetilde{K}_2)$ containing $z$.  Then,  $U_j$ is simply connected and $\partial U_j \cap \partial \BD$ consists of exactly two connected components that we denote by $\gamma$ and $\widetilde{\gamma}$ respectively.  Also,  we have that $B_{\lambda^{1+j-\zeta}}(z) \subset U_j$ and by arguing as in the proof of \Cref{lem:Beurling},  it holds that 
\begin{align*}
\max\{\dist(z,\gamma) ,  \dist(z,\widetilde{\gamma})\} \geq (\text{diam}(\widetilde{K}_1) \wedge \text{diam}(\widetilde{K}_2)) / 2.
\end{align*}
Without loss of generality,  we can assume that
\begin{align*}
\dist(z,\gamma) \geq (\text{diam}(\widetilde{K}_1) \wedge \text{diam}(\widetilde{K}_2) )/ 2.
\end{align*}
Moreover, using the exact same argument as in the proof of \Cref{lem:Beurling}, we obtain that there exists a constant $c>0$ depending only on $\delta$ and $\lambda$ such that the extremal distance in $U_j$ between $\widetilde{K}_1$ and $\widetilde{K}_2$ is at most $c / j$. Also, there exists $r_j>0$ and a conformal transformation $\phi_j : U_j \to V_{r_j}$ that maps $\partial U_j \cap \partial \BD$ to the left and right sides of $V_{r_j}$ such that $\phi_j(\eta_{x,j}(\sigma_{x,j})), \phi_j(\widetilde{\eta}_{x,j}(\widetilde{\sigma}_{x,j})) \in T_{r_j}$ and $\Re(\phi_j(\eta_{x,j}(\sigma_{x,j}))) < \Re(\phi_j(\widetilde{\eta}_{x,j}(\widetilde{\sigma}_{x,j})))$. Here, we let $\{(V_r,T_r,B_r)\}_{r >0}$ and $\{V_{x,y}\}_{x<y}$ be as in the proof of \Cref{prop:distance_across_rectangle}. Note that there exists a constant $\widetilde{c}>0$ depending only on $\delta$ such that $r_j \geq \widetilde{c} j$ by the above estimate on extremal distances.

\substepn{step:hit2Ej-macro}{A macroscopic bound on the real parts} We claim that there exists a constant $c_1>0$ depending only on $\delta$ and $\lambda$ such that at least one of $\Re(\phi_j(\eta_{x,j}(\sigma_{x,j})))$ and $r_j - \Re(\phi_j(\widetilde{\eta}_{x,j}(\widetilde{\sigma}_{x,j})))$ is at least $c_1 j$. Indeed, since $r_j \geq \widetilde{c} j$, we can assume that $-\Re(\phi_j(\eta_{x,j}(\sigma_{x,j}))) + \Re(\phi_j(\widetilde{\eta}_{x,j}(\widetilde{\sigma}_{x,j}))) \geq r_j /3$, since otherwise the claim is obviously true. Suppose that $\gamma$ is mapped to the right side of $V_{r_j}$ under $\phi_j$. Let $I_j$ denote the connected component of $U_j \cap \partial B_{\lambda^j}(z)$ such that $\overline{I_j}$ intersects both the left side of $\widetilde{\eta}_{x,j}|_{[0,\widetilde{\sigma}_{x,j}]}$ and the right side of $\eta_{y,j}|_{[0,\tau_{y,j}]}$. Let also $w$ denote the point on $I_j$ such that $\Im(\phi_j(w)) = 1/2$. Note that if $\lambda \in (0,1)$ is sufficiently small (depending only on $\delta$), it follows from the Beurling estimate that the planar Brownian motion starting from $w$ intersects $\gamma$ before intersecting $\widetilde{K}_1 \cup \widetilde{K}_2$ for the first time with probability at most
\begin{align*}
C_1 \left(\frac{\lambda^j}{\text{diam}(\widetilde{K}_1) \wedge \text{diam}(\widetilde{K}_2)}\right)^{1/2}
\end{align*}
for some universal constant $C_1>0$. Moreover, in order for a planar Brownian motion starting from $w$ to exit $U_j$ on $\phi_j^{-1}([0,\Re(\phi_j(\widetilde{\eta}_{x,j}(\widetilde{\sigma}_{x,j})))] \times \{1\})$, it must cross the annulus $A_{\lambda^{j+1-\zeta},\lambda^j}(z)$ between its inner and outer boundaries without intersecting $\widetilde{\eta}_{x,j}|_{[0,\widetilde{\sigma}_{x,j}]}$. Hence, by applying the Beurling estimate again and possibly taking $C_1$ to be large (in a universal way), we obtain that the probability that a planar Brownian motion starting from $w$ exits $U_j$ for the first time on $[0,\Re(\phi_j(\widetilde{\eta}_{x,j}(\widetilde{\sigma}_{x,j})))] \times \{1\}$ is at most $C_1 \lambda^{(1-\zeta)/2}$. Therefore, by taking $\lambda \in (0,1)$ sufficiently small (in a way that depends only on $\delta$ and $\zeta$), we obtain that there exists a constant $\hat{c}>0$ depending only on $\delta$ such that
\begin{align*}
r_j - \Re(\phi_j(\widetilde{\eta}_{x,j}(\widetilde{\sigma}_{x,j}))) \geq \hat{c},\,\, \phi_j(w) \in (\Re(\phi_j(\widetilde{\eta}_{x,j}(\widetilde{\sigma}_{x,j}))),r_j) \times (0,1).
\end{align*}

Set $\widetilde{r}_j\defeq r_j - \Re(\phi_j(\widetilde{\eta}_{x,j}(\widetilde{\sigma}_{x,j}))) > 0$. Since both the left and right sides of $V_{r_j}$ correspond to circular arcs of $\partial \BD$, we can apply the Schwarz reflection principle across both boundaries. The starting point $\phi_j(w)$ lies in $(\Re(\phi_j(\widetilde{\eta}_{x,j}(\widetilde{\sigma}_{x,j}))),r_j) \times (0,1)$, so its distance to either the left or right boundary of the extended rectangle
\begin{align*}
\left(2\Re(\phi_j(\widetilde{\eta}_{x,j}(\widetilde{\sigma}_{x,j}))) - r_j, 2r_j - \Re(\phi_j(\widetilde{\eta}_{x,j}(\widetilde{\sigma}_{x,j})))\right) \times (0,1)
\end{align*}
is at most $2\widetilde{r}_j$. Thus, applying \Cref{lem:extremal_length} to the extended rectangle (scaled by $\pi$) implies that there exist universal constants $a_1,a_2>0$ such that the probability that a planar Brownian motion starting from $\phi_j(w)$ exits the extended rectangle on its left or right side is at least $a_1 e^{-a_2 \widetilde{r}_j}$. By the conformal invariance of Brownian motion, together with the reflection symmetry just described (a Brownian motion exiting the extended rectangle on its left or right side exits $U_j$ on $\gamma$ or on its reflection, and the two have the same probability), this implies that the probability that a planar Brownian motion starting from $w$ exits $U_j$ for the first time on $\gamma$ is at least $a_1 e^{-a_2 \widetilde{r}_j}$. But we have already shown that the latter probability is at most
\begin{align*}
C_1 \left(\frac{\lambda^j}{\text{diam}(\widetilde{K}_1) \wedge \text{diam}(\widetilde{K}_2)}\right)^{1/2} \leq C_1 \left(\frac{\lambda^j}{\delta}\right)^{1/2}.
\end{align*}

Combining everything, we obtain that there exists a constant $c_1>0$ depending only on $\lambda$ and $\delta$ such that $\widetilde{r}_j \geq c_1 j$. Similarly, if $\gamma$ is mapped to the left side of $V_{r_j}$ under $\phi_j$, we obtain that $\Re(\phi_j(\eta_{x,j}(\sigma_{x,j}))) \geq c_1 j$ possibly by taking $c_1>0$ to be smaller (in a way that depends only on $\lambda$ and $\delta$). This proves that at least one of $\Re(\phi_j(\eta_{x,j}(\sigma_{x,j})))$ and $r_j - \Re(\phi_j(\widetilde{\eta}_{x,j}(\widetilde{\sigma}_{x,j})))$ is at least $c_1 j$.

\substepn{step:hit2Ej-transfer}{Transferring the rectangle estimate} We may assume without loss of generality that $\gamma$ is mapped to the left side of $V_{r_j}$ since the exact same argument works in the case that $\gamma$ is mapped to the right side of $V_{r_j}$.  Then,  we have that
$\Re(\phi_j(\eta_{x,j}(\sigma_{x,j})))$ is at least $c_1j$.  Moreover, since $A_{\lambda^{j + 1},\lambda^{j + 1 - \zeta}}(z) \subset U_j$ disconnects $B_{\lambda^{j + 1}}(z)$ and $\partial U_j$ (i.e., the planar Brownian motion starting from a suitable point disconnects $B_{\lambda^{j + 1}}(z)$ and $\partial U_j$ before exiting $U_j$ for the first time with probability at least $c_2$ for some $c_2 = c_2(\zeta, \lambda) \in (0, 1)$), it follows that 
	\begin{gather}\label{eq:hitting-two-metric-balls-proof-7}
		\dist(\phi_j(B_{\lambda^{j+1}}(z)),\partial V_{r_j}) \geq c_3; \\
		\sup\left\{\Re(z) : z \in \phi_j(B_{\lambda^{j + 1}}(z))\right\} - \inf\left\{\Re(z) : z \in \phi_j(B_{\lambda^{j + 1}}(z))\right\} \le c_3^{-1} \notag
	\end{gather}
	for some $c_3 = c_3(\zeta, \lambda) \in (0, 1)$.   Moreover,  the Beurling estimate implies that the probability that a planar Brownian motion starting from $\phi_j(z)$ exits $V_{r_j}$ for the first time on the left boundary of $V_{r_j}$ is at most $\lesssim \lambda^{j/2}$,  where the implicit constant depends only on $\delta$.  It follows from combining \Cref{lem:extremal_length} and~\eqref{eq:hitting-two-metric-balls-proof-7} that there exists a constant $c_4 \in (0,c_1/2)$ depending only on $\zeta,\lambda$,  and $\delta$ such that
	\begin{equation}\label{eq:nice_ineq}
	\Re(\phi_j(w)) \geq 2c_4 j \quad \text{for all} \quad w \in B_{\lambda^{j+1}}(z).
	\end{equation}
	
	Let $\Gamma_{V_{r_j}}$ be a non-nested CLE$_4$ in $V_{r_j}$. Recall that by the locality property, given 
\begin{align*}	
\SCB_{x,j}^-, \eta_{x,j}|_{[0,\sigma_{x,j}]}, \widetilde{\eta}_{x,j}|_{[0,\widetilde{\sigma}_{x,j}]}, \SCB_{y,j}, 
\end{align*}
and the event $F_j^2 \cap E_j$, the restriction of $\Gamma_{\BD}$ to $U_j^{\star}$ is conditionally a non-nested CLE$_4$ on $U_j^{\star}$ with one wired boundary arc on $[\eta_{x,j}(\sigma_{x,j}),\widetilde{\eta}_{x,j}(\widetilde{\sigma}_{x,j})]_{\partial U_j}^\circlearrowright$, i.e., a chordal SLE$_4$ curve in $U_j^{\star}$ from $\eta_{x,j}(\sigma_{x,j})$ to $\widetilde{\eta}_{x,j}(\widetilde{\sigma}_{x,j})$, together with a non-nested CLE$_4$ in the complementary connected components. Let $\Gamma_{V_{r_j}}$ and $\phi_j(\Gamma_{\BD}|_{U_j})$ be coupled as in \Cref{lem:from-CLE-2-SLE}, i.e., there is a chordal SLE$_{8/3}$ curve $\gamma$ in $V_{r_j}$ from $\phi_j(\eta_{x,j}(\sigma_{x,j}))$ to $\phi_j(\widetilde\eta_{x,j}(\widetilde\sigma_{x,j}))$ that is independent of $\Gamma_{V_{r_j}}$, such that $\phi_j(\SCL_{x,j} \setminus (\eta_{x,j}|_{[0, \sigma_{x,j}]} \cup \widetilde\eta_{x,j}|_{[0, \widetilde\sigma_{x,j}]}))$ is given by the rightmost boundary of the set obtained by attaching to $\gamma$ all the loops of $\Gamma_{V_{r_j}}$ that it intersects, and the collection of loops of $\phi_j(\Gamma_{\BD}|_{U_j})$ is given by the restriction of $\Gamma_{V_{r_j}}$ to the complementary connected component of $\phi_j(\SCL_{x,j} \setminus (\eta_{x,j}|_{[0, \sigma_{x,j}]} \cup \widetilde\eta_{x,j}|_{[0, \widetilde\sigma_{x,j}]}))$. Write $W_j$ for the connected component of $V_{r_j}$ that is to the right of $\gamma$. Note that $W_j^\star$ is a complementary connected component of $\phi_j(\SCL_{x,j} \setminus (\eta_{x,j}|_{[0, \sigma_{x,j}]} \cup \widetilde\eta_{x,j}|_{[0, \widetilde\sigma_{x,j}]}))$. Furthermore, let $\Gamma_{V_{r_j}}$, $\phi_j(\Gamma_{\BD}|_{U_j})$, $D_{\Gamma_{V_{r_j}}}^{V_{r_j}}$, and $D_{\phi_j(\Gamma_{\BD}|_{U_j})}^{W_j^\star}$ be coupled as in \Cref{def:weak_axioms}, Axiom~\eqref{it:weak_axiom_locality} (locality).  Then,  combining~\eqref{eq:hitting-two-metric-balls-proof-7},  ~\eqref{eq:nice_ineq} with \Cref{lem:avoiding_two_metric_balls} and the fact that $r_j \geq c_1 j$,  we obtain that there exists a deterministic constant $c_5 = c_5(\zeta,\lambda,\delta)>0$ such that with conditional probability at least $1-c_5 / j$,  we have that
	\begin{align*}
	D_{\Gamma_{V_{r_j}}|_{V_{c_4 / 2 j ,  c_4 j}^{\star}}}^{V_{c_4 / 2 j,  c_4 j}^{\star}}(T_{r_j} ,  B_{r_j}) < D_{\Gamma_{V_{r_j}}}^{V_{r_j}}(\phi_j(B_{\lambda^{j+1}}(z)),B_{r_j}).
	\end{align*}
	
	By the monotonicity property in Axiom~\eqref{it:weak_axiom_locality}, the $D_{\Gamma_{\BD}}^{\BD}$-distance from $\SCB_{y,j}$ to $B_{\lambda^{j + 1}}(z)$ without using $\SCB_{x,j}$ is a.s.\ at least $D_{\Gamma_{V_{r_j}}}^{V_{r_j}}(\phi_j(B_{\lambda^{j + 1}}(z)), B_{r_j})$. Moreover, it follows from a similar argument to the argument applied in \Cref{step:avoid-wide} of  the proof of \Cref{lem:avoiding_two_metric_balls} that it holds with exponentially high probability as $j \to \infty$, at a rate depending only on $\zeta,\delta$ and $\lambda$, that $\phi_j(\SCL_{x,j} \setminus (\eta_{x,j}|_{[0, \sigma_{x,j}]} \cup \widetilde\eta_{x,j}|_{[0, \widetilde\sigma_{x,j}]})) \cap V_{ c_4j / 2,c_4j} =\emptyset$,  in which case 
\begin{align*}
D_{\Gamma_{V_{r_j}}|_{V_{c_4 j /2 ,  c_4 j}^{\star}}}^{V_{c_4 j/2 ,  c_4 j}^{\star}}(T_{r_j} ,  B_{r_j}) = D_{\Gamma_{\BD}|_{\phi_j^{-1}(V_{c_4 j/2,c_4 j}^{\star})}}^{\phi_j^{-1}(V_{c_4 j / 2,c_4 j}^{\star})}(\phi_j^{-1}(T_{r_j}),\phi_j^{-1}(B_{r_j})).	
\end{align*}	

 \substepn{step:hit2Ej-concl}{Conclusion} Thus, we conclude that there exists $c_6 = c_6(\zeta, \lambda) > 0$ such that it holds with conditional probability at least $1 - c_6/j$ given $\SCB_{x,j}^-,  \eta_{x,j}|_{[0,\sigma_{x,j}]},\widetilde{\eta}_{x,j}|_{[0,\widetilde{\sigma}_{x,j}]},  \SCB_{y,j},F_j^2 \cap E_j$,  that the $D_{\Gamma_{\BD}}^{\BD}$-distance between $\SCB_{y,j}$ and $\SCB_{x,j}$ is strictly less than the $D_{\Gamma_{\BD}}^{\BD}$-distance from $\SCB_{y,j}$ to $B_{\lambda^{j+1}}(z)$ without using $\SCB_{x,j}$ More precisely $D^\BD_{\Gamma_\BD}(\SCB_{x,j}, \SCB_{y,j}) <\inf\{t\ge 0: \SCB_t(\SCB_{y,j}; D^\BD_{\Gamma_\BD})  \cap B_{\lambda^{j+1}}(z) \neq \emptyset \text{ and }  \SCB_t(\SCB_{y,j}; D^\BD_{\Gamma_\BD}) \cap \SCB_{x,j} = \emptyset\}$.  Clearly,  if the latter occurs,  the event $E_{j+1}$ does not occur.  This completes the proof of~\eqref{eq:hitting-two-metric-balls-proof-5} in the case that both $\widetilde{K}_1$ and $\widetilde{K}_2$ intersect $\partial \BD$.  
 
 By arguing as in \Cref{step:hit2-Fj1} and using the same argument as in the case that both $\widetilde{K}_1$ and $\widetilde{K}_2$ intersect $\partial \BD$,  we obtain that ~\eqref{eq:hitting-two-metric-balls-proof-5} still holds in the case that either $\widetilde{K}_1 \cap \partial \BD = \emptyset$ or $\widetilde{K}_2 \cap \partial \BD = \emptyset$.

\step{step:hit2-concl}{Bounding $\BP\lbrack E_{j + 1} \mid F_j^3 \cap E_j\rbrack$ and concluding the proof} It follows from a similar (but easier) argument to the argument applied in the proof of~\eqref{eq:hitting-two-metric-balls-proof-5} that
	\begin{equation}\label{eq:hitting-two-metric-balls-proof-6}
		\BP\lbrack E_{j + 1} \mid F_j^3 \cap E_j\rbrack = O(1/j) \quad \text{as } j \to \infty, 
	\end{equation}
	at a rate depending only on $\zeta$ and $\lambda$. (The proof of~\eqref{eq:hitting-two-metric-balls-proof-6} is exactly the same as the proof of~\eqref{eq:hitting-two-metric-balls-proof-5}, except that we do not need to concern ourselves with the chordal SLE$_4$ from $\eta_{x,j}(\sigma_{x,j})$ to $\widetilde\eta_{x,j}(\widetilde\sigma_{x,j})$.) Since $\zeta$ was arbitrary, by combining~\eqref{eq:hitting-two-metric-balls-proof-2},~\eqref{eq:hitting-two-metric-balls-proof-3},~\eqref{eq:hitting-two-metric-balls-proof-4},~\eqref{eq:hitting-two-metric-balls-proof-5}, and~\eqref{eq:hitting-two-metric-balls-proof-6}, we complete the proof of \Cref{prop:hitting-two-metric-balls}.
	\end{proof}

\section{Vertical versus horizontal distance across a rectangle}
\label{sec:vertical_vs_horizontal}
	
Throughout this section, $D$ denotes a weak geodesic CLE$_4$ metric coupling in the sense of \Cref{def:weak_axioms}. Next, we state and prove the following lemma, which will play a crucial role in subsequent work \cite{kkmt2026cle4_part3}. It states that if we have the rectangle $V_a$ (for any fixed $a > 0$) and a non-nested CLE$_4$ on $V_a$, then with positive probability, the distance between the left and right boundaries of $V_a$ with respect to the CLE$_4$ metric is strictly less than the distance between the top and bottom boundaries of $V_a$.

\begin{lemma}\label{lem:vertical_vs_horizontal}
Fix $a>0$ and let $\Gamma_{V_a}$ be a non-nested CLE$_4$ in $V_a$. Then, there exists $r_0>0$ such that with positive probability, we have that 
\begin{align*}
D_{\Gamma_{V_a}}^{V_a}(\{0\} \times [0,1], \{a\} \times [0,1]) <r_0 \quad \text{and} \quad 2r_0 \le D_{\Gamma_{V_a}}^{V_a}([0,a] \times \{0\}, [0,a] \times \{1\}).
\end{align*}
Moreover, the following holds with positive probability. Let $L_2 = \{0\} \times [1/5,4/5]$ and let $L = \{0\} \times [0, 1]$ be the left boundary of $V_a$. Let $\partial_\mathrm{R} \SCB_{r_0}(L; D^{V_a}_{\Gamma_{V_a}})$, the \emph{right boundary} of the metric ball (it faces the right side of $V_a$ but may also meet $L$), be the intersection of $\SCB_{r_0}(L; D^{V_a}_{\Gamma_{V_a}})$ with the closure of the union of the connected components $C$ of $V_a \setminus \SCB_{r_0}(L; D^{V_a}_{\Gamma_{V_a}})$ such that $\overline{C} \cap (\{a\} \times (0,1)) \neq \emptyset$. Then:
	\begin{enumerate}[label=(\alph*),ref=\alph*]
		\item\label{it: lemma crossing a} We have $D_{\Gamma_{V_a}}^{V_a}(L_2, \{a\} \times [1/5,4/5]) < r_0$;
		\item\label{it: lemma crossing b} The right boundary $\partial_\mathrm{R} \SCB_{r_0}(L; D^{V_a}_{\Gamma_{V_a}})$ intersects $L_2$;
		\item\label{it: lemma crossing c} We have $D_{\Gamma_{V_a}}^{V_a}([0,a] \times \{0\}, [0,a] \times \{1/5\}) \ge 2r_0$ and $D_{\Gamma_{V_a}}^{V_a}([0,a] \times \{4/5\}, [0,a] \times \{1\}) \ge 2r_0$.
	\end{enumerate} 

\end{lemma}

\begin{figure}
	\centering
	\includegraphics[width=\linewidth]{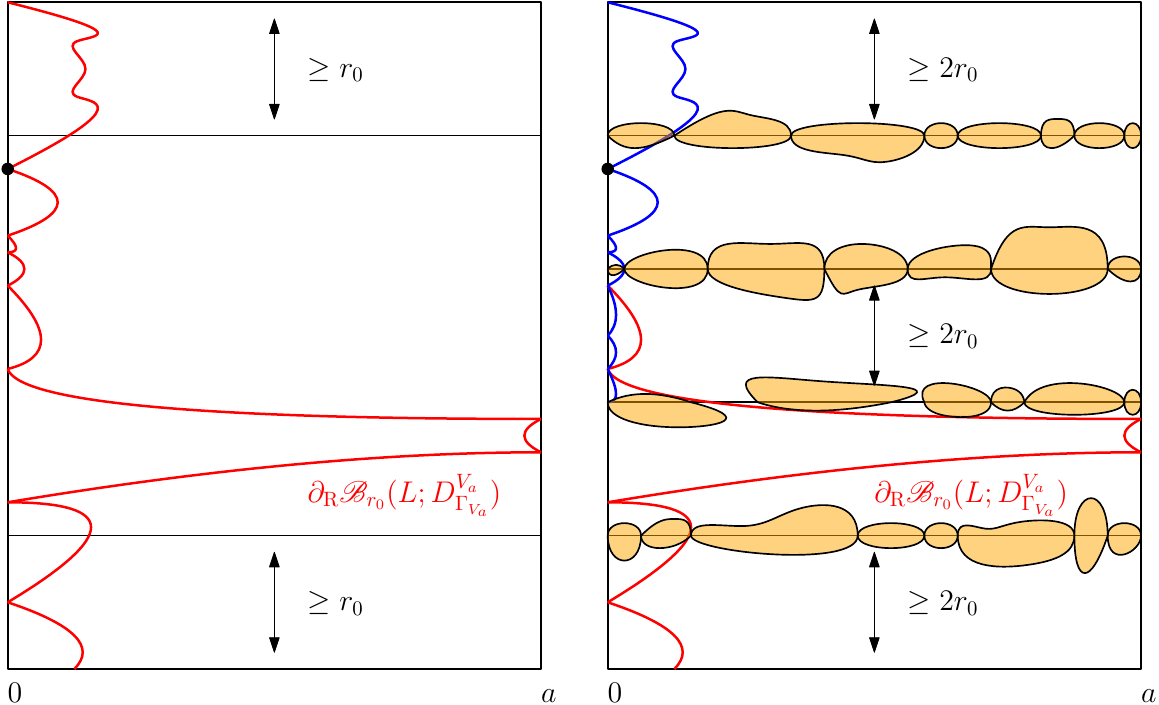}
	\caption{\textbf{Left:} Illustration of the statement of \Cref{lem:vertical_vs_horizontal}. The $D^{V_a}_{\Gamma_{V_a}}$-distance from top to bottom across the top and bottom rectangles is at least $r_0$, and the metric ball $\SCB_{r_0}(L; D^{V_a}_{\Gamma_{V_a}})$ hits the right-hand side of $V_a$ and its right boundary intersects $L_2$ at the black point. \textbf{Right:} Illustration of the proof of \Cref{lem:vertical_vs_horizontal}. The rectangle $V_a$ is divided into five rectangles from bottom to top, $G_1, \ldots, G_5$. The distance from top to bottom across each of $G_1, G_3$, and $G_5$ is at least $2r_0$. The distance from left to right across $G_2$ is smaller than $r_0$. Moreover, the right boundary of the metric ball starting from $\{0\} \times [2/5, 1]$ (in blue) intersects the left boundary of $G_4$. The right boundary of the metric ball starting from $L$ (in red) agrees with the blue boundary in $G_4$ and $G_5$ since the distance from top to bottom across $G_3$ is at least $2r_0$.}
	\label{fig:crossing-rectangle}
\end{figure}

See the left-hand side of \Cref{fig:crossing-rectangle} for an illustration of the above lemma. \Cref{lem:vertical_vs_horizontal} will follow from \Cref{lem:loop_surrounds_ball,lem:no_loop_crossing}, which are stated and proved below. We start by stating and proving \Cref{lem:loop_surrounds_ball}. It states that for any fixed $r \in (0,1)$, with positive probability, some loop of a CLE$_4$ on $\BD$ surrounds $B_r(0)$.

\begin{lemma}\label{lem:loop_surrounds_ball}
Fix $r \in (0,1)$ and suppose that $\Gamma_{\BD}$ is a CLE$_4$ on $\BD$. Then, there exists $p \in (0,1)$, depending only on $r$, such that with probability at least $p$, $B_r(0)$ is contained in the interior of some loop in $\Gamma_{\BD}$.
\end{lemma}

\begin{proof}
    This follows from the construction of \Cref{subsec:brownian_loop_soup}: the Brownian loop measure of the collection of Brownian loops contained in $A_{r,1}(0)$ and disconnecting its inner and outer boundaries is positive, so with positive probability some soup loop surrounds $B_r(0)$; the loop of $\Gamma_\BD$ bounding the outermost cluster containing or surrounding it then has $B_r(0)$ in its interior.
\end{proof}

\begin{lemma}\label{lem:no_loop_crossing}
Fix $a>0$ and $0 < r_1 < r_2 < 1$, and let $\Gamma_{V_a}$ be a non-nested CLE$_4$ in $V_a$. Then, with positive probability, no loop of $\Gamma_{V_a}$ intersects both the top and the bottom boundary of the rectangle $(0,a) \times (r_1, r_2)$.
\end{lemma}

\begin{proof}
    Consider the coupling of CLE$_4$ with a GFF as its level lines as in \Cref{subsec:nested_cle_gff}. Then, it holds with positive probability that the level line of the GFF with height $0$, started from $(r_1 + r_2)\ri/2$ and targeted at $a + (r_1+r_2)\ri/2$, hits $\{a\} \times (r_1, r_2)$ before hitting $(0,a) \times \{r_1\}$ or $(0,a) \times \{r_2\}$ (by \cite[Lemma~2.8]{miller2017intersections} and \cite[Lemma~A.1]{CoInCLERiemSph}, as in \Cref{sec:distances_across_rectangles}). On this event, there cannot exist a loop of CLE$_4$ intersecting both $(0,a) \times \{r_1\}$ and $(0,a) \times \{r_2\}$: the stopped level line joins the left and right sides of $(0,a) \times (r_1, r_2)$ within its closure, so such a loop would cross it, which is impossible since the loops of $\Gamma_{V_a}$ are level loops of heights $\pm \lambda$, 
    and always stay on the same side of a height-$0$ level line (cf.\ \Cref{thm:level_line_interaction}).
\end{proof}

Before proving \Cref{lem:vertical_vs_horizontal}, we will give a brief description of the main ideas behind its proof and describe the setup. See also the right-hand side of \Cref{fig:crossing-rectangle} for some of the ideas of the proof.

For each $j \in \{1,2,3,4\}$, let $\Gamma_j$ denote the collection of loops in $\Gamma_{V_a}$ which intersect $(0,a) \times \{j/5\}$. Then, \Cref{lem:no_loop_crossing}, together with the FKG inequality for PPPs (see \cite[Theorem 20.4]{LP18LecturePoissonProcess}) applied to the Brownian loop soup of \Cref{subsec:brownian_loop_soup} (the three events $\{\Gamma_j \cap \Gamma_{j+1} = \emptyset\}$ are decreasing in the soup, as adding a loop enlarges clusters and hence the loops of $\Gamma_{V_a}$), implies that there exists $p \in (0,1)$ such that with probability at least $p$, we have that
\begin{align}\label{eqn:loops_not_intersecting}
\Gamma_1 \cap \Gamma_2 = \Gamma_2 \cap \Gamma_3 = \Gamma_3 \cap \Gamma_4 = \emptyset
\end{align}
and that $\bigcup_{\SCL \in \Gamma_1 \cup \Gamma_2} \overline{\mathop{\mathrm{int}}(\SCL)}$ does not separate $\{0\} \times (1/5, 2/5)$ from $\{a\} \times (1/5, 2/5)$ within $(0,a) \times (1/5, 2/5)$. (This fourth event is also decreasing in the soup, since the filled loops only grow, and it has positive probability: the stopped level line in the proof of \Cref{lem:no_loop_crossing}, run for the strip $(0,a) \times (1/5, 2/5)$, avoids every interior of loop.)

Let $U$ denote the connected component of $V_a \setminus \overline{\bigcup_{\SCL \in \Gamma_1 \cup \Gamma_2} \SCL}$ whose boundary intersects both $\{0\} \times (1/5, 2/5)$ and $\{a\} \times (1/5, 2/5)$. On the above event, $U$ exists and is unique: each loop of $\Gamma_1 \cup \Gamma_2$ is at positive distance from the two segments, and an accumulation of infinitely many loops of $\Gamma_1 \cup \Gamma_2$ lies on one of the horizontal lines, so each segment lies on the boundary of a single component, and a path realizing the non-separation joins the two. In particular, the two segments are free boundary arcs of $U$, and $\ri/5$, $2\ri/5$, $a + \ri/5$, $a + 2\ri/5$ denote the prime ends of $U$ obtained as their endpoint limits.
Suppose that we are working on the event that~\eqref{eqn:loops_not_intersecting} occurs. Let $\phi$ be the conformal mapping from $U$ onto $\BD$ such that $\phi(\ri/5) = -\ri$, $\phi(2\ri/5) = \ri$, and $\phi(a + 2\ri/5) = 1$. On the event~\eqref{eqn:loops_not_intersecting}, the four quantities appearing below are a.s.\ positive; hence there exists a deterministic constant $r_0 \in (0,1/2)$ such that, after possibly decreasing $p \in (0,1)$, with probability at least $p$, we have in addition to~\eqref{eqn:loops_not_intersecting} that $\dist(\phi(a + \ri/5), \{-\ri, 1\}) \geq r_0$ and the $D_{\Gamma_{V_a}}^{V_a}$-distance between $[0,a] \times \{1/5\}$ and $[0,a] \times \{0\}$ is at least $2r_0$, and that the same is true for the $D_{\Gamma_{V_a}}^{V_a}$-distance between $[0,a] \times \{4/5\}$ (resp.\ $[0,a] \times \{2/5\}$) and $[0,a] \times \{1\}$ (resp.\ $[0,a] \times \{3/5\}$). These two inequalities for the $D_{\Gamma_{V_a}}^{V_a}$-distance can be obtained using the fact that, on the event that~\eqref{eqn:loops_not_intersecting} occurs, the $D_{\Gamma_{V_a}}^{V_a}$-distances which are considered are a.s.\ nonzero: a zero distance would force the closure of the union of the filled loops meeting one segment to meet the other, by right-continuity and closedness of the ball process; each such loop has diameter at least $1/5$, so only finitely many approach the other segment, and each is disjoint from it (loops avoid $\partial V_a$, and~\eqref{eqn:loops_not_intersecting} handles the interior pair). In particular, we already have condition~\eqref{it: lemma crossing c} of \Cref{lem:vertical_vs_horizontal}. 

Moreover, let $V^\dagger$ be the conformal rectangle defined as the connected component of $((0,a)\times (2/5, 1)) \setminus \overline{\bigcup_{\SCL \in \Gamma_2} \SCL }$, whose top boundary is $[0,a] \times \{1\}$, whose bottom boundary is included in the closure of the union of loops in $\Gamma_2$ (a.s., this closure covers $(0,a) \times \{2/5\}$: the interiors of the loops of $\Gamma_2$ cover a dense open subset of this segment, and the endpoints of each maximal covered subinterval lie on loops of $\Gamma_2$), and whose left and right boundaries are $\{0\}\times [2/5, 1]$ and $\{a\}\times [2/5, 1]$, respectively (loops of $\Gamma_2$ accumulate on the vertical sides only at the corners $2\ri/5$ and $a + 2\ri/5$, since such a loop within distance $\delta$ of a side point at height at least $2/5 + 2\delta$ has diameter at least $\delta$); we mark the corners $2\ri/5$, $\ri$, $a + \ri$, and $a + 2\ri/5$. By possibly taking $p$ and $r_0$ to be smaller, on the same event as above, there exists a connected component $C$ of $V^\dagger \setminus \SCB_{r_0}(\{0\} \times [2/5, 1]; D^{V^\dagger}_{\Gamma_{V_a}\vert_{V^\dagger}})$ such that $\overline{C}$ contains the right-hand side of $V^\dagger$. Actually, by possibly taking $p$ and $r_0$ to be even smaller, thanks to \Cref{prop BCLE rho to zero}, and using that the metric ball starting from the left-hand side of $V^\dagger$ is included in the metric ball starting from $\partial V^\dagger$, on the same event as above, we have that
\begin{equation}\label{eq right boundary intersects left boundary}
	\overline{C} \cap \SCB_{r_0}(\{0\} \times [2/5, 1]; D^{V^\dagger}_{\Gamma_{V_a}\vert_{V^\dagger}}) \cap (\{0\} \times (3/5, 4/5)) \neq \emptyset.
\end{equation}
Since crossing the strip $(0,a) \times (2/5, 3/5)$ costs at least $2r_0$, the ball $\SCB_{r_0}(\{0\} \times [0,1]; D^{V_a}_{\Gamma_{V_a}})$ can enter $V^\dagger$ only through its left side, where it agrees with $\SCB_{r_0}(\{0\} \times [2/5, 1]; D^{V^\dagger}_{\Gamma_{V_a}\vert_{V^\dagger}})$ (the arc analog of the ball consistency in Axiom~\eqref{it:weak_axiom_locality}), and $C$ lies in a component of $V_a \setminus \SCB_{r_0}(\{0\} \times [0,1]; D^{V_a}_{\Gamma_{V_a}})$ whose closure meets $\{a\} \times (0,1)$. Hence~\eqref{eq right boundary intersects left boundary} implies that $\partial_\mathrm{R}\SCB_{r_0}(\{0\} \times [0,1]; D^{V_a}_{\Gamma_{V_a}})$ intersects $\{0\}\times (3/5, 4/5)$ (in the notation of \Cref{lem:vertical_vs_horizontal}). See the right-hand side of \Cref{fig:crossing-rectangle}. In particular,~\eqref{eq right boundary intersects left boundary} implies condition~\eqref{it: lemma crossing b} of \Cref{lem:vertical_vs_horizontal}.

The main idea of the proof of \Cref{lem:vertical_vs_horizontal} is the following. First, we note that \Cref{lem:loop_surrounds_ball} implies that it is a positive probability event that we can find a loop $\SCL$ in $\Gamma_{V_a}|_U$ whose $D_{\Gamma_{V_a}|_U}^U$-distance to all of the marked boundary arcs of the conformal rectangle $U$ is small but fixed. Then, conditionally on the above event, it is very likely that the $D_{\Gamma_{V_a}|_U}^U$-distances between $\SCL$ and each of the two marked boundary arcs of $U$ contained in $(\{0\} \times (0,1)) \cup (\{a\} \times (0,1))$ are much smaller than the $D_{\Gamma_{V_a}}^{V_a}$-distance between $(0,a) \times \{0\}$ and $(0,a) \times \{1\}$. Thus, combining this with the locality property of $D_{\Gamma_{V_a}}^{V_a}$ and the fact that $\SCL$ is a loop in $\Gamma_{V_a}$, we will complete the proof of \Cref{lem:vertical_vs_horizontal}.

Fix $r \in (0,1)$ sufficiently close to $1$ (to be chosen), and let $\mathcal{R}_1$ (resp.\ $\mathcal{R}_2$) be the conformal rectangle bounded by the clockwise arc of $\partial \BD$ from $-\ri $ to $\ri$, the segment $[\ri,\ri r]$, the counterclockwise arc of $\partial B_r(0)$ from $\ri r$ to $-\ri r$, and the segment $[-\ri r,-\ri ]$ (resp.\ the counterclockwise arc of $\partial \BD$ from $\phi(a + \ri/5)$ to $1$, the segment $[1,r]$, the clockwise arc of $\partial B_r(0)$ from $r$ to $r e^{\ri \, \text{arg}(\phi(a + \ri/5))}$ and the segment $[r e^{\ri \, \text{arg}(\phi(a + \ri/5))},\phi(a + \ri/5)]$). Let $\phi_1$ be the unique conformal mapping from $\mathcal{R}_1$ onto the rectangle $(0,w_1) \times (0,1)$ such that $\phi_1(-\ri) = 0$, $\phi_1(-\ri r) = w_1$, $\phi_1(\ri r) = w_1 + \ri$, and $\phi_1(\ri) = \ri$ for some $w_1 \in (0,\infty)$. Similarly we let $\phi_2$ be the unique conformal mapping from $\mathcal{R}_2$ onto the rectangle $(0,w_2) \times (0,1)$ such that $\phi_2(r e^{\ri \, \text{arg}(\phi(a + \ri/5))}) = 0$, $\phi_2(\phi(a + \ri/5)) = w_2$, $\phi_2(1) = w_2 + \ri$, and $\phi_2(r) = \ri$ for some $w_2 \in (0,\infty)$. (Note that $\mathcal{R}_1$ is deterministic while $\mathcal{R}_2$ is random.)

\begin{figure}
	\centering
	\includegraphics[width=\linewidth]{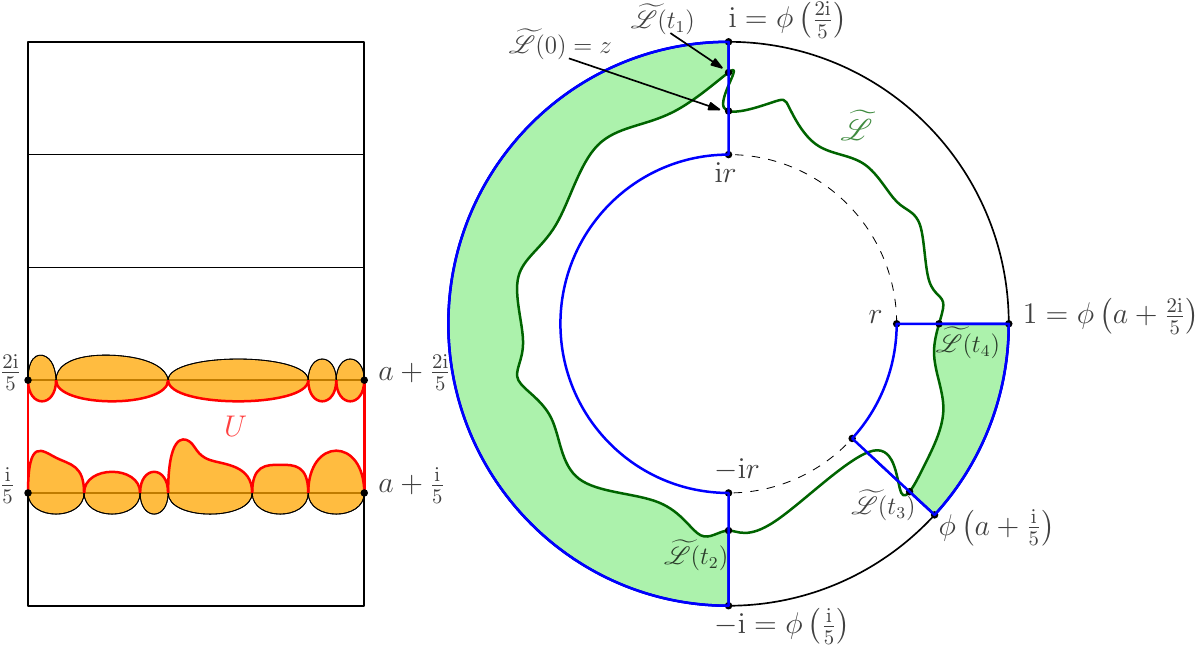}
	\caption{\textbf{Left:} Illustration of the topological rectangle $U$ in the proof of \Cref{lem:vertical_vs_horizontal}. \textbf{Right:} Illustration of the image of $U$ by $\phi$. The topological rectangles $\mathcal{R}_1$ and $\mathcal{R}_2$ have their sides colored in blue, while the topological rectangles $\widetilde{\mathcal{R}}_1$ and $\widetilde{\mathcal{R}}_2$ are colored in green. The loop $\widetilde{\SCL}$ is in dark green.}
	\label{fig:UR1R2}
\end{figure}

\begin{proof}[Proof of \Cref{lem:vertical_vs_horizontal}]
\step{step:vh-outline}{Outline and setup} 
We begin by describing the main steps of the proof. Suppose that we have the setup described in the paragraphs following the statement of the lemma and recall that we have fixed $p \in (0,1), r_0 \in (0,1/2)$ such that with probability at least $p$, we have that \begin{itemize}
	\item The event~\eqref{eqn:loops_not_intersecting} occurs and $\dist(\phi(a + \ri/5), \{-\ri, 1\}) \geq r_0$;
	\item The $D_{\Gamma_{V_a}}^{V_a}$-distance between $[0,a] \times \{1/5\}$ and $[0,a] \times \{0\}$ is at least $2r_0$, and that the same is true for the $D_{\Gamma_{V_a}}^{V_a}$-distance between $[0,a] \times \{4/5\}$ (resp.\ $[0,a] \times \{2/5\}$) and $[0,a] \times \{1
		\}$ (resp.\ $[0,a] \times \{3/5\}$);
	\item We have~\eqref{eq right boundary intersects left boundary}.
\end{itemize}
From now on, we assume that we are working on the event that the above events occur. 

First, we will show in \Cref{step:vh-rects} that the rectangles $\mathcal{R}_1$ and $\mathcal{R}_2$ are not too wide in the sense that both $w_1$ and $w_2$ are bounded by some deterministic constant. Next, in \Cref{step:vh-thin}, we will construct thin rectangles $\widetilde{\mathcal{R}}_1$ and $\widetilde{\mathcal{R}}_2$ satisfying the following property. The arc $\phi([\ri/5,2\ri/5])$ (resp.\ $\phi([a + \ri/5,a + 2\ri/5])$) is one of the marked boundary arcs of $\widetilde{\mathcal{R}}_1$ (resp.\ $\widetilde{\mathcal{R}}_2$), and there is a marked boundary arc $\widetilde{I}_1$ (resp.\ $\widetilde{I}_2$) of $\widetilde{\mathcal{R}}_1$ (resp.\ $\widetilde{\mathcal{R}}_2$) such that $\widetilde{I}_1 \cup \widetilde{I}_2 \subset \widetilde{\SCL}$ for some loop $\widetilde{\SCL}$ in $\phi(\Gamma_{V_a}|_U)$ such that the $D_{\phi(\Gamma_{V_a}|_U)|_{\widetilde{\mathcal{R}}_1}}^{\widetilde{\mathcal{R}}_1}$ (resp.\ $D_{\phi(\Gamma_{V_a}|_U)|_{\widetilde{\mathcal{R}}_2}}^{\widetilde{\mathcal{R}}_2}$)-distance between $\phi([\ri/5,2\ri/5])$ and $\widetilde{I}_1$ (resp.\ $\phi([a + \ri/5,a + 2\ri/5])$ and $\widetilde{I}_2$) is very small. In particular, this will imply that the $D_{\phi(\Gamma_{V_a}|_U)}^{\phi(U)}$-distance between $\phi([\ri/5,2\ri/5])$ and $\phi([a + \ri/5,a + 2\ri/5])$ is very small. Finally, in \Cref{step:vh-concl}, we will complete the proof of the lemma using the conformal invariance property of the CLE$_4$ metric.

\step{step:vh-rects}{The conformal rectangles ${\mathcal{R}}_1$ and ${\mathcal{R}}_2$ are not too wide}

\substepn{step:vhr-R1}{The rectangle $\mathcal{R}_1$} Fix $q \in (\max\{p, 1/2\},1)$, so that $1 - 2(1-q) > 0$. Note that there exists $b_0 \in (0,1)$ such that for all $b \in (0,b_0)$, the following holds with probability at least $q$. Let $\Gamma_{V_b}$ be a CLE$_4$ on $V_b$. Then the $D_{\Gamma_{V_b}}^{V_b}$-distance between $\{0\} \times (0,1)$ and $\{b\} \times (0,1)$ is at most $r_0 / 100$. (This follows from \Cref{cor:moment_across_rectangle}, applied after a rotation and conformal invariance to identify this distance with the top-bottom distance across a $(1/b) \times 1$ rectangle.) We claim that we can choose $r \in (0,1)$ sufficiently close to $1$ (depending only on $r_0$ and $b_0$) such that $\max\{w_1,w_2\} < b_0$. Indeed, fix $z \in \mathcal{R}_1$. Then if $r \in (0,1)$ is sufficiently close to $1$ in a universal manner, we have that either $\dist(z,[\ri r,\ri]) \geq 1/4$ or $\dist(z,[-\ri, -\ri r]) \geq 1/4$. Suppose that the former case holds. Then the Beurling estimate implies that there exists a universal constant $C \in (0,\infty)$ such that the probability that a complex Brownian motion starting from $z$ exits $\mathcal{R}_1$ on the segment $[\ri r,\ri]$ is at most $C(4(1-r))^{1/2}$. Similarly, if $\dist(z,[-\ri, -\ri r]) \geq 1/4$, then the probability that a complex Brownian motion starting from $z$ exits $\mathcal{R}_1$ on $[-\ri, -\ri r]$ is at most $C (4(1-r))^{1/2}$. In either case, the minimum of these two probabilities is at most $C(4(1-r))^{1/2}$; taking the supremum over $z \in \mathcal{R}_1$ and using the conformal invariance of complex Brownian motion, we obtain
\begin{align*}
	\Theta\left(\pi / w_1\right) \le 2 C \left(4(1-r)\right)^{1/2},
\end{align*}
where $\Theta$ is as in \Cref{lem:extremal_length}, applied to $\mathcal{R}_1$ with the two pairs of marked sides interchanged (so the relevant modulus is $1/w_1$). Thus, choosing $r$ close enough to $1$ (depending only on $b_0$) that $2C(4(1-r))^{1/2} < \Theta(\pi / b_0)$, we obtain that $w_1 < b_0$ since $\Theta$ is continuous and strictly decreasing (see \Cref{lem:extremal_length}).

\substepn{step:vhr-R2}{The rectangle $\mathcal{R}_2$} Similarly, if $z \in \mathcal{R}_2$, then we have that either $\dist(z,[r,1]) \geq r_0 / 4$ or $\dist(z,[r e^{\ri \arg(\phi(a + \ri/5))}, \phi(a + \ri/5)]) \geq r_0 / 4$ since $\min\{|\phi(a + \ri/5)-1|, |\phi(a + \ri/5)+\ri|\} \geq r_0$.  Therefore,  by arguing as above,  we obtain that $\Theta(\pi / w_2) \le 2C(4(1-r)/r_0)^{1/2} < \Theta(\pi / b_0)$ and so $w_2 < b_0$.  This proves the claim.

\step{step:vh-thin}{Constructing thin rectangles $\widetilde{\mathcal{R}}_1$ and $\widetilde{\mathcal{R}}_2$}

\substepn{step:vht-loop}{Conditioning and the loop $\widetilde{\SCL}$} We construct conformal rectangles $\widetilde{\mathcal{R}}_1$ and $\widetilde{\mathcal{R}}_2$ as described in \Cref{step:vh-outline}. First, by the domain Markov property and conformal invariance of the $\CLE_4$, conditionally on the loops of $\Gamma_{V_a}$ which intersect $([0,a] \times \{1/5\}) \cup ([0,a] \times \{2/5\})$, 
$\phi(\Gamma_{V_a}\vert_U)$ is a CLE$_4$ on $\BD$. Therefore, \Cref{lem:loop_surrounds_ball} implies that by taking $p \in (0,1)$ to be smaller, we have that the conditional probability given $U$ that there exists a loop $\widetilde{\SCL} \in \phi(\Gamma_{V_a}|_U)$ such that $\widetilde{\SCL}$ surrounds $B_r(0)$ is at least $p$. On this event, we let $\widetilde\SCL$ be the outermost such loop.

\substepn{step:vht-build}{Construction of $\widetilde{\mathcal{R}}_1$ and $\widetilde{\mathcal{R}}_2$} Since $\widetilde\SCL$ separates $\partial B_r(0)$ from $\partial\BD$, it crosses $\mathcal{R}_1$ between $[\ri r, \ri]$ and $[-\ri, -\ri r]$: there are $t_1 < t_2$ with $\widetilde\SCL(t_1) \in [\ri r, \ri]$, $\widetilde\SCL(t_2) \in [-\ri, -\ri r]$, and $\widetilde\SCL|_{(t_1,t_2)} \subset \mathcal{R}_1$. Let $\widetilde{\mathcal{R}}_1$ be the conformal rectangle bounded by the clockwise arc of $\partial \BD$ from $-\ri$ to $\ri$ and the segments $[\ri,\widetilde{\SCL}(t_1)]$, $\widetilde{\SCL}([t_1,t_2])$ and $[\widetilde{\SCL}(t_2),-\ri]$. Similarly, $\widetilde\SCL$ contains a crossing of $\mathcal{R}_2$: there are times $t_3 < t_4$ such that $\widetilde\SCL(t_3) \in [r e^{\ri \arg(\phi(a + \ri/5))}, \phi(a + \ri/5)]$, $\widetilde\SCL(t_4) \in [r,1]$, and $\widetilde\SCL|_{(t_3,t_4)} \subset \mathcal{R}_2$. Let $\widetilde{\mathcal{R}}_2$ be the conformal rectangle bounded by the counterclockwise arc of $\partial \BD$ from $\phi(a + \ri/5)$ to $1$ and the segments $[1,\widetilde{\SCL}(t_4)]$, $\widetilde{\SCL}([t_3,t_4])$ and $[\widetilde{\SCL}(t_3),\phi(a + \ri/5)]$. Let $\widetilde{\phi}_1$ be the conformal mapping from $\widetilde{\mathcal{R}}_1$ onto a rectangle $(0,\widetilde{W}_1) \times (0,1)$ such that $\widetilde{\phi}_1(-\ri) = 0$, $\widetilde{\phi}_1(\widetilde{\SCL}(t_2)) = \widetilde{W}_1$, $\widetilde{\phi}_1(\widetilde{\SCL}(t_1)) = \widetilde{W}_1+\ri$, and $\widetilde{\phi}_1(\ri) = \ri $ for some $\widetilde{W}_1 \in (0,\infty)$. Since $\widetilde{\mathcal{R}}_1 \subset \mathcal{R}_1$ and since for all $z \in \widetilde{\mathcal{R}}_1$, the probability that a planar Brownian motion starting from $z$ stopped when it hits $\partial \widetilde{\mathcal{R}}_1$ hits the left or right boundaries of $\widetilde{\mathcal{R}}_1$ is at least the probability that a planar Brownian motion starting from $z$ stopped when it hits $\partial{\mathcal{R}}_1$ hits the left or right boundaries of ${\mathcal{R}}_1$ (this comes from the strong Markov property of the Brownian motion), we see that 
$\widetilde{W}_1 \leq w_1 < b_0$. Similarly we let $\widetilde{\phi}_2$ be the conformal mapping from $\widetilde{\mathcal{R}}_2$ onto a rectangle $(0,\widetilde{W}_2) \times (0,1)$ such that $\widetilde{\phi}_2(\widetilde{\SCL}(t_3)) = 0$, $\widetilde{\phi}_2(\phi(a + \ri/5)) = \widetilde{W}_2$, $\widetilde{\phi}_2(1) = \widetilde{W}_2 + \ri$, and $\widetilde{\phi}_2(\widetilde{\SCL}(t_4)) = \ri $ for some $\widetilde{W}_2 \in (0,\infty)$. The same reasoning gives $\widetilde{W}_2 \leq w_2 < b_0$.

\substepn{step:vht-estimate}{The distance estimate} The choice of $q$ implies that conditionally on $U$ and on the event that there exists $\widetilde{\SCL} \in \phi(\Gamma_{V_a}|_U)$ such that $\widetilde{\SCL}$ surrounds $B_r(0)$, we have a.s.\ that, with conditional probability at least $1-2(1-q)>0$, the distance between the clockwise arc of $\partial \BD$ from $-\ri$ to $\ri$ and $\widetilde{\SCL}([t_1,t_2])$ with respect to $D_{\phi(\Gamma_{V_a}|_U)|_{\widetilde{\mathcal{R}}_1}}^{\widetilde{\mathcal{R}}_1}$, and the distance between the counterclockwise arc of $\partial \BD$ from $\phi(a + \ri/5)$ to $1$ and $\widetilde{\SCL}([t_3,t_4])$ with respect to $D_{\phi(\Gamma_{V_a}|_U)|_{\widetilde{\mathcal{R}}_2}}^{\widetilde{\mathcal{R}}_2}$, are both at most $r_0 / 100$. Hence, we obtain that if the above event occurs, we have that the $D_{\phi(\Gamma_{V_a}|_U)}^{\phi(U)}$-distance between the clockwise arc of $\partial \BD$ from $-\ri$ to $\ri$ and the counterclockwise arc of $\partial \BD$ from $\phi(a + \ri/5)$ to $1$ is at most $r_0 / 50$: both crossings end on the same loop $\widetilde\SCL$, a single point of the metric, so their costs add, and each $\widetilde{\mathcal{R}}_j$-distance dominates the corresponding $D_{\phi(\Gamma_{V_a}|_U)}^{\phi(U)}$-distance by the monotonicity in Axiom~\eqref{it:weak_axiom_locality}. 

\step{step:vh-concl}{Conclusion of the proof} Combining \Cref{step:vh-outline,step:vh-rects,step:vh-thin}, we obtain that with probability at least $p^2 (1-2(1-q))>0$ (the event of \Cref{step:vh-outline}, times the conditional probability at least $p$ that some loop surrounds $B_r(0)$, times the conditional probability at least $1-2(1-q)$ of the two crossing estimates), the following holds. The $D_{\Gamma_{V_a}}^{V_a}$-distance between $\{0\} \times [1/5,2/5]$ and $\{a\} \times [1/5,2/5]$ is at most $r_0 / 50$ and the $D_{\Gamma_{V_a}}^{V_a}$-distance between $[0,a] \times \{1/5\}$ (resp.\ $[0, a] \times \{4/5\}$, $[0, a] \times \{2/5\}$) and $[0,a] \times \{0\}$ (resp.\ $[0,a] \times \{1\}$, $[0, a] \times \{3/5\}$) is at least $2r_0$, as well as the fact that~\eqref{eq right boundary intersects left boundary} holds. Note that this gives conditions~\eqref{it: lemma crossing a}, \eqref{it: lemma crossing b}, and~\eqref{it: lemma crossing c} of the lemma. Suppose that the above event occurs. To conclude, it is enough to show that
\begin{align*}
D_{\Gamma_{V_a}}^{V_a}(\{0\} \times [0,1], \{a\} \times [0,1]) \leq \frac{r_0}{50} < 2r_0 \leq D_{\Gamma_{V_a}}^{V_a}([0,a] \times \{0\}, [0,a] \times \{1\}).
\end{align*}
Recall that 
\begin{align*}
D_{\Gamma_{V_a}}^{V_a}([0,a] \times \{0\}, [0,a] \times \{1\}) = \inf\{t \geq 0 : \SCB_t([0,a] \times \{1\}; D_{\Gamma_{V_a}}^{V_a}) \cap ([0,a] \times \{0\}) \neq \emptyset\}.
\end{align*}
Furthermore, since $\SCB_t([0,a] \times \{1\}; D_{\Gamma_{V_a}}^{V_a})$ is connected (Axiom~\eqref{it:weak_axiom_geodesic}\eqref{it:weak_axiom_geodesic_0}) and the segment $[0,a] \times \{4/5\}$ is closed, so that it cannot be evaded through $\partial V_a$, the ball has to intersect $[0,a] \times \{4/5\}$ before intersecting $[0,a] \times \{0\}$. Thus, 
\[
D_{\Gamma_{V_a}}^{V_a}([0,a] \times \{0\}, [0,a] \times \{1\}) \ge D_{\Gamma_{V_a}}^{V_a}([0,a] \times \{4/5\}, [0,a] \times \{1\}) \ge 2r_0.
\]
Moreover, we have
\[
D_{\Gamma_{V_a}}^{V_a}(\{0\} \times [0,1], \{a\} \times [0,1]) \le D_{\Gamma_{V_a}}^{V_a}(\{0\} \times [1/5, 2/5], \{a\} \times [1/5, 2/5]) \le \frac{r_0}{50}.
\]
This completes the proof of the lemma.
\end{proof}

\section{Existence of geodesics}
\label{section:existence-of-geodesics}

The present section is devoted to proving the existence of geodesics (cf.~\Cref{def:axioms}, Axiom~\eqref{it:axiom_geodesic}). The proof strategy is as follows. By the conformal invariance property, it suffices to consider the case of $\BD$. We fix two loops $\SCL_1, \SCL_2 \in \Gamma_{\BD}$. We consider the loops $\SCL$ with $D_{\Gamma_\BD}^\BD(\SCL_1, \SCL) + D_{\Gamma_\BD}^\BD(\SCL, \SCL_2) = D_{\Gamma_\BD}^\BD(\SCL_1, \SCL_2)$, and construct finite families of such loops, totally ordered by the associated distances, becoming finer and finer, with successive loops joined by chains of dyadic squares of vanishing mesh. Using the estimates from \Cref{sec:distances_across_rectangles} and \Cref{sec:hitting_two_metric_balls}, we show that such families exist at every scale and converge to a geodesic connecting $\SCL_1$ and $\SCL_2$.

\begin{proposition}\label{prop:existence_geodesics}
Let $D$ be a weak geodesic CLE$_4$ metric coupling. Then $D$ satisfies Axiom~\eqref{it:axiom_geodesic} of \Cref{def:axioms}. Moreover, almost surely, for all $\SCL_1, \SCL_2 \in \Gamma_U$ and for all $D^\BD_{\Gamma_\BD}$-geodesic $P$ from $\SCL_1$ to $\SCL_2$, the set $\{D_{\Gamma_{\BD}}^{\BD}(\SCL_1,\SCL): \SCL \in \Gamma, \ \SCL\cap P([0,1]) \neq \emptyset\}$ forms a dense subset of the interval $[0, D_{\Gamma_{\BD}}^{\BD}(\SCL_1, \SCL_2)]$.
\end{proposition}

For $\SCL_1, \SCL_2 \in \Gamma_{\BD}$, we shall write $\Gamma_{\SCL_1,\SCL_2}$ for the collection of loops $\SCL \in \Gamma_{\BD}$ with $D_{\Gamma_{\BD}}^{\BD}(\SCL_1, \SCL) + D_{\Gamma_{\BD}}^{\BD}(\SCL, \SCL_2) = D_{\Gamma_{\BD}}^{\BD}(\SCL_1, \SCL_2)$; we recall from \Cref{def:weak_axioms}, Axiom~\eqref{it:weak_axiom_geodesic}\eqref{it:weak_axiom_geodesic_4}, the set $K_{\SCL_1,\SCL_2}$, the closure of the collection of points $x \in \overline\BD$ such that $D_{\Gamma_{\BD}}^{\BD}(\SCL_1, x) + D_{\Gamma_{\BD}}^{\BD}(x, \SCL_2) = D_{\Gamma_{\BD}}^{\BD}(\SCL_1, \SCL_2)$ (here $D_{\Gamma_\BD}^\BD(\SCL, x)$ is the distance to the point $x$ defined in \Cref{subsec:setup}), and the connected component $\widetilde K_{\SCL_1,\SCL_2}$ of $K_{\SCL_1,\SCL_2}$ containing $\SCL_1$ and $\SCL_2$, provided by that axiom. We shall write $\widetilde\Gamma_{\SCL_1,\SCL_2} \defeq \{\SCL \in \Gamma_{\SCL_1,\SCL_2} : \SCL \subset \widetilde K_{\SCL_1,\SCL_2}\}$; note that $D_{\Gamma_\BD}^\BD(\SCL_1, \cdot)$ is constant on each loop, so each loop of $\Gamma_{\SCL_1,\SCL_2}$ (resp.\ $\widetilde{\Gamma}_{\SCL_1,\SCL_2}$) is contained in $K_{\SCL_1,\SCL_2}$ (resp.\ $\widetilde{K}_{\SCL_1,\SCL_2}$). For each $n \in \BN$, we shall write $\CS_{\SCL_1,\SCL_2}^n$ for the collection of dyadic squares of side length $2^{-n}$ that intersect $\overline{K_{\SCL_1,\SCL_2} \setminus \bigcup_{\SCL \in \Gamma_{\SCL_1,\SCL_2}} \overline{\mathop{\mathrm{int}}(\SCL)}}$ (the full $K_{\SCL_1,\SCL_2}$ is used here, which only enlarges $\CS_{\SCL_1,\SCL_2}^n$). All of the definitions and results of this section are also made and proved with one or both of $\SCL_1, \SCL_2$ replaced by a deterministic connected arc of $\partial\BD$ (two disjoint arcs in the latter case), using the boundary-arc clauses of the axioms of \Cref{def:weak_axioms} in place of the loop clauses; we will not repeat the arguments.

We first record some basic facts regarding these constructions.

\begin{lemma}\label{lem:building-geodesic}
\begin{enumerate}
	\item\label{it:building-geodesic-5} Almost surely, for each $\SCL_1, \SCL_2 \in \Gamma_{\BD}$, we have $D_{\Gamma_{\BD}}^{\BD}(\SCL_1, x) + D_{\Gamma_{\BD}}^{\BD}(x, \SCL_2) = D_{\Gamma_{\BD}}^{\BD}(\SCL_1, \SCL_2)$ for all $x \in K_{\SCL_1,\SCL_2}$.
\item\label{it:building-geodesic-0} Almost surely, $\Gamma_{\SCL_1,\SCL} \subset \Gamma_{\SCL_1,\SCL_2}$, $\widetilde\Gamma_{\SCL_1,\SCL} \subset \widetilde\Gamma_{\SCL_1,\SCL_2}$, $K_{\SCL_1,\SCL} \subset K_{\SCL_1,\SCL_2}$, and $\widetilde K_{\SCL_1,\SCL} \subset \widetilde K_{\SCL_1,\SCL_2}$ for all $\SCL_1, \SCL_2 \in \Gamma_{\BD}$ and $\SCL \in \Gamma_{\SCL_1,\SCL_2}$. 
\item\label{it:building-geodesic-2} Almost surely, $\#\CS_{\SCL_1,\SCL_2}^n \le 2^{n \cdot o(1)}$ as $n \to \infty$ for all $\SCL_1, \SCL_2 \in \Gamma_{\BD}$. 
\item\label{it:building-geodesic-6} Almost surely, for each $\SCL_1, \SCL_2 \in \Gamma_{\BD}$,
\begin{equation*}
	\sup\nolimits_{S \in \CS_{\SCL_1,\SCL_2}^n} D_{\Gamma_{\BD}}^{\BD}\left(\SCB^S(\SCL_1; D_{\Gamma_{\BD}}^{\BD}), \SCB^S(\SCL_2; D_{\Gamma_{\BD}}^{\BD})\right) \to 0 \quad \text{as } n \to \infty,
\end{equation*}
where, as in the statement of \Cref{prop:hitting-two-metric-balls}, we set $\SCB^{S}(\SCL; D_{\Gamma_{\BD}}^{\BD}) \defeq \SCB_{D_{\Gamma_{\BD}}^{\BD}(\SCL, S)}(\SCL; D_{\Gamma_{\BD}}^{\BD})$.

\end{enumerate}
\end{lemma}

\begin{proof}
\stepn{step:bg-1}{Proof of assertion~\eqref{it:building-geodesic-5}} Let $x \in K_{\SCL_1,\SCL_2}$. By definition, there exists a sequence $\{x_n\}_{n \in \BN} \subset \overline\BD$ such that $\lvert x_n - x\rvert \to 0$ as $n \to \infty$ and $D_{\Gamma_{\BD}}^{\BD}(\SCL_1, x_n) + D_{\Gamma_{\BD}}^{\BD}(x_n, \SCL_2) = D_{\Gamma_{\BD}}^{\BD}(\SCL_1, \SCL_2)$ for all $n \in \BN$. By lower semi-continuity (cf.~\Cref{lem:lower-semi-continuity}), 
\begin{equation*}
	D_{\Gamma_{\BD}}^{\BD}(\SCL_1, x) \le \liminf_{n \to \infty} D_{\Gamma_{\BD}}^{\BD}(\SCL_1, x_n) \quad \text{and} \quad D_{\Gamma_{\BD}}^{\BD}(x, \SCL_2) \le \liminf_{n \to \infty} D_{\Gamma_{\BD}}^{\BD}(x_n, \SCL_2). 
\end{equation*}
Thus, 
\begin{equation*}
	D_{\Gamma_{\BD}}^{\BD}(\SCL_1, x) + D_{\Gamma_{\BD}}^{\BD}(x, \SCL_2) \le \liminf_{n \to \infty} \left( D_{\Gamma_{\BD}}^{\BD}(\SCL_1, x_n) + D_{\Gamma_{\BD}}^{\BD}(x_n, \SCL_2) \right) = D_{\Gamma_{\BD}}^{\BD}(\SCL_1, \SCL_2). 
\end{equation*}
Therefore, by the triangle inequality, we must have $D_{\Gamma_{\BD}}^{\BD}(\SCL_1, x) + D_{\Gamma_{\BD}}^{\BD}(x, \SCL_2) = D_{\Gamma_{\BD}}^{\BD}(\SCL_1, \SCL_2)$. This completes the proof of assertion~\eqref{it:building-geodesic-5}.

\stepn{step:bg-2}{Proof of assertion~\eqref{it:building-geodesic-0}} It suffices to show that $\Gamma_{\SCL_1,\SCL} \subset \Gamma_{\SCL_1,\SCL_2}$ and $K_{\SCL_1,\SCL} \subset K_{\SCL_1,\SCL_2}$. (Indeed, $\widetilde K_{\SCL_1,\SCL}$ is then a connected subset of $K_{\SCL_1,\SCL_2}$ containing $\SCL_1$, so it lies in the connected component $\widetilde K_{\SCL_1,\SCL_2}$ of $K_{\SCL_1,\SCL_2}$ containing $\SCL_1$.) For simplicity, we only consider the former; the proof of the latter is entirely similar. Let $\SCL^\prime \in \Gamma_{\SCL_1,\SCL}$. Then
\begin{align*}
D_{\Gamma_{\BD}}^{\BD}(\SCL_1, \SCL_2) &= D_{\Gamma_{\BD}}^{\BD}(\SCL_1, \SCL) + D_{\Gamma_{\BD}}^{\BD}(\SCL, \SCL_2) \\
&= D_{\Gamma_{\BD}}^{\BD}(\SCL_1, \SCL^\prime) + D_{\Gamma_{\BD}}^{\BD}(\SCL^\prime, \SCL) + D_{\Gamma_{\BD}}^{\BD}(\SCL, \SCL_2) \\
&\ge D_{\Gamma_{\BD}}^{\BD}(\SCL_1, \SCL^\prime) + D_{\Gamma_{\BD}}^{\BD}(\SCL^\prime, \SCL_2) \\
&\ge D_{\Gamma_{\BD}}^{\BD}(\SCL_1, \SCL_2). 
\end{align*}
Thus, we conclude that $D_{\Gamma_{\BD}}^{\BD}(\SCL_1, \SCL^\prime) + D_{\Gamma_{\BD}}^{\BD}(\SCL^\prime, \SCL_2) = D_{\Gamma_{\BD}}^{\BD}(\SCL_1, \SCL_2)$, i.e., $\SCL^\prime \in \Gamma_{\SCL_1,\SCL_2}$. This completes the proof of assertion~\eqref{it:building-geodesic-0}.

\stepn{step:bg-3}{Proof of assertion~\eqref{it:building-geodesic-2}} Fix deterministic $x, y \in \BD$. Fix $\delta > 0$. Write $E_\delta$ for the event that $B_{2\delta}(x) \subset \mathop{\mathrm{int}}(\SCL(x))$ and $B_{2\delta}(y) \subset \mathop{\mathrm{int}}(\SCL(y))$. Since $\BP[E_{\delta}] \to 1$ as $\delta \to 0$ (a deterministic point is a.s.\ surrounded by a loop of $\Gamma_\BD$, cf.\ \Cref{subsec:setup}), it suffices to show that, a.s.\ on the event $E_\delta$, $\#\CS_{\SCL(x),\SCL(y)}^n \le 2^{n \cdot o(1)}$ as $n \to \infty$. (This gives assertions~\eqref{it:building-geodesic-2} and~\eqref{it:building-geodesic-6} for all pairs of loops: intersect the almost sure events over $x, y \in \BD \cap \BQ^2$ and $\delta \in \BQ_{>0}$; every $\SCL_1, \SCL_2 \in \Gamma_\BD$ equal $\SCL(x), \SCL(y)$ for such $x, y$, with $E_\delta$ occurring for some such $\delta$, as the loops have nonempty interiors.) On $E_\delta$, every $S \in \CS_{\SCL(x),\SCL(y)}^n$ with $2^{-n}\sqrt{2} < \delta$ satisfies $\dist(S,x) \wedge \dist(S,y) \ge \delta$: otherwise $S \subset B_{2\delta}(x) \subset \mathop{\mathrm{int}}(\SCL(x))$ (or the same with $y$ in place of $x$), an open set which is disjoint from $\overline{K_{\SCL(x),\SCL(y)} \setminus \bigcup_{\SCL \in \Gamma_{\SCL(x),\SCL(y)}} \overline{\mathop{\mathrm{int}}(\SCL)}}$. We conclude from \Cref{prop:hitting-two-metric-balls} that
\begin{multline*}
	\BE\!\left[ \#\CS_{\SCL(x),\SCL(y)}^n \one_{E_\delta}\right] \\
	\le \BE\!\left[ \#\!\left\{S : S \cap \BD \neq \emptyset, \ \dist(S, x) \wedge \dist(S, y) \ge \delta, \ \SCB^S(\SCL(x); D_{\Gamma_{\BD}}^{\BD}) \cap \SCB^S(\SCL(y); D_{\Gamma_{\BD}}^{\BD}) = \emptyset\right\}\right] \\
	\le 2^{n \cdot o(1)} \quad \text{as } n \to \infty.
\end{multline*}
(For the last bound we used that $S \subset B_{2^{-n}}(z_S)$ for $z_S$ the center of $S$, hence $\SCB^{B_{2^{-n}}(z_S)}(\SCL; D_{\Gamma_\BD}^\BD) \subset \SCB^S(\SCL; D_{\Gamma_\BD}^\BD)$, and applied \Cref{prop:hitting-two-metric-balls} with $\varepsilon = 2^{-n}$ to each of the $O(2^{2n})$ squares in question.) Now, one deduces assertion~\eqref{it:building-geodesic-2} using Markov's inequality and the Borel--Cantelli lemma.

\stepn{step:bg-4}{Proof of assertion~\eqref{it:building-geodesic-6}}

\substepn{step:bg4-setup}{Setup} Fix deterministic $x, y \in \BD$. Fix $\delta > 0$. Let the event $E_\delta$ be as above. Fix $\varepsilon > 0$. It suffices to show that, a.s.\ on the event $E_\delta$,
\begin{equation*}
	\sup\nolimits_{S \in \CS_{\SCL(x),\SCL(y)}^n} D_{\Gamma_{\BD}}^{\BD}\left(\SCB^S(\SCL(x); D_{\Gamma_{\BD}}^{\BD}), \SCB^S(\SCL(y); D_{\Gamma_{\BD}}^{\BD})\right) < \varepsilon
\end{equation*}
for all sufficiently large $n \in \BN$. Write $F(S)$ for the event that $\SCB^S(\SCL(x); D_{\Gamma_{\BD}}^{\BD}) \cap \SCB^S(\SCL(y); D_{\Gamma_{\BD}}^{\BD}) = \emptyset$. 

\substepn{step:bg4-bothbdry}{The case where both balls intersect $\partial\BD$} We first consider the case where both $\SCB^S(\SCL(x); D_{\Gamma_{\BD}}^{\BD})$ and $\SCB^S(\SCL(y); D_{\Gamma_{\BD}}^{\BD})$ intersect $\partial \BD$ (i.e., the case where each connected component of $\BD \setminus (\SCB^S(\SCL(x); D_{\Gamma_{\BD}}^{\BD}) \cup \SCB^S(\SCL(y); D_{\Gamma_{\BD}}^{\BD}))$ is simply connected). Note that there is a deterministic constant $C = C(\delta) > 0$ such that, on the event $E_\delta \cap F(S)$, the extremal distance between $\SCB^S(\SCL(x); D_{\Gamma_{\BD}}^{\BD})$ and $\SCB^S(\SCL(y); D_{\Gamma_{\BD}}^{\BD})$ is at most $C/n$ (cf.~\Cref{lem:Beurling}). Combining this with \Cref{cor:moment_across_rectangle} (transferred by uniformizing the complementary component, with the two balls as top and bottom sides and the two arcs of $\partial\BD$ as free sides) and with Axiom~\eqref{it:weak_axiom_locality}, which gives the conditional law of the loops and the metric there, it follows that, given $\SCB^S(\SCL(x); D_{\Gamma_{\BD}}^{\BD})$, $\SCB^S(\SCL(y); D_{\Gamma_{\BD}}^{\BD})$, the event $E_\delta \cap F(S)$, and the event that both $\SCB^S(\SCL(x); D_{\Gamma_{\BD}}^{\BD})$ and $\SCB^S(\SCL(y); D_{\Gamma_{\BD}}^{\BD})$ intersect $\partial \BD$, the conditional law of $D_{\Gamma_{\BD}}^{\BD}\left(\SCB^S(\SCL(x); D_{\Gamma_{\BD}}^{\BD}), \SCB^S(\SCL(y); D_{\Gamma_{\BD}}^{\BD})\right)$ is a.s.\ stochastically dominated by $C^\prime/n$ times a geometric random variable of success probability $\ge 1/C^\prime$, where $C^\prime = C^\prime(\delta) \ge 1$ is a deterministic constant. (Recall from the proof of \Cref{cor:moment_across_rectangle} that the distance between the top and bottom sides of an $r \times 1$ rectangle is stochastically dominated by $1/r$ times a geometric random variable with universal success probability.)
Thus, conditional on $\SCB^S(\SCL(x); D_{\Gamma_{\BD}}^{\BD})$, $\SCB^S(\SCL(y); D_{\Gamma_{\BD}}^{\BD})$, the event $E_\delta \cap F(S)$, and the event that both $\SCB^S(\SCL(x); D_{\Gamma_{\BD}}^{\BD})$ and $\SCB^S(\SCL(y); D_{\Gamma_{\BD}}^{\BD})$ intersect $\partial \BD$, the probability that $D_{\Gamma_{\BD}}^{\BD}\left(\SCB^S(\SCL(x); D_{\Gamma_{\BD}}^{\BD}), \SCB^S(\SCL(y); D_{\Gamma_{\BD}}^{\BD})\right) \ge \varepsilon$ is at most $O(\re^{-\alpha n})$ for some deterministic constant $\alpha = \alpha(\delta, \varepsilon) > 0$.

\substepn{step:bg4-general}{The general case} In general,  we fix deterministic line segments $I$ and $J$ from $x$ and $y$,  respectively,  to $\partial \BD$,  in such a way that $\dist(I,J) \ge \delta^\prime$ and $\dist(I \cup J, S) \ge \delta^\prime$, where $\delta^\prime \defeq \delta/4$ (such a pair exists for every $S$ as above once $2^{-n}\sqrt{2} < \delta^\prime$, and we fix one for each $S$).  Let $I^{\star}$ (resp.\@ $J^{\star}$) denote the closure of the union of the loops in $\Gamma_{\BD}$ that intersect $I$ (resp.\@ $J$),  each intersected with the closed Euclidean $\delta^\prime/2$-neighborhood of $I$ (resp.\ $J$).  Then,  arguing as in \Cref{step:avoid-wide} of the proof of \Cref{lem:avoiding_two_metric_balls} (where a level line of height $u \in (0,\pi/2)$ crossing a unit box separates, by \Cref{thm:level_line_interaction}, the loops meeting $I \cup J$ from the target region),  we obtain that there exist deterministic constants $C_1,\alpha_1>0$ depending only on $\delta$,  such that the following holds a.s.  Conditionally on $\SCB^S(\SCL(x) ;  D_{\Gamma_{\BD}}^{\BD}),  \SCB^S(\SCL(y) ; D_{\Gamma_{\BD}}^{\BD})$,  and on the event $E_{\delta} \cap F(S)$,  we have with conditional probability at least $1 - C_1 e^{-\alpha_1 n}$ that there are no loops in $\Gamma_{\BD}$ intersecting both $I \cup J$ and $B_{2^{-n/2}}(S)$.  Write $G(S)$ for the latter event. On $G(S)$ we have $(I^\star \cup J^\star) \cap B_{2^{-n/2}}(S) = \emptyset$, so that $S \setminus (\SCB^S(\SCL(x) ; D_{\Gamma_{\BD}}^{\BD}) \cup \SCB^S(\SCL(y) ; D_{\Gamma_{\BD}}^{\BD}))$ is contained in $\BD \setminus (\SCB^S(\SCL(x) ; D_{\Gamma_{\BD}}^{\BD}) \cup \SCB^S(\SCL(y) ; D_{\Gamma_{\BD}}^{\BD}) \cup I^{\star} \cup J^{\star})$; we let $W$ be a connected component of the latter set whose closure meets both of the stopped balls.  Furthermore,  arguing as in the case that both $\SCB^S(\SCL(x) ; D_{\Gamma_{\BD}}^{\BD})$ and $\SCB^S(\SCL(y) ; D_{\Gamma_{\BD}}^{\BD})$ intersect $\partial \BD$,  we obtain that there exist constants $C^{\prime},  \alpha^{\prime}>0$ depending only on $\delta$ and $\varepsilon$,  such that conditionally on $\SCB^S(\SCL(x) ; D_{\Gamma_{\BD}}^{\BD}),  \SCB^S(\SCL(y) ; D_{\Gamma_{\BD}}^{\BD})$,  and on the event $E_{\delta} \cap F(S) \cap G(S)$,  with conditional probability at most $C^{\prime} e^{-\alpha^{\prime} n}$,  we have that
\begin{align*}
D_{\Gamma_{\BD}}^{\BD}(\SCB^S(\SCL(x) ; D_{\Gamma_{\BD}}^{\BD}),  \SCB^S(\SCL(y) ; D_{\Gamma_{\BD}}^{\BD}); W) \geq \varepsilon.
\end{align*}
Since
\begin{align*}
D_{\Gamma_{\BD}}^{\BD}(\SCB^S(\SCL(x) ; D_{\Gamma_{\BD}}^{\BD}),\SCB^S(\SCL(y) ; D_{\Gamma_{\BD}}^{\BD})) \leq D_{\Gamma_{\BD}}^{\BD}(\SCB^S(\SCL(x) ; D_{\Gamma_{\BD}}^{\BD}) ,  \SCB^S(\SCL(y) ; D_{\Gamma_{\BD}}^{\BD}) ; W),
\end{align*}
it follows that,  a.s.,  
\begin{multline*}
	\BP\!\left\lbrack D_{\Gamma_{\BD}}^{\BD}\left(\SCB^S(\SCL(x); D_{\Gamma_{\BD}}^{\BD}), \SCB^S(\SCL(y); D_{\Gamma_{\BD}}^{\BD})\right) \ge \varepsilon \, \middle\vert \, \SCB^S(\SCL(x); D_{\Gamma_{\BD}}^{\BD}),  \SCB^S(\SCL(y); D_{\Gamma_{\BD}}^{\BD}), E_\delta \cap F(S) \cap G(S)\right\rbrack \\ \le C^\pprime \re^{-\alpha^\pprime n},
\end{multline*}
where $C^\pprime = C^\pprime(\delta, \varepsilon) > 0$ and $\alpha^\pprime = \alpha^\pprime(\delta, \varepsilon) > 0$ are deterministic constants. On the other hand, it follows from \Cref{prop:hitting-two-metric-balls} that there is a deterministic constant $C^{\pprime\prime} = C^{\pprime\prime}(\delta) > 0$ such that $\BP\lbrack E_\delta \cap F(S)\rbrack \le C^{\pprime\prime} 2^{-n(2 + o(1))}$. Thus, we conclude that there exists a constant $\alpha = \alpha(\delta,\varepsilon)>0$ such that
\begin{align*}
\BP\!\left[\left\{D_{\Gamma_{\BD}}^{\BD}(\SCB^S(\SCL(x) ; D_{\Gamma_{\BD}}^{\BD}) ,  \SCB^S(\SCL(y) ; D_{\Gamma_{\BD}}^{\BD})) \geq \varepsilon \right \} \cap E_{\delta} \right] \leq 2^{-n(2+\alpha + o(1))}
\end{align*}
as $n \to \infty$, for every dyadic square $S$ with side length $2^{-n}$ such that $\dist(S,x) \wedge \dist(S,y) \geq \delta$,  where the rate of convergence depends only on $\varepsilon$ and $\delta$.

\substepn{step:bg4-concl}{Conclusion via the Borel--Cantelli lemma} Now, one deduces assertion~\eqref{it:building-geodesic-6} using the Borel--Cantelli lemma, after summing over the $O(2^{2n})$ dyadic squares $S$ of side length $2^{-n}$ with $\dist(S,x) \wedge \dist(S,y) \ge \delta$. The remaining squares do not contribute: on $E_\delta$ they lie in $\mathop{\mathrm{int}}(\SCL(x)) \cup \mathop{\mathrm{int}}(\SCL(y))$, where the two stopped balls intersect, so the distance vanishes; this is also why only the squares of $F(S)$ matter above.
\end{proof}

\begin{proof}[Proof of \Cref{prop:existence_geodesics}]
\stepn{step:eg-order}{The partial order on the loops of a geodesic} Fix $\SCL_1, \SCL_2 \in \Gamma_{\BD}$. For $\SCL, \SCL^\prime \in \widetilde\Gamma_{\SCL_1,\SCL_2}$, we shall write $\SCL \preceq \SCL^\prime$ if the following equivalent conditions are satisfied:
\begin{itemize}
\item $\SCL \in \Gamma_{\SCL_1,\SCL^\prime}$, i.e., $D_{\Gamma_{\BD}}^{\BD}(\SCL_1, \SCL) + D_{\Gamma_{\BD}}^{\BD}(\SCL, \SCL^\prime) = D_{\Gamma_{\BD}}^{\BD}(\SCL_1, \SCL^\prime)$.
\item $\SCL^\prime \in \Gamma_{\SCL,\SCL_2}$, i.e., $D_{\Gamma_{\BD}}^{\BD}(\SCL, \SCL^\prime) + D_{\Gamma_{\BD}}^{\BD}(\SCL^\prime, \SCL_2) = D_{\Gamma_{\BD}}^{\BD}(\SCL, \SCL_2)$.
\item $D_{\Gamma_{\BD}}^{\BD}(\SCL_1, \SCL) + D_{\Gamma_{\BD}}^{\BD}(\SCL, \SCL^\prime) + D_{\Gamma_{\BD}}^{\BD}(\SCL^\prime, \SCL_2) = D_{\Gamma_{\BD}}^{\BD}(\SCL_1, \SCL_2)$.
\end{itemize}
(The three identities are equivalent for $\SCL, \SCL^\prime \in \widetilde\Gamma_{\SCL_1,\SCL_2}$: adding $D(\SCL^\prime, \SCL_2)$ to the first and using that $\SCL^\prime$ is intermediate gives the third, which with the triangle inequality forces equality in the first two.) Note that $(\widetilde\Gamma_{\SCL_1,\SCL_2}, \preceq)$ is a partially ordered set (antisymmetry holds since distinct loops are at positive $D_{\Gamma_\BD}^\BD$-distance, $D_{\Gamma_\BD}^\BD$ being a.s.\ a metric on $\Gamma_\BD$ by Axiom~\eqref{it:weak_axiom_geodesic}) with least element $\SCL_1$ and greatest element $\SCL_2$. We shall write $\SCL \xleftrightarrow n \SCL^\prime$ if there exists a sequence $S_0, S_1, \ldots, S_m \in \CS_{\SCL_1,\SCL_2}^n$ such that $S_0 \cap \SCL \neq \emptyset$, $S_m \cap \SCL^\prime \neq \emptyset$, and $\partial S_{j - 1} \cap \partial S_j \neq \emptyset$ for all $j \in [1, m]_{\BZ}$. 

\stepn{step:eg-families}{Construction of the finite families}

\substepn{step:egf-claim}{The claim} Next, we \emph{claim} that there exists a sequence of finite collections 
\begin{equation*}
\{\SCL_1, \SCL_2\} \subset \Gamma_{\SCL_1,\SCL_2}^1 \subset \Gamma_{\SCL_1,\SCL_2}^2 \subset \cdots \subset \widetilde\Gamma_{\SCL_1,\SCL_2}
\end{equation*}
satisfying the following conditions:
\begin{enumerate}
\item For each $n \in \BN$, the set $(\Gamma_{\SCL_1,\SCL_2}^n, \preceq)$ is totally ordered. 
\item For each $n \in \BN$ and $\SCL \in \Gamma_{\SCL_1,\SCL_2}^n \setminus \{\SCL_1, \SCL_2\}$, we have $\diam(\SCL) \ge 2^{-n}$. 
\item For each $n \in \BN$, if we write $\Gamma_{\SCL_1,\SCL_2}^n = \{\SCL_1 = \SCL_0^\prime \preceq \SCL_1^\prime \preceq \cdots \preceq \SCL_N^\prime = \SCL_2\}$, then $\SCL_{j - 1}^\prime \xleftrightarrow n \SCL_j^\prime$ for all $j \in [1, N]_{\BZ}$. 
\end{enumerate}

\substepn{step:egf-build}{Construction of $\Gamma_{\SCL_1,\SCL_2}^1$} Let us construct $\Gamma_{\SCL_1,\SCL_2}^1$. Since $\widetilde K_{\SCL_1,\SCL_2}$ is connected, 
the graph $(\widetilde\Gamma_{\SCL_1,\SCL_2}, \bullet \xleftrightarrow {1} \bullet)$ is connected, and hence so is the graph
\begin{equation*}
\left(\left\{\SCL \in \widetilde\Gamma_{\SCL_1,\SCL_2} : \diam(\SCL) \ge 2^{-1}\right\} \cup \{\SCL_1, \SCL_2\}, \bullet \xleftrightarrow {1} \bullet\right), 
\end{equation*}
(since loops of diameter smaller than $2^{-1}$ cannot intersect two dyadic squares of side length $2^{-1}$ that are not adjacent). If $\SCL_1 \xleftrightarrow 1 \SCL_2$, then we are done; otherwise, we may choose $\SCA_0 \in \widetilde\Gamma_{\SCL_1,\SCL_2} \setminus \{\SCL_1,\SCL_2\}$ with $\diam(\SCA_0) \ge 2^{-1}$ (such a loop exists since $\SCL_1$ and $\SCL_2$ are non-adjacent vertices of the above connected graph, so any path joining them in the graph passes through a third vertex). Now, in a similar vein, the graphs
\begin{multline*}
\left(\left\{\SCL \in \widetilde\Gamma_{\SCL_1,\SCA_0} : \diam(\SCL) \ge 2^{-1}\right\} \cup \{\SCL_1\}, \bullet \xleftrightarrow {1} \bullet\right) \\
\text{and} \quad \left(\left\{\SCL \in \widetilde\Gamma_{\SCA_0,\SCL_2} : \diam(\SCL) \ge 2^{-1}\right\} \cup \{\SCL_2\}, \bullet \xleftrightarrow {1} \bullet\right)
\end{multline*}
are connected. If $\SCL_1 \xleftrightarrow 1 \SCA_0$ and $\SCA_0 \xleftrightarrow 1 \SCL_2$, then we are done; otherwise, we may choose $\SCA_1 \in \widetilde\Gamma_{\SCL_1,\SCA_0} \setminus \{\SCL_1,\SCA_0\}$ (resp.\ $\SCA_2 \in \widetilde\Gamma_{\SCA_0,\SCL_2} \setminus \{\SCA_0,\SCL_2\}$) with $\diam(\SCA_1) \ge 2^{-1}$ (resp.\ $\diam(\SCA_2) \ge 2^{-1}$). Since the number of loops with diameter $\ge 2^{-1}$ is finite, this procedure must terminate after finitely many steps. This completes the construction of $\Gamma_{\SCL_1,\SCL_2}^1$.

\substepn{step:egf-induct}{The inductive step} The construction of $\Gamma_{\SCL_1,\SCL_2}^{n + 1}$ given $\Gamma_{\SCL_1,\SCL_2}^n$ follows from a similar argument, applied between consecutive loops of $\Gamma_{\SCL_1,\SCL_2}^n$: by \Cref{lem:building-geodesic}\eqref{it:building-geodesic-0}, $\widetilde K_{\SCL_{j-1}^\prime,\SCL_j^\prime} \subset \widetilde K_{\SCL_1,\SCL_2}$, so the same argument runs there.

\stepn{step:eg-paths}{Construction of the approximating paths} Next, for each $n \in \BN$, choose a continuous path 
\begin{equation*}
P^n \colon [0, 1] \to \bigcup_{\SCL \in \Gamma_{\SCL_1,\SCL_2}^n} \SCL \cup \bigcup_{S \in \CS_{\SCL_1,\SCL_2}^n} S
\end{equation*}
from $\SCL_1$ to $\SCL_2$ satisfying the following conditions:
\begin{enumerate}[label=(\Alph*)]
\item\label{it:condition A path P n} If we write $\Gamma_{\SCL_1,\SCL_2}^n = \{\SCL_1 = \SCL_0^\prime \preceq \SCL_1^\prime \preceq \cdots \preceq \SCL_N^\prime = \SCL_2\}$, then there exist times $0 = s_0 < t_0 < s_1 < t_1 < \cdots < s_N < t_N = 1$ such that $P^n([s_j, t_j]) \subset \SCL_j^\prime$ for all $j \in [0, N]_{\BZ}$, and for each $j \in [1, N]_{\BZ}$, there exists a sequence $S_0, S_1, \ldots, S_m \in \CS_{\SCL_1,\SCL_2}^n$ such that $P^n([t_{j - 1}, s_j]) \subset \bigcup_{k \in [0, m]_{\BZ}} S_k$, $S_0 \cap \SCL_{j - 1}^\prime \neq \emptyset$, $S_m \cap \SCL_j^\prime \neq \emptyset$, and $\partial S_{k - 1} \cap \partial S_k \neq \emptyset$ for all $k \in [1, m]_{\BZ}$. 
\item\label{it:condition B path P n} Let $0 = s_0 < t_0 < s_1 < t_1 < \cdots < s_N < t_N = 1$ be as above. Then (by abuse of notation) $P^{n^\prime}|_{[s_j, t_j]} = P^n|_{[s_j, t_j]}$ for all $j \in [0, N]_{\BZ}$ and $n^\prime \ge n$. (More precisely, if $0 = s_0^n < t^n_0 < s^n_1 < \cdots$ denotes the times corresponding to $P^n$ and if $P^n([s^n_j, t^n_j]) \subset \SCL$ and $P^{n^\prime}([s^{n^\prime}_{j^\prime}, t^{n^\prime}_{j^\prime}])\subset \SCL$ with $\SCL \in \Gamma_{\SCL_1,\SCL_2}^n$, then $s^n_j= s^{n^\prime}_{j^\prime}$, $t^n_j= t^{n^\prime}_{j^\prime}$ and $P^{n^\prime}\vert_{[s^{n^\prime}_{j^\prime}, t^{n^\prime}_{j^\prime}]}= P^n\vert_{[s^n_j, t^n_j ]}$.)
\end{enumerate}
Such a path exists: each loop is a continuous curve, hence path-connected, and consecutive squares of a chain share a boundary point, through which the path can pass.

\stepn{step:eg-concl}{Conclusion of the proof} Next, let us show the uniform convergence of $P^n$. The maximum diameter of a chain of squares in $\CS^n_{\SCL_1, \SCL_2}$ in Condition \ref{it:condition A path P n} that $P^n$ satisfies goes to zero as $n\to \infty$. Indeed, by \Cref{lem:building-geodesic}\eqref{it:building-geodesic-2}, there are $2^{n\cdot o(1)}$ squares in $\CS^n_{\SCL_1, \SCL_2}$ and each of them has side-length $2^{-n}$. Moreover, Condition \ref{it:condition B path P n} shows that the path $P^{n+1}$ agrees with $P^n$ on the loops that are used by $P^n$. Therefore, the sequence $\{P^n\}_{n \in \BN}$ converges uniformly on $[0,1]$. Let $P$ be its limit. It remains to verify that $P$ is a $D_{\Gamma_{\BD}}^{\BD}$-geodesic from $\SCL_1$ to $\SCL_2$. For each $n \in \BN$, since $\Gamma_{\SCL_1,\SCL_2}^n = \{\SCL_1 = \SCL_0^\prime \preceq \SCL_1^\prime \preceq \cdots \preceq \SCL_N^\prime = \SCL_2\}$ is a totally ordered set, it follows that 
\begin{equation*}
	D_{\Gamma_{\BD}}^{\BD}(\SCL_1, \SCL_2) = \sum_{j = 1}^{N} D_{\Gamma_{\BD}}^{\BD}(\SCL_{j - 1}^\prime, \SCL_j^\prime). 
\end{equation*}
Moreover, the path $P$ is admissible. Assume by contradiction that there exists a time $t\in [0,1]$ such that $P(t)\not\in P([0,1]) \cap \overline{\bigcup_{\SCL\in \Gamma} \SCL}$. Note that necessarily $t \not\in\{0,1\}$. Then, there exists $\varepsilon>0$ such that $B_\varepsilon(P(t)) \cap \overline{\bigcup_{\SCL\in \Gamma} \SCL} = \emptyset$. This is impossible by the observation that we made at the beginning of \Cref{step:eg-concl}. Thus, $P$ is a $D_{\Gamma_{\BD}}^{\BD}$-geodesic from $\SCL_1$ to $\SCL_2$.

Finally, it remains to verify that $\bigcup_{n \in \BN} \{D_{\Gamma_{\BD}}^{\BD}(\SCL_1, \SCL) : \SCL \in \Gamma_{\SCL_1,\SCL_2}^n\}$ forms a dense subset of the interval $[0, D_{\Gamma_{\BD}}^{\BD}(\SCL_1, \SCL_2)]$. Suppose by way of contradiction that this is false. Then there exists $0 \le s < t \le D_{\Gamma_{\BD}}^{\BD}(\SCL_1, \SCL_2)$ such that for each $n \in \BN$, there exist $\SCL_1^n, \SCL_2^n \in \Gamma_{\SCL_1,\SCL_2}^n$ consecutive in $(\Gamma_{\SCL_1,\SCL_2}^n, \preceq)$ (which is why the chain below is provided by condition~(3)) and $S_0^n, S_1^n, \ldots, S_m^n \in \CS_{\SCL_1,\SCL_2}^n$ such that $D_{\Gamma_{\BD}}^{\BD}(\SCL_1, \SCL_1^n) \le s$, $D_{\Gamma_{\BD}}^{\BD}(\SCL_1, \SCL_2^n) \ge t$, $S_0^n \cap \SCL_1^n \neq \emptyset$, $S_m^n \cap \SCL_2^n \neq \emptyset$, and $\partial S_{j - 1}^n \cap \partial S_j^n \neq \emptyset$ for all $j \in [1, m]_{\BZ}$. Let $S^n$ be the smallest dyadic square that contains $S_0^n, S_1^n, \ldots, S_m^n$. By definition, the side length $\scl(S^n)$ of $S^n$ satisfies $S^n \in \CS_{\SCL_1,\SCL_2}^{\log_{2}(1/\scl(S^n))}$ and $\scl(S^n) \le 100 \cdot \#\CS_{\SCL_1,\SCL_2}^n \cdot 2^{-n} \to 0$ as $n \to \infty$ (cf.~\Cref{lem:building-geodesic}\eqref{it:building-geodesic-2}). Thus, it follows from \Cref{lem:building-geodesic}\eqref{it:building-geodesic-6} that 
\begin{equation*}
	D_{\Gamma_{\BD}}^{\BD}(\SCL_1^n, \SCL_2^n) \le D_{\Gamma_{\BD}}^{\BD}\left(\SCB^{S^n}(\SCL_1; D_{\Gamma_{\BD}}^{\BD}), \SCB^{S^n}(\SCL_2; D_{\Gamma_{\BD}}^{\BD})\right) \to 0 \quad \text{as } n \to \infty.
\end{equation*}
This contradicts the fact that $D_{\Gamma_{\BD}}^{\BD}(\SCL^n_1, \SCL^n_2) = D_{\Gamma_{\BD}}^{\BD}(\SCL_1, \SCL^n_2) - D_{\Gamma_{\BD}}^{\BD}(\SCL_1, \SCL^n_1) \ge t - s$. Thus, $\bigcup_{n \in \BN} \{D_{\Gamma_{\BD}}^{\BD}(\SCL_1, \SCL) : \SCL \in \Gamma_{\SCL_1,\SCL_2}^n\}$ is dense in $[0, D_{\Gamma_{\BD}}^{\BD}(\SCL_1, \SCL_2)]$.

This completes the proof of \Cref{prop:existence_geodesics}. 
\end{proof}
\begin{remark}
	Note that the above proof produces, for any $\SCL_1, \SCL_2 \in \Gamma_U$, a $D^U_{\Gamma_U}$-geodesic $P$ from $\SCL_1$ to $\SCL_2$ along which $t\mapsto D^U_{\Gamma_U}(\SCL_1, P(t))$ is continuous, and the same holds if we replace one of the two loops or both of them by a segment of $\partial U$. (By conformal invariance it suffices to treat $U = \BD$, and the boundary-arc versions require the corresponding versions of \Cref{lem:building-geodesic}.)
\end{remark}

The following are two immediate corollaries regarding the behavior of these geodesics.

\begin{corollary}\label{lem:Hausdorff_dim_geodesic}
Almost surely, for each $\SCL_1, \SCL_2 \in \Gamma_{\BD}$, any $D_{\Gamma_{\BD}}^{\BD}$-geodesic $P$ connecting $\SCL_1$ and $\SCL_2$ has the property that the subset $\{P(t) : P(t) \notin \SCL \text{ for any } \SCL \in \Gamma_{\BD}\}$ is of Hausdorff dimension zero. A similar statement holds with two deterministic and disjoint connected arcs of $\partial\BD$ in place of $\SCL_1$ and $\SCL_2$.
\end{corollary}

\begin{proof}
Every point of a geodesic $P$ from $\SCL_1$ to $\SCL_2$ lies in $K_{\SCL_1,\SCL_2}$ by the additivity of the distances along $P$, and $K_{\SCL_1,\SCL_2}$ misses the interior of every loop of $\Gamma_{\SCL_1,\SCL_2}$, as shown in the proof of \Cref{prop:existence_geodesics}. Hence the subset in question lies in $\overline{K_{\SCL_1,\SCL_2} \setminus \bigcup_{\SCL \in \Gamma_{\SCL_1,\SCL_2}} \overline{\mathop{\mathrm{int}}(\SCL)}}$, and the claim follows from \Cref{lem:building-geodesic}\eqref{it:building-geodesic-2}.
\end{proof}

\begin{corollary}\label{lem:geodesic_avoid_boundary}
Almost surely, for each $\SCL_1, \SCL_2 \in \Gamma_{\BD}$, any $D_{\Gamma_{\BD}}^{\BD}$-geodesic connecting $\SCL_1$ and $\SCL_2$ does not touch $\partial\BD$. 
A similar statement holds with two deterministic and disjoint connected arcs of $\partial\BD$ in place of $\SCL_1$ and $\SCL_2$ (in which case the $D_{\Gamma_{\BD}}^{\BD}$-geodesic touches $\partial\BD$ only at its endpoints).
\end{corollary}

\begin{proof}
For simplicity, we only consider the case of loops of $\Gamma_{\BD}$; the case of two deterministic and disjoint connected arcs of $\partial\BD$ is entirely similar.
By the conformal invariance property, it suffices to consider the case of $\BH$, provided we apply the argument below for two conformal maps from $\BD$ onto $\BH$ whose inverses send $\infty$ to two distinct points of $\partial\BD$, so that every boundary point is covered. Fix $x, y \in \BH$ and $R > 0$; as in the proof of \Cref{lem:building-geodesic}, it suffices to treat $\SCL_1 = \SCL(x)$ and $\SCL_2 = \SCL(y)$ for $x, y$ in a countable dense set. It suffices to show that, a.s., no $D_{\Gamma_{\BH}}^{\BH}$-geodesic $P$ connecting $\SCL(x)$ and $\SCL(y)$ touches $B_R(0) \cap \partial\BH$. Moreover, if one touches it at a point $w$, then by the additivity of the distances along $P$ the balls $\SCB^{B_\varepsilon(z)}$ below are disjoint for the $z \in (\varepsilon/2)\BZ$ nearest to $w$, so that, by a union bound,
\begin{align*}
	&\BP\!\left[ \text{some geodesic } P \text{ touches } B_R(0) \cap \partial\BH\right] \\
	&\qquad \le \sum_{z \in B_R(0) \cap (\varepsilon/2)\BZ}
	\BP\!\left[\SCB^{B_\varepsilon(z)}(\SCL(x); D_{\Gamma_{\BH}}^{\BH}) \cap \SCB^{B_\varepsilon(z)}(\SCL(y); D_{\Gamma_{\BH}}^{\BH}) = \emptyset\right], 
\end{align*}
By \Cref{prop:boundary-hitting-two-metric-balls}, each of the $O(R/\varepsilon)$ terms of this sum is $\varepsilon^{4 + o(1)}$, so the sum is $O(\varepsilon^{3 + o(1)})$ as $\varepsilon \to 0$. Since the left-hand side does not depend on $\varepsilon$, it vanishes.
\end{proof}

\begin{corollary}\label{lem:strong-locality}
Let $D$ be a weak geodesic CLE$_4$ metric coupling, and let $U \subsetneq \BC$ be a deterministic simply connected domain. Let $V \subset U$ be a deterministic simply connected subdomain and let $\{V_j\}_j$ denote the connected components of $V^\star$. Suppose that the random variables $(\Gamma_U, D_{\Gamma_U}^U, \{D_{\Gamma_U |_{V_j}}^{V_j}\}_j)$ are coupled as in Axiom~\eqref{it:weak_axiom_locality} in \Cref{def:weak_axioms}. Then, we have that $D_{\Gamma_U |_{V_j}}^{V_j} = D_{\Gamma_U}^U(\bullet, \bullet; V_j)$ for all $j$ a.s. In particular, for all $j$, $(\Gamma_U|_{V_j}, D_{\Gamma_U}^U(\bullet, \bullet; V_j))$ and $(\Gamma_{V_j}, D_{\Gamma_{V_j}}^{V_j})$ have the same law. In other words, $D$ satisfies Axiom~\eqref{it:axiom_locality} of \Cref{def:axioms} (the conditional independence over $j$ required comes from Axiom~\eqref{it:weak_axiom_locality} of \Cref{def:weak_axioms}).
\end{corollary}

\begin{proof}

Recall that by Axiom~\eqref{it:weak_axiom_locality} in \Cref{def:weak_axioms}, we have that conditionally on the $\sigma$-algebra given in~\eqref{eq:weak_axiom_locality}, the pairs $(\Gamma_U |_{V_j}, D_{\Gamma_U |_{V_j}}^{V_j})$ are independent and their conditional laws are that of $(\Gamma_{V_j}, D_{\Gamma_{V_j}}^{V_j})$ respectively. Therefore, combining with \Cref{prop:existence_geodesics,lem:geodesic_avoid_boundary}, we obtain that it is a.s.\ the case that the metrics $D_{\Gamma_U |_{V_j}}^{V_j}$ are geodesic metrics and that any geodesic between two loops in $\Gamma_U |_{V_j}$ does not intersect $\partial V_j$. It follows that for all $\SCL_1,\SCL_2 \in \Gamma_U |_{V_j}$ and all $j$, we have that
\begin{equation}\label{eq:strong-locality-proof}
	D_{\Gamma_U|_{V_j}}^{V_j}(\SCL_1, \SCL_2) = \inf_{\SCL_0^\prime, \SCL_1^\prime, \ldots, \SCL_n^\prime} \sum_{k = 1}^n D_{\Gamma_U|_{V_j}}^{V_j}(\SCL_{k - 1}^\prime, \SCL_k^\prime), 
\end{equation}
where $\SCL_0^\prime, \SCL_1^\prime, \ldots, \SCL_n^\prime$ range over all sequences of loops of $\Gamma_U|_{V_j}$ such that $\SCL_1 = \SCL_0^\prime$, $\SCL_2 = \SCL_n^\prime$, and $D_{\Gamma_U|_{V_j}}^{V_j}(\SCL_{k - 1}^\prime, \SCL_k^\prime) < D_{\Gamma_U|_{V_j}}^{V_j}(\SCL_k^\prime, \partial V_j)$ for all $k \in [1, n]_{\BZ}$. 

Moreover, by Axiom~\eqref{it:weak_axiom_locality} in \Cref{def:weak_axioms}, we have that $\SCB_t(\SCL; D_{\Gamma_U}^U) = \SCB_t(\SCL; D_{\Gamma_U|_{V_j}}^{V_j})$ for all $\SCL \in \Gamma_U|_{V_j}, \ t \in [0, D_{\Gamma_U}^U(\SCL, \partial V_j)]$. This implies that
\begin{align*}
D_{\Gamma_U|_{V_j}}^{V_j}(\SCL_{k - 1}^\prime, \SCL_k^\prime) = D_{\Gamma_U}^U(\SCL_{k - 1}^\prime, \SCL_k^\prime),
\end{align*}
where $\SCL_0^\prime, \SCL_1^\prime, \ldots, \SCL_n^\prime$ are as in~\eqref{eq:strong-locality-proof}. In particular, we have that 
\begin{align*}
D_{\Gamma_U|_{V_j}}^{V_j}(\SCL_1, \SCL_2) = \inf_{\SCL_0^\prime, \SCL_1^\prime, \ldots, \SCL_n^\prime} \sum_{k = 1}^n D_{\Gamma_U}^U(\SCL_{k - 1}^\prime, \SCL_k^\prime), 
\end{align*}
where $\SCL_0^\prime, \SCL_1^\prime, \ldots, \SCL_n^\prime$ range over all sequences of loops of $\Gamma_U|_{V_j}$ such that $\SCL_1 = \SCL_0^\prime$, $\SCL_2 = \SCL_n^\prime$, and $D_{\Gamma_U}^U(\SCL_{k - 1}^\prime, \SCL_k^\prime) < D_{\Gamma_U}^U(\SCL_k^\prime, \partial V_j)$ for all $k \in [1, n]_{\BZ}$. By the definition of the internal metric, the same is true with $D_{\Gamma_U}^U(\bullet, \bullet; V_j)$ in place of $D_{\Gamma_U |_{V_j}}^{V_j}$. This completes the proof.
\end{proof}

\begin{remark}\label{rem:almost_strong_metric}
Note that we have not yet proven that $D_{\Gamma_U}^U$ is a.s.\ determined by $\Gamma_U$. This will be achieved in the next paper. 
\end{remark}

\begin{proof}[Proof of \Cref{thm:existence_geodesics}.]
\stepn{step:teg-forward}{Every weak geodesic coupling is a geodesic coupling} By \Cref{prop:existence_geodesics} and \Cref{lem:strong-locality}, we know that every weak geodesic CLE$_4$ metric coupling satisfies Axioms~\eqref{it:axiom_geodesic} and~\eqref{it:axiom_locality} of \Cref{def:axioms}; since the remaining axioms of \Cref{def:axioms} (conformal invariance and uniform exploration) are common to the two definitions, it is a geodesic CLE$_4$ metric coupling. Let us show the converse, which is easier to prove.
	
\stepn{step:teg-easy}{The converse: the axioms that transfer directly} Let $D$ be a geodesic CLE$_4$ metric coupling. It is straightforward to see that Axiom~\eqref{it:axiom_locality} from \Cref{def:axioms} implies Axiom~\eqref{it:weak_axiom_locality} from \Cref{def:weak_axioms}. Similarly, Axioms~\eqref{it:weak_axiom_geodesic}\eqref{it:weak_axiom_geodesic_0} and~\eqref{it:weak_axiom_geodesic}\eqref{it:weak_axiom_geodesic_4} are a direct consequence of the existence of geodesics. Moreover, Axiom~\eqref{it:weak_axiom_geodesic}\eqref{it:weak_axiom_geodesic_1} is contained in Axiom~\eqref{it:axiom_geodesic} from \Cref{def:axioms}.
	 
\stepn{step:teg-hard}{The converse: Axiom~\eqref{it:weak_axiom_geodesic}\eqref{it:weak_axiom_geodesic_2}} Let $x, y \in U_\BQ$ and $t \in [0, D^U_{\Gamma_U}(\SCL(x), \SCL(y))]$. Let $P \colon [0,1] \to U$ be a $D^U_{\Gamma_U}$-geodesic from $\SCL(x)$ to $\SCL(y)$. Let $r>0$. Assume that 
	 \[
	 \SCB_t(\SCL(x); D^U_{\Gamma_U}) \cap \SCB_r(\SCL(y); D^U_{\Gamma_U}) = \emptyset.
	 \]
	 Then, by continuity of $P$ and since $ \SCB_t(\SCL(x); D^U_{\Gamma_U}) $ and $ \SCB_r(\SCL(y); D^U_{\Gamma_U})$ are closed subsets of $\overline{U}$, there is an interval $[s_1, s_2]$ with $s_1< s_2$ such that $P([s_1, s_2])$ positive diameter and does not intersect $ \SCB_t(\SCL(x); D^U_{\Gamma_U}) \cup \SCB_r(\SCL(y); D^U_{\Gamma_U}) $. Since $\bigcup_{\SCL \in \Gamma_U}\SCL \cap P$ is dense in $P$ (because $P$ is an admissible path), we deduce that there is a loop $\SCL$ of $\Gamma_U$ meeting $P([s_1,s_2])$; since a loop which intersects a metric ball is contained in it (the ball being the closure of the union of the domains surrounded by the loops it contains), we have that
	 \[
	\SCL \cap \left( \SCB_t(\SCL(x); D^U_{\Gamma_U}) \cup \SCB_r(\SCL(y); D^U_{\Gamma_U}) \right) = \emptyset .
	 \]
	 But then, this means that
	 \[
	 D^U_{\Gamma_U}(\SCL(x), \SCL(y))=D^U_{\Gamma_U}(\SCL(x), \SCL)+ D^U_{\Gamma_U}(\SCL, \SCL(y)) > t+r.
	 \]
	 Since this holds for every $r$ with $\SCB_t(\SCL(x)) \cap \SCB_r(\SCL(y)) = \emptyset$, and since the balls intersect once $r$ exceeds $D^U_{\Gamma_U}(\SCB_t(\SCL(x); D^U_{\Gamma_U}), \SCL(y))$, letting $r$ increase to the latter gives, a.s.,
	 \[
	 D^U_{\Gamma_U}(\SCL(x), \SCL(y)) \ge D^U_{\Gamma_U}(\SCB_t(\SCL(x); D^U_{\Gamma_U}), \SCL(y)) + t.\]
	 Let us show the converse inequality. Assume that 
	 \[
	 \SCB_t(\SCL(x); D^U_{\Gamma_U}) \cap \SCB_r(\SCL(y); D^U_{\Gamma_U}) \neq \emptyset.
	 \]
	 By Axiom~\eqref{it:axiom_geodesic}\eqref{it:axiom_geodesic_twoball}, we have $D^U_{\Gamma_U}(\SCL(x), \SCL(y)) \le t+r$. Letting $r$ decrease to $D^U_{\Gamma_U}(\SCB_t(\SCL(x); D^U_{\Gamma_U}), \SCL(y))=\inf\{r\ge 0: \SCB_t(\SCL(x); D^U_{\Gamma_U}) \cap \SCB_r(\SCL(y); D^U_{\Gamma_U})  \neq \emptyset\}$ gives a.s.
	 \[
	 D^U_{\Gamma_U}(\SCL(x), \SCL(y)) \le D^U_{\Gamma_U}(\SCB_t(\SCL(x); D^U_{\Gamma_U}), \SCL(y)) + t.
	 \]
	  This proves Axiom~\eqref{it:weak_axiom_geodesic}\eqref{it:weak_axiom_geodesic_2} from \Cref{def:weak_axioms}.
\end{proof}

\begin{proof}[Proof of \Cref{thm:behavior-of-geodesics}.]
\Cref{thm:behavior-of-geodesics} follows from the second part of \Cref{prop:existence_geodesics}, from \Cref{lem:Hausdorff_dim_geodesic,lem:geodesic_avoid_boundary}, together with conformal invariance (Axiom~\eqref{it:weak_axiom_conformal_invariance} of \Cref{def:weak_axioms}), since the corollaries are stated for $\BD$.
\end{proof}

\begin{appendix}
\crefalias{section}{appendix}

\section{\texorpdfstring{Bichordal SLE$_4$ four-arm exponents}{Bichordal SLE4 four-arm exponents}}
\label{sec:cle_4_four_arm_exponents}

In this appendix, we prove \Cref{lem:bichordal-boundary-4A}. In order to prove \Cref{lem:bichordal-boundary-4A}, we will review some results from \cite{zhan2019two}.

Note that the hypergeometric function ${}_2F_1(1,0; 2;\bullet)$ appearing in \cite[Theorems~6.2 and~6.4]{zhan2019two} is actually identically equal to one in our case. For fixed points $v_-, w_-, w_+, v_+ \in \BR$ such that $v_- < w_- < w_+ < v_+$, we set 
\begin{equation}\label{eqn:form_of_G_1}
    G_1(w_-, w_+; v_-, v_+) \defeq \lvert w_+ - v_+\rvert \cdot \lvert w_- - v_-\rvert \cdot \lvert w_+ - v_-\rvert \cdot \lvert w_- - v_+\rvert. 
\end{equation}
Moreover, we set 
\begin{equation}\label{eqn:form_of_G_2}
    G_2(w_-, w_+; v_-, v_+) \defeq \lvert w_+ - w_-\rvert \cdot \lvert v_+ - v_-\rvert \cdot \lvert w_+ - v_-\rvert \cdot \lvert w_- - v_+\rvert. 
\end{equation}

We will use the following result from \cite{zhan2019two}.

\begin{proposition}[{\cite[Theorem~6.2]{zhan2019two}}]
\label{prop:arm_exponent_interior}
Let $v_-, w_-, w_+, v_+ \in \BR$ with $v_- < w_- < w_+ < v_+$ be such that $0 \in [v_-, v_+]$. Let $(\eta_+,\eta_-)$ be a bichordal SLE$_4$ in $\BH$ with link pattern $\{\{w_+, v_+\}, \{w_-, v_-\}\}$. Then, there exists a universal constant $C \in (0,\infty)$ such that
\begin{equation*}
    \BP[\text{both } \eta_+ \text{ and } \eta_- \text{ exit } B_L(0)] \leq C L^{-4} G_1(w_-, w_+; v_-, v_+) \left(1 + C((v_+ - v_-) / L)^{5/6}\right)
\end{equation*}
for all $L > 0$.
\end{proposition}

\Cref{prop:arm_exponent_interior} will be used to prove \Cref{lem:bichordal-boundary-4A} in the case that neither of the curves $\eta_{12}$ and $\eta_{34}$ appearing in the statement of \Cref{lem:bichordal-boundary-4A} separates $z_0$ from the other curve.

Next, we state the other main result from \cite{zhan2019two} that we are going to use.

\begin{proposition}[{\cite[Theorem~6.4]{zhan2019two}}]
\label{prop:arm_exponent_boundary}
Let $v_-, w_-, w_+, v_+ \in \BR$ with $v_- < w_- < w_+ < v_+$ be such that $0 \in [v_-, v_+]$. Let $(\eta_w,\eta_v)$ be a bichordal SLE$_4$ in $\BH$ with link pattern $\{\{w_+, w_-\}, \{v_+, v_-\}\}$. Then, there exists a universal constant $C \in (0,\infty)$ such that
\begin{align*}
\BP[\eta_w \text{ exits } B_L(0)] \leq C L^{-4} G_2(w_-, w_+; v_-, v_+) \left(1 + C((v_+ - v_-) / L)^{5/6}\right)
\end{align*}
for all $L > 0$.
\end{proposition}

\Cref{prop:arm_exponent_boundary} will be used to prove \Cref{lem:bichordal-boundary-4A} in the case that one of $\eta_{12}$ and $\eta_{34}$ separates $z_0$ from the other curve.

Now, we are ready to prove \Cref{lem:bichordal-boundary-4A}.

\begin{proof}[Proof of \Cref{lem:bichordal-boundary-4A}.]
\stepn{step:b4a-setup}{Setup} Let $(\eta_{12},\eta_{34})$ denote a bichordal SLE$_4$ on $\BD$ such that $\eta_{12}$ connects $z_1$ to $z_2$ and $\eta_{34}$ connects $z_3$ to $z_4$. Write $I_{x,y}$ for the counterclockwise arc of $\partial\BD$ from $x$ to $y$. In each case below, fix $w \in \partial\BD$ in a different connected component from $z_0$ (of $\partial\BD \setminus (I_{z_1,z_2} \cup I_{z_3,z_4})$ in \Cref{step:b4a-nosep}, of $I_{z_1,z_2} \cup I_{z_3,z_4}$ in \Cref{step:b4a-sep}), and let $\psi(z) = c(z-w)/(z-z_0)$, where $\vert c\vert = 1$ is chosen so that $\psi$ maps $\BD$ onto $\BH$ with $\psi(I_{w,z_0}) \subset \BR_+$; thus $\psi(w) = 0$ and $\psi(z_0) = \infty$. Since $w$ and $z_0$ are separated by the arcs, both boundary arcs from $z_0$ to $w$ contain some $z_i$, so $\vert z_0 - w\vert \ge \delta$, the chord length being monotone in the arc length. Then, we have the following two cases.

\stepn{step:b4a-nosep}{Neither $\eta_{12}$ nor $\eta_{34}$ separates $z_0$ from the other curve} In that case, we have that $z_0 \notin I_{z_1,z_2} \cup I_{z_3,z_4}$, and $w$ is chosen in a different connected component of $\partial\BD \setminus (I_{z_1,z_2} \cup I_{z_3,z_4})$ from $z_0$.

If $w \in I_{z_4, z_1}$ and $z_0 \in I_{z_2,z_3}$, we set $v_- = \psi(z_3), w_- = \psi(z_4), w_+ = \psi(z_1)$, and $v_+ = \psi(z_2)$. Then, we have that $(\psi(\eta_{34}), \psi(\eta_{12}))$ has the law of a bichordal SLE$_4$ on $\BH$ with link pattern $\{\{v_-, w_-\}, \{w_+, v_+\}\}$, and $0 \in [v_-, v_+]$. If $w \in I_{z_2,z_3}$ and $z_0 \in I_{z_4,z_1}$, we set $v_- = \psi(z_1), w_- = \psi(z_2), w_+ = \psi(z_3)$, and $v_+ = \psi(z_4)$. Then, we have that $(\psi(\eta_{12}), \psi(\eta_{34}))$ has the law of a bichordal SLE$_4$ on $\BH$ with link pattern $\{\{v_-, w_-\}, \{v_+, w_+\}\}$, and $0 \in [v_-, v_+]$.

Since the points $z_1, z_2, z_3, z_4$ are at distance at least $\delta$ from $z_0$, their images under $\psi$ are bounded. Hence, it is easy to see that there exists $M \in (0,\infty)$ depending only on $\delta$ such that $[v_-, v_+] \subset [-M,M]$. Moreover, note that $[w_-, w_+] \subset [v_-, v_+]$. Therefore,
  we obtain by the explicit form of $G_1(w_-, w_+; v_-, v_+)$ in~\eqref{eqn:form_of_G_1} that there exists $\widetilde{M} \in (0,\infty)$ depending only on $\delta$ such that $G_1(w_-, w_+; v_-, v_+) \leq \widetilde{M}$. Moreover, we have that
\begin{align*}
\psi(\BD \cap B_\varepsilon(z_0)) \subset \BH \setminus B_{\delta/(2\varepsilon)}(0), \quad \forall\varepsilon \in (0, \delta / 2).
\end{align*}
Indeed, $\vert \psi(z)\vert = \vert z - w\vert / \vert z - z_0\vert \ge (\delta - \varepsilon)/\varepsilon \ge \delta/(2\varepsilon)$ for $z \in \BD \cap B_\varepsilon(z_0)$ and $\varepsilon \le \delta/2$, by the bound $\vert z_0 - w \vert \ge \delta$ above.
Thus, the claim of the lemma follows by combining the above with \Cref{prop:arm_exponent_interior}.

\stepn{step:b4a-sep}{One of $\eta_{12}$ and $\eta_{34}$ separates $z_0$ from the other curve} In that case, we have that $z_0 \in I_{z_1,z_2} \cup I_{z_3,z_4}$, and here $w$ is chosen in a different connected component of $I_{z_1,z_2} \cup I_{z_3,z_4}$ from $z_0$ (note that this union, rather than its complement, is the relevant set in this case).

\substepn{step:b4as-first}{The case $z_0 \in I_{z_1,z_2}$} If $z_0 \in I_{z_1,z_2}$ and $w \in I_{z_3,z_4}$, we set $v_- = \psi(z_2), w_- = \psi(z_3), w_+ = \psi(z_4)$, and $v_+ = \psi(z_1)$. Then, we have that $(\psi(\eta_{34}), \psi(\eta_{12}))$ has the law of a bichordal SLE$_4$ on $\BH$ with link pattern given by $\{\{w_-,w_+\}, \{v_-, v_+\}\}$. Again, as in \Cref{step:b4a-nosep}, we have that there exists $M \in (0,\infty)$ depending only on $\delta$ such that $[v_-, v_+] \subset [-M,M]$. Thus, combining this with the explicit form of $G_2(w_-, w_+; v_-, v_+)$ in~\eqref{eqn:form_of_G_2}, we obtain that there exists a constant $\widetilde{M} \in (0,\infty)$ depending only on $\delta$ such that $G_2(w_-, w_+; v_-, v_+) \leq \widetilde{M}$. Moreover, we have again that
\begin{align*}
\psi(\BD \cap B_\varepsilon(z_0)) \subset \BH \setminus B_{\delta/(2\varepsilon)}(0), \quad \forall\varepsilon \in (0, \delta / 2).
\end{align*}
Indeed, $\vert \psi(z)\vert = \vert z - w\vert / \vert z - z_0\vert \ge (\delta - \varepsilon)/\varepsilon \ge \delta/(2\varepsilon)$ for $z \in \BD \cap B_\varepsilon(z_0)$ and $\varepsilon \le \delta/2$, by the bound $\vert z_0 - w \vert \ge \delta$ above.
This implies that if both $\eta_{12}$ and $\eta_{34}$ intersect $\BD \cap B_\varepsilon(z_0)$, we have that $\psi(\eta_{34})$ intersects $\BH \setminus B_{\delta/(2\varepsilon)}(0)$. Therefore, the claim of the lemma follows by combining the above with \Cref{prop:arm_exponent_boundary}.

\substepn{step:b4as-second}{The case $z_0 \in I_{z_3,z_4}$} If $w \in I_{z_1,z_2}$ and $z_0 \in I_{z_3,z_4}$, then we set $v_- = \psi(z_4), w_- = \psi(z_1), w_+ = \psi(z_2)$, and $v_+ = \psi(z_3)$. Then, we have that $(\psi(\eta_{12}), \psi(\eta_{34}))$ has the law of a bichordal SLE$_4$ on $\BH$ with link pattern given by $\{\{w_-, w_+\}, \{v_-, v_+\}\}$.

Therefore, we conclude the proof of the lemma by using \Cref{prop:arm_exponent_boundary} and arguing as in the previous paragraph. Finally, we note that by possibly increasing the constant $C(\delta)$, the bound $C \varepsilon^4$ trivially holds for all $\varepsilon \ge \delta/2$ since the probability is at most 1.
\end{proof}

\begin{remark}
We note that \Cref{lem:bichordal-bulk-4A,lem:bichordal-boundary-4A} can also be derived without referring to the results in \cite{zhan2019two} by combining the SLE$_4$ four-arm estimates (see \cite[Theorem~1.3]{U2C4ASimCLE}) with arguments similar to those used in \cite{kkmt2026cle4_part3}.
More precisely, we would obtain the following version of \Cref{lem:bichordal-boundary-4A} (resp.\ \Cref{lem:bichordal-bulk-4A}). Suppose that we have the same setup as in the statement of \Cref{lem:bichordal-boundary-4A} (resp.\ \Cref{lem:bichordal-bulk-4A}). Then, for all $\nu \in (0,1)$, there exists $\varepsilon_0 \in (0,1)$ depending only on $\delta,\nu$ such that the following holds for all $\varepsilon \in (0,\varepsilon_0)$. The probability that both $\eta_{12}$ and $\eta_{34}$ intersect $B_{\varepsilon}(z_0)$ (resp.\ $B_{\varepsilon}(0)$) is at most $\varepsilon^{4(1-\nu)}$ (resp.\ $\varepsilon^{2(1-\nu)}$). We chose not to include the aforementioned proof since the current argument keeps the proofs of \Cref{lem:bichordal-bulk-4A,lem:bichordal-boundary-4A} shorter.
\end{remark}

\end{appendix}

\bibliographystyle{alpha}
\bibliography{references}

\end{document}